\documentclass[11pt]{article}
\usepackage{fullpage}
\usepackage{amsmath,amssymb,amsthm}
\usepackage[english]{babel}
\usepackage[left=3cm,right=3cm,top=3cm,bottom=4cm]{geometry}
\usepackage{caption}

\usepackage{tikz,tikz-cd}
\usetikzlibrary{snakes}
\usetikzlibrary{decorations.pathmorphing,shapes}
\usetikzlibrary{positioning}
\usepackage{braket}
\usepackage{mathtools,stmaryrd}

\SetSymbolFont{stmry}{bold}{U}{stmry}{m}{n}
\usepackage{lmodern,bm,bbm,mathrsfs}
\usepackage{manfnt}

\usepackage{enumitem}
\newcounter{dummy}
\makeatletter
\newcommand{\customitem}[1][]{\item[#1]\refstepcounter{dummy}\def\@currentlabel{#1}}
\makeatother

\usepackage[scr=esstix]{mathalfa}

\usepackage[backref=page, colorlinks=true, linkcolor=blue, citecolor=red, menucolor=green]{hyperref}
\newtheorem*{theorem}{Theorem}
\newtheorem*{corollary}{Corollary}
\newtheorem*{lemma}{Lemma}
\newtheorem*{proposition}{Proposition}

\theoremstyle{definition}
\newtheorem*{definition}{Definition}
\newtheorem*{remark}{Remark}

\newtheorem*{example}{Example}

\newtheorem*{assumption}{Assumption}

\renewcommand{\bar}[1]{\overline{#1}}
\newcommand{\cat}[1]{\mathscr{#1}}
\renewcommand{\hat}[1]{\widehat{#1}}

\newcommand{\labeleq}[1]{\stackrel{\mathclap{\normalfont\mbox{#1}}}{=}}
\renewcommand{\tilde}[1]{\widetilde{#1}}
\renewcommand{\vec}[1]{\bm{{#1}}}
\newcommand{\bA}{\mathbb{A}}
\newcommand{\bC}{\mathbb{C}}
\newcommand{\bE}{\mathbb{E}}
\newcommand{\bF}{\mathbb{F}}
\newcommand{\bh}{\mathbbm{h}}
\newcommand{\bk}{\mathbbm{k}}
\newcommand{\bK}{\mathbb{K}}
\newcommand{\bfK}{\mathbf{K}}
\newcommand{\bL}{\mathbb{L}}
\newcommand{\bM}{\mathbb{M}}
\newcommand{\bN}{\mathbb{N}}
\newcommand{\bP}{\mathbb{P}}
\newcommand{\bQ}{\mathbb{Q}}
\newcommand{\bR}{\mathbb{R}}
\newcommand{\bT}{\mathbb{T}}
\newcommand{\bZ}{\mathbb{Z}}
\newcommand{\cA}{\mathcal{A}}
\newcommand{\cB}{\mathcal{B}}

\newcommand{\cE}{\mathcal{E}}
\newcommand{\cF}{\mathcal{F}}
\newcommand{\cG}{\mathcal{G}}
\newcommand{\cH}{\mathcal{H}}
\newcommand{\PcH}{\,{}^{\cat{P}}\!\mathcal{H}}
\newcommand{\cI}{\mathcal{I}}
\newcommand{\cK}{\mathcal{K}}
\newcommand{\cL}{\mathcal{L}}
\newcommand{\cM}{\mathcal{M}}
\newcommand{\cN}{\mathcal{N}}
\newcommand{\cO}{\mathcal{O}}
\newcommand{\cOb}{\mathcal{O}b}
\newcommand{\cQ}{\mathcal{Q}}
\newcommand{\cT}{\mathcal{T}}
\newcommand{\cV}{\mathcal{V}}

\newcommand{\fB}{\mathfrak{B}}
\newcommand{\fc}{\mathfrak{c}}
\newcommand{\fC}{\mathfrak{C}}
\newcommand{\ff}{\mathfrak{f}}
\newcommand{\fF}{\mathfrak{F}}
\newcommand{\fK}{\mathfrak{K}}

\newcommand{\fM}{\mathfrak{M}}
\newcommand{\fMum}{\mathfrak{Mum}}
\newcommand{\fn}{\mathfrak{n}}
\newcommand{\fN}{\mathfrak{N}}
\newcommand{\fQ}{\mathfrak{Q}}
\newcommand{\fS}{\mathfrak{S}}
\newcommand{\fW}{\mathfrak{W}}
\newcommand{\fX}{\mathfrak{X}}
\newcommand{\fY}{\mathfrak{Y}}
\newcommand{\fZ}{\mathfrak{Z}}
\newcommand{\scE}{\mathscr{E}}
\newcommand{\scF}{\mathscr{F}}
\newcommand{\scI}{\mathscr{I}}
\newcommand{\sA}{\mathsf{A}}
\newcommand{\sB}{\mathsf{B}}
\newcommand{\se}{\mathsf{e}}
\newcommand{\spt}{\mathsf{p}}
\newcommand{\sG}{\mathsf{G}}

\newcommand{\sP}{\mathsf{P}}
\newcommand{\sS}{\mathsf{S}}
\newcommand{\sT}{\mathsf{T}}
\newcommand{\sU}{\mathsf{U}}
\newcommand{\sV}{\mathsf{V}}
\newcommand{\sVW}{\mathsf{VW}}
\newcommand{\svw}{\mathsf{vw}}
\newcommand{\sY}{\mathsf{Y}}
\newcommand{\sZ}{\mathsf{Z}}
\newcommand{\sz}{\mathsf{z}}
\newcommand{\BS}{\mathrm{BS}}
\newcommand{\DT}{\mathrm{DT}}
\newcommand{\PT}{\mathrm{PT}}
\newcommand{\vw}{\mathrm{vw}}
\newcommand{\inert}{\mathrm{inert}}
\newcommand{\loc}{\mathrm{loc}}
\newcommand{\pe}{\mathrm{pe}}
\newcommand{\pl}{\mathrm{pl}}
\newcommand{\pt}{\mathrm{pt}}
\newcommand{\red}{\mathrm{red}}
\newcommand{\sst}{\mathrm{sst}}
\newcommand{\st}{\mathrm{st}}
\newcommand{\vac}{\mathbf{1}}
\newcommand{\vir}{\mathrm{vir}}
\newcommand{\Pn}{\,{}^{\cat{P}}\!n}
\newcommand\blackbullet{\raisebox{-.1ex}{\scalebox{1.5}{$\bullet$}}}
\DeclareMathOperator{\ad}{ad}
\DeclareMathOperator{\Aut}{Aut}
\DeclareMathOperator{\End}{End}
\DeclareMathOperator{\Ext}{Ext}

\DeclareMathOperator{\ch}{ch}
\DeclareMathOperator{\CH}{CH}
\DeclareMathOperator{\Char}{char}
\DeclareMathOperator{\cocone}{cocone}
\DeclareMathOperator{\Pcoker}{{}^{\cat{P}}\!coker}
\DeclareMathOperator{\coker}{coker}
\DeclareMathOperator{\cone}{cone}
\DeclareMathOperator{\ev}{ev}
\DeclareMathOperator{\fix}{\mathsf{fix}}
\DeclareMathOperator{\Fr}{Fr}
\DeclareMathOperator{\fr}{fr}
\DeclareMathOperator{\cFr}{\mathcal{F}{\it r}}
\DeclareMathOperator{\Frac}{Frac}
\DeclareMathOperator{\Frs}{Frs}
\DeclareMathOperator{\GL}{GL}
\DeclareMathOperator{\Hilb}{Hilb}
\DeclareMathOperator{\fHilb}{\mathfrak{Hilb}}
\DeclareMathOperator{\Hom}{Hom}
\DeclareMathOperator{\cHom}{\mathcal{H}{\it om}}
\DeclareMathOperator{\id}{id}
\DeclareMathOperator{\im}{im}
\DeclareMathOperator{\inj}{inj}

\DeclareMathOperator{\NS}{NS}
\DeclareMathOperator{\Per}{\cat{Per}}
\DeclareMathOperator{\PGL}{PGL}
\DeclareMathOperator{\Pic}{Pic}
\DeclareMathOperator{\fPic}{\mathfrak{Pic}}
\DeclareMathOperator{\pr}{pr}
\DeclareMathOperator{\rank}{rank}
\DeclareMathOperator{\Res}{Res}

\DeclareMathOperator{\SO}{SO}
\DeclareMathOperator{\Spec}{Spec}
\DeclareMathOperator{\SU}{SU}
\DeclareMathOperator{\supp}{supp}
\DeclareMathOperator{\Sym}{Sym}
\DeclareMathOperator{\td}{td}
\DeclareMathOperator{\tot}{tot}
\DeclareMathOperator{\tr}{tr}
\DeclareMathOperator{\vdim}{vdim}

\DeclarePairedDelimiter{\inner}{\langle}{\rangle}
\newcommand{\abs}[1]{\left| #1 \right|}
\DeclarePairedDelimiterX{\lseries}[1]{(}{)}{\mkern-2mu\delimsize(#1\delimsize)\mkern-2mu}
\DeclarePairedDelimiterX{\pseries}[1]{[}{]}{\mkern-2mu\delimsize[#1\delimsize]\mkern-2mu}

\newcommand\bigbullet{\raisebox{-.1ex}{\scalebox{1.5}{$\bullet$}}}

\tikzset{%
  vertex/.style={shape=circle,fill=black,minimum size=6pt,inner sep=0},
  framing/.style={shape=rectangle,fill=black,minimum size=6pt,inner sep=0},
  baseline={([yshift=-0.8ex]current bounding box.center)}
}

\title{Wall-crossing for invariants of equivariant 3CY categories}
\author{Nikolas Kuhn, Henry Liu, Felix Thimm}
\date{\today}

\begin{document}

\maketitle

\begin{abstract}
   We provide a wall-crossing framework for operational enumerative
   invariants of equivariant 3-Calabi--Yau categories arising from
   virtual cycles. The strategy follows ideas of Joyce's ``universal''
   wall-crossing framework \cite{Joyce2021}, using the authors'
   symmetrized pullback technique to preserve the symmetry of the
   (almost-perfect) obstruction theories throughout. As an
   application, we define and study wall-crossings of simple type
   between operational equivariant Donaldson--Thomas (DT),
   Pandharipande--Thomas (PT), and Bryan--Steinberg (BS) vertices. In
   particular, we give an explicit DT/PT descendent vertex
   correspondence in the Calabi--Yau limit. As another application, we
   construct and prove wall-crossing formulas for operational refined
   semistable Vafa--Witten invariants.
\end{abstract}

\tableofcontents

\section{Introduction}

\subsection{General motivation}

\subsubsection{} 

Enumerative geometry studies (virtual) invariants of moduli spaces of
objects of interest: curves, quiver representations, and coherent
sheaves, among others. For these invariants to be finite and
well-defined, we restrict to stable objects of a given topological
type $\alpha$, e.g. Gieseker-stable sheaves whose Chern character is
$\alpha$. Since stability conditions often come in families, a natural
question is: how do the invariants change as the stability condition
varies?

The goal of this paper is to provide a formula describing exactly this
change. Such formulas are called {\it wall-crossing formulas} due to
the prototypical setting where there is a geometric space of stability
conditions with a wall-and-chamber decomposition such that enumerative
invariants are locally constant within each chamber. Walls typically
consist of those stability conditions for which {\it strictly
  semistable} objects, i.e. semistable but not stable objects, exist.

\subsubsection{}

Historically, the study of {\it variation of GIT} \cite{Dolgachev1998,
  Thaddeus1996} --- the dependence of GIT quotients $X \sslash_\cL G$
on the choice of polarization $\cL$ --- formed one source of interest
in wall-crossing phenomena. The basic ``master space'' or ``symplectic
cut'' technique \cite{Kalkman1995, Lerman1995} first used to obtain
wall-crossing formulas in this setting\footnote{Independently of
\cite{Kalkman1995, Lerman1995}, Jeffrey and Kirwan \cite{Jeffrey1995}
established a different technique which produces the same
wall-crossing formula. This technique is now known as {\it
  Jeffrey--Kirwan integration} and consequently may be viewed as a
special form of wall-crossing.} forms the crucial foundation for the
more sophisticated machinery in this paper.

Interest in {\it BPS state counting} from physics, particularly in
those 4d supersymmetric gauge theories arising from compactifying
string theory on a Calabi--Yau $3$-fold --- see \cite{Denef2000}, for
example --- led to Joyce--Song and Kontsevich--Soibelman's results
\cite{Joyce2012, Kontsevich2010} on {\it generalized Donaldson--Thomas
  (DT) invariants} and their wall-crossing formulas. In contrast to
integrals over GIT quotients, these invariants are ``motivic'' and
count objects in categories which are suitably ``$3$-Calabi--Yau
(3CY)''. Their wall-crossing formula is written in terms of the
product in a motivic {\it Hall algebra}, and is the precursor to the
wall-crossing formula in this paper.

\subsubsection{}

In this paper, we study enumerative invariants of {\it equivariant 3CY
  categories} which arise from {\it virtual cycles}, instead of
motivic invariants, and provide a geometric framework which works
equally well in equivariant cohomology and equivariant K-theory.
Applied to the category of coherent sheaves on a smooth projective
Calabi--Yau $3$-fold, in non-equivariant cohomology, we immediately
recover the Joyce--Song wall-crossing formula
\cite{Joyce2012}.\footnote{When the stable(=semistable) locus is a
proper scheme with a symmetric perfect obstruction theory, its motivic
and virtual cycle invariants coincide \cite[Theorem
  4.18]{Behrend2009}.}

As the (equivariant) K-theoretic setting requires more technical care,
we focus on it for the majority of this paper. Recent work of Joyce
\cite{Joyce2021} provides the desired geometric framework in ordinary
cohomology and for categories which are ``$\le 2$-dimensional''. Tools
for its extension to equivariant K-theory, in particular a definition
of an appropriate {\it equivariant K-homology} functor, were developed
by the second author in \cite{liu_eq_k_theoretic_va_and_wc}. Tools for
its extension to equivariant 3CY categories, in particular the {\it
  symmetrized pullback} operation on symmetric obstruction theories,
were developed by the authors in \cite{klt_dtpt}. Thus, most of the
new technical content in this paper lies in the synthesis of these
tools with Joyce's machinery in full generality.

\subsubsection{}

A hallmark of DT-type theories is that their moduli stacks are
naturally equipped with {\it symmetric} or {\it self-dual} obstruction
theories \cite{Behrend2008}. (More strongly, perhaps the moduli stacks
are even $(-1)$-shifted symplectic \cite{Pantev2013}, though we will
not use this formalism.)

Although the setup in this paper encapsulates much more than the world
of DT-type theories, it remains crucial throughout to work with
obstruction theories which are symmetric up to an overall equivariant
weight $\kappa \neq 1$, for the simple reason that they are
automatically {\it perfect} on stable loci (see
Lemma~\ref{lem:perfect-symmetric-obstruction-theory}) and therefore
may be used to construct virtual cycles; otherwise, there is no
mechanism in general to rule out the presence of higher obstructions.
Thus, it is important to consistently symmetrize or preserve the
symmetry of objects related to the obstruction theory and virtual
cycle.

In particular, one such example of symmetrization is the
aforementioned symmetrized pullback operation on symmetric obstruction
theories. It forms the technical backbone of this paper and requires
us to use the more general notion of {\it almost-perfect obstruction
  theory (APOT)} of Kiem and Savvas \cite{kiem_savvas_apot}. Master
space wall-crossing in the 3CY setting therefore constitutes a highly
non-trivial and interesting application of their techniques.

\subsubsection{}
\label{sec:intro-lie-bracket}

A surprising observation by Joyce in \cite{Joyce2021} is that, for
moduli stacks of {\it linear} objects with obstruction theories that
are bilinear functions of the universal family, the complicated term
in wall-crossing formulas obtained by master space techniques may be
naturally interpreted as a {\it vertex algebra} structure on the
(equivariant) (K-)homology of the moduli stack, and the wall-crossing
formulas are naturally written in terms of the induced Lie bracket
$[-, -]$ \cite[\S 4]{Borcherds1986}. This geometrically-constructed
Lie bracket $[-, -]$ is, morally, the replacement for the motivic Hall
algebra in the earlier Joyce--Song machinery, and it is the key player
in the wall-crossing formulas in both \cite{Joyce2021} and this paper.

There exists a significant amount of work attempting to control
enumerative quantities via certain actions of quantum groups on the
cohomology or K-theory of moduli spaces; see \cite{Okounkov2018} for
an introduction. We hope that the vertex algebras in this paper may be
used similarly, to obtain constraints of a different flavor than what
is offered by quantum groups. We expect the two structures to be
compatible whenever they are comparable \cite{Liu2025}. Equivariant
and multiplicative refinements of vertex algebras, inspired by the
natural extension of \cite{Joyce2021} to equivariant K-theory, were
defined and explored in \cite{liu_eq_k_theoretic_va_and_wc}.

\subsubsection{}

Finally, we emphasize that the wall-crossing formula proved in this
paper is an equality of {\it operational K-homology classes}, not just
of numerical quantities. To be precise, typical K-theoretic
wall-crossing formulas arising from master space techniques have the
form
\begin{equation} \label{eq:master-space-numerical-wcf}
  \chi(M_+, i_+^*\cE) - \chi(M_-, i_-^*\cE) = \chi(\bM_0, \cE')
\end{equation}
where $i_\pm\colon M_\pm \subset \fM$ are two different
semistable=stable loci, $\cE$ is a perfect complex on $\fM$, and
$f\colon \bM_0 \to \fM$ is some auxiliary space constructed from $\fM$
with $\cE'$ constructed explicitly from $f^*\cE$. We view this instead
as an equality of {\it operators} such as $\chi(M_\pm, i_\pm^*(-))$.
Such operators have the opposite functorial properties as perfect
complexes, e.g. pullbacks of perfect complexes become pushforwards of
these operators. Our wall-crossing formula takes place in an {\it
  (torus-equivariant) operational K-homology theory}
$K_\circ^{\sT}(-)$ consisting of such operators. The operational
analogue of \eqref{eq:master-space-numerical-wcf}, in our setting, is
\[ (i_+)_* \chi(M_+, -) - (i_-)_* \chi(M_-, -) = [\phi, \psi], \]
where $[-, -]$ is the Lie bracket mentioned in
\S\ref{sec:intro-lie-bracket}, and $\phi$ and $\psi$ are some
operators such that $[\phi, \psi](\cE)$ is exactly the right hand side
of \eqref{eq:master-space-numerical-wcf}.

Operational wall-crossing formulas are stronger than numerical ones
like \eqref{eq:master-space-numerical-wcf}, but weaker than possible
wall-crossing formulas for classes in a ``genuine'' K-homology theory
which may contain classes like $\cO_{M_\pm}$ or virtual structure
sheaves $\cO_{M_\pm}^{\vir}$. For standard applications to enumerative
geometry, we believe operational wall-crossing formulas suffice.

\subsection{The geometric setting}

\subsubsection{}
\label{sec:abelian-category-moduli-stack}

Let $\cat{A}$ be a Noetherian abelian category which is $\bC$-linear,
meaning that its morphism sets are $\bC$-vector spaces. Assume that
$\cat{A}$ admits a moduli stack
\[ \fM = \bigsqcup_{\alpha \in K(\cat{A})} \fM_\alpha \]
which is Artin and locally of finite type. Here $K(\cat{A})$ is an
appropriate quotient of the Grothendieck group $K_0(\cat{A})$, and
$\fM_\alpha$ is the component parameterizing objects of class
$\alpha$. Direct sum and composition with scaling automorphisms induce
maps
\begin{align}
  \Phi_{\alpha,\beta} &\colon \fM_\alpha \times \fM_\beta \to \fM_{\alpha+\beta} \nonumber \\
  \Psi_\alpha &\colon [\pt/\bC^\times] \times \fM_\alpha \to \fM_\alpha \label{eq:psi}
\end{align}
respectively, which make $\fM$ into a graded monoidal stack
(Definition~\ref{def:graded-monoidal-stack}) --- roughly, a monoid
object in the category of Artin stacks, with
$[\pt/\bC^\times]$-action.

\subsubsection{}
\label{sec:equivariant-3CY-category}

Let $\sT = (\bC^\times)^r$ for $r > 0$. Throughout this paper, there
will be a distinguished $\sT$-weight
\[ 1 \neq \kappa \in \Hom(\sT, \bC^\times) \]
such that $\kappa^{1/2}$ does not exist, and we let $\tilde\sT \to
\sT$ be a double cover where $\kappa^{1/2}$ does exist. We say that
$\cat{A}$ is a {\it $\sT$-equivariant 3CY category} if $\fM$ admits a
$\sT$-action \cite{Romagny2005}, commuting with $\Psi$ (so that the
$\sT$-action descends to $\fM_\alpha^\pl$ and makes $\Pi^\pl_\alpha$
equivariant), and there are $\sT$-equivariant perfect complexes
\[ \scE_{\alpha,\beta} \text{ on } \fM_\alpha \times \fM_\beta, \]
for every $\alpha$ and $\beta$, satisfying the following properties.
\begin{itemize}
\item (Bilinear) Let $\cL$ be the weight-$1$ line bundle on
  $[\pt/\bC^\times]$. Then there are isomorphisms
  \begin{gather}\label{intro:eq:bilinear-element}
  \begin{aligned}
    (\Phi_{\alpha,\beta} \times \id)^*(\scE_{\alpha+\beta,\gamma}) &\cong \scE_{\alpha,\gamma} \oplus \scE_{\beta,\gamma}, \qquad &(\Psi_\alpha \times \id)^*(\scE_{\alpha,\beta}) &\cong \cL^\vee \boxtimes \scE_{\alpha,\beta}, \\
    (\id \times \Phi_{\beta,\gamma})^*(\scE_{\alpha,\beta+\gamma}) &\cong \scE_{\alpha,\beta} \oplus \scE_{\alpha,\gamma}, \qquad &(\id \times \Psi_\beta)^*(\scE_{\alpha,\beta}) &\cong \cL \boxtimes \scE_{\alpha,\beta}.
  \end{aligned}
\end{gather}

\item ($\kappa$-symmetric) Let $(12)\colon \fM_\beta \times \fM_\alpha
  \to \fM_\alpha \times \fM_\beta$ swap the two factors. Then there
  are isomorphisms
  \[ \Xi_{\alpha,\beta}\colon \scE_{\beta,\alpha} \xrightarrow{\sim} \kappa^{-1} \otimes (12)^*\scE_{\alpha,\beta}^\vee[-3] \]
  compatible with $\Phi$ and $\Psi$, e.g. there are commutative diagrams
  \[ \begin{tikzcd}[column sep=huge]
      (\Phi_{\alpha,\beta} \times \id)^*(\scE_{\alpha+\beta,\gamma}) \ar{r}{(\Phi_{\alpha,\beta} \times \id)^*\Xi_{\gamma,\alpha+\beta}} \ar{d}{\cong} & (\Phi_{\alpha,\beta} \times \id)^*(\kappa^{-1} \otimes (12)^*\scE_{\gamma,\alpha+\beta}^\vee[-3]) \ar{d}{\cong}\\
      \scE_{\alpha,\gamma}\boxplus \scE_{\beta,\gamma} \ar{r}{\Xi_{\gamma,\alpha}\boxplus\Xi_{\gamma,\beta}} & (\kappa^{-1} \otimes (12)^*\scE_{\gamma,\alpha}^\vee[-3])\boxplus (\kappa^{-1} \otimes (12)^*\scE_{\gamma,\beta}^\vee[-3])
  \end{tikzcd} \]
  for every $\alpha, \beta$, where the vertical arrows are given by
  the first isomorphism in \eqref{intro:eq:bilinear-element}. The
  analogous diagrams corresponding to the other isomorphisms in
  \eqref{intro:eq:bilinear-element} must also commute.

\item (Obstruction theory) Let $\Delta\colon \fM \hookrightarrow \fM
  \times \fM$ be inclusion of the diagonal. Then there are
  $\sT$-equivariant obstruction theories
  (Definition~\ref{def:obstruction-theory})
  \begin{equation} \label{eq:obstruction-theory}
    \varphi_\alpha\colon \Delta^*\scE_{\alpha,\alpha}^\vee[-1] \to \bL_{\fM_\alpha}
  \end{equation}
  compatible with $\Phi$, i.e. there are commutative diagrams
  \begin{equation*}
    \begin{tikzcd}[column sep=huge]
      \Phi_{\alpha,\beta}^*\Delta^*\scE_{\alpha+\beta,\alpha+\beta}^\vee[-1] \ar{r}{\Phi_{\alpha,\beta}^*\varphi_{\alpha+\beta}} \ar{d} & \Phi_{\alpha,\beta}^*\bL_{\fM_{\alpha+\beta}} \ar{d}\\
      \Delta^*\scE_{\alpha,\alpha}^\vee[-1]\boxplus \Delta^*\scE_{\beta,\beta}^\vee[-1] \ar{r}{\varphi_\alpha\boxplus\varphi_\gamma} & \bL_{\fM_\alpha}\boxplus\bL_{\fM_\beta}
    \end{tikzcd}
  \end{equation*}
  for every $\alpha, \beta$. Here, the upper-left corner of the
  diagram is identified with
  $(\Delta^*\scE_{\alpha,\alpha}^\vee[-1]\boxplus
  \Delta^*\scE_{\beta,\beta}^\vee[-1])\oplus
  \scE_{\alpha,\beta}^\vee[-1] \oplus
  (12)^*\scE_{\beta,\alpha}^\vee[-1]$ using
  \eqref{intro:eq:bilinear-element}, so that the left vertical arrow
  is the projection, and the right vertical arrow is the natural
  morphism of cotangent complexes.
\end{itemize}
In particular, $\fM$ becomes a graded monoidal $\sT$-stack with
symmetric bilinear element
(Definition~\ref{def:graded-monoidal-stack}). We let
\[ \chi(\alpha, \beta) \coloneqq \rank \scE_{\alpha,\beta}. \]

Note that this is completely different from the usual notion of
``$d$-Calabi--Yau category'' in the literature.

\subsubsection{}

\begin{example}
  Here is the prototypical example of this setup. Let $X$ be a smooth
  quasi-projective $3$-fold with $\sT$-action, which is {\it
    equivariantly Calabi--Yau} in the sense that its canonical bundle
  is
  \[ \cK_X \cong \kappa \otimes \cO_X, \]
  as $\sT$-equivariant sheaves, for some $\sT$-weight $\kappa \neq 1$.
  Let $\cat{A} = \cat{Coh}(X)$ be the category of compactly-supported
  coherent sheaves on $X$, and define $K(\cat{A}) \subset H^*(X; \bQ)$
  to be the image of the Chern character. The moduli stack $\fM$ of
  $\cat{A}$ is an Artin stack locally of finite type \cite[Theorem
    2.1.1]{Lieblich2006a}, and its monoidal structures
  $\Phi_{\alpha,\beta}$ and $\Psi_\alpha$ are given by direct sum and
  (composition with) scaling automorphisms of sheaves. Let
  $\scE_{\alpha,\beta} \coloneqq R\pi_*R\cHom(\scF_\alpha,
  \scF_\beta)$ where $\scF_\alpha, \scF_\beta$ are the (pullbacks of
  the) universal families on $\pi\colon \fM_\alpha \times \fM_\beta
  \times X$ for $\fM_\alpha$ and $\fM_\beta$ respectively. This is
  obviously bilinear, Serre duality makes it $\kappa$-symmetric, and
  it forms a natural $\sT$-equivariant obstruction theory on
  $\fM_\alpha$ via the Atiyah class \cite{Ricolfi2021,Kuhn2024}.
\end{example}

\subsubsection{}

For $\alpha \neq 0$, let $\fM_\alpha^\pl$ denote the {\it
  $\bC^\times$-rigidification} \cite{Abramovich2008} of $\fM_\alpha$
with respect to the group $\bC^\times$ of scaling automorphisms.
Roughly, this means to quotient by the $[\pt/\bC^\times]$-action given
by $\Psi$, so that there is a natural $\bC^\times$-gerbe
\[ \Pi^\pl_\alpha\colon \fM_\alpha \to \fM_\alpha^\pl. \]
We make the assumption that
\begin{itemize}
\item (rigidified obstruction theory) $\fM_\alpha^\pl$ has a
  $\sT$-equivariant $\kappa$-symmetric obstruction theory
  \begin{equation} \label{eq:obstruction-theory-pl}
    \varphi_\alpha^\pl\colon \bE_\alpha^\pl \to \bL_{\fM_\alpha^\pl}
  \end{equation}
  which is $\kappa$-symmetrically compatible
  (Definition~\ref{def:obstruction-theories-compatibility}) under
  $\Pi^\pl_\alpha$ with \eqref{eq:obstruction-theory}.
\end{itemize}
In practice, this can often be obtained from
\eqref{eq:obstruction-theory} using
Lemma~\ref{lem:obstruction-theory-pl}. \footnote{Conversely, if a
symmetric obstruction theory \eqref{eq:obstruction-theory-pl} is given
such that $\bE_\alpha^\pl = -\scE_{\alpha,\alpha}^\vee + \cO - \kappa
\cO^\vee$ in K-theory, then it is in fact unnecessary to check that
\eqref{eq:obstruction-theory} is an obstruction theory. This is
because \eqref{eq:obstruction-theory} is only ever used to construct
symmetric obstruction theories on auxiliary stacks
(Definition~\ref{def:auxiliary-stack}) via symmetrized pullback
(Theorem~\ref{thm:APOTs}\ref{it:APOT-sym-pullback}). So, whenever a
symmetrized pullback along a smooth morphism $f$ of $\varphi_\alpha$
is desired, one can simply take the symmetrized pullback along the
smooth morphism $\Pi^\pl \circ f$ of $\varphi_\alpha^\pl$.}

\subsubsection{}

A {\it weak stability condition} is a function $\tau$ from effective
classes $\alpha \neq 0$ into some totally-ordered set, such that if
$\alpha = \beta + \gamma$ is a decomposition into effective classes
$\beta,\gamma \neq 0$ then either
\begin{equation} \label{eq:stability-condition}
  \tau(\beta) \ge \tau(\alpha) \ge \tau(\gamma) \;\text{ or }\; \tau(\beta) \le \tau(\alpha) \le \tau(\gamma),
\end{equation}
which we refer to as the {\it weak see-saw property}. This is weaker
than the traditional notion of {\it stability condition}, which
satisfies the {\it see-saw property}: either $\tau(\beta) >
\tau(\alpha) > \tau(\gamma)$ or $\tau(\beta) = \tau(\alpha) =
\tau(\gamma)$ or $\tau(\beta) < \tau(\alpha) < \tau(\gamma)$.

Given $E \in \cat{A}$ of class $\alpha$, the notation $\tau(E)$ means
$\tau(\alpha)$, and we refer to it as the {\it ($\tau$-)slope} of $E$.
Let $C(\cat A) \subset K(\cat A)$ be the cone of non-zero effective
classes. It is convenient to write $\alpha > \beta$ if $\alpha - \beta
\in C(\cat A)$.

An object $E \in \cat{A}$ is {\it $\tau$-stable} (resp. {\it
  $\tau$-semistable}) if $\tau(E') < \tau(E/E')$ (resp. $\tau(E') \le
\tau(E/E')$) for all sub-objects $0 \neq E' \subsetneq E$ and is {\it
  strictly $\tau$-semistable} if it is $\tau$-semistable but not
$\tau$-stable. Let
\[ \fM_\alpha^\st(\tau) \subset \fM_\alpha^{\sst}(\tau) \subset \fM_\alpha^\pl \]
be the moduli substacks parameterizing $\tau$-stable and
$\tau$-semistable objects, respectively. Since stable objects in
$\cat{A}$ only have scaling automorphisms, and those are removed by
rigidification, $\fM_\alpha^\st(\tau)$ is an algebraic space.

\subsection{Summary of main results}

\subsubsection{}

The overall strategy can be summarized as follows. It is difficult to
directly define enumerative invariants of $\fM^{\sst}_\alpha(\tau)$
since strictly $\tau$-semistable objects of class $\alpha$ may exist,
and, furthermore, master space techniques for wall-crossing work best
only at {\it simple} walls, i.e. semistable objects decompose into at
most two pieces at the wall, and typically walls will not be simple.
Instead, we assume that appropriate {\it framing functors} $\Fr$ exist
--- roughly, exact functors from an open subcategory of $\cat{A}$ to
vector spaces --- using which we construct auxiliary categories
$\tilde{\cat{A}}^{Q(\Fr)}$ of ``objects framed by a quiver $Q$'',
whose moduli stack of objects of class $\alpha$ with framing dimension
vector $\vec d$ is denoted $\tilde\fM^{Q(\Fr)}_{\alpha,\vec d}$ (see
\S\ref{sec:auxiliary-stacks}). For appropriate $Q$:
\begin{itemize}
\item there is a weak stability condition $\tau^Q$ on
  $\tilde{\cat{A}}^{Q(\Fr)}$, compatible with $\tau$, with no strictly
  $\tau^Q$-semistable objects;
\item the forgetful map $\pi\colon \tilde\fM^{Q(\Fr)}_{\alpha,\vec d}
  \to \fM_\alpha$ is smooth, and the symmetric obstruction theory on
  $\fM$ admits a symmetrized pullback to symmetric APOTs on all
  semistable=stable loci on $\tilde\fM^{Q(\Fr)}_{\alpha,\vec d}$,
  yielding (symmetrized) virtual cycles;
\item a wall-crossing problem between $\tau$ and $\mathring \tau$ on
  $\cat{A}$ lifts to a wall-crossing problem between $\tau^Q$ and
  $\mathring \tau^Q$, which may (and will, for us) have more walls,
  but each wall will be simple.
\end{itemize}
Thus we may define enumerative invariants of
$\tilde\fM^{Q(\Fr),\sst}_{\alpha,\vec d}(\tau^Q)$ and study their
wall-crossing via master space techniques. Such an approach was
pioneered by Mochizuki \cite{Mochizuki2009} in his study of Donaldson
invariants, and later generalized in \cite{Joyce2021}. Finally,
enumerative invariants of the original semistable loci
$\fM^{\sst}_\alpha(\tau)$ may be defined formally, as a ``logarithm''
of the enumerative invariants on $\tilde\fM^{Q(\Fr),\sst}_{\alpha,\vec
  d}(\tau^Q)$. That these {\it semistable invariants} are well-defined
and independent of choices is the first main theorem of this paper
(Theorem~\ref{thm:sst-invariants}). Ultimately, the wall-crossing
formula (Theorem~\ref{thm:wcf}) is written in terms of the semistable
invariants.

\subsubsection{}

\begin{assumption} \label{assump:semistable-invariants}
  Let $\tau$ be a weak stability condition on $\cat{A}$. We make the
  following technical assumptions, in order to construct semistable
  invariants. 
  \begin{enumerate}[label = (\alph*)]
  \item \label{assump:it:tau-artinian} $\cat{A}$ admits {\it
    $\tau$-Harder--Narasimhan (HN) filtrations}: a non-zero object $A
    \in \cat{A}$ admits a finite chain of sub-objects $0 = A_0
    \subsetneq A_1 \subsetneq \cdots \subsetneq A_n = A$ whose factors
    $B_j \coloneqq A_j/A_{j-1}$ are $\tau$-semistable and $\tau(B_1) >
    \tau(B_2) > \cdots > \tau(B_n)$.\footnote{In \cite[Assumption
        5.2(a)]{Joyce2021}, Joyce obtains $\tau$-HN filtrations by
      assuming $\cat{A}$ is {\it $\tau$-Artinian}, i.e. that there is
      no infinitely strictly descending chain of objects $\cdots
      \subsetneq A_3 \subsetneq A_2 \subsetneq A_1 = A$ in $\cat{A}$
      with $\tau(A_{n+1}) \ge \tau(A_n/A_{n+1})$ for all $n \ge 1$.
      This is a strictly stronger assumption; the existence of
      $\tau$-HN filtrations is all that is genuinely used in
      wall-crossing.}

  \item \label{assump:it:semistable-loci} $\tau$-(semi)stability is
    {\it open}: $\fM^{\st}_\alpha(\tau) \subset
    \fM^{\sst}_\alpha(\tau) \subset \fM_\alpha^\pl$ are open substacks
    of finite type for all $\alpha \in C(\cat{A})$.

  \item \label{assump:it:framing-functor} There exists a set $\Frs$ of
    framing functors (Definition~\ref{bg:def:framing-functor}) such
    that for any finite collection of classes $\{\alpha_i\}_{i \in I}
    \subset C(\cat{A})$, there exists $\Fr \in \Frs$ such that
    $\fM_{\alpha_i}^{\sst}(\tau) \subset \fM_{\alpha_i}^{\Fr,\pl}$ for
    all $i \in I$.
  \end{enumerate}
  For our weak stability conditions on auxiliary stacks (see
  Definition~\ref{def:flag-invariant}), we furthermore assume the
  following.
  \begin{enumerate}[resume, label = (\alph*)]
  \item \label{assump:it:rank-function} There exists a ``rank
    function'' $r\colon C(\cat{A}) \to \bZ$ such that
    \begin{itemize}
    \item if $A \in \cat{A}$ is $\tau$-semistable then $r(A) > 0$,
      and moreover
    \item if $A' \subsetneq A$ has $\tau(A') = \tau(A/A')$ then $r(A)
      = r(A') + r(A/A')$ and $r(A'), r(A/A') > 0$.
    \end{itemize}

  \item \label{assump:it:semi-weak-stability} Let $0 < \beta < \alpha$
    be the classes of $\tau$-semistable objects $0 \neq B \subsetneq
    A$ in $\cat{A}$ with $\tau(\beta) = \tau(\alpha-\beta)$. Then for
    the class $\beta'$ of any sub-object $0 \neq B' \subsetneq
    B$: \footnote{It is straightforward to prove that this condition
    is equivalent to: $\tau(\beta') = \tau(\alpha-\beta')$ if and only
    if $\tau(\beta') = \tau(\beta-\beta')$; see
    Lemma~\ref{wc:lem:R-sets}\ref{it:R-sets-ii}. We present them in
    their current form in order to match better with the more general
    condition in Assumption~\ref{es:assump:semistable-invariants}.}
    \begin{itemize}
    \item $\tau(\beta') = \tau(\alpha-\beta') < \tau(\beta-\beta')$
      does not occur;
    \item $\tau(\beta') < \tau(\alpha-(\beta-\beta')) =
      \tau(\beta-\beta')$ does not occur.
    \end{itemize}
  \end{enumerate}
  Finally, we need to be able to construct virtual cycles and their
  enumerative invariants.
  \begin{enumerate}[resume, label = (\alph*)]
  \item \label{assump:it:properness} The following algebraic spaces
    are proper and have the resolution property (see
    Remark~\ref{rem:localization-resolution-property}):
    \begin{itemize}
    \item $\fM_\alpha^{\sst}(\tau)^\sT$ for all $\alpha \in
      C(\cat{A})$ with no strictly $\tau$-semistable
      objects; \footnote{This assumption is only used to obtain
        Theorem~\ref{thm:sst-invariants}\ref{item:vss-no-strictly-semistables}
        and can be omitted otherwise.}
    \item $\tilde\fM_{\alpha,1}^{Q(\Fr),\sst}(\tau^Q)^\sT$ (the
      auxiliary stack in Definition~\ref{def:pair-invariant}) for all
      $\alpha \in C(\cat{A})$ and $\Fr \in \Frs$;
    \item $\bM_\alpha^{\sT_w}$ (the master space in
      Proposition~\ref{prop:pairs-master-space-fixed-loci}) for all
      $w$ in Lemma~\ref{lem:master-space-no-poles}, and all $\alpha
      \in C(\cat{A})$ and $\Fr_1, \Fr_2 \in \Frs$.
    \end{itemize}
  \end{enumerate}
\end{assumption}

\subsubsection{}

\begin{theorem}[Operational semistable invariants] \label{thm:sst-invariants}
  Suppose $\tau$ is a weak stability condition on $\cat{A}$ for which
  Assumption~\ref{assump:semistable-invariants} holds. Then there
  exists a unique collection
  \begin{equation}
    \left(\sz_{\alpha}(\tau)\in K_{\circ}^{\tilde{\sT}}(\fM_{\alpha})^\pl_{\loc,\bQ}\right)_{\alpha\in C(\cat{A})}
  \end{equation}
  of operational K-homology classes satisfying the following properties:
  \begin{enumerate}[label = (\roman*)]
  \item \label{item:vss-support} $\sz_\alpha(\tau)$ is supported on
    $(\Pi^\pl_\alpha)^{-1}(\fM_\alpha^{\sst}(\tau))$;
  \item \label{item:vss-no-strictly-semistables} for any $\alpha$ for
    which $\fM_\alpha^{\st}(\tau) = \fM_{\alpha}^{\sst}(\tau)$, 
    \[ (\Pi^\pl_\alpha)_*\sz_{\alpha}(\tau) = \chi\left(\fM^{\sst}_{\alpha}(\tau), \hat\cO^{\vir}_{\fM_{\alpha}^{\sst}(\tau)}\otimes -\right); \]
  \item \label{item:vss-isomorphic-moduli} if $\tau'$ is another
    weak stability condition on $\cat{A}$ for which
    Assumption~\ref{assump:semistable-invariants} holds, and
    $\fM_\alpha^{\sst}(\tau) = \fM_\alpha^{\sst}(\tau')$ for all
    $\alpha$, then $\sz_\alpha(\tau) = \sz_\alpha(\tau')$ for all
    $\alpha$;
  \item \label{item:vss-pairs-relation} for any framing functor $\Fr
    \in \Frs$ such that $\fM_\alpha^{\sst}(\tau) \subset
    \fM_\alpha^{\Fr,\pl}$, in the notation of
    Definition~\ref{def:pair-invariant} and \S\ref{sec:pairs-maps},
    \begin{equation} \label{eq:sstable-def-intro}
      I_*\tilde\sZ_{\alpha,1}^{\Fr}(\tau^Q) = \sum_{\substack{n>0 \\ \alpha = \alpha_1+\cdots+\alpha_n\\ \forall i: \,\tau(\alpha_i) = \tau(\alpha)\\ \;\;\fM_{\alpha_i}^{\sst}(\tau) \neq \emptyset}} \frac{1}{n!} \left[\iota^Q_*\sz_{\alpha_n}(\tau), \left[\cdots,\left[\iota^Q_*\sz_{\alpha_2}(\tau), \left[\iota^Q_*\sz_{\alpha_1}(\tau), \partial\right]\right]\cdots\right]\right]
    \end{equation}
    in
    $K_\circ^{\tilde\sT}(\tilde\fM^{Q(\Fr)}_{\alpha,1})^\pl_{\loc,\bQ}$,
    with Lie bracket $[-, -]$ defined by
    Theorem~\ref{thm:auxiliary-stack-vertex-algebra}.
  \end{enumerate}
\end{theorem}

To be clear, the Lie brackets on the right hand side of
\eqref{eq:sstable-def-intro} depend non-trivially on $\Fr$ via the
dimension function $\fr$.

The existence of the rank function
(Assumption~\ref{assump:semistable-invariants}\ref{assump:it:rank-function}),
along with \ref{item:vss-support}, ensures that the sum in
\eqref{eq:sstable-def-intro} is a finite sum.
  
\subsubsection{}

\begin{remark}
  Theorem~\ref{thm:sst-invariants} is analogous to \cite[Theorem
    5.7]{Joyce2021}, but property~\ref{item:vss-pairs-relation} is
  slightly different. Namely, for us, the formula
  \eqref{eq:sstable-def-intro} which implicitly defines the semistable
  invariants $\sz_\alpha(\tau)$ takes place on the auxiliary stack
  $\tilde\fM_{\alpha,1}^{Q(\Fr),\pl}$, not the original stack
  $\fM_\alpha^{\pl}$. This is not merely an aesthetic choice. Although
  there is a natural forgetful map between these stacks, pushforward
  along this map is not a Lie algebra homomorphism and therefore does
  {\it not} recover a Joyce-style formula from ours. The 3CY setting,
  working with symmetric obstruction theories, behaves fundamentally
  differently from Joyce's setting and genuinely requires our version
  of the formula; see Remarks~\ref{rem:joyce-lie-bracket-vs-ours-1}
  and \ref{rem:joyce-lie-bracket-vs-ours-2}.
\end{remark}

\subsubsection{}
\label{sec:universal-coefficients}

\begin{definition}[Universal coefficients] \label{def:universal-coefficients}
  Let $n \ge 1$ and $\alpha_1, \ldots, \alpha_n$ be effective classes.
  If for all $i = 1, \ldots, n-1$, either
  \begin{enumerate}[label=(\alph*)]
  \item \label{it:S-case-1} $\tau(\alpha_i) \le \tau(\alpha_{i+1})$
    and $\tau'(\alpha_1 + \cdots + \alpha_i) >
    \tau'(\alpha_{i+1} + \cdots + \alpha_n)$, or
  \item \label{it:S-case-2} $\tau(\alpha_i) > \tau(\alpha_{i+1})$
    and $\tau'(\alpha_1 + \cdots + \alpha_i) \le
    \tau'(\alpha_{i+1} + \cdots + \alpha_n)$,
  \end{enumerate}
  then define $S(\alpha_1, \ldots, \alpha_n; \tau, \tau')
  \coloneqq (-1)^r$ where $r$ is the number of $i = 1, \ldots, n-1$
  satisfying \ref{it:S-case-1}. Otherwise define $S(\alpha_1, \ldots,
  \alpha_n; \tau, \tau') \coloneqq 0$.

  We use $S$ to define universal coefficients $U$. A {\it double
    grouping} is a choice of integers
  \begin{align*}
    0 &= a_0 < a_1 < \cdots < a_m = n, \qquad \text{ for some } 1 \le m \le n, \\
    0 &= b_0 < b_1 < \cdots < b_l = m, \qquad \text{ for some } 1 \le l \le m,
  \end{align*}
  which defines classes $\beta_i \coloneqq \alpha_{a_{i-1}+1} + \cdots
  + \alpha_{a_i}$ and $\gamma_j \coloneqq \beta_{b_{j-1}+1} + \cdots +
  \beta_{b_j}$. A double grouping is {\it $(\tau,
    \tau')$-permissible} if:
  \begin{itemize}
  \item ($\tau$-permissible) $\tau(\beta_i) = \tau(\alpha_j)$ for
    every $1 \le i \le m$ and $a_{i-1} < j \le a_i$;
  \item ($\tau'$-permissible) $\tau'(\gamma_j) =
    \tau'(\alpha_1 + \cdots + \alpha_n)$ for every $1 \le j \le
    l$.
  \end{itemize}
  Then, summing over $(\tau,\tau')$-permissible double groupings,
  define
  \[ U(\alpha_1, \ldots, \alpha_n; \tau, \tau') \coloneqq \sum_{\text{groupings}} \frac{(-1)^{l-1}}{l} \cdot \prod_{j=1}^l S(\beta_{b_{j-1}+1}, \ldots, \beta_{b_j}; \tau,\tau') \cdot \prod_{i=1}^m \frac{1}{(a_i - a_{i-1})!}. \]
  Finally, we use $U$ to implicitly define universal coefficients
  $\tilde U$, using the following Lemma~\ref{lem:Utilde-definition}.
\end{definition}

\subsubsection{}

\begin{lemma}[{\cite[Theorem 5.4]{joyce-config-iv}}] \label{lem:Utilde-definition}
  Let $L$ be a free graded Lie algebra over $\bQ$ generated by symbols
  $\epsilon(\beta)$ of degree $\beta$, for all effective classes
  $\beta$. Then, assuming that the right hand sum below has finitely
  many non-zero terms, there exist $\tilde U(\alpha_1, \ldots,
  \alpha_n; \tau, \tau') \in \bQ$ such that
  \begin{equation} \label{eq:U-vs-Utilde}
    \begin{aligned}
      &\sum_{\substack{n \ge 1\\\alpha = \alpha_1+\cdots+\alpha_n}} \tilde U(\alpha_1,\ldots,\alpha_n; \tau, \tau') \left[\left[ \cdots \left[[\epsilon(\alpha_1), \epsilon(\alpha_2)], \epsilon(\alpha_3)\right], \ldots\right], \epsilon(\alpha_n)\right] \\
      &= \sum_{\substack{n \ge 1\\\alpha = \alpha_1+\cdots+\alpha_n}} U(\alpha_1,\ldots,\alpha_n; \tau, \tau') \epsilon(\alpha_1) \epsilon(\alpha_2) \cdots \epsilon(\alpha_n)
    \end{aligned}
  \end{equation}
  in the universal enveloping algebra of $L$, i.e. after expanding
  $[f, g] = fg - gf$.
\end{lemma}

\subsubsection{}

\begin{definition}[Dominance conditions] \label{def:dominates-at}
  Let $\tau$ and $\mathring\tau$ be two weak stability conditions. We
  say {\it $\tau$ dominates $\mathring\tau$} if, for any $\alpha_1,
  \alpha_2 \in C(\cat{A})$,
  \[ \mathring\tau(\alpha_1) \le \mathring\tau(\alpha_2) \implies \tau(\alpha_1) \le \tau(\alpha_2). \]
  
  Given a class $\alpha \in C(\cat{A})$, we say {\it $\tau$ dominates
    $\mathring{\tau}$ at $\alpha$}, if for any decomposition $\alpha =
  \alpha_1 + \cdots + \alpha_n$, $\alpha_1, \ldots, \alpha_n \in
  C(\cat{A})$, with either
  \begin{enumerate}[label=(\alph*)]
  \item \label{it:wcf-U-case-i} $U(\alpha_1, \ldots, \alpha_n; \tau,
    \mathring\tau) \neq 0$ and $\fM_{\alpha_i}^{\sst}(\tau)\neq
    \emptyset$ for $i = 1, \ldots, n$, or
  \item \label{it:wcf-U-case-ii} $U(\alpha_1, \ldots, \alpha_n;
    \mathring\tau, \tau) \neq 0$ and
    $\fM_{\alpha_i}^{\sst}(\mathring\tau) \neq \emptyset$ for $i =
    1, \ldots, n$,
  \end{enumerate}
  then $\tau(\alpha_i) = \tau(\alpha)$ for all $i = 1, \ldots, n$,
  and in case \ref{it:wcf-U-case-ii} we also have
  $\fM_{\alpha_i}^{\sst}(\mathring\tau) \subset
  \fM_{\alpha_i}^{\sst}(\tau)$ for $i = 1, \ldots, n$.

  Finally, given a class $\alpha \in C(\cat{A})$, define the sets
  \begin{align*}
    R_{\alpha} &= \{\alpha\} \cup \{\beta \in C(\cat{A}) : \alpha-\beta\in C(\cat{A}), \, \tau(\beta) = \tau(\alpha-\beta), \, \fM_{\beta}^{\sst}(\tau), \fM_{\alpha-\beta}^{\sst}(\tau)\neq \emptyset\}, \\
    \mathring R_{\alpha} &= \{\alpha\} \cup \{\beta \in C(\cat{A}) : \alpha-\beta\in C(\cat{A}), \, \mathring\tau(\beta) = \mathring\tau(\alpha-\beta), \, \fM_{\beta}^{\sst}(\mathring\tau), \fM_{\alpha-\beta}^{\sst}(\mathring\tau)\neq \emptyset\}.
  \end{align*}
  We say {\it $\tau$ weakly dominates $\mathring{\tau}$ at $\alpha$}
  if $\mathring R_\alpha \subseteq R_\alpha$ and, for any $\beta\in
  \mathring{R}_\alpha$, we have
  $\fM_\beta^\sst(\mathring{\tau})\subseteq \fM_\beta^\sst(\tau)$.
\end{definition}

It is easy to see that if $\tau$ dominates $\mathring{\tau}$, then
$\tau$-stable implies $\mathring\tau$-stable implies
$\mathring\tau$-semistable implies $\tau$-semistable. Less obviously,
if $\tau$ dominates $\mathring{\tau}$, then it also dominates
$\mathring{\tau}$ at any $\alpha \in C(\cat{A})$ \cite[Thm.
  3.11]{Joyce2021}. Finally, if $\tau$ dominates $\mathring\tau$ at
$\alpha$, then it also weakly dominates $\mathring\tau$ at $\alpha$
\cite[Prop. 10.2]{Joyce2021} (see also
Lemma~\ref{wc:lem:R-sets}\ref{it:R-sets-i}).

\subsubsection{}

\begin{assumption} \label{assump:wall-crossing}
  Let $\tau$ and $\mathring\tau$ be weak stability conditions on
  $\cat{A}$. We make the following technical assumptions, for all
  $\alpha \in C(\cat{A})$, in order to obtain a wall-crossing formula.
  \begin{enumerate}[label = (\alph*)]
  \item \label{assump:it:tau-circ} The weak stability condition $\tau$
    satisfies Assumption~\ref{assump:semistable-invariants}. In
    addition, $\cat{A}$ admits $\mathring\tau$-HN filtrations
    (Assumption~\ref{assump:semistable-invariants}\ref{assump:it:tau-artinian})
    and $\mathring\tau$-(semi)stability is open
    (Assumption~\ref{assump:semistable-invariants}\ref{assump:it:semistable-loci}).

  \item \label{assump:it:lambda} For any $\alpha$, there exists a
    group homomorphism
    \[ \lambda\colon K(\cat{A})\to \bR \]
    such that for any class $\beta \in R_\alpha$, we have
    $\lambda(\beta) > 0$ (resp. $\lambda(\beta) < 0$) if and only if
    $\mathring\tau(\beta) > \mathring\tau(\alpha-\beta)$ (resp.
    $\mathring\tau(\beta) <
    \mathring\tau(\alpha-\beta)$). \footnote{This is a slightly weaker
      condition on $\lambda$ than in \cite[Assumption
        5.2(d)]{Joyce2021} where $\mathring\tau(\beta)$ is compared
      with $\mathring\tau(\alpha)$ instead of
      $\mathring\tau(\alpha-\beta)$. This weaker condition is needed
      to ensure the auxiliary weak stability conditions in
      \eqref{wc:eq:joyce-framed-stack-stability} are genuinely weak
      stability conditions. Where necessary, we explicitly indicate
      how this affects our proof of the dominant wall-crossing formula
      in \S\ref{sec:wall-crossing}, in comparison to Joyce's
      proof.\label{footnote:lambda}}

  \item \label{assump:it:properness-wcf} The following algebraic
    spaces are proper and have the resolution property (see
    Remark~\ref{rem:localization-resolution-property}):
    \begin{itemize}
    \item $\tilde\fM_{\alpha,\vec d}^{\vec Q(\Fr),\sst}(\tau_x^s)^\sT$
      (the auxiliary stack in Definition~\ref{def:flag-invariant}) for
      all $\Fr \in \Frs$ and $(\alpha, \vec d)$ with no strictly
      $\tau^s_x$-semistable objects; \footnote{It is enough to satisfy
      this assumption for $(s, x) \in [0,1] \times \{0\}$ and $(s,x)
      \in \{0,1\} \times [-1,0]$. By
      Lemma~\ref{lem:sst-pair}\ref{it:sst-pair-a}, this subsumes the
      properness assumption on
      $\tilde\fM_{\alpha,1}^{Q(\Fr),\sst}(\tau^Q)^\sT$ in
      Assumption~\ref{assump:semistable-invariants}\ref{assump:it:properness}.}
    \item $\bM_{\alpha,\vec d}^{\sT_w}$ for all $w$ in
      Lemma~\ref{lem:master-space-no-poles}, where $\bM_{\alpha,\vec
        d}$ is the master space in the proof of
      Proposition~\ref{wc:prop:horizontal-wc}, for all classes
      $(\alpha,\vec d)$ and $\Fr \in \Frs$.
    \end{itemize}
  \end{enumerate}
\end{assumption}

\subsubsection{}

\begin{theorem}[Dominant wall-crossing formula] \label{thm:wcf}
  Let $\tau$ and $\mathring\tau$ be weak stability conditions on
  $\cat{A}$ for which Assumption~\ref{assump:wall-crossing} holds.
  Suppose $\tau$ weakly dominates $\mathring{\tau}$ at $\alpha \in
  C(\cat{A})$. Then the operational semistable invariants of
  Theorem~\ref{thm:sst-invariants} satisfy
  \begin{equation} \label{eq:wcf}
    \sz_\alpha(\mathring\tau) = \sum_{\substack{n>0\\\alpha = \alpha_1 + \cdots + \alpha_n\\\forall i:\, \tau(\alpha_i) = \tau(\alpha)\\\fM^{\sst}_{\alpha_i}(\tau)\neq \emptyset}}\tilde U\left(\alpha_1,\dots,\alpha_n;\tau,\mathring\tau\right)\left[\left[\cdots\left[\sz_{\alpha_1}(\tau),\sz_{\alpha_2}(\tau)\right],\cdots\right],\sz_{\alpha_n}(\tau)\right]
  \end{equation}
  in $K_\circ^{\tilde\sT}(\fM_\alpha)^\pl_{\loc,\bQ}$, with Lie
  bracket $[-, -]$ defined by
  Theorem~\ref{thm:mVOA-monoidal-stack-lie-algebra} (or
  Theorem~\ref{thm:auxiliary-stack-vertex-algebra}).
\end{theorem}

As in Theorem~\ref{thm:sst-invariants}, the existence of the rank
function ensures that the sum is finite.

\subsubsection{}

\begin{remark} \label{intro:remark:dominant-wcf-conditions}
  Theorem~\ref{thm:wcf} is analogous to \cite[Theorem 5.8]{Joyce2021},
  but the dominance condition for $\tau$ and $\mathring\tau$ is
  weaker; Joyce assumes that $\tau$ dominates $\mathring{\tau}$ at
  $\alpha$. Joyce's assumption in fact unnecessary, and even Joyce's
  proof of his wall-crossing formula uses only the weak dominance
  condition except in the proof of \cite[Prop. 10.13]{Joyce2021}. The
  analogous step in our proof is
  Lemma~\ref{lem:sst-pair}\ref{it:sst-pair-b}, which we prove with
  only the weak dominance condition.
\end{remark}

\subsubsection{}
\label{sec:general-wall-crossing}

Morally, one should imagine that $\tau$ and $\mathring\tau$ originate
from a family of weak stability conditions with a wall-and-chamber
decomposition, and the dominance condition expresses that $\tau$ is a
weak stability condition on some wall and $\mathring\tau$ is a weak
stability condition in an adjacent chamber. Clearly, a
``chamber-to-chamber'' wall-crossing problem may be factorized into
two ``chamber-to-wall'' wall-crossing problems. Using the transitivity
property (Lemma~\ref{wc:lem:U-properties}\ref{it:U-composition}) of
the universal coefficients $\tilde U$, the dominant wall-crossing
formula immediately implies the {\it general wall-crossing formula}
\[ \sz_\alpha(\tau') = \sum_{\substack{n>0\\\alpha = \alpha_1 + \cdots + \alpha_n\\\forall i:\, \fM^{\sst}_{\alpha_i}(\tau)\neq \emptyset}}\tilde U\left(\alpha_1,\dots,\alpha_n;\tau,\tau'\right)\left[\left[\cdots\left[\sz_{\alpha_1}(\tau),\sz_{\alpha_2}(\tau)\right],\cdots\right],\sz_{\alpha_n}(\tau)\right] \]
between any two weak stability conditions $\tau$ and $\tau'$ connected
by a path of dominant wall-crossings; see \cite[Assumption 5.3, \S
  11]{Joyce2021} for details.

\subsubsection{}

\begin{remark}
  We believe that
  Assumptions~\ref{assump:semistable-invariants}\ref{assump:it:tau-artinian},
  \ref{assump:it:semistable-loci}, and
  \ref{assump:wall-crossing}\ref{assump:it:tau-circ} are genuinely
  important for studying wall-crossing, while
  Assumptions~\ref{assump:semistable-invariants}\ref{assump:it:framing-functor},
  \ref{assump:it:rank-function}, \ref{assump:it:semi-weak-stability},
  and \ref{assump:wall-crossing}\ref{assump:it:lambda} are only
  artifacts of our (and Joyce's \cite{Joyce2021}) geometric setup for
  proving the wall-crossing formula. Namely, the latter assumptions
  are only really used to construct auxiliary moduli stacks and weak
  stability conditions on them (Definition~\ref{def:flag-invariant};
  see also Remark~\ref{rem:joyce-framed-stack-stability}), and future
  work may render them unnecessary.

  For readers familiar with \cite[Assumption 5.2]{Joyce2021}, note
  that
  Assumption~\ref{assump:semistable-invariants}\ref{assump:it:rank-function}
  is weaker and
  Assumption~\ref{assump:semistable-invariants}\ref{assump:it:semi-weak-stability}
  is new (see footnote~\ref{footnote:joyce-framed-stack-stability}).
  Note that
  Assumption~\ref{assump:semistable-invariants}\ref{assump:it:semi-weak-stability}
  is automatically satisfied if $\tau$ is a stability condition.

  Both
  Assumptions~\ref{assump:semistable-invariants}\ref{assump:it:rank-function}
  and \ref{assump:it:semi-weak-stability} only need to be checked on
  the classes of sub- and quotient objects of $\tau$-semistable
  objects; both $r$ and $\lambda$ only need to be defined on such
  classes. Often, such sub- and quotient objects remain
  $\tau$-semistable, so that it suffices to define $r$ and $\lambda$
  only on $\tau$-semistable objects. In fact, if $r$ only needs to be
  defined for finitely many classes $\beta$, then we can take it to be
  the dimension function $\fr$ of a framing functor $\Fr$ such that
  $\fM_{\beta}^{\sst}(\tau) \subset \fM_{\beta}^{\Fr,\pl}$ for all
  such $\beta$ (which exists by
  Assumption~\ref{assump:semistable-invariants}\ref{assump:it:framing-functor});
  e.g. the injectivity condition $\End(E)
  \hookrightarrow \End(\Fr(E))$ on $\Fr$ yields the desired positivity
  of $r$.
\end{remark}

\subsubsection{}
\label{sec:comparison-with-joyce-song}

This wall-crossing framework can be viewed as a (K-)homological
generalization, as well as a $\kappa$-refinement, of the {\it motivic}
wall-crossing framework by Joyce--Song \cite{Joyce2012} and
Kontsevich--Soibelman \cite{Kontsevich2010}. Namely, our wall-crossing
formula \eqref{eq:wcf} has the exact same shape as \cite[Theorem
  3.14]{Joyce2012}, and evaluation of our operational K-homology
classes on $\cO$ is the analogue of the Joyce--Song \cite[Theorem
  5.14]{Joyce2012} and the Kontsevich--Soibelman \cite[Theorem
  2]{Kontsevich2010} integration maps. In particular, the formula (see
Proposition~\ref{prop:rigidity})
\[ [\phi_\alpha, \psi_\beta](\cO) = [\chi(\alpha,\beta)]_\kappa \cdot \phi_\alpha(\cO) \psi_\beta(\cO) \]
is (the $\kappa$-refinement of) the Lie bracket for the quantum torus
underlying both Joyce--Song \cite[Def. 5.13]{Joyce2012} and
Kontsevich--Soibelman's \cite[\S 4.2]{Kontsevich2010} frameworks. Here
$[n]_\kappa$ is the {\it (symmetric) quantum integer}
\begin{equation} \label{eq:quantum-integer}
  [n]_\kappa \coloneqq (-1)^{n-1} \frac{\kappa^{\frac{n}{2}} - \kappa^{-\frac{n}{2}}}{\kappa^{\frac{1}{2}} - \kappa^{-\frac{1}{2}}} \in \bZ[\kappa^{\pm\frac{1}{2}}],
\end{equation}
which satisfies $\lim_{\kappa \to 1} [n]_\kappa = (-1)^{n-1} n$. The
slightly unconventional sign $(-1)^{n-1}$ is a stylistic choice to
save on signs elsewhere.

Recall that motivic enumerative invariants are defined (for
semistable=stable loci) using Behrend functions, in contrast to our
enumerative invariants defined using virtual cycles. The two
approaches agree when the underlying stable locus is proper
\cite{Behrend2009}, explaining why our wall-crossing formula matches
those of Joyce--Song and Kontsevich--Soibelman. However, they disagree
in general, and there is some indication that virtual cycles may be
the correct approach to consider in physics, see e.g. \cite[\S
  1.10]{Tanaka2020}.

\subsubsection{}
\label{sec:generalizations}

In \S\ref{sec:additional-features}, we provide the following
generalizations of Theorems~\ref{thm:sst-invariants} and
\ref{thm:wcf}.
\begin{itemize}
\item (\S\ref{sec:cohomological-version}) Cohomological versions
  exist, of exactly the same shape. Invariants live in
  $A^\sT_*(\fM)^\pl_{\loc,\bQ}$ where $A^\sT_*(-)$ is dual to {\it
    equivariant operational Chow cohomology} in the same way
  $K_\circ^\sT(-)$ is dual to $K^\circ_\sT(-)$. There is a notion of
  additive equivariant vertex algebra, and we explain how to replace
  all K-theoretic ingredients with cohomological analogues. More
  general cohomology theories may also be used; see
  Remark~\ref{rem:cohomology-theory}.

\item (\S\ref{sec:restrictions}) The abelian category $\cat{A}$ may be
  replaced by an exact subcategory $\cat{B} \subset \cat{A}$ closed
  under isomorphisms and direct summands. Furthermore, the moduli
  stack $\fM$ of $\cat{B}$ may be replaced by certain locally closed
  moduli substacks $\fN \subset \fM$ closed under direct sums and
  summands. For example, $\fM$ may be some moduli of sheaves on a
  space $X$ and $\fN$ the moduli substack of those sheaves with
  prescribed restriction to a divisor $D \subset X$ (see e.g.
  \S\ref{sec:intro:equivariant-vertices}). Finally, with some care, we
  need only require that all assumptions hold for a subset
  $C(\cat{B})_{\pe} \subset C(\cat{B})$ of ``permissible'' classes,
  not necessarily closed under direct sum.

\item (\S\ref{sec:reduced-obstruction-theories}) We may also consider
  wall-crossing for {\it (Kiem--Li) reduced virtual cycles} and their
  enumerative invariants, in the setting where the obstruction theory
  on $\fM_\alpha$ comes with $o_\alpha\in \bZ_{\ge 0}$ cosections. The
  wall-crossing formula for such reduced invariants is actually
  simpler because its sum contains an additional constraint $o_\alpha
  = o_{\alpha_1} + \cdots + o_{\alpha_n}$. In English, the two sides
  of the wall-crossing formula are reduced by some number of
  cosections, and, if these numbers don't match, the resulting term
  vanishes. Our setup for reduced virtual cycles follows
  \cite{Kiem2013}, generalizing the setup in \cite{Joyce2021}.
\end{itemize}

\subsubsection{}
\label{sec:joyce-comparison-and-differences}

Our 3CY wall-crossing framework and proof strategy is based on and
closely follows Joyce's ``universal wall-crossing'' framework
\cite{Joyce2021}. In a nutshell, Joyce's moduli stacks all carry {\it
  perfect} obstruction theories, while we ensure throughout that ours
carry {\it symmetric} obstruction theories. We highlight three main
technical differences.

First, the auxiliary stacks and master spaces are equipped with
symmetric APOTs, compatible with the one on $\fM$, via a {\it
  symmetrized pullback} operation
(Theorem~\ref{thm:APOTs}\ref{it:APOT-sym-pullback}) on obstruction
theories, constructed in the authors' earlier work \cite{klt_dtpt}.
This is non-trivial even without symmetrization. Joyce circumvents
this by assuming the existence of derived versions of all moduli
stacks and morphisms between them.

Consequently, the key geometric wall-crossing steps
(Theorems~\ref{thm:pairs-master-relation} and
\ref{wc:prop:horizontal-wc}) must be written in terms of the Lie
bracket on the (operational K-homology of) auxiliary stacks. In
Joyce's setting, they may be written using only the Lie bracket on the
(homology of the) original stack $\fM$. As explained in
Remark~\ref{rem:joyce-lie-bracket-vs-ours-1}, this is an inevitable
symptom of symmetrization.

Consequently, in the proof of the wall-crossing formula
(Theorem~\ref{thm:wcf}), many results may be proved directly in the
operational K-homology of the auxiliary stacks instead of in an
artificial Lie algebra. This makes some of the vanishing results more
complicated, notably Lemma~\ref{wc:lem:vanishing-basic}, but leads to
a more conceptual proof.

\subsection{Summary of applications}

\subsubsection{}
\label{sec:wcf-pure-type-setting}

In \S\ref{sec:wall-crossings-simple-type}, we study the simplest
non-trivial case of our wall-crossing formula (Theorem~\ref{thm:wcf}):
a continuous family of weak stability conditions $\{\tau_t\}_{t \in
  [-1,1]}$ which is ``of simple type''
(Definition~\ref{def:wall-crossing-simple-type}). One may imagine that
$\{\tau_t\}_t$ arises from a family of weak Bridgeland stability
conditions with central charges as depicted in
Figure~\ref{fig:wc-simple-type} involving effective classes of type
$A$ and of type $B$. Roughly, there is only one wall, at $t=0$, and at
this wall objects of type $A$ may split off strictly
$\tau_0$-semistable pieces of type $B$.

\begin{figure}[!ht]
  \centering
  \captionsetup{width=.8\linewidth}
  \includegraphics[scale=1.1]{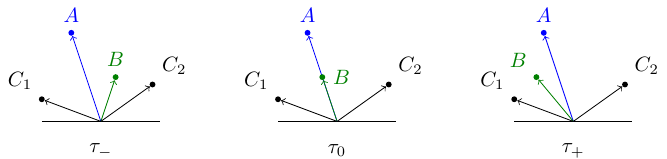}
  \caption{A one-parameter family of weak stability conditions
    defining a wall-crossing problem of simple type. Central charges
    of all classes not of type $A$ and $B$ stay constant and far away
    from central charges of classes of type $A$ and $B$.}
  \label{fig:wc-simple-type}
\end{figure}

\subsubsection{}

\begin{proposition} \label{prop:wcf-simple-type}
  Let $\{\tau_t\}_{t \in [-1,1]}$ define a wall-crossing problem of
  simple type. For any class $\alpha \in A$,
  \begin{equation} \label{eq:wcf-simple-type}
    \sz_\alpha(\tau_1) = \sum_{\substack{n \ge 0\\\alpha = \alpha' + \beta_1 + \cdots + \beta_n\\\alpha' \in A, \; \forall k: \, \beta_k \in B}} \frac{1}{n!} \left[\left[\cdots[\sz_{\alpha'}(\tau_{-1}), \sz_{\beta_1}(\tau_{-1})], \ldots\right], \sz_{\beta_n}(\tau_{-1})\right].
  \end{equation}
\end{proposition}

This is essentially a calculation of the universal coefficients
$\tilde U$, and is proved from first principles in
\S\ref{sec:wall-crossings-simple-type}.

By introducing a grading or otherwise so that all the following
expressions are well-defined, we may define
\[ \sz^A(\tau_t) \coloneqq \sum_{\alpha \in A} \sz_\alpha(\tau_t), \quad \sz^B(\tau_t) \coloneqq \sum_{\beta \in B} \sz_\beta(\tau_t), \]
and use the operator notation $\ad(x)(-) \coloneqq [x, -]$ to write
\eqref{eq:wcf-simple-type} more cleanly as
\[ \sz^A(\tau_1) = \exp\left(-\ad(\sz^B(\tau_{-1}))\right) \sz^A(\tau_{-1}). \]

\subsubsection{}
\label{sec:DT-theories}

In \S\ref{sec:DT-PT} and \S\ref{sec:PT-BS}, we apply
Proposition~\ref{prop:wcf-simple-type} to Donaldson--Thomas-type
theories counting curves in smooth quasi-projective 3-folds $X$:
\begin{itemize}
\item the {\it Donaldson--Thomas} ({\it DT}) theory \cite{Maulik2006}
  of ideal sheaves $\cI_C$ of $1$-dimensional subschemes $C
  \subset X$;
\item the {\it Pandharipande--Thomas} ({\it PT}) theory
  \cite{Pandharipande2009} of {\it stable pairs} $\cO_X
  \xrightarrow{s} \cF$ on $X$, meaning that $\coker(s)$ is
  $0$-dimensional and $\cF$ is pure $1$-dimensional;
\item if $\pi\colon X \to X_0$ is a resolution of singularities of
  relative dimension $\le 1$, the {\it Bryan--Steinberg} ({\it
    $\BS(\pi)$}) theory of {\it $\pi$-stable pairs} $\cO_X
  \xrightarrow{s} \cF$ on $X$, meaning that $R\pi_*\coker(s)$ is a
  $0$-dimensional sheaf and $\cF$ is $1$-dimensional and admits only
  the zero map from sheaves $\cT$ such that $R\pi_*\cT$ is a
  $0$-dimensional sheaf \cite{Bryan2016}.
\end{itemize}
Viewing an ideal sheaf $\cI_C$ also as a pair $\cO_X
\twoheadrightarrow \cO_C$, these are three (weak) stability chambers
for pairs $\cO_X \to \cF$ on $X$ with proper support and $\dim\supp\cF
\le 1$. The DT/PT chambers and PT/BS chambers are related by
wall-crossings of simple type, well-studied due to their relevance to
e.g. rationality problems \cite{Bridgeland2011,Pandharipande2013} and
the DT crepant resolution conjecture \cite{Bryan2012,Beentjes2022}.

When $X$ is toric, the bulk of the complexity in these wall-crossings
may be captured by certain quasi-projective ``model'' geometries.
Generating series of their associated universal enumerative
invariants, called {\it vertices}, have a long history in mathematical
physics \cite{Aganagic2005, Iqbal2008, Li2009,Iqbal2009}, with rich
connections to Gromov--Witten theory \cite{Maulik2011}, geometric
representation theory \cite{Okounkov2017}, and statistical mechanics
\cite{Okounkov2006,Jenne2021}, for a non-representative sample. We
summarize their construction in \S\ref{sec:intro:equivariant-vertices}
below; see \S\ref{sec:DT-PT} and \S\ref{sec:PT-BS} for details.

\subsubsection{}
\label{sec:intro:equivariant-vertices}

For DT/PT, consider $\bC^3 \subset X \coloneqq (\bP^1)^3$ with the
scaling action of $\sT \coloneqq (\bC^\times)^3$ and with
$\sT$-invariant boundary divisors at infinity denoted $\iota_i\colon
D_i \to (\bP^1)^3$ for $i=1,2,3$. For $M \in \{\DT, \PT\}$ and integer
partitions $\lambda,\mu,\nu$ specifying $\sT$-fixed points in
$\Hilb(D_i)$ for $i = 1,2,3$, let
\[ M_{(\lambda,\mu,\nu),n}^{\sst} = \left\{ M\text{-stable } [\cO_X \xrightarrow{s} \cF] : \begin{array}{c}\ch(\cF) = (0,0,\beta,n), \\ L^1\iota_i^*\cF = 0 \text{ for } i=1,2,3, \\ ({[\cO_{D_i} \xrightarrow{s} \iota_i^*\cF]})_{i=1}^3 = (\lambda,\mu,\nu)\end{array}\right\} \]
where $\beta \coloneqq (|\lambda|,|\mu|,|\nu|) \in H_2((\bP^1)^3;
\bZ)$. These moduli spaces carry symmetric perfect obstruction
theories, using which we define the (``3-legged'') {\it DT} and {\it
  PT vertices}
\begin{equation} \label{eq:DT-PT-vertices}
  \sV_{\lambda,\mu,\nu}^M \coloneqq \sum_{n \in \bZ} Q^n \chi\left(M_{(\lambda,\mu,\nu),n}^{\sst}, \hat\cO^\vir \otimes -\right).
\end{equation}

For PT/BS, consider $\cA_m \times \bC \subset X \coloneqq \cA_m \times
\bP^1$, where $\pi\colon \cA_m \times \bP^1 \to \bC^2/\bZ_{m+1} \times
\bP^1$ is the minimal smooth crepant resolution, with the action of
$\sT \coloneqq (\bC^\times)^3$ induced from its natural action on
$\bC^2 \times \bP^1$, and $\sT$-invariant boundary divisor at infinity
denoted $\iota\colon D \to \cA_m \times \bP^1$. For $M \in \{\PT,
\BS\}$ and integer partitions $\lambda_1,\ldots,\lambda_{m+1}$
specifying a $\sT$-fixed point $\vec\lambda =
(\lambda_1,\ldots,\lambda_{m+1})$ in $\Hilb(\cA_m)$, let
\[ M_{\vec\lambda,\beta_{\cA},n}^{\sst} = \left\{ M\text{-stable } [\cO_X \to \cF] : \begin{array}{c}\ch(\cF) = (0,0,(\beta_{\bP}, \beta_{\cA}),n), \\ L^1\iota^*\cF = 0 \\ {[\cO_D \xrightarrow{s} \iota^*\cF]} = \vec\lambda \end{array}\right\} \]
where $\beta_{\bP} \coloneqq \textstyle\sum_{i=1}^m |\lambda_i| \in
H_2(\bP^1; \bZ)$. These again carry symmetric perfect obstruction
theories, using which we define the (``1-legged'') {\it PT$(\pi)$} and
{\it BS$(\pi)$ vertices} \footnote{While $\bC^3$ is the only basic
building block for DT-type theories of smooth toric $3$-folds, the
basic building blocks of DT-type theories of toric crepant resolutions
are crepant resolutions of $[\bC^3/G]$ where $G \subset \SO(3)$ or $G
\subset \SU(2)$ is a finite subgroup \cite[Lemma 24]{Bryan2009}, and
there should be a BS vertex for each such basic crepant resolution.
Thus, $\pi$ must be recorded in the notation.}
\begin{equation} \label{eq:PT-BS-vertices}
  \sV_{\vec\lambda}^{M(\pi)} \coloneqq \sum_{\substack{\beta_{\cA} \in H_2(\cA_m; \bZ)\\n \in \bZ}} A^{\beta_{\cA}} Q^n \chi\left(M_{\vec\lambda,\beta_{\cA},n}^{\sst}, \hat\cO^\vir \otimes -\right).
\end{equation}

These DT/PT and $\PT(\pi)$/$\BS(\pi)$ vertices agree with those of
\cite[\S 4]{Maulik2006} (DT), \cite[\S 4]{Pandharipande2009a} (PT),
and \cite[\S 2.2]{Liu2021} ($\BS(\pi)$).

\subsubsection{}

\begin{theorem}[Operational equivariant vertex correspondences] \label{thm:vertex-correspondences} 
  \leavevmode
  \begin{enumerate}[label=(\roman*)]
  \item \label{it:dt-pt-vertex} (DT/PT) For all $\lambda, \mu, \nu$,
    \begin{equation} \label{eq:DT-PT}
      I_*\sV_{\lambda,\mu,\nu}^{\DT} = \exp\left(\ad(\sz)\right) I_*\sV_{\lambda,\mu,\nu}^{\PT}
    \end{equation}
    where $\sz \coloneqq \sum_{m \in \bZ} \sz_m Q^m$ is a series of
    K-homology classes independent of $\lambda, \mu, \nu$.
  \item \label{it:pt-bs-vertex} (PT/BS) For all $\vec\lambda$,
    \begin{equation} \label{eq:PT-BS}
      I_*\sV_{\vec\lambda}^{\PT(\pi)} = \bigg(\prod^{\rightarrow}_{s \in \bQ_{>0}} \exp\left(\ad(\sz^{\pi,s})\right)\bigg) I_*\sV_{\vec\lambda}^{\BS(\pi)},
    \end{equation}
    where $\sz^{\pi,s} \coloneqq \sum_{\mu(\beta_{\cA},n)=s}
    \sz_{\beta_{\cA},n}^{\pi,s} A^{\beta_{\cA}} Q^n$ (see
    \S\ref{sec:PT-BS-WCF-proof} and Remark~\ref{rem:PT-BS-wcf-shape})
    is a series of K-homology classes independent of $\vec\lambda$.
  \end{enumerate}
\end{theorem}

These are equalities of formal series in $Q$ (and $A$) valued in
$\tilde\sT$-equivariant K-homology. The discussion of
\S\ref{sec:cohomological-version} applies, yielding the same results
in homology as well.

The proof of \ref{it:dt-pt-vertex} in \S\ref{sec:DT-PT} is more or
less a review of the content of \cite[\S 6.1]{klt_dtpt}, which builds
on the setup of \cite{toda_dtpt}. The analogous proof of
\ref{it:pt-bs-vertex} in \S\ref{sec:PT-BS} is new and is inspired by
the setup of \cite{Toda2013}.

\subsubsection{}
\label{sec:intro:primary-vertex-correspondences}

Evaluating \eqref{eq:DT-PT} and \eqref{eq:PT-BS} on the trivial
cohomology/K-theory classes $1$ or $\cO$ immediately produces the {\it
  equivariant primary vertex correspondences} (see
\S\ref{sec:DT-PT-consequences} and \S\ref{sec:PT-BS-consequences})
\begin{align}
  \sV_{\lambda,\mu,\nu}^{\DT}(\cO) &= \sV_{\emptyset,\emptyset,\emptyset}^{\DT}(\cO) \cdot \sV_{\lambda,\mu,\nu}^{\PT}(\cO), \label{eq:DT-PT-primary} \\
  \sV_{\vec\lambda}^{\PT(\pi)}(\cO) &= \sV_{\vec\emptyset}^{\PT(\pi)}(\cO) \cdot \sV_{\vec\lambda}^{\BS(\pi)}(\cO), \label{eq:PT-BS-primary}
\end{align}
where $\emptyset$ is the trivial integer partition and $\vec\emptyset
\coloneqq (\emptyset,\ldots,\emptyset)$. Note here, that $\sV_{\vec\emptyset}^{\PT(\pi)}$ only fixes an empty curve class in the $\bP^1$-direction and hence corresponds to the PT invariants counting exceptional curves. While the (K-theoretic) DT/PT
primary vertex correspondence \eqref{eq:DT-PT-primary} had previously
existed in the literature as a conjecture \cite[Conjecture
  4]{Pandharipande2009a} \cite[Equation (16)]{Nekrasov2016} --- and
was proved by the authors ``by hand'' in \cite{klt_dtpt} --- the
(K-theoretic) PT/BS primary vertex correspondence
\eqref{eq:PT-BS-primary} is genuinely new and is a significant
refinement of previous non-equivariant cohomological results for
compact CY 3-folds \cite{Toda2013} \cite[Theorem 6]{Bryan2016}
\cite[\S 1.1]{Beentjes2022}. Combined with the BS/quasimaps primary
vertex correspondence \cite{Liu2021}, it completes the K-theoretic
refinement of a connection between the quantum cohomology of
$\Hilb(\cA_m)$ and sheaf counting on $\cA_m \times \bP^1$
\cite{Maulik2009, Maulik2011}.

\subsubsection{}

In \S\ref{sec:DT-descendent-transformations}, we go one step beyond
the simplest case of \S\ref{sec:intro:primary-vertex-correspondences}
and evaluate \eqref{eq:DT-PT} and \eqref{eq:PT-BS} on products of
cohomological {\it descendent classes}
\begin{equation} \label{eq:DT-descendents}
  \tau_n(\xi) \coloneqq \pi_{\fN *}((\ch_n(\scF^\pl) \cup \pi_X^*(\xi)) \cap \pi_{\fN}^*(-)) \in A_{\sT}^{n-3+\deg \xi}(\fN^\pl; \bQ),
\end{equation}
for $n \ge 0$ and homogeneous $\xi \in \CH^{\sT}_*(X)$. Here
$\scF^\pl$ is the universal sheaf of $\cF$ on $\fN^\pl \times X$ (see
Definition~\ref{def:DT-rigidified-descendents}), and $\pi_{\fN}$ and
$\pi_X$ are projections to the two factors. The task of understanding
integrals of polynomials in descendent classes, as well as structures
in the algebra of descendent classes themselves, has a long history
\cite{Pandharipande2018} motivated by the analogous (apparently
harder) task for {\it Gromov--Witten (GW) theory}, see e.g.
\cite{Maulik2006a,Pandharipande2018a}, along with the celebrated GW/DT
correspondence \cite{Maulik2011}.

Certainly, one may also consider K-theoretic descendents and analogues
of the following Corollary~\ref{cor:descendent-correspondences} and
Theorem~\ref{thm:DT-PT-descendents}. But because K-theoretic
descendents are rather obscure in the current literature, this will be
addressed more comprehensively in future work instead.

\subsubsection{}

\begin{corollary}[Equivariant cohomological descendent vertex correspondences] \label{cor:descendent-correspondences}
  Let
  \[ (\sV, \sV_0, \sV') = (\sV_{\lambda,\mu,\nu}^{\DT}, \sV_{\emptyset,\emptyset,\emptyset}^{\DT}, \sV_{\lambda,\mu,\nu}^{\PT}) \text{ or } (\sV_{\vec\lambda}^{\PT(\pi)}, \sV_{\vec\emptyset}^{\PT(\pi)}, \sV_{\vec\lambda}^{\BS(\pi)}). \]
  For a monomial $\ff$ in the descendents $\{\tau_n(\xi) : n > 0, \,
  \xi \in \CH^{\sT}_*(X)\}$,
  \[ \sV(\ff) = \sum_{\ff'} c_{\ff}^{\ff'}(\sV_0) \cdot \sV'(\ff') \]
  for monomials $\ff'$ and coefficients $c_{\ff}^{\ff'}(\sV_0)$ whose
  $Q^n$ coefficient is a polynomial, over $\bQ$, in
  \[ (\sV_0)_m(\ff'') \text{ for } m \le n \text{ and monomials } \ff'' \]
  where $(\sV_0)_m$ denotes the coefficient of $Q^m$ in $\sV_0$.
\end{corollary}

The coefficients $c_{\ff}^{\ff'}(\sV_0)$ are inexplicit but are
effectively computable. The existence of {\it some} expressions
$c_{\ff}^{\ff'}(\sV_0)$, not necessarily a polynomial in the $\DT_0$
or $\PT_0(\pi)$ descendent vertices, was a folklore conjecture
probably already anticipated by \cite{Pandharipande2009}.

\subsubsection{}

\begin{theorem}[Explicit equivariant cohomological DT/PT descendent vertex correspondence] \label{thm:DT-PT-descendents}
  Let
  \[ \tau^{(k)}(\xi) \coloneqq \sum_{n \ge 0} k^{n-3+\deg \xi} \tau_n(\xi). \]
  Let $\spt \in \CH_0^\sT(X)$ be the class of $0 \in \bC^3 \subset X$
  and
  \[ \sigma\{k_1, \ldots, k_N\} \coloneqq \prod_{i=1}^N (1 - \tau^{(k_i)}(\spt)). \]
  Then, for integers $N \ge 0$ and $k_1, \ldots, k_N \in \bZ$,
  \begin{equation} \label{eq:DT-PT-descendent-transformation}
    \begin{aligned}
      \sV_{\lambda,\mu,\nu}^{\DT}\left(\sigma\{k_i\}_{i \in \underline{N}}\right) \equiv \sum_{\substack{n>0\\m_1,\ldots,m_n>0\\ S_1 \sqcup \cdots \sqcup S_n = \underline{N}\\ \forall i: \, S_i^1 \sqcup \cdots \sqcup S_i^{m_i} = S_i}}
      &(-1)^n \cdot \sV_{\lambda,\mu,\nu}^{\PT}\Big(\sigma\big\{\sum_{j \in S_i} k_j\big\}_{i=1}^n\Big) \cdot \sV_{\emptyset,\emptyset,\emptyset}^{\DT}(1) \\[-2.5em]
      &\quad \cdot \prod_{i=1}^n (m_i-1)! \prod_{j=1}^{m_i} \frac{\sV_{\emptyset,\emptyset,\emptyset}^{\DT}\left(\sigma\{k_\ell\}_{\ell \in S_i^j}\right)}{-\sV_{\emptyset,\emptyset,\emptyset}^{\DT}(1)} \bmod{\hbar}
    \end{aligned}
  \end{equation}
  where $\underline{N} \coloneqq \{1, \ldots, N\}$, all set partitions
  are into non-empty subsets (i.e. all $S_i^j \neq \emptyset$), and
  $\hbar \in \bh^2_\sT$ is the Calabi--Yau weight.
\end{theorem}
  
The notation $\{-\}$ reminds us that $\sigma$ depends only on the set
$\{k_1, \ldots, k_N\}$ and not the ordering of its elements.

\subsubsection{}

\begin{remark}
  We comment on various aspects of
  \eqref{eq:DT-PT-descendent-transformation}.

  First, one can use \eqref{eq:DT-PT-descendent-transformation} to
  effectively obtain a descendent vertex correspondence (modulo
  $\hbar$) for arbitrary descendents; see
  Remark~\ref{rem:DT-PT-descendents-basis}.

  Second, vertices may be glued together to form partition functions
  for toric geometries $X$. This gluing is compatible with
  \eqref{eq:DT-PT-descendent-transformation} and produces a descendent
  correspondence at the level of partition functions; see
  Remark~\ref{rem:DT-PT-partition-functions}.

  Third, the formula \eqref{eq:DT-PT-descendent-transformation} takes
  place in the Calabi--Yau limit $\hbar \to 0$, which is well-defined
  \cite[\S 4.10]{Maulik2006} and greatly simplifies (descendent)
  vertices, but even in this simplified setting we are unaware of any
  previously-known general formulas or conjectures. The previous
  state-of-the-art appears to be \cite[Conjecture
    5.3.1]{Oblomkov2019}, which is the $N=1$ case (i.e. one descendent
  insertion) of Theorem~\ref{thm:DT-PT-descendents}. \footnote{To be
  precise, Hagborg--Oblomkov's conjecture is a transformation formula
  for {\it $\sT$-equivariant} and {\it capped} vertices with one
  descendent insertion. Capped vertices are related to our vertices by
  a change of basis known as the {\it capping operator} \cite[\S
    2.3]{Maulik2011}, which becomes the identity operator upon taking
  $\hbar \to 0$.}

  Finally, we produce formulas like
  \eqref{eq:DT-PT-descendent-transformation} by explicitly computing
  the Lie bracket evaluated on descendents. Such explicit computations
  can be done equally well without taking the Calabi--Yau limit, and
  also for PT/BS descendent vertices. As the resulting formulas will
  be much more complicated, we computed only the simplest case,
  i.e. DT/PT in the Calabi--Yau limit.
\end{remark}

\subsubsection{}

In \S\ref{sec:VW}, we apply our main results to {\it Vafa--Witten (VW)
  theory}, which we view as a flavor of DT theory with reduced virtual
cycles. As constructed mathematically by Tanaka and Thomas
\cite{Tanaka2017,Tanaka2020}, VW invariants count Gieseker-stable
compactly-supported sheaves on a local surface, with an appropriately
symmetrically-reduced symmetric obstruction theory. Moreover, in
\cite{Thomas2020}, Thomas provides a generalization to K-theory which
{\it refines} these VW invariants. This is a natural setting where our
K-theoretic wall-crossing framework is applicable, and indeed we
apply it to define and study {\it operational refined semistable VW
  invariants} $\svw_\alpha(H)$ in equivariant K-homology. Their
evaluations $\svw_\alpha(H)(\cO)$ recover the numerical refined
semistable VW invariants of \cite{Thomas2020, liu_ss_vw}.\footnote{To
  avoid confusion, note that our $\svw_\alpha(H)(\cO)$ is called
  $\sVW_\alpha(t)$ in these references, and their notation
  $\svw_\alpha(t)$ denotes the {\it Behrend-weighted Euler
    characteristic} version of VW invariants, which are ``incorrect''
  and do not appear in this paper.}

In physics, Vafa and Witten \cite{Vafa1994} originally introduced and
studied these numerical (unrefined) invariants in order to test the
prediction from S-duality that appropriate generating series of VW
invariants are modular forms. In recent years, this prediction has
attracted growing mathematical interest, see e.g. \cite{Goettsche2022,
  Jiang2022, Goettsche2025}, and we expect our results to enable the
study of modularity phenomena at the level of operational K-homology
classes.

\subsubsection{}
\label{sec:intro-VW-setup}

We briefly introduce the notation necessary to state the main
Theorems~\ref{thm:sst-invariants} and \ref{thm:wcf} in the setting of
VW theory. See \S\ref{sec:VW-setup} for details.

Let $S$ be a smooth projective surface and let $C_+(S) \subset
H^{\text{even}}(S; \bQ)$ be the sub-monoid of Chern characters $(r,
c_1, \ch_2)$ of coherent sheaves on $S$ with rank $r > 0$. For each
$\alpha = (r, c_1, \ch_2) \in C_+(S)$, choose a line bundle
$\cL(\alpha) \in \Pic(S)$ such that $c_1(\cL(\alpha)) = c_1$, such
that $\cL(\alpha+\beta) = \cL(\alpha) \otimes \cL(\beta)$ for all
$\alpha, \beta \in C_+(S)$. Define the moduli stacks
\[ \fN_{\alpha,\cL(\alpha)} \coloneqq \left\{\text{Higgs pairs } (\bar\cE, \phi) : \begin{array}{c} \alpha = (\rank \bar\cE, c_1(\bar\cE), \ch_2(\bar\cE)) \\ \det \bar\cE = \cL(\alpha), \, \tr \phi = 0 \end{array}\right\} \]
parameterizing {\it fixed determinant} and {\it trace-free} Higgs
sheaves. These moduli stacks carry the natural action of $\sT
\coloneqq \bC^\times$ scaling the Higgs field $\phi$ with $\sT$-weight
denoted $\kappa$, and $\sT$-equivariant symmetric obstruction theories
which may be used to construct virtual cycles on stable loci. For
brevity, we write $\fN_{\alpha,\cL}$ instead of
$\fN_{\alpha,\cL(\alpha)}$.

Let $H$ be a polarization of $S$, and let $\tau^H$ denote $H$-Gieseker
stability for Higgs pairs (Definition~\ref{def:VW-stability}).
Conventional VW theory uses Gieseker stability, but one may consider
other variations such as slope stability, Uhlenbeck stability
\cite{Tajakka2023}, and more exotic stability conditions. The
following results are stated for $\tau^H$ but hold equally well for
slope stability (which is used in their proofs).

\subsubsection{}

\begin{theorem}[Operational refined semistable VW invariants] \label{thm:VW-sst-invariants}
  There exists a unique collection
  \[ \left(\svw_\alpha(H) \in K_\circ^{\tilde\sT}(\fN_{\alpha,\cL})^\pl_{\loc,\bQ}\right)_{\alpha \in C_+(S)} \]
  of operational K-homology classes satisfying the following
  properties:
  \begin{enumerate}[label = (\roman*)]
  \item $\svw_\alpha(H)$ is supported on
    $(\Pi_\alpha^\pl)^{-1}(\fN_{\alpha,\cL}^{\sst}(\tau^H))$;
  \item for any $\alpha$ where $\fN_{\alpha,\cL}^{\sst}(\tau^H) =
    \fN_{\alpha,\cL}^{\st}(\tau^H)$,
    \[ (\Pi_\alpha^\pl)_*\svw_\alpha(H) = \chi\left(\fN^{\sst}_{\alpha,\cL}(\tau^H), \hat\cO^{\vir}_{\fN_{\alpha,\cL}^{\sst}(\tau^H)} \otimes -\right); \]
  \item if $H'$ is another polarization, and
    $\fN_{\alpha,\cL}^{\sst}(\tau^H) =
    \fN_{\alpha,\cL}^{\sst}(\tau^{H'})$ for all $\alpha$, then
    $\svw_\alpha(H) = \svw_\alpha(H')$ for all $\alpha$;
  \item for integers $k \gg 0$ \footnote{It suffices to take $k$ such
  that $\cE(k)$ is globally generated for any $H$-Gieseker semistable
  sheaf $\cE$ of class $\alpha$} and the \emph{refined pairs
  invariant} $\tilde\sVW_\alpha(H; k)$ associated to the refined pairs
    stack $\pi\colon \tilde\fN_{(\alpha,\cL),1}^{Q(k)} \to
    \fN_{\alpha,\cL}$ (\S\ref{sec:VW-auxiliary-stack}),
    \begin{enumerate}[label = (\arabic*)]
    \item if $h^1(\cO_S) = h^2(\cO_S) = 0$, then
      \begin{equation} \label{eq:VW-sst-invariants-pg-zero}
        I_*\tilde\sVW_\alpha(H; k) = \sum_{\substack{n>0 \\ \alpha = \alpha_1+\cdots+\alpha_n\\ \forall i: \,\tau^H(\alpha_i) = \tau^H(\alpha)}} \frac{1}{n!} \left[\iota^Q_*\svw_{\alpha_n}(H), \left[\cdots,\left[\iota^Q_*\svw_{\alpha_2}(H), \left[\iota^Q_*\svw_{\alpha_1}(H), \partial\right]\right]\cdots\right]\right],
      \end{equation}
    \item otherwise if $h^2(\cO_S) \neq 0$ then
      \begin{equation} \label{eq:VW-sst-invariants-pg-positive}
        \pi_*I_*\tilde\sVW_\alpha(H; k) = [\chi(\alpha(k))]_\kappa \cdot \svw_\alpha(H),
      \end{equation}
      and if $h^1(\cO_S) \neq 0$ then the same holds \emph{modulo $(1
        - \kappa)^{h^1(\cO_S)}$} (see \S\ref{sec:VW-proof-end}).
    \end{enumerate}
  \end{enumerate}
\end{theorem}

Note that if $\tau^H(\alpha_i) = \tau^H(\alpha)$ and $\rank \alpha >
0$ then $\rank \alpha_i > 0$ as well, so indeed the right hand side of
\eqref{eq:VW-sst-invariants-pg-zero} is well-defined. (Rank zero
objects have strictly larger $\tau^H$; see
Definition~\ref{def:VW-stability}.)

\subsubsection{}

\begin{theorem}[Wall-crossing formula for operational refined semistable VW invariants] \label{thm:VW-WCF}
  Let $H_1$ and $H_2$ be two polarizations on $S$.
  \begin{enumerate}[label = (\arabic*)]
  \item If $h^1(\cO_S) = h^2(\cO_S) = 0$, then
    \begin{equation} \label{eq:VW-WCF-pg-zero}
      \svw_\alpha(H_2) = \sum_{\substack{n > 0 \\ \alpha = \alpha_1 + \cdots + \alpha_n\\\forall i: \rank(\alpha_i) > 0}} \tilde U(\alpha_1,\ldots,\alpha_n; \tau^{H_1}, \tau^{H_2}) \left[\left[\cdots\left[\svw_{\alpha_1}(H_1),\svw_{\alpha_2}(H_1)\right],\cdots\right],\svw_{\alpha_n}(H_1)\right].
    \end{equation}

  \item Otherwise if $h^2(\cO_S) \neq 0$, then
    \[ \svw_\alpha(H_2) = \svw_\alpha(H_1) \]
    and if $h^1(\cO_S) \neq 0$ then the same holds \emph{modulo $(1 -
      \kappa)^{h^1(\cO_S)}$} (see \S\ref{sec:VW-proof-end}).
  \end{enumerate}
\end{theorem}

The proofs of Theorems~\ref{thm:VW-sst-invariants} and
\ref{thm:VW-WCF} are given in \S\ref{sec:VW-WCF}. The independence on
$H$ when $h^2(\cO_S) > 0$ agrees with a similar result in Donaldson
theory \cite[\S 2.5]{Witten1994} and with expectations from physics
(in unrefined Vafa--Witten theory) \cite[\S 5]{Vafa1994}.

\subsubsection{}

\begin{corollary}[Numerical refined semistable VW invariants and their wall-crossing formulas] \label{cor:numerical-VW}
  \leavevmode
  \begin{enumerate}[label = (\roman*)]
  \item \label{it:numerical-VW-sst-invariants} There exist
    $\vw_\alpha(H) \in \bQ(\kappa^{\frac{1}{2}})$, namely
    $\vw_\alpha(H) \coloneqq \svw_\alpha(H)(\cO)$, such that:
    \begin{enumerate}[label = (\arabic*)]
    \item if $h^1(\cO_S) = h^2(\cO_S) = 0$, then
      \begin{equation} \label{eq:numerical-VW-sst-invariants-pg-zero}
        \tilde\sVW_\alpha(H; k)(\cO) = \sum_{\substack{n>0 \\ \alpha = \alpha_1+\cdots+\alpha_n\\ \forall i: \,\tau^H(\alpha_i) = \tau^H(\alpha)}} \frac{1}{n!} \prod_{i=1}^n \Big[\chi(\alpha_i(k)) + \chi\Big(\alpha_i, \sum_{j=1}^{i-1} \alpha_j\Big)\Big]_\kappa \vw_{\alpha_i}(H);
      \end{equation}
    \item otherwise if $h^2(\cO_S) \neq 0$ then
      \[ \tilde\sVW_\alpha(H; k)(\cO) = [\chi(\alpha(k))]_\kappa \cdot \vw_\alpha(H), \]
      and if $h^1(\cO_S) \neq 0$ then the same holds \emph{modulo $(1 -
      \kappa)^{h^1(\cO_S)}$} (see \S\ref{sec:VW-proof-end}).
    \end{enumerate}

  \item \label{it:numerical-VW-WCF} Given two polarizations $H_1$ and
    $H_2$ on $S$:
    \begin{enumerate}[label = (\arabic*)]
    \item if $h^1(\cO_S) = h^2(\cO_S) = 0$, then
      \begin{equation} \label{eq:numerical-VW-WCF-pg-zero}
        \vw_\alpha(H_2) = \sum_{\substack{n > 0 \\ \alpha = \alpha_1 + \cdots + \alpha_n\\\forall i: \rank(\alpha_i) > 0}} \tilde U(\alpha_1,\ldots,\alpha_n; \tau^{H_1}, \tau^{H_2}) \prod_{i=1}^n \Big[\chi\Big(\alpha_i, \sum_{j=1}^{i-1} \alpha_j\Big)\Big]_\kappa \vw_{\alpha_i}(H);
      \end{equation}
    \item otherwise if $h^2(\cO_S) \neq 0$ then $\vw_\alpha(H_2) =
      \vw_\alpha(H_1)$, and if $h^1(\cO_S) \neq 0$ then the same holds
      \emph{modulo $(1 - \kappa)^{h^1(\cO_S)}$} (see
      \S\ref{sec:VW-proof-end}).
    \end{enumerate}
  \end{enumerate}
\end{corollary}

Part~\ref{it:numerical-VW-sst-invariants} is the main theorem of
\cite{liu_ss_vw}, originally a conjecture of
\cite{Thomas2020}. \footnote{In the $h^1(\cO_S)$ case, Thomas
  conjectures that the caveat ``modulo $(1 - \kappa)^{h^1(\cO_S)}$''
  is unnecessary. We are unsure whether this is true.} Thus
Theorem~\ref{thm:VW-sst-invariants} may be viewed as the ``true''
source of such numerical VW invariants.

Part~\ref{it:numerical-VW-WCF} is new, although when $h^2(\cO_S) > 0$
it was known in special cases as a consequence of a universality
theorem for semistable=stable numerical VW invariants \cite[Cor.
  1.12]{Laarakker2020}.

\subsection{Notation, conventions, acknowledgements}

\subsubsection{}

We work over $\bC$. All Artin stacks, i.e. algebraic stacks, are
assumed to be locally of finite type. Given an algebraic group $\sG$
(i.e. group scheme of finite type), let $\cat{Art}_\sG$ be the strict
$2$-category of Artin stacks with $\sG$-action \cite{Romagny2005}. We
will rarely use the $2$-categorical structure. Equalities of morphisms
will take place in the homotopy category of $\cat{Art}_\sG$. When
$\sG$ is trivial, we omit the subscript $\sG$ from all relevant
objects.

\subsubsection{}

Where reasonable, we align our notation with Joyce's notation
\cite{Joyce2021}. Furthermore, we use the following notational
conventions.
\begin{itemize}
\item If $\fX = \bigsqcup_\alpha \fX_\alpha$ is a stack and $\cF$ is
  an object on $\fX$, then $\cF_\alpha$ is the restriction of $\cF$ to
  the component $\fX_\alpha$, and similarly for morphisms from $\fX$.

\item If $f$ is a function of classes $\alpha$, and $E \in \cat{A}$ is
  an object, it is often convenient to write $f(E)$ to mean $f$
  applied to the class of $E$.

\item In tuples, e.g. of real numbers, an entry $*$ denotes an
  arbitrary valid value. For example, the condition $\ch(I) = (1, 0,
  *, *)$ denotes the conditions $\ch_0(I) = 1$ and $\ch_1(I) = 0$.

\item The $i$-th cohomology sheaf of $E \in D\cat{Coh}(X)$ is denoted
  $\cH^i(E) \in \cat{Coh}(X)$, while its $i$-th hypercohomology is
  denoted $H^i(X, E) \in \cat{Vect}_{\bC}$ or $H^i(E)$ for short. In
  particular, if $\cE \in \cat{Coh}(X)$, its sheaf cohomology is
  denoted $H^i(\cE)$. Let $h^i(\cE) \coloneqq \dim H^i(\cE)$.
\end{itemize}

\subsubsection{}

This project is very much inspired by the monumental work
\cite{Joyce2021} of Joyce. A significant portion of the content in
\S\ref{sec:semistable-invariants} and \S\ref{sec:wall-crossing} is
based on the framework presented there.

During the arduous two-year-long course of this project, we benefited
from fruitful discussions with Dominic Joyce, Ivan Karpov, Miguel
Moreira, Georg Oberdieck, and Richard Thomas. H.L. would also like to
thank Yukinobu Toda for drawing attention to \cite{Toda2013}. The name
``inert''
(Assumption~\ref{es:assump:semistable-invariants}\ref{es:assump:it:inert-classes})
was suggested by Gemini 3.

N.K. was supported by Research Council of Norway grant number 302277 -
”Orthogonal gauge duality and non-commutative
geometry” and by EPSRC grant number EP/X040674/1. H.L. was
supported by World Premier International Research Center Initiative
(WPI), MEXT, Japan.

\section{Background and setup}
\label{sec:background}

\subsection{Equivariant K-theory}
\label{sec:k-theory}

\subsubsection{}

\begin{definition} \label{def:equivariant-k-theory}
  Given an Artin stack $\fX \in \cat{Art}_\sT$, let
  \begin{align*}
    K_\sT(\fX) &\coloneqq K_0(D_{\cat{Coh},\sT}(\fX)), \\
    K_\sT^\circ(\fX) &\coloneqq K_0(D_{\cat{Perf},\sT}(\fX)),
  \end{align*}
  be the Grothendieck groups of $\sT$-equivariant quasi-coherent
  complexes on $\fX$ which are coherent and perfect, respectively.
  Both are modules for the representation ring
  \[ \bk_\sT \coloneqq K_\sT(\pt) \cong \bZ[t^\mu : \mu \in \Char\sT]. \]
  Let $\bk_{\sT,\loc} \coloneqq \Frac \bk_\sT$. Given a
  $\bk_\sT$-module $M$, let $M_\loc \coloneqq M \otimes_{\bk_\sT}
  \bk_{\sT,\loc}$. We refer to $K_\sT(\fX)_{\loc}$ and
  $K^\circ_\sT(\fX)_{\loc}$ as {\it localized} K-groups.
\end{definition}

\subsubsection{}
\label{sec:thomason-trobaugh-vs-quillen}

In the literature, our $K(-)$ is often called {\it G-theory} and our
$K^\circ(-)$ is often called {\it Thomason--Trobaugh K-theory}. We
will not use this terminology past this subsection. One may also
consider the {\it Quillen K-theory}
\[ K_0(\cat{Vect}_\sT(\fX)) \]
where $\cat{Vect}_\sT(\fX)$ is the exact category of $\sT$-equivariant
locally free $\cO_{\fX}$-modules. A locally free $\cO_{\fX}$-module is
in particular a perfect complex, so there is a natural map
\[ K_0(\cat{Vect}_\sT(\fX)) \to K_\sT^\circ(\fX). \]
If $\fX$ has the {\it $\sT$-equivariant resolution property}, i.e.
every $\sT$-equivariant coherent $\cO_{\fX}$-module is a quotient of a
$\sT$-equivariant locally free $\cO_{\fX}$-module, then this map is an
isomorphism \cite{Totaro2004}.

While Quillen K-theory has the advantage that it naturally admits
characteristic classes (e.g. see \S\ref{def:k-theoretic-wedge}),
Thomason--Trobaugh K-theory is better-behaved for stacks.

\subsubsection{}
\label{sec:k-theory-basic-properties}

We review some basic properties of $K_\sT(-)$ and $K_\sT^\circ(-)$.
Unless indicated otherwise, we will only consider $K_\sT(-)$ only for
finite-type algebraic spaces.
\begin{itemize}
\item The (derived) tensor product $\cE \otimes -$, for $\cE \in
  K_\sT^\circ(\fX)$, is a well-defined $\bk_\sT$-linear operator on
  both $K^\circ_\sT(\fX)$ and $K_\sT(\fX)$. In particular,
  $K^\circ_\sT(\fX)$ is a $\bk_\sT$-algebra and $K_\sT(\fX)$ is a
  $K^\circ_\sT(\fX)$-module.

\item The canonical inclusion $D_{\cat{Perf},\sT}(\fX) \subset
  D_{\cat{Coh},\sT}(\fX)$ induces a $\bk_\sT$-linear map
  $K_\sT^\circ(\fX) \to K_\sT(\fX)$ which is neither injective nor
  surjective in general.

\item A $\sT$-equivariant morphism $f\colon \fX \to \fY$ induces a
  functorial pullback $f^*\colon K^\circ_\sT(\fY) \to
  K^\circ_\sT(\fX)$. If $f$ is proper and representable, there is a
  functorial pushforward $f_*\colon K_\sT(\fX) \to K_\sT(\fY)$
  satisfying the {\it projection formula}
  \[ f_*(\cF) \otimes \cE = f_*(\cF \otimes f^*\cE), \qquad \cF \in K_\sT(\fX) \text{ or } K_\sT^\circ(\fX), \; \cE \in K_\sT^\circ(\fY). \]
  If in addition $f$ is of finite Tor-amplitude, there is a functorial
  pushforward $f_*\colon K^\circ_\sT(\fX) \to K^\circ_\sT(\fY)$
  satisfying the {\it projection formula}
  \[ \cF \otimes f_*\cE = f_*(f^*\cF \otimes \cE), \qquad \cF \in K_\sT(\fY) \text{ or } K_\sT^\circ(\fY), \; \cE \in K_\sT^\circ(\fX). \]
  In particular, both $f^*$ and $f_*$ are $\bk_\sT$-linear.
\end{itemize}

\subsubsection{}

\begin{definition} \label{def:k-theoretic-wedge}
  Given $\cE \in \cat{Vect}_\sT(\fX)$ and a formal variable $z$, let
  \[ \wedge_{-z}^\bullet(\cE) \coloneqq \sum_i (-z)^i \wedge^i(\cE) = \prod_\cL (1 - z\cL) \in K_0(\cat{Vect}_\sT(\fX))[z] \]
  be its exterior algebra in K-theory; the product ranges over
  (K-theoretic) Chern roots $\cL$ of $\cE$, by the splitting
  principle. Whenever a square root of $\det(\cE)$ is given, define
  also the {\it symmetrized} version
  \begin{align}
    \hat\wedge_{-z}^\bullet(\cE)
    &\coloneqq z^{-\frac{1}{2} \rank \cE} \det(\cE)^{-\frac{1}{2}} \wedge_{-z}^\bullet(\cE) \label{eq:k-theoretic-wedge-symmetrized} \\
    &= \prod_\cL (z^{-1/2}\cL^{-1/2} - z^{1/2}\cL^{1/2}) \in K_0(\cat{Vect}_\sT(\fX))[z^{\pm \frac{1}{2}}], \nonumber
  \end{align}
  Extend both $\wedge_{-z}^\bullet(-)$ and
  $\hat\wedge_{-z}^\bullet(-)$ to $K_0(\cat{Vect}_\sT(\fX))$ by
  defining
  \begin{equation} \label{eq:k-theoretic-wedge}
    \wedge_{-z}^\bullet(\cE_1 - \cE_2) \coloneqq \frac{\wedge_{-z}^\bullet(\cE_1)}{\wedge_{-z}^\bullet(\cE_2)} \in K_0(\cat{Vect}_\sT(\fX))[z]\pseries*{(1-z)^{-1}},
  \end{equation}
  using the splitting principle and the formula
  \begin{equation} \label{eq:k-theoretic-inverse-chern-root}
    \frac{1}{\wedge_{-z}^\bullet(\cL)} \coloneqq \frac{1}{\cL(1 - z) - (\cL - 1)} \coloneqq \cL^\vee \sum_{k \ge 0} \frac{(1 - \cL^\vee)^k}{(1 - z)^{k+1}},
  \end{equation}
  for Chern roots $\cL$ of $\cE_2$, to define the inverse of
  $\wedge_{-z}^\bullet(\cE_2)$.
\end{definition}

For $\cE \in \cat{Vect}_\sT(\fX)$, note the crucial symmetry
properties
\begin{equation} \label{eq:wedge-symmetry}
  \begin{aligned}
    \wedge_{-z}^\bullet(\cE) &= (-z)^{\rank \cE} \det(\cE) \wedge_{-z^{-1}}^\bullet(\cE^\vee) \\
    \hat\wedge_{-z}^\bullet(\cE) &= (-1)^{\rank \cE} \hat\wedge_{-z^{-1}}^\bullet(\cE^\vee).
  \end{aligned}
\end{equation}

\subsubsection{}

\begin{definition} \label{def:k-theoretic-euler-class}
  Suppose $\sT$ acts trivially on $\fX$. Then every $\sT$-equivariant
  locally free $\cO_{\fX}$-module decomposes as
  \[ \cE \eqqcolon \sum_\mu t^\mu \cE_\mu \]
  for a finite set of $\sT$-weights $\mu$. Its {\it K-theoretic Euler
    class} is
  \[ \se_z(\cE) \coloneqq \bigotimes_\mu \wedge_{-z^{-1} t^{-\mu}}^\bullet(\cE_\mu^\vee) \in K_0(\cat{Vect}_\sT(\fX))[z^{-1}] \]
  Extend this multiplicatively, using \eqref{eq:k-theoretic-wedge}, to
  $K_0(\cat{Vect}_\sT(\fX))$:
  \begin{equation} \label{eq:k-theoretic-euler-class}
    \se_z(\cE) \in K_0(\cat{Vect}_\sT(\fX))[z^{-1}]\pseries*{(1 - z^{-1} t^\mu)^{-1} : \mu \in \Char(\sT)}.
  \end{equation}
  For a given $\cE$, only the $\sT$-weights $t^\mu$ occuring in $\cE$
  will appear in $\se_z(\cE)$, so in particular only finitely many $(1
  - z^{-1} t^\mu)^{-1}$ need to be adjoined.

  Define $\hat\se(\cE)$ in the same way, using $\hat\wedge$ instead of
  $\wedge$.

  When $\cE$ has no piece of trivial $\sT$-weight, $\se(\cE) \coloneqq
  \se_1(\cE)$ and $\hat\se(\cE) \coloneqq \hat\se_1(\cE)$ are
  well-defined. Note that this is very different from
  $\wedge_{-z^{-1}}^\bullet(-)^\vee$ and
  $\hat\wedge_{-z^{-1}}^\bullet(-)^\vee$, which may not be
  well-defined at $z=1$.

  Note that $\se_z(-\cE)$ is well-defined and therefore an inverse for
  $\se_z(\cE)$ since $\se_z(-)$ is multiplicative by definition. By
  convention, we write this inverse as $\frac{1}{\se_z(\cE)}$.
\end{definition}

\subsubsection{}

\begin{lemma} \label{lem:k-theory-finiteness}
  Let $X$ be a finite-type algebraic space with trivial $\sT$-action,
  and
  \[ I^\circ(X) \subset K_0(\cat{Vect}(X)) \]
  denote the {\it augmentation ideal} of rank-$0$ elements. Then
  \begin{equation} \label{eq:k-theory-augmentation-nilpotence}
    I^\circ(X)^{\otimes N} K(X) = 0, \qquad \forall N \gg 0.
  \end{equation}
  Consequently, \eqref{eq:k-theoretic-euler-class} defines an operator
  \[ \se_z(\cE) \otimes \in \End(K_\sT(X))[z^{-1}]\left[(1 - z^{-1} t^\mu)^{-1} : \mu \in \Char(\sT)\right] \]
  which satisfies the symmetry property \eqref{eq:wedge-symmetry} as
  rational functions of $z$. Its expansion as a formal power series
  around $z^{-1} = 0$ is
  \begin{equation} \label{eq:k-theoretic-euler-class-expansion}
    \se_z(\cE_1 - \cE_2) \otimes \leadsto \sum_{i,j \ge 0} z^{-i-j} (-1)^i \wedge^i(\cE_1^\vee) \otimes \Sym^j(\cE_2^\vee) \otimes \in \End(K_\sT(X))\pseries*{z^{-1}},
  \end{equation}
  for $\cE_1, \cE_2 \in \cat{Vect}_\sT(X)$.
\end{lemma}

\begin{proof}
  The nilpotence \eqref{eq:k-theory-augmentation-nilpotence} is
  \cite[Lemma 2.4]{Edidin2000}. If $\cE$ is a vector bundle, then by
  the splitting principle
  \begin{align*}
    \wedge_{-z}^\bullet(\cE)
    &= \prod_\cL ((1 - \cL) + \cL(1 - z)) \\
    &\in \det(\cE) (1 - z)^{\rank \cE} \cdot (1 + I^\circ(X)[(1-z)^{-1}])
  \end{align*}
  since the coefficient of $(1 - z)^{-k}$ is a degree-$k$ symmetric
  polynomial in the variables $1 - \cL$. Hence both it and its inverse
  $\wedge_{-z}^\bullet(-\cE)$ have the property that the $(1 -
  z)^{-N}$ coefficient lies in $I^\circ(X)^{\otimes M(N)}$ for some
  function $M(N)$ such that $\lim_{N \to \infty} M(N) = \infty$. By
  \eqref{eq:k-theory-augmentation-nilpotence}, multiplication by the
  Euler class therefore becomes a finite sum. Finiteness of the sum
  means we may treat the Chern roots $\cL$ as generic non-zero complex
  numbers without worrying about convergence issues. The symmetry
  \eqref{eq:wedge-symmetry} and expansion
  \eqref{eq:k-theoretic-euler-class-expansion} follow immediately.
\end{proof}

\subsubsection{}

\begin{definition}[{\cite[Definition 3.1.6]{liu_eq_k_theoretic_va_and_wc}}] \label{def:k-theoretic-residue-map}
  Given a rational function $f \in \bk_{\sT \times \bC^\times,\loc}$
  with the coordinate on $\bC^\times$ denoted by $z$, let $f_+ \in
  \bk_{\sT,\loc}\lseries{z}$ and $f_- \in
  \bk_{\sT,\loc}\lseries{z^{-1}}$ be its formal series expansion
  around $z=0$ and $z=\infty$ respectively. The {\it
    ($\sT$-equivariant) K-theoretic residue map} is the
  $\bk_{\sT,\loc}$-module homomorphism
  \begin{align*}
    \rho_K\colon \bk_{\sT \times \bC^\times,\loc} &\to \bk_{\sT,\loc} \\
    f &\mapsto z^0 \text{ term in } (f_- - f_+).
  \end{align*}
  Treating $\sT$-weights as generic non-zero complex numbers, this is
  equivalent to
  \begin{align}
    \rho_K f
    &= -\left(\Res_{z=0} + \Res_{z=\infty}\right)\left(f \, \frac{dz}{z}\right) \label{eq:k-theoretic-residue-map-0-infty} \\
    &= \sum_{p \in \bC^\times} \Res_{z=p} \left(f \, \frac{dz}{z}\right) \label{eq:k-theoretic-residue-map-finite-poles}
  \end{align}
  where the second equality is the residue theorem. When there is
  ambiguity, we write $\rho_{K,z}$ to emphasize that the residue is
  being taken in the variable $z$.
\end{definition}

\subsubsection{}

\begin{lemma}[Projective bundle formula] \label{lem:projective-bundle-formula}
  Let $X$ be a finite-type algebraic space acted on by $\sT$, and $\cV
  \in \cat{Vect}_\sT(X)$. Let
  \[ \pi\colon \bP_X(\cV) \to X. \]
  be the natural projection, and $s = \cO_{\bP_X(\cV)}(1)$ be the dual of
  the universal line bundle on $\bP_X(\cV)$.
  \begin{enumerate}[label = (\roman*)]
  \item There is an isomorphism of $K_\sT^\circ(X)$-modules
    \[ K_\sT(\bP_X(\cV)) \cong \bigoplus_{k=0}^{\rank \cV-1} (s^k \otimes \pi^* K_\sT(X)). \]
  \item If $\sT$ acts trivially on $X$, then for any $f(s) \in
    K_\sT(X)[s^\pm]$ or $K_\sT^\circ(X)[s^\pm]$,
    \begin{equation} \label{eq:projective-bundle-formula}
      \pi_* f(s) = \rho_{K,z} \frac{f(z)}{\se_z(\cV)}.
    \end{equation}
  \end{enumerate}
\end{lemma}

When $\sT$ acts trivially, $K_\sT(X) = K(X) \otimes_{\bZ} \bk_\sT$ and
$K_\sT^\circ(X) = K^\circ(X) \otimes_{\bZ} \bk_\sT$; this is required
to define $1/\se_z(\cV)$ and to make sense of the operation $\rho_K$.
When the base $X$ is unambiguous, we omit the subscript on
$\bP_X(\cV)$ and write $\bP(\cV)$.

\begin{proof}
  The first claim follows from a standard argument using the Beilinson
  resolution of the diagonal, see e.g. \cite[\S 5.7]{Chriss1997}. The
  second claim is a special case of Jeffrey--Kirwan integration, see
  e.g. \cite[\S 2]{Szenes2004} or \cite[Appendix A]{Aganagic2018}, but
  we will verify it by direct calculation. By linearity, it is enough
  to take $f(s) = s^k$ for $k \in \bZ$. Let $r \coloneqq \rank \cV$.
  Then the left hand side is
  \begin{equation} \label{eq:projective-bundle-tautological-pushforward}
    \pi_* s^k = \begin{cases} \Sym^k \cV^\vee, & k \ge 0, \\ 0, & -r<k<0, \\ (-1)^{r+1}\det \cV \otimes \Sym^{-k-r} \cV, & k\le -r. \end{cases}
  \end{equation}
  To compute the right hand side, we use the expansion
  \eqref{eq:k-theoretic-euler-class-expansion} and the symmetry
  \eqref{eq:wedge-symmetry} to expand $1/\se_z(\cV)$ as a
  formal power series in $z^{-1}$ and $z$. By the definition of
  $\rho_K$, the right hand side is the difference of
  \begin{equation} \label{eq:projective-bundle-residue}
    \begin{aligned}
      z^0 \text{ term in } z^k \cdot (-z)^r \det(\cV) \otimes \sum_{i \ge 0} z^i \Sym^i(\cV) &= \begin{cases} (-1)^r \det(\cV) \otimes \Sym^{-k-r}(\cV), & k \le -r, \\ 0, & \text{else}, \end{cases} \\
      z^0 \text{ term in } z^k \cdot \sum_{i \ge 0} z^{-i} \Sym^i(\cV^\vee) &= \begin{cases} \Sym^k(\cV^\vee) & k \ge 0, \\ 0, & \text{else}. \end{cases}
    \end{aligned}
  \end{equation}
  We see that these match up exactly.
\end{proof}

\subsection{Equivariant operational K-homology}
\label{sec:equivariant-k-homology}

\subsubsection{}

\begin{definition}[{\cite[\S 2.2]{liu_eq_k_theoretic_va_and_wc}}] \label{def:operational-k-homology}
  Let $\bK^\sT_\circ(\fX)$ be the $\bk_{\sT}$-module which consists of
  all collections
  \[ \phi\coloneqq \{K^\circ_\sT(\fX \times S) \xrightarrow{\phi_S} K^\circ_\sT(S)\}_S \]
  of homomorphisms of $K^\circ_\sT(S)$-modules, for all $S \in
  \cat{Art}_\sT$ which we call the {\it base}, which obey the
  following axioms.
  \begin{enumerate}[label = (\alph*)]
  \item (Naturality) For any morphism $h\colon S \to S'$ in
    $\cat{Art}_\sT$, there is a commutative square
    \[ \begin{tikzcd}
        K^\circ_{\sT}(\fX \times S') \ar{r}{(\id \times h)^*} \ar{d}{\phi_{S'}} & K^\circ_{\sT}(\fX \times S) \ar{d}{\phi_S} \\
        K^\circ_{\sT}(S') \ar{r}{h^*} & K^\circ_{\sT}(S).
      \end{tikzcd} \]
  \item (Equivariant localization) There exists a proper algebraic
    space $\fF_\phi$ with the resolution property (see
    Remark~\ref{rem:localization-resolution-property}),
    equipped with the trivial $\sT$-action, and a map $\fix_\phi\colon
    \fF_\phi \to \fX$ in $\cat{Art}_\sT$, such that $\phi$ factors as
    \[ \begin{tikzcd}
        K_{\sT}^\circ(\fX \times S) \ar{rr}{\phi_S} \ar{dr}[swap]{(\fix_\phi \times \id)^*} && K_{\sT}^\circ(S) \\
        & K_{\sT}^\circ(\fF_\phi \times S) \ar{ur}[swap]{\phi_S^{\sT}}
      \end{tikzcd} \]
    for homomorphisms $\{\phi_S^{\sT}\}_S$ of
    $K_{\sT}^\circ(S)$-modules which themselves satisfy all other
    axioms, i.e. forming an element $\phi^\sT \in \bK_\circ^\sT(\fF)$.
  \item (Finiteness) for any proper algebraic space $S$ with the
    resolution property,
    \begin{equation} \label{eq:operational-k-homology-finiteness-condition}
      \phi_S^\sT\left(I^\circ(\fF \times S)^{\otimes N}\right) = 0, \qquad \forall N \gg 0.
    \end{equation}
  \end{enumerate}
  The sum $\phi + \psi$ of two elements $\phi, \psi \in
  \bK_\circ^\sT(\fX)$ still satisfies the equivariant localization and
  finiteness axioms by setting $\fF_{\phi + \psi} \coloneqq \fF_\phi
  \sqcup \fF_\psi$.
  
  Similarly, define the {\it localized} groups
  $\bK_\circ^\sT(\fX)_{\loc}$ by replacing all groups
  $K_{\sT}^\circ(-)$ with the localized groups
  $K_{\sT}^\circ(-)_{\loc}$.
\end{definition}

\subsubsection{}
\label{sec:k-homology-shorthand}

To prevent notational clutter, and for clarity, we generally write
formulas involving elements $\phi$ in a K-homology group
$\bK_\circ^\sT(\fX)$ in terms of just the functional $\phi_S$ for $S =
\pt$. It will always be clear how to extend the formula to $\phi_S$ for
general $S$. For instance, in \eqref{eq:k-homology-pushforward} below,
the definition $(f_*\phi)(\cE) \coloneqq \phi(f^*\cE)$ means that
\[ f_*\phi \coloneqq \{(f_*\phi)_S\colon K^\circ_\sT(\fY \times S) \to K^\circ_\sT(S)\} \in \bK^\sT_\circ(\fY) \]
is defined by $(f_*\phi)_S(\cE) \coloneqq \phi_S((f \times \id)^*\cE)$
for every base $S$.

\subsubsection{}
\label{sec:k-homology-properties}

We list some properties of $\bK_\circ^\sT(-)$, mostly inherited from
$K_\sT^\circ(-)$.
\begin{itemize}
\item Tensor product on $K^\circ_\sT(\fX)$ induces a cap product
  \begin{equation} \label{eq:k-homology-cap}
    \cap\colon \bK^\sT_\circ(\fX) \otimes K^\circ_\sT(\fX) \to \bK^\sT_\circ(\fX), \qquad (\phi \cap \cF)(\cE) \coloneqq \phi(\cE \otimes \cF).
  \end{equation}
  This is well-defined since $I^\circ(\fF_\phi)$ is an ideal.

\item A $\sT$-equivariant morphism $f\colon \fX \to \fY$ induces a
  functorial pushforward
  \begin{equation} \label{eq:k-homology-pushforward}
    f_*\colon \bK^\sT_\circ(\fX) \to \bK^\sT_\circ(\fY), \qquad (f_*\phi)(\cE) \coloneqq \phi(f^*\cE)
  \end{equation}
  which satisfies the {\it push-pull} formula
  \[ f_*(\phi) \cap E = f_*(\phi \cap f^*E), \qquad \phi \in \bK_\circ^\sT(\fX), \; E \in K^\circ_\sT(\fY). \]

\item A proper and representable $\sT$-equivariant morphism $f\colon
  \fX \to \fY$ of finite Tor-amplitude induces a functorial pullback
  \begin{equation} \label{eq:k-homology-pullback}
    f^*\colon \bK^\sT_\circ(\fY) \to \bK^\sT_\circ(\fX), \qquad (f^*\phi)(\cE) \coloneqq \phi(f_*\cE),
  \end{equation}
  which respects cap product, i.e.
  \[ f^*\phi \cap f^*E = f^*(\phi \cap E), \qquad \phi \in \bK_\circ^\sT(\fY), \; E \in K^\circ_\sT(\fY). \]

\item There is an external tensor product
  \begin{equation} \label{eq:k-homology-boxtimes}
    \boxtimes\colon \bK_\circ^\sT(\fX) \otimes \bK_\circ^\sT(\fY) \to \bK_\circ^\sT(\fX \times \fY), \qquad (\phi \boxtimes \psi)_S \coloneqq \psi_S \circ \phi_{\fY \times S}.
  \end{equation}
\end{itemize}
All these operations are homomorphisms of $\bk_\sT$-modules.

\subsubsection{}

\begin{definition} \label{def:k-homology-theory}
  View $\bK_\circ^\sT(-)$
  (Definition~\ref{def:operational-k-homology}) as a covariant functor
  from $\cat{Art}_\sT$ into $\bk_\sT$-modules via
  \eqref{eq:k-homology-pushforward}. A {\it $\sT$-equivariant
    operational K-homology theory} $\bfK_\circ^\sT(-) =
  \bfK_\circ^\sT(-)$ is a sub-functor
  \[ \bfK_\circ^\sT(-) \subset \bK_\circ^\sT(-) \]
  which is closed under the cap product \eqref{eq:k-homology-cap}, the
  pullback \eqref{eq:k-homology-pullback}, and the external tensor
  product \eqref{eq:k-homology-boxtimes}. Similarly, use
  $\bK_\circ^\sT(-)_{\loc}$ to define the notion of {\it localized}
  K-homology theory.

  We say $\bfK_\circ^\sT(-)$ is a {\it commutative} K-homology theory
  (see \cite[\S 2.2]{Fulton1981}) if for any $\fX, \fX' \in
  \cat{Art}_\sT$ and $\phi \in \bfK_\circ^\sT(\fX)$ and $\psi \in
  \bfK_\circ^\sT(\fX')$,
  \[ \phi \boxtimes \psi = \psi \boxtimes \phi. \]
\end{definition}

\subsubsection{}

\begin{definition} \label{def:supported-on}
  Let $\bfK_\circ^\sT(-)$ be a K-homology theory. An element $\phi \in
  \bfK_\circ^\sT(\fX)$ is {\it supported on} a $\sT$-invariant
  substack $\iota\colon \fX' \hookrightarrow \fX$ if there exists
  $\phi' \in \bfK_\circ^\sT(\fX')$ such that
  \[ \iota_* \phi' = \phi. \]
  As for usual homology class or for sheaves, we often omit $\iota_*$
  and simply write $\phi' = \phi$.
\end{definition}

\subsubsection{}

\begin{definition} \label{def:concrete-k-homology}
  Here is the ``concrete'' K-homology theory we will use throughout
  this paper. Define $K_\circ^\sT(\fX)_{\loc}$ to be the set of all
  triples $\phi = (Z_\phi, \fix_\phi, \cF_\phi)$ where:
  \begin{itemize}
  \item $Z_\phi$ is a proper algebraic space with the resolution
    property;
  \item $\fix_\phi\colon Z_\phi \to \fX$ is a $\sT$-equivariant
    morphism for the trivial $\sT$-action on $Z_\phi$;
  \item $\cF_\phi \in K_\sT(Z_\phi)_{\loc}$ is a K-theory element.
  \end{itemize}
  Equip it with the group operation where $\fix_{\phi+\psi}$ is the
  obvious map from $Z_{\phi+\psi} \coloneqq Z_\phi \sqcup Z_\psi$ and
  $\cF_{\phi+\psi} \coloneqq \cF_\phi + \cF_\psi$. For the remainder
  of this paper, $K_\circ^\sT(\fX)_{\loc}$ will denote its image
  inside $\bK_\circ^\sT(\fX)_{\loc}$, defined by the following
  Proposition~\ref{prop:actual-k-homology}. 
\end{definition}

\subsubsection{}

\begin{proposition}[{\cite[Propositions 2.2.10, 2.3.5]{liu_eq_k_theoretic_va_and_wc}}] \label{prop:actual-k-homology}
  $K_\circ^\sT(-)_{\loc}$ defines a localized commutative K-homology
  theory. Furthermore, for the trivial $\sT$-action,
  \begin{align}
    K_\circ^\sT(\pt)_{\loc} &= \bK_\circ^\sT(\pt)_{\loc} = \bk_{\sT,\loc}, \nonumber \\
    K_\circ^\sT([\pt/\bC^\times])_{\loc} &= \bK_\circ^\sT([\pt/\bC^\times])_{\loc} = \bk_{\sT,\loc}[\xi], \label{eq:k-homology-BGm}
  \end{align}
  where if $K([\pt/\bC^\times]) = \bZ[s^\pm]$ then $\xi^k(s^n) =
  (-1)^k \binom{n}{k}$.
\end{proposition}

\begin{proof}[Proof sketch.]
  Let $\phi = (Z_\phi, \fix_\phi, \cF_\phi) \in
  K_\circ^\sT(\fX)_{\loc}$. Given a base $S \in \cat{Art}_\sT$, define
  \begin{align}
    \phi_S\colon K^\circ_\sT(\fX \times S)_{\loc} &\to K^\circ_\sT(S)_{\loc} \nonumber \\
    \cE &\mapsto (\pi_S)_*(\pi_Z^*\cF_\phi \otimes (\fix_\phi \times \id)^*(\cE)) \label{eq:actual-k-homology-elements}
  \end{align}
  where $\pi_Z$ and $\pi_S$ are the projections from $Z_\phi \times S$
  to the $Z_\phi$ and $S$ factors respectively. Since both
  $\pi_Z^*\cF_\phi$ and $(\fix_\phi \times \id)^*\cE$ are flat with
  respect to $\pi_S$, the output of $\phi_S$ indeed defines an element
  of $K^\circ_\sT(S)_{\loc}$. Some straightforward base change
  formulas show that $\{\phi_S\}_S$ defines an element of
  $\bK_\circ^\sT(\fX)_{\loc}$; in particular, the finiteness axiom
  follows from Lemma~\ref{lem:k-theory-finiteness}.

  The verification that $K_\circ^\sT(-)_{\loc}$ is closed under cap
  product, pushforward, pullback, and external tensor is
  straightforward. Note that the product of two algebraic spaces with
  the resolution property still has the resolution property
  \cite{Totaro2004, Gross2017}.

  To compute $\bK_\circ^\sT(\fX)$ for $\fX = \pt$ and
  $[\pt/\bC^\times]$, it suffices by the naturality axiom to compute
  the sub-module of $\Hom_{\bZ}(K^\circ(\fX), \bZ)$ consisting of
  elements satisfying the finiteness axiom. This is easy since
  $K^\circ([\pt/G]) = R(G)$ is the representation ring of $G$. To show
  that $K_\circ^\sT(\fX) = \bK_\circ^\sT(\fX)$, it then suffices to
  verify that all elements in $\bK_\circ^\sT(\fX)$ may be written in
  the form \eqref{eq:actual-k-homology-elements}. Explicitly,
  \[ (\xi^k_S)^\sT = (-1)^k (\pi_S)_*\left(\pi_Z^*\cO_{\bP^k}(-k) \otimes (\fix \times \id)^*(-)\right) \]
  where $Z \coloneqq \bP^k$ and $\fix\colon Z = (\bC^{k+1} \setminus
  \{0\})/\bC^\times \hookrightarrow [\bC^{k+1}/\bC^\times] \to
  [\pt/\bC^\times]$ is the natural open embedding followed by
  projection.
\end{proof}

Note that \eqref{eq:k-homology-BGm} agrees with the topological
K-homology of the topological realization $B\bC^\times =
\bC\bP^\infty$ \cite[\S II.4]{Adams1974}.

\subsubsection{}
\label{sec:universal-invariants-shorthand}

It is clear from the proof of Proposition~\ref{prop:actual-k-homology}
that we view an element $\phi = (Z_\phi, \fix_\phi, \cF_\phi) \in
K_\circ^\sT(\fX)_{\loc}$ as an operator
\[ \phi = \chi\left(Z_\phi, \cF_\phi \otimes \fix_\phi^*(-)\right)\colon K_\sT^\circ(\fX)_{\loc} \to \bk_{\sT,\loc} \]
which behaves well with respect to base change. Hence we adopt this
shorthand notation for elements in $K_\circ^\sT(-)$, cf.
\S\ref{sec:k-homology-shorthand}. For instance, given $\psi = (Z_\psi,
\fix_\psi, \cF_\psi) \in K_\circ^\sT(\fX)_{\loc}$,
\begin{equation} \label{eq:universal-invariants-tensor-product}
  \phi \boxtimes \psi = \chi\left(Z_\phi \times Z_\psi, \left(\cF_\phi \boxtimes \cF_\psi\right) \otimes \left(\fix_\phi \times \fix_\psi\right)^*(-)\right).
\end{equation}

\subsubsection{}

\begin{remark} \label{rem:cohomology-theory}
  More generally, our wall-crossing framework only requires a pair of
  functors $F_*^\sT(-)$ (homology) and $F^*_\sT(-)$ (cohomology), and
  a pairing $F_*^\sT(X) \otimes F^*_\sT(X) \to F^*_\sT(\pt)$ on
  algebraic spaces $X \in \cat{Art}_\sT$. The cohomology functor
  $F^*_\sT(-)$ must have analogues of all the properties of
  $K_\sT^\circ(-)$ given in \S\ref{sec:k-theory}. In particular, it
  must have Euler classes, a residue map, and a projective bundle
  formula. The homology functor $F_*^\sT(-)$ must be commutative
  (Definition~\ref{def:k-homology-theory}) and satisfy the analogue of
  \eqref{eq:k-homology-BGm} along with all the properties given in
  this subsection. A suitable $F_*^\sT(-)$ may always be constructed
  from $F^*_\sT(-)$ in an ``operational'' manner, as in
  Definition~\ref{def:concrete-k-homology} or
  Definition~\ref{def:homology-theory-operational} below.

  In \S\ref{sec:cohomological-version}, we explain how to take
  $F^*_\sT(-)$ to be {\it operational Chow cohomology}, and
  $F_*^\sT(-)$ to be its operational dual. See also
  Remark~\ref{rem:khan-homology}.
\end{remark}

\subsection{Monoidal stacks and vertex algebras}
\label{sec:vertex-algebra}

\subsubsection{}
\label{sec:series-rings}

In this subsection, we review a symmetrized version of the
constructions of \cite{liu_eq_k_theoretic_va_and_wc}. For formal
variables $z_1, \ldots, z_m$ and a $\bk_\sT$-module $M$, define
\begin{equation} \label{eq:equivariant-polynomial-ring}
  M\left[(1 - z_1 \cdots z_m)^{-1}\right]_\sT \coloneqq M\left[(1 - t^\mu z_1^{i_1} \cdots z_m^{i_m})^{-1} : \begin{array}{c} \mu \in \Char(\sT) \\ i_1,\ldots,i_m \in \bZ \setminus \{0\}\end{array}\right]
\end{equation}
Similarly define $M[(1-z_1\cdots z_m)^\pm]_\sT$. In the case of
multiple ``generators'', we use the usual convention for polynomial
rings that, for instance,
\[ M\left[(1 - z)^{-1}, (1 - w)^{-1}\right]_\sT \coloneqq M\left[(1 - z)^{-1}\right]_\sT\left[(1 - w)^{-1}\right]_\sT \]
for variables $z, w$. In the case of a single variable $z$, let
\[ M\lseries*{1 - z}_{\sT} \coloneqq M\pseries*{1-z}\left[(1 - z)^{-1}\right]_\sT. \]
Finally, define $M[(1 - z_1 \cdots z_m)^{-1}]_\sT^{(1)}$ by replacing
$\bZ \setminus \{0\}$ in \eqref{eq:equivariant-polynomial-ring} with
$\{1, -1\}$. This also defines $M[(1 - z_1 \cdots z_m)^\pm]_\sT^{(1)}$
and $M\lseries{1-z}_\sT^{(1)}$.

If $M$ is a ring, then all these modules are also rings under the
usual multiplication of polynomials and power series.

\subsubsection{}

\begin{definition} \label{def:expansion}
  For two formal variables $z$ and $w$, the {\it (multiplicative)
    expansion}
  \[ \iota_{z,w}\colon \bZ\left[(1 - zw)^\pm\right] \hookrightarrow \bZ\left[(1-w)^\pm\right]\pseries*{1 - z} \]
  is the injective ring homomorphism given by
  \begin{align*}
    \iota_{z,w}(1 - zw)^k
    &\coloneqq \left((1-z) + (1-w) - (1-z)(1-w)\right)^k \\
    &= (1-w)^k \left(1 - (1-z)(1 - (1-w)^{-1})\right)^k.
  \end{align*}
  Note that $1 - (1-z)(1 - (1-w)^{-1})$ is invertible, so this is
  well-defined for all $k \in \bZ$. The {\it (multiplicative)
    equivariant expansion}
  \[ \iota_{z,w}^\sT\colon \bk_\sT\left[(1 - zw)^\pm\right]_\sT \to \bk_\sT\left[(1-w)^\pm\right]_\sT\pseries*{1-z}. \]
  is the $\bk_\sT$-algebra homomorphism given by applying the
  expansion
  \[ \iota_{z^i,t^\mu w^j}\colon \bk_\sT\left[(1 - t^\mu z^i w^j)^\pm\right] \to \bk_\sT\left[(1 - t^\mu w^j)^\pm\right]\pseries*{1 - z^i} \]
  to the monomial $(1 - t^\mu z^i w^j)^k$, and using the isomorphism
  $\bZ\pseries{1-z^{-1}} \cong \bZ\pseries{1-z}$ as necessary.
\end{definition}

\subsubsection{}

\begin{definition} \label{def:multiplicative-vertex-algebra}
  A {\it $\sT$-equivariant multiplicative vertex algebra} is
  the data of:
  \begin{enumerate}[label = (\roman*)]
  \item a $\bk_\sT$-module $V$ of {\it states} with a distinguished
    vacuum vector $\vac \in V$;
  \item a multiplicative {\it translation operator} $D(z)
    \in \End(V)\pseries*{1 - z}$, i.e. $D(z)D(w) = D(zw)$;
  \item a {\it vertex product} $Y(-, z)\colon V \otimes V \to
    V\lseries{1-z}_\sT$.
  \end{enumerate}
  This data must satisfy the following axioms:
  \begin{enumerate}[label = (\alph*)]
  \item (vacuum) $Y(\vac, z) = \id$ and $Y(a, z)\vac \in V\pseries{1 -
    z}$ with $Y(a,1)\vac = a$;
  \item (skew symmetry) $Y(a, z)b = D(z) Y(b, z^{-1}) a$;
  \item (weak associativity) $Y(Y(a,z) b, w) \equiv Y(a, zw) Y(b, w)$,
    where $\equiv$ means that when applied to any $c \in V$, both
    sides are equivariant expansions of the same element in
    \begin{equation} \label{eq:mvoa-ope-ring}
      V\pseries*{1 - z, 1 - w}\left[(1 - z)^{-1}, (1 - w)^{-1}, (1 - zw)^{-1}\right]_{\sT}.
    \end{equation}
  \end{enumerate}
  A {\it morphism} $f\colon (V, \vac, D, Y) \to (V', \vac', D', Y')$
  of two such vertex algebras is a homomorphism $f\colon V \to V'$ of
  $\bk_\sT$-modules such that
  \[ f(\vac) = \vac', \qquad f \circ D = D' \circ f, \qquad f \circ Y = Y' \circ (f \otimes f). \]
\end{definition}

\subsubsection{}

\begin{definition} \label{def:graded-monoidal-stack}
  A {\it graded monoidal $\sT$-stack with symmetric bilinear element} is the
  data of:
  \begin{enumerate}[label = (\roman*)]
  \item an Artin stack $\fM = \bigsqcup_\alpha \fM_\alpha$, where
    $\alpha$ ranges over an additive monoid and the torus $\sT$ acts
    on each $\fM_\alpha$;
  \item for every $\alpha$ and $\beta$, $\sT$-equivariant morphisms
    \begin{align*}
      &\Phi_{\alpha,\beta}\colon \fM_\alpha \times \fM_\beta \to \fM_{\alpha+\beta}, \\
      &\Psi_\alpha \colon [\pt/\bC^\times] \times \fM_\alpha \to \fM_\alpha,
    \end{align*}
    and elements $\scE_{\alpha,\beta} \in K_\sT^\circ(\fM_\alpha \times
    \fM_\beta)$.
  \end{enumerate}
  This data must satisfy the following axioms, for all $\alpha$ and
  $\beta$.
  \begin{enumerate}[label = (\alph*)]
  \item (Identity) $\fM_0 = *$ consists of a single point.
  \item ($\Phi$ is monoidal) $\Phi_{0,\alpha} = \id = \Phi_{\alpha,0}$
    and $\Phi_{\alpha,\beta+\gamma} \circ (\id \times
    \Phi_{\beta,\gamma} \times \id) = \Phi_{\alpha+\beta,\gamma} \circ
    (\Phi_{\alpha,\beta} \times \id)$.
  \item ($\Psi$ is $[\pt/\bC^\times]$-action) Define $\Omega\colon
    [\pt/\bC^\times]^2 \to [\pt/\bC^\times]$ using multiplication on
    on $\bC^\times$. Then
    \begin{align*}
      \Psi_{\alpha+\beta} \circ (\id \times \Phi_{\alpha,\beta}) &= \Phi_{\alpha,\beta} \circ (\Psi_\alpha^{(12)}, \Psi_\alpha^{(13)}) \\
      \Psi_\alpha \circ (\id \times \Psi_\alpha) &= \Psi_\alpha \circ (\Omega \times \id)
    \end{align*}
    where superscripts $(ij)$ mean to act on the $i$-th and $j$-th
    factors.
  \item ($\scE_{\alpha,\beta}$ is bilinear) Let $\cL_{[\pt/\bC^\times]}
    \in K_\sT([\pt/\bC^\times])$ be the weight-$1$ representation.
    Then
    \begin{equation} \label{eq:mVOA-bilinear-complex-conditions}
      \begin{alignedat}{2}
        (\Phi_{\alpha,\beta} \times \id)^*(\scE_{\alpha+\beta,\gamma}) &= \scE_{\alpha,\gamma} \oplus \scE_{\beta,\gamma}, \qquad & (\Psi_\alpha \times \id)^*(\scE_{\alpha,\beta}) &= \cL^\vee \boxtimes \scE_{\alpha,\beta}, \\
        (\id \times \Phi_{\beta,\gamma})^*(\scE_{\alpha,\beta+\gamma}) &= \scE_{\alpha,\beta} \oplus \scE_{\alpha,\gamma}, \qquad & (\id \times \Psi_\beta)^*(\scE_{\alpha,\beta}) &= \cL \boxtimes \scE_{\alpha,\beta}.
      \end{alignedat}
    \end{equation}
    Some obvious pullbacks along projections to various factors have
    been omitted.
  \end{enumerate}
  We say that the bilinear element is {\it $\kappa$-symmetric} (cf.
  Definition~\ref{def:obstruction-theory}) if:
  \begin{enumerate}[resume, label = (\alph*)]
  \item ($\scE_{\alpha,\beta}$ is $\kappa$-symmetric) If $(12)\colon
    \fM_\beta \times \fM_\alpha \to \fM_\alpha \times \fM_\beta$ swaps
    the two factors, then
    \[ \scE_{\beta,\alpha} = -\kappa^{-1} \otimes (12)^*\scE_{\alpha,\beta}^\vee. \]
  \end{enumerate}
  A {\it morphism} $f\colon (\fM, \Phi, \Psi, \scE) \to (\fM', \Phi',
  \Psi', \scE')$ of two graded monoidal stacks with (symmetric)
  bilinear elements is a collection of morphisms $f_\alpha\colon
  \fM_\alpha \to \fM'_\alpha$ in $\cat{Art}_\sT$, for every $\alpha$,
  such that
  \[ \Phi'_{\alpha,\beta} \circ (f_\alpha \times f_\beta) = f_{\alpha+\beta} \Phi_{\alpha,\beta}, \qquad \Psi'_\alpha \circ f_\alpha = f_\alpha \circ \Psi_\alpha, \qquad (f_\alpha \times f_\beta)^*\scE_{\alpha,\beta} = \scE'_{\alpha,\beta}. \]
\end{definition}

\subsubsection{}

\begin{definition} \label{def:degree-map}
  Let the rigidification $\Pi\colon \fX \to \fX^\pl$ be a $\sT$-equivariant
  $\bC^\times$-gerbe. Then there is a natural $\sT$-equivariant map
  \[ \Psi\colon [\pt/\bC^\times] \times \fX \to \fX \]
  given by $(\pt, x) \mapsto x$ on $\bC$-points and $(\lambda, f)
  \mapsto (\lambda \cdot \id_x) \circ f$ on the associated stabilizer
  groups. It induces a {\it degree operator}
  \[ z^{\deg}\colon K_\sT^\circ(\fX) \xrightarrow{\Psi^*} K_\sT^\circ([\pt/\bC^\times] \times \fX) \cong K_\sT^\circ(\fX)[z^\pm] \]
  where we identify $K_\sT^\circ([\pt/\bC^\times]) =
  K_\sT([\pt/\bC^\times]) = \bk_\sT[z^\pm]$. If $\fX = \prod_i \fX_i$
  is a product of such $\bC^\times$-gerbes, we write $z^{\deg_i}$ for
  the operator defined by acting on the $i$-th factor by $\Psi^*$.
  Note that, by construction,
  \begin{equation} \label{eq:degree-map-pl-stack}
    z^{\deg} \Pi^*\cE = \Pi^*\cE, \qquad \cE \in K_\sT^\circ(\fX^\pl).
  \end{equation}
\end{definition}

\subsubsection{}

\begin{theorem}[{\cite[Theorem 3.3.5]{liu_eq_k_theoretic_va_and_wc}}] \label{thm:mVOA-monoidal-stack}
  Let $\fM = \bigsqcup_\alpha \fM_\alpha$ be a graded monoidal
  $\sT$-stack with $\kappa$-symmetric bilinear elements
  $\scE_{\alpha,\beta}$. Let $\bfK_\circ^{\tilde\sT}(-)$ be a
  commutative K-homology theory such that
  $\bfK^{\tilde\sT}_\circ([*/G]) = \bK^{\tilde\sT}_\circ([*/G])$ for
  $G = 1$ and $G = \bC^\times$. Then
  \[ \bfK^{\tilde\sT}_\circ(\fM) \coloneqq \bigoplus_\alpha \bfK^{\tilde\sT}_\circ(\fM_\alpha) \]
  has the structure of a $\tilde\sT$-equivariant multiplicative vertex
  algebra.
  \begin{itemize}
  \item The vacuum $\vac \in \bfK_\circ^{\tilde\sT}(\fM_0)$ is given by
    the identity map $\bk_{\tilde\sT} \to \bk_{\tilde\sT}$.

  \item The translation operator is
    \[ D(z)\phi \coloneqq \sum_{k \ge 0} (1-z)^k \Psi_*\left(\xi^k \boxtimes \phi\right) \in \bfK_\circ^{\tilde\sT}(\fM)\pseries{1-z}, \]
    where $\bfK_\circ^{\tilde\sT}([\pt/\bC^\times]) =
    \bk_{\tilde\sT}[\xi]$ as in Proposition~\ref{prop:actual-k-homology}.
    Explicitly, $(D(z)\phi)(\cE) = \phi(z^{\deg} \cE)$ where
    $z^{\deg}$ is the degree operator associated to $\Psi$
    (Definition~\ref{def:degree-map}).

  \item The vertex product is given on $\phi \in
    \bfK^{\tilde\sT}_\circ(\fM_\alpha)$ and $\psi \in
    \bfK^{\tilde\sT}_\circ(\fM_\beta)$ by
    \begin{equation} \label{eq:monoidal-stack-vertex-product}
      \begin{aligned}
        Y(\phi, z) \psi
        &\coloneqq (\Phi_{\alpha,\beta})_* (D(z) \times \id) \left((\phi \boxtimes \psi) \cap \hat\Theta_{\alpha,\beta}(z)\right) \\
        K^\circ_\sT(\fM_{\alpha+\beta}) \ni \cE
        &\mapsto (\phi \boxtimes \psi)\left(\hat\Theta_{\alpha,\beta}(z) \otimes z^{\deg_1} \Phi_{\alpha,\beta}^* \cE \right) \in \bk_{\tilde\sT}[(1-z)^\pm]_\sT^{(1)},
      \end{aligned}
    \end{equation}
    where (see \S\ref{sec:monoidal-stack-vertex-product-well-defined}
    for details)
    \begin{equation} \label{eq:mVOA-theta}
      \hat\Theta_{\alpha,\beta}(z) \coloneqq \hat\se_{z^{-1}}\left(\scE_{\alpha,\beta}\right) \otimes \hat\se_{z}\left((12)^*\scE_{\beta,\alpha}\right).
    \end{equation}
  \end{itemize}
  Furthermore, this construction is functorial: if $f\colon \fM \to
  \fM'$ is a morphism of graded monoidal $\sT$-stacks with symmetric
  bilinear elements, then
  \[ f_*\colon \bfK_\circ^{\tilde\sT}(\fM) \to \bfK_\circ^{\tilde\sT}(\fM') \]
  is a morphism of $\tilde\sT$-equivariant multiplicative vertex
  algebras.
\end{theorem}

This theorem, as well as the contents of the remainder of this
subsection, continue to hold with localized K-homology theories. By
Proposition~\ref{prop:actual-k-homology}, the ``concrete'' K-homology
theory $K_\circ^{\tilde\sT}(-)_{\loc}$ satisfies the conditions of
this theorem.

\begin{proof}[Proof sketch.]
  All vertex algebra axioms, as well as functoriality, follow
  completely formally from this and the graded monoidal structure of
  $\fM$, and the easy algebraic exercise that
  \[ (w^{\deg} \times \id)\hat\Theta(z) = \iota_{w,z}^{\tilde\sT}\hat\Theta(wz). \]
  Note that \cite[Theorem 3.3.5]{liu_eq_k_theoretic_va_and_wc} uses
  $\se$ (and $\sT$) instead of $\hat\se$ (and $\tilde\sT$) in
  \eqref{eq:mVOA-theta}, but this makes no difference in the
  verification of the axioms.
\end{proof}

\subsubsection{}
\label{sec:monoidal-stack-vertex-product-well-defined}

The formula \eqref{eq:mVOA-theta} is not well-defined, as written.
Instead, in \eqref{eq:monoidal-stack-vertex-product}, we {\it define}
\[ (\phi \boxtimes \psi) \cap \hat\Theta_{\alpha,\beta}(z) \coloneqq \fix_*\left[(\phi^\sT \boxtimes \psi^\sT) \cap \fix^*\hat\Theta_{\alpha,\beta}(z)\right] \]
where $\fix$ denotes $\fix_{\phi \boxtimes \psi}\colon \fF \to
\fM_\alpha \times \fM_\beta$ from the equivariant localization axiom
for the K-homology element $\phi \boxtimes \psi$. On the right hand
side, since $\fF$ is an algebraic space with trivial
$\tilde\sT$-action and the resolution property, the K-theoretic Euler
classes in
\begin{equation} \label{eq:Theta-on-fixed-locus}
  \fix^*\hat\Theta_{\alpha,\beta}(z) \coloneqq \hat\se_{z^{-1}}\left(\fix^*\scE_{\alpha,\beta}\right) \otimes \hat\se_{z}\left(\fix^* (12)^*\scE_{\beta,\alpha}\right) \in K_0(\cat{Vect}_{\tilde\sT}(\fF))[z^\pm]\pseries*{(1 - z^\pm t^\mu)^{-1}}
\end{equation}
are well-defined. (No square root of $z$ is required in
\eqref{eq:Theta-on-fixed-locus} because $\rank \scE_{\alpha,\beta} =
-\rank \scE_{\beta,\alpha}$.) Finally, from
\eqref{eq:k-theoretic-inverse-chern-root}, the degree-$N$ terms in
this power series ring are elements in $K^\circ(\fF)^{\otimes (N-K)}$
for some constant $K$ independent of $N$. Since
\[ (\phi^\sT \boxtimes \psi^\sT)_S(I^\circ(\fF \times S)^{\otimes (N-K)}) = 0 \qquad \forall N \gg 0 \]
by the finiteness axiom for $\phi \boxtimes \psi$, it follows that
\[ (\phi^\sT \boxtimes \psi^\sT) \cap \fix^*\hat\Theta_{\alpha,\beta}(z) \in \bfK_\circ^{\tilde\sT}(\fF)[(1-z)^\pm]_\sT^{(1)}. \]

\subsubsection{}

\begin{definition} \label{def:pl-groups}
  Given a $\sT$-equivariant multiplicative vertex algebra $(V, \vac,
  D, Y)$, let $\im(1 - D(z)) \subset V$ denote the
  $\bk_\sT$-submodule generated by the coefficients of
  \[ (1 - D(z))a \in V\pseries*{1-z} \]
  for all $a \in V$, and define
  \[ V^\pl \coloneqq V / \im(1 - D(z)). \]
\end{definition}

\subsubsection{}

\begin{lemma} \label{lem:k-homology-pl}
  The morphism $\Pi_\alpha^\pl\colon \fM_\alpha \to \fM_\alpha^\pl$
  induces
  \[ (\Pi_\alpha^\pl)_*\colon \bfK_\circ^{\tilde\sT}(\fM_\alpha)^\pl \to \bfK_\circ^{\tilde\sT}(\fM_\alpha^\pl). \]
  If $\Pi_\alpha^\pl$ admits a section $I_\alpha$, then this is an
  isomorphism of $\bk_\sT$-modules with inverse $(I_\alpha)_*$.
\end{lemma}

Note that $\Pi_\alpha^\pl$ admits a section if and only if it is a
trivial $[\pt/\bC^\times]$-bundle, i.e. $\fM_\alpha = \fM_\alpha^\pl
\times [\pt/\bC^\times]$ \cite[Lemma 3.21]{Laumon2000}.

This lemma will be how we relate abstract invariants in
$K_\circ^{\tilde\sT}(\fM_\alpha)^\pl_{\loc}$ with enumerative
invariants in $K_\circ^{\tilde\sT}(\fM_\alpha^\pl)_{\loc}$.

\begin{proof}
  It suffices to check that $(\Pi_\alpha^\pl)_* \im(1 - D(z)) = 0$.
  This follows from \eqref{eq:degree-map-pl-stack}. See \cite[Lemma
    3.3.17]{liu_eq_k_theoretic_va_and_wc} for details.
\end{proof}

\subsubsection{}

\begin{lemma} \label{lem:pl-group-functoriality}
  Let $f\colon \fM \to \fM'$ be a morphism of graded monoidal
  $\sT$-stacks.
  \begin{enumerate}[label = (\roman*)]
  \item The degree operator $z^{\deg}$ commutes with $f^*\colon
    K^\circ_\sT(\fM') \to K^\circ_\sT(\fM)$.
  \item The translation operator $D(z)$ commutes with $f_*\colon
    \bfK_\circ^\sT(\fM) \to \bfK_\circ^\sT(\fM')$.
  \end{enumerate}
  Consequently $f_*$ induces a map $f_*\colon \bfK_\circ^\sT(\fM)^\pl
  \to \bfK_\circ^\sT(\fM')^\pl$ of $\bk_\sT$-modules.
\end{lemma}

\begin{proof}
  Straightforward, since $z^{\deg}$ is defined by a pullback and
  pullbacks are functorial.
\end{proof}

\subsubsection{}

\begin{theorem}[{\cite[Theorem 3.2.13]{liu_eq_k_theoretic_va_and_wc}}] \label{thm:mVOA-monoidal-stack-lie-algebra}
  The multiplicative vertex algebra structure on
  $\bfK_\circ^{\tilde\sT}(\fM)$ induces a Lie algebra structure on
  $\bfK_\circ^{\tilde\sT}(\fM)^\pl$, with Lie bracket given by
  \begin{align} 
    (\kappa^{-\frac{1}{2}} - \kappa^{\frac{1}{2}}) \cdot [\phi, \psi](\cE)
    &\coloneqq \rho_{K,z} \left(Y(\tilde\phi, z)\tilde \psi\right)(\cE) \nonumber \\
    &= \rho_{K,z} \left[ (\tilde\phi \boxtimes \tilde\psi)\left(\hat\Theta_{\alpha,\beta}(z) \otimes z^{\deg_1} \Phi_{\alpha,\beta}^*\cE\right)\right] \label{eq:monoidal-stack-lie-bracket}
  \end{align}
  for $\phi \in \bfK_\circ^{\tilde\sT}(\fM_\alpha)^\pl$ and $\psi \in
  \bfK_\circ^{\tilde\sT}(\fM_\beta)^\pl$, and $\tilde\phi \in
  \bfK_\circ^{\tilde\sT}(\fM_\alpha)$ and $\tilde\psi \in
  \bfK_\circ^{\tilde\sT}(\fM_\beta)$ are any lifts of $\phi$ and
  $\psi$ respectively.
\end{theorem}

\begin{proof}[Proof sketch.]
  While $Y(\tilde\phi, z)\tilde\psi$ may involve power series in
  $1-z$, it becomes a rational function of $z$ upon application to any
  given K-theory class. Hence $\rho_{K,z}$ may be applied. Then
  \eqref{eq:monoidal-stack-lie-bracket} produces a bilinear operation
  \[ [-, -]\colon \bfK_\circ^{\tilde\sT}(\fM) \otimes \bfK_\circ^{\tilde\sT}(\fM) \to \bfK_\circ^{\tilde\sT}(\fM) \]
  which one can check preserves the submodule $\im(1 - D(z))$. Hence
  $[-, -]$ induces a bilinear operation on
  $\bfK_\circ^{\tilde\sT}(\fM)^\pl$. Finally, the anti-symmetry and
  Jacobi identities for $[-, -]$ follow from skew symmetry and weak
  associativity for the vertex algebra.
\end{proof}

\subsubsection{}

\begin{proposition}[Rigidity] \label{prop:rigidity}
  For any $\phi \in \bfK_\circ^{\tilde\sT}(\fM_\alpha)^\pl$ and $\psi
  \in \bfK_\circ^{\tilde\sT}(\fM_\beta)^\pl$
  \[ [\phi, \psi](\cO) = (\phi \boxtimes \psi)\left(\left[\rank \scE_{\alpha,\beta}\right]_\kappa \cdot (\cO \boxtimes \cO)\right) \]
  where $[n]_\kappa$ is the quantum integer from
  \eqref{eq:quantum-integer}.
\end{proposition}

In general, $\rank \scE_{\alpha,\beta}$ may not be a constant function
on $\fM_\alpha \times \fM_\beta$. In practice, it is, in which case we
denote it by $\chi(\alpha, \beta)$ and we may write
\[ (\phi \boxtimes \psi)\left(\left[\rank \scE_{\alpha,\beta}\right]_\kappa\right) = [\chi(\alpha,\beta)]_\kappa \phi(\cO) \psi(\cO). \]

\begin{proof}
  Plug $\cE = \cO$ into the formula
  \eqref{eq:monoidal-stack-lie-bracket} for the Lie bracket. Since
  $\Phi_{\alpha,\beta}^*\cO = \cO \boxtimes \cO$ and $z^{\deg}\cO =
  \cO$, it suffices to show that
  \[ \rho_K \fix^*\hat\Theta_{\alpha,\beta}(z) = (-1)^{-\rank \fix^*\scE_{\alpha,\beta}} \left(\kappa^{\frac{1}{2} \rank \fix^*\scE_{\alpha,\beta}} - \kappa^{-\frac{1}{2} \rank \fix^*\scE_{\alpha,\beta}}\right). \]
  This is the exact same computation as in the proof of
  Corollary~\ref{cor:symmetrized-projective-bundle-formula}.
\end{proof}

\subsection{Auxiliary framed stacks}
\label{sec:auxiliary-stacks}

\subsubsection{}

\begin{definition}\label{bg:def:framing-functor}
  A {\it framing functor} for the abelian category $\cat{A}$ is a
  $\bC$-linear exact functor
  \[ \Fr\colon \cat{A}^{\Fr} \to \cat{Vect} \]
  on a full exact $\sT$-invariant subcategory $\cat{A}^{\Fr} \subset
  \cat{A}$, closed under isomorphisms in $\cat{A}$ and direct summands
  in $\cat{A}$ (i.e. if $E, F \in \cat{A}$ with $E \oplus F \in
  \cat{A}^{\Fr}$, then $E, F \in \cat{A}^{\Fr}$), such that:
  \begin{enumerate}[label=(\alph*)]
  \item the moduli substack $\fM^{\Fr}_\alpha \subset \fM_\alpha$, of
    objects in $\cat{A}^{\Fr}$ of class $\alpha$, is open, and $\Fr$
    induces morphisms of moduli stacks
    \begin{align*}
      \fM^{\Fr}_\alpha &\to \bigsqcup_{d \ge 0} [\pt/\GL(d)] \\
      \fM^{\Fr,\pl}_\alpha &\to \bigsqcup_{d \ge 0} [\pt/\PGL(d)];
    \end{align*}
  \item $\Hom(E, E) \to \Hom(\Fr(E), \Fr(E))$ is injective for all $E
    \in \cat{A}^{\Fr}$, so $\Fr(E) \neq 0$ for $E \neq 0$;
  \item $\fr(E) \coloneqq \dim \Fr(E)$ depends only on the class
    $\alpha$ of $E \in \cat{A}^{\Fr}$.
  \end{enumerate}
  We usually write $\fr(\alpha)$ instead of $\fr(E)$.
\end{definition}

\subsubsection{}

\begin{example} \label{ex:framing-functor}
  Suppose $\cat{A} = \cat{Coh}(X)$ is the abelian category of
  compactly-supported coherent sheaves on a smooth quasi-projective
  Calabi--Yau $3$-fold $X$. Let $\cO_X(1)$ denote any very ample line
  bundle on $X$ and consider the $\bC$-linear functor
  \[ \Fr_k\colon \cF \mapsto H^0(\cF \otimes \cO_X(k)). \]
  Let $\cat{A}^{\Fr_k} \subset \cat{A}$ be the full subcategory
  consisting of objects $\cF$ which are {\it $k$-regular} \cite[\S
    1.7]{huybrechts_lehn_geom}, i.e. such that $H^i(\cF(k-i)) = 0$ for
  all $i > 0$. Here are some basic properties of $k$-regularity, for
  any $k \in \bZ$:
  \begin{itemize}
  \item $\cat{A}^{\Fr_k} \subset \cat{A}$ is an exact subcategory;
  \item the moduli substack $\fM^{\Fr_k} \subset \fM$ parameterizing
    objects in $\cat{A}^{\Fr_k} \subset \cat{A}$ is an open substack;
  \item if $\cF$ is $k$-regular then $\cF(k)$ is globally generated,
    so in particular $\Hom(\cF, \cF) \cong \Hom(\cF(k), \cF(k)) \to
    \Hom(H^0(\cF(k)), H^0(\cF(k)))$ is injective;
  \item if $\cF$ is $k$-regular then $H^i(\cF(k)) = 0$ for $i > 0$, so
    the dimension of $\Fr_k(\cF) = H^0(\cF(k)) = \chi(X, \cF(k))$
    depends only on the class of $\cF$;
  \item if $\fS \subset \fM$ is any finite type substack, then there
    exists $k \gg 0$ such that $\fS \subset \fM^{\Fr_k}$.
  \end{itemize}
  Hence $\Fr_k$ is a framing functor that additionally satisfies the
  condition in
  Assumption~\ref{assump:semistable-invariants}\ref{assump:it:framing-functor}.
  We will use (mild modifications of) $\Fr_k$ throughout
  \S\ref{sec:DT-PT-BS} and \S\ref{sec:VW}.
\end{example}

\subsubsection{}

\begin{definition} \label{def:auxiliary-stack}
  Let $Q$ be an acyclic quiver with edges $Q_1$ and vertices $Q_0 =
  Q_0^f \sqcup Q_0^o$ split into {\it ordinary} vertices
  $\blackbullet\in Q_0^o$ and {\it framing} vertices $\blacksquare \in
  Q_0^f$, such that ordinary vertices have no outgoing arrows. Given
  an abelian category $\cat{A}$ and a tuple
  \[ \vec\Fr \coloneqq (\Fr_v)_{v \in Q_0^o} \]
  of framing functors for $\cat{A}$, let
  $\tilde{\cat{A}}^{Q(\vec\Fr)}$ be the exact category of triples $(E,
  \vec V, \vec\rho)$ where:
  \begin{itemize}
  \item $E \in \cat{A}^{\vec\Fr} \coloneqq \bigcap_{v \in Q_0^o}
    \cat{A}^{\Fr_v}$;
  \item $\vec V = (V_v)_{v\in Q_0^f}$ are finite-dimensional vector
    spaces; set $V_v\coloneqq \Fr_v(E)$ for $v\in Q_0^o$;
  \item $\vec \rho = (\rho_e)_{e\in Q_1}$ are morphisms $\rho_e\colon
    V_{t(e)}\to V_{h(e)}$, where $t$ and $h$ denote tail and head.
  \end{itemize}
  A morphism $f\colon (E, \vec V, \vec\rho) \to (E', \vec V',
  \vec\rho')$ in $\cat{A}^{Q(\vec\Fr)}$ is given by a morphism $E \to
  E'$ in $\cat{A}$, which induces linear maps $f_v\colon \Fr_v(E) \to
  \Fr_v(E')$ for $v \in Q_0^o$, along with linear maps $f_v\colon V_v
  \to V'_v$ for $v \in Q_0^f$, such that $f_{t(e)} \circ \rho_e =
  \rho'_e \circ f_{h(e)}$ for all $e \in Q_1$. The group $\bC^\times$
  of scaling automorphisms acts diagonally on $(V_v)_{v \in Q_0}$.
\end{definition}

\subsubsection{}\label{bg:sec:aux-stacks}

A triple $(E, \vec V, \vec\rho)$ has class $(\alpha, \vec d)$ where
$\alpha$ is the class of $E$ and $\vec d \coloneqq (\dim V_v)_{v \in
  Q_0^f}$. Let
\[ \tilde\fM^{Q(\vec\Fr)} = \bigsqcup_{\alpha,\vec d} \tilde\fM^{Q(\vec\Fr)}_{\alpha,\vec d} \]
be the moduli stack of objects in $\cat{A}^{Q(\vec\Fr)}$. Clearly the
$\sT$-action lifts from $\fM$ to $\tilde\fM^{Q(\vec\Fr)}$. Let
$(\cV_v)_{v \in Q_0}$ be the universal bundles for $(V_v)_{v \in
  Q_0}$, and in particular let $(\cFr_v(\cE))_{v \in Q_0^o}$ denote
the universal bundles for $(\Fr_v(E))_{v \in Q_0^o}$. Let
\[ \pi_{\fM^{\vec\Fr}_\alpha}\colon \tilde\fM^{Q(\vec\Fr)}_{\alpha,\vec d} \to \fM^{\vec\Fr}_\alpha \]
be the forgetful map to the moduli stack $\fM^{\vec\Fr}_\alpha$ of
objects in $\cat{A}^{\vec\Fr}$ with class $\alpha$. Then
\[ \tilde\fM^{Q(\vec\Fr)}_{\alpha,\vec d} = \tot\bigg(\bigoplus_{e \in Q_1} \cV_{t(e)}^\vee \otimes \cV_{h(e)}\bigg) \to \fM^{\vec\Fr}_{\alpha} \times \prod_{v \in Q_0^f} [\pt/\GL(d_v)]. \]
So $\pi_{\fM^{\vec\Fr}_\alpha}$ is smooth. Define the bilinear element
\begin{equation} \label{eq:framed-stack-forgetful-map-cotangent}
  \bF^{Q(\vec\Fr)} \coloneqq \sum_{e \in Q_1} \cV^\vee_{t(e)} \boxtimes \cV_{h(e)} - \sum_{v \in Q_0^f} \cV^\vee_v \boxtimes \cV_v \in K_\sT^\circ\left(\tilde\fM^{Q(\vec\Fr)} \times \tilde\fM^{Q(\vec\Fr)}\right).
\end{equation}
The restriction of $\bF^{Q(\vec\Fr)}$ to the diagonal is the relative
tangent complex of $\pi_{\fM^{\vec\Fr}_\alpha}$.

\subsubsection{}

\begin{definition}[``De-rigidification''] \label{def:de-rigidification}
  Suppose $\vec d$ is a dimension vector with $d_i = 1$ for some $i$.
  Then the rigidification map
  \[ \Pi_{\alpha,\vec d}^\pl\colon \tilde\fM_{\alpha,\vec d}^{Q(\vec\Fr)} \to \tilde\fM_{\alpha,\vec d}^{Q(\vec\Fr),\pl} \]
  is a {\it trivial} $\bC^\times$-gerbe because the group of scaling
  automorphisms may be identified with $\Aut(V_i) = \bC^\times$. In
  other words, rigidification has a non-canonical description as the
  choice of an isomorphism $\bC \xrightarrow{\sim} V_i$, and it has a
  section
  \[ I_{\alpha,\vec d}\colon \tilde\fM^{Q(\vec\Fr),\pl}_{\alpha,\vec d} \to \tilde\fM^{Q(\vec\Fr)}_{\alpha,\vec d} \]
  given by forgetting this isomorphism. We often write omit the
  subscript on $I_{\alpha,\vec d}$ and just write $I$ when there is no
  ambiguity.
\end{definition}

\subsubsection{}

\begin{theorem} \label{thm:auxiliary-stack-vertex-algebra}
  Let $\fM = \bigsqcup_\alpha \fM_\alpha$ be a graded monoidal
  $\sT$-stack with $\kappa$-symmetric bilinear elements
  $\scE_{\alpha,\beta}$, and suppose the graded monoidal structure is
  given by direct sum and scaling automorphisms. Then
  $\tilde\fM^{Q(\vec\Fr)}$ is also a graded monoidal $\sT$-stack, with
  $\kappa$-symmetric bilinear elements
  \[ \tilde\scE^{Q(\vec\Fr)}_{(\alpha,\vec d),(\beta,\vec e)} \coloneqq \scE_{\alpha,\beta} - \bF^{Q(\vec\Fr)}_{(\alpha,\vec d),(\beta,\vec e)} + \kappa^{-1} \cdot (12)^*(\bF^{Q(\vec\Fr)}_{(\beta,\vec e),(\alpha,\vec d)})^\vee. \]
  Let $\bfK_\circ^{\tilde\sT}(-)$ be a commutative K-homology theory
  such that $\bfK^{\tilde\sT}_\circ([*/G]) =
  \bK^{\tilde\sT}_\circ([*/G])$ for $G = 1$ and $G = \bC^\times$. Then
  $\bfK_\circ^{\tilde\sT}(\tilde\fM^{Q(\vec\Fr)})$ becomes a
  $\tilde\sT$-equivariant multiplicative vertex algebra with vertex
  product
  \[ (Y(\tilde\phi, z)\tilde\psi)(\tilde\cE) \coloneqq (\tilde\phi \boxtimes \tilde\psi)\left(\hat\Theta^{Q(\vec\Fr)}_{(\alpha,\vec d),(\beta,\vec e)}(z) \otimes z^{\deg_1}\Phi_{(\alpha,\vec d),(\beta,\vec e)}^*\tilde\cE\right) \]
  for $\tilde\phi \in K_\circ^{\tilde\sT}(\tilde\fM_{\alpha,\vec
    d}^{Q(\vec\Fr)})$ and $\tilde\psi \in
  K_\circ^{\tilde\sT}(\tilde\fM_{\beta,\vec e}^{Q(\vec\Fr)})$, where
  \[ \hat\Theta^{Q(\vec\Fr)}_{(\alpha,\vec d),(\beta,\vec e)}(z) \coloneqq \hat\se_{z^{-1}}\left(\tilde\scE^{Q(\vec\Fr)}_{(\alpha,\vec d),(\beta,\vec e)}\right) \otimes \hat\se_{z}\left((12)^*(\tilde\scE^{Q(\vec\Fr)}_{(\beta,\vec e),(\alpha,\vec d)})\right) \]
  is defined as in
  \S\ref{sec:monoidal-stack-vertex-product-well-defined}. The induced
  Lie algebra structure on
  $\bfK_\circ^{\tilde\sT}(\tilde\fM^{Q(\vec\Fr)})^\pl$ has Lie bracket
  \[ (\kappa^{-\frac{1}{2}} - \kappa^{\frac{1}{2}}) \cdot [\phi, \psi](\cE) \coloneqq \rho_{K,z} \left(Y(\tilde\phi, z)\tilde\psi\right)(\cE) \]
  for $\phi \in \bfK_\circ^{\tilde\sT}(\tilde\fM_{\alpha,\vec
    d}^{Q(\vec\Fr)})^\pl$ and $\psi \in
  \bfK_\circ^{\tilde\sT}(\tilde\fM_{\beta,\vec e}^{Q(\vec\Fr)})^\pl$,
  where $\tilde\phi \in \bfK_\circ^{\tilde\sT}(\tilde\fM_{\alpha,\vec
    d}^{Q(\vec\Fr)})$ and $\tilde\psi \in
  \bfK_\circ^{\tilde\sT}(\tilde\fM_{\beta,\vec e}^{Q(\vec\Fr)})$ are
  arbitrary lifts of $\phi$ and $\psi$ respectively. It satisfies
  \[ [\phi, \psi](\cO) = (\phi \boxtimes \psi)\left(\left[\rank \tilde\scE^{Q(\vec \Fr)}_{(\alpha,\vec d),(\beta,\vec e)}\right]_\kappa \cdot (\cO \boxtimes \cO)\right). \]
\end{theorem}

Note that the natural isomorphisms $\iota^Q_\alpha\colon \fM_\alpha
\to \fM^{Q(\vec\Fr)}_{\alpha,\vec 0}$ given by $E \mapsto (E, \vec 0,
\vec 0)$ are morphisms of graded monoidal $\sT$-stacks, so this
theorem is compatible with and generalizes
Theorem~\ref{thm:mVOA-monoidal-stack}.

\begin{proof}
  The graded monoidal structure on $\tilde\fM^{Q(\vec\Fr)}$ is also
  given by direct sum and scaling automorphisms. This is well-defined
  because framing functors are compatible with both. Also, from
  \eqref{eq:framed-stack-forgetful-map-cotangent}, it is clear that
  $\tilde\scE_{\alpha,\beta}^{Q(\vec \Fr)}$ is still bilinear with
  weight $\pm 1$ in its two factors, and $\kappa$-symmetric. The
  remaining claims summarize Theorem~\ref{thm:mVOA-monoidal-stack},
  Theorem~\ref{thm:mVOA-monoidal-stack-lie-algebra}, and
  Proposition~\ref{prop:rigidity}, when applied to
  $\tilde\fM^{Q(\vec\Fr)}$.
\end{proof}

\subsection{Obstruction theories}

\subsubsection{}

Throughout this subsection, let $\fX$ be an Artin stack over a smooth
equidimensional base $\fB$, all in $\cat{Art}_\sT$. Let
$D^-_{\cat{QCoh},\sT}(\fX)$ be its derived category of bounded-above
$\sT$-equivariant complexes with quasi-coherent cohomology
\cite{Olsson2007}. Let $\bL_{\fX/\fB} \in D^-_{\cat{QCoh},\sT}(\fX)$
denote the cotangent complex \cite{Illusie1971}.

\subsubsection{}

\begin{definition} \label{def:obstruction-theory}
  A $\sT$-equivariant\footnote{Since all obstruction theories in this
    paper are $\sT$-equivariant, we will omit writing
    ``$\sT$-equivariant''.} {\it obstruction theory} on $\fX$ relative
  to $\fB$ is an object $\bE_{\fX/\fB}$ with a morphism
  \[ \varphi\colon \bE_{\fX/\fB} \to \bL_{\fX/\fB} \]
  in $D^-_{\cat{QCoh},\sT}(\fX)$, such that $\cH^i(\varphi)$ for $i
  \ge 0$ are isomorphisms and $\cH^{-1}(\varphi)$ is surjective. An
  obstruction theory is:
  \begin{itemize}
  \item {\it perfect} if $\bE$ is a perfect complex of amplitude
    contained in $[-1,1]$;
  \item {\it $\kappa$-symmetric}, for a $\sT$-weight $\kappa$, if
    $\bE$ is a perfect complex and there is an isomorphism
    \[ \Xi\colon \bE\xrightarrow{\sim} \kappa \otimes \bE^\vee[1] \]
    such that $\Xi^\vee[1]=\Xi$. (We often just write {\it symmetric}
    when the precise weight $\kappa$ is irrelevant.)
  \end{itemize}
  The notions of perfect obstruction theory (POT) and symmetric POT
  coincide with the ones in \cite{Behrend1997, Behrend2008} if $\fX$
  is a Deligne--Mumford stack, because of the requirement that
  $\cH^1(\varphi)$ is an isomorphism.

  The {\it obstruction sheaf} of an obstruction theory is
  $\cOb_\varphi \coloneqq \cH^1(\bE^\vee)$.
\end{definition}

\subsubsection{}

\begin{lemma} \label{lem:perfect-symmetric-obstruction-theory}
  Let $X \subset \fX$ be a $\sT$-invariant open locus which is an
  algebraic space. Then the restriction to $X$ of a symmetric
  obstruction theory on $\fX$ is automatically perfect.
\end{lemma}

\begin{proof}
  Let $\varphi\colon \bE_\fX \to \bL_\fX$ be the $\kappa$-symmetric
  obstruction theory. Open immersions are flat, so $\varphi|_X$ is
  still a $\kappa$-symmetric obstruction theory. Since $X$ is an
  algebraic space, for $i > 0$,
  \[ \cH^{i}(\bE_{\fX})\Big|_X \cong \cH^{i}(\bL_\fX)\Big|_X = \cH^{i}(\bL_X) = 0. \]
  But then by symmetry, for $i < -1$,
  \[ \cH^i(\bE_\fX)\Big|_X \cong \kappa \otimes \cH^{-1-i}(\bE_\fX)^\vee\Big|_X = 0 \]
  as well. This implies $\bE_\fX$ has amplitude in $[-1, 0]$.
\end{proof}

\subsubsection{}

\begin{definition} \label{def:obstruction-theories-compatibility}
  Let $f\colon \fX \to \fY$ be a smooth morphism over $\fB$ in
  $\cat{Art}_\sT$. Suppose $\varphi_{\fX/\fB}\colon \bE_{\fX/\fB} \to
  \bL_{\fX/\fB}$ and $\varphi_{\fY/\fB}\colon \bE_{\fY/\fB} \to
  \bL_{\fY/\fB}$ are obstruction theories.
  \begin{itemize}
  \item They are {\it compatible under $f$} if they fit into a
    morphism of exact triangles
    \begin{equation} \label{eq:obstruction-theories-compatibility}
      \begin{tikzcd}
        \bL_f[-1] \ar{r} \ar[equal]{d} & f^*\bE_{\fY/\fB} \ar{r} \ar{d}{f^*\varphi_{\fY/\fB}} & \bE_{\fX/\fB} \ar{r}{+1} \ar{d}{\varphi_{\fX/\fB}} & {} \\
        \bL_f[-1] \ar{r} & f^*\bL_{\fY/\fB} \ar{r} & \bL_{\fX/\fB} \ar{r}{+1} & {}
      \end{tikzcd}
    \end{equation}
    where the bottom row is the exact triangle of cotangent complexes.

  \item They are {\it $\kappa$-symmetrically compatible under $f$},
    for a $\sT$-weight $\kappa$, if they fit into morphisms of exact
    triangles
    \begin{equation} \label{eq:obstruction-theories-symmetric-compatibility}
      \begin{tikzcd}
        \bL_f[-1] \ar[equals]{d} \ar{r} & \kappa \otimes \bF^\vee[1] \ar{d}{\eta^\vee[1]} \ar{r}{\zeta} & \bE_{\fX/\fB} \ar{d}{\zeta^\vee[1]} \ar{r}{+1} & {} \\
        \bL_f[-1] \ar[equals]{d} \ar{r} & f^*\bE_{\fY/\fB} \ar{d}{f^*\varphi_{\fY/\fB}} \ar{r}{\eta} & \bF \ar{d} \ar{r}{+1} & {} \\
        \bL_f[-1] \ar{r} & f^*\bL_{\fY/\fB} \ar{r} & \bL_{\fX/\fB} \ar{r}{+1} & {}
      \end{tikzcd}
    \end{equation}
    such that the third column is $\varphi_{\fX/\fB}$.
  \end{itemize}
  A {\it smooth pullback} (resp. {\it symmetrized (smooth) pullback})
  of $\varphi_{\fY/\fB}$ along $f$ is any obstruction theory
  $\varphi_{\fX/\fB}$ which is compatible (resp. symmetrically
  compatible) with $\varphi_{\fY/\fB}$ under $f$.
\end{definition}

\subsubsection{}

\begin{lemma} \label{lem:obstruction-theory-pl}
  Let $\Pi^\pl\colon \fX \to \fX^\pl$ be a $\sT$-equivariant
  $\bC^\times$-gerbe over $\fB$, and suppose $\varphi\colon
  \bE_{\fX/\fB} \to \bL_{\fX/\fB}$ is an obstruction theory such that
  $\bE_{\fX/\fB}$ has trivial $\bC^\times$-weight.
  \begin{enumerate}[label = (\roman*)]
  \item There is an obstruction theory $\tilde\varphi^\pl\colon \tilde
    \bE_{\fX^\pl/\fB} \to \bL_{\fX^\pl/\fB}$ compatible under
    $\Pi^\pl$ with $\varphi$, unique up to isomorphism.
  \item Suppose furthermore that $\varphi$ is $\kappa$-symmetric, and
    that \eqref{eq:traceless-Ext3-condition} below is the zero map.
    Then there is a $\kappa$-symmetric obstruction theory
    $\varphi^\pl\colon \bE_{\fX^\pl/\fB} \to \bL_{\fX^\pl/\fB}$
    symmetrically compatible under $\Pi^\pl$ with $\varphi$, unique up
    to isomorphism.
  \end{enumerate}
\end{lemma}

In practice, if $\fX$ is a moduli stack of perfect complexes with its
natural obstruction theory, the condition that
\eqref{eq:traceless-Ext3-condition} is the zero map is automatically
satisfied, e.g. using the argument in \cite[Cor.
  3.18]{Tanaka2020}.\footnote{The argument there is written for a
stable locus in $\fX^\pl$, but works for the entirety of $\fX^\pl$ as
well.}

We implicitly use throughout the proof that objects (and morphisms
between them) on $\fX$ of trivial $\bC^\times$-weight are exactly
those pulled back from $\fX^\pl$ along $\Pi^\pl$.

\begin{proof}
  Consider the compatibility diagram
  \eqref{eq:obstruction-theories-compatibility} for the morphism
  $\Pi^\pl$. The desired $\tilde\varphi^\pl$ is given as
  $\cocone(\nu)$ in the diagram
  \[ \begin{tikzcd}
    \tilde \bE_{\fX^\pl/\fB} \ar[dashed]{r} \ar[dashed]{d}{\tilde\varphi^\pl} & \bE_{\fX/\fB} \ar{r}{\nu} \ar{d}{\varphi} & \bL_{\Pi^\pl} \ar[equals]{d} \ar[dashed]{r}{+1} & {} \\
    (\Pi^\pl)^*\bL_{\fX^\pl/\fB} \ar{r} & \bL_{\fX/\fB} \ar{r} & \bL_{\Pi^\pl} \ar{r}{+1} & {}
  \end{tikzcd} \]
  yielding a morphism of exact triangles. It is still an obstruction
  theory by long exact sequence. Recall that $\tilde\varphi^\pl$ is
  unique up to a homotopy in $\Hom(\tilde\bE_{\fX^\pl/\fB},
  \bL_{\Pi^\pl}[-1])$, but this vanishes for degree reasons because
  $\bL_{\Pi^\pl}$ is a locally free sheaf in cohomological degree $1$,
  and thus $\tilde\varphi^\pl$ is genuinely unique (up to
  isomorphism).

  If further $\varphi$ is $\kappa$-symmetric, then the solid
  commutative square below exists and may be completed into the
  commutative $3 \times 3$ diagram
  \[ \begin{tikzcd}
    \kappa \otimes \bL_{\Pi^\pl}^\vee[1] \ar[dashed,equals]{r} \ar[dashed]{d}{\mu} & \kappa \otimes \bL_{\Pi^\pl}^\vee[1] \ar{r} \ar{d}{\kappa \otimes \nu^\vee[1]} & 0 \ar{d} \ar[dashed]{r}{+1} & {} \\
    \tilde\bE_{\fX^\pl/\fB} \ar[dashed]{r} \ar[dashed]{d}{\zeta} & \bE_{\fX/\fB} \ar{r}{\nu} \ar[dashed]{d} & \bL_{\Pi^\pl} \ar[dashed]{r}{+1} \ar[dashed,equals]{d} & {} \\
    \bE_{\fX^\pl/\fB} \ar[dashed]{d}{+1} \ar[dashed]{r}{\kappa \otimes \zeta^\vee[1]} & \kappa \otimes \tilde\bE_{\fX^\pl/\fB}^\vee[1] \ar[dashed]{r}{\kappa \otimes \mu^\vee[1]} \ar[dashed]{d}{+1} & \bL_{\Pi^\pl} \ar[dashed]{r}{+1} \ar[dashed]{d}{+1} & {} \\
       {} & {} & {} & {}
  \end{tikzcd} \]
  whose rows and columns are exact triangles. Here, as before, we are
  using cohomological degree arguments to conclude: the solid square
  indeed commutes because $\nu \circ (\kappa \otimes \nu^\vee[1]) \in
  \Hom(\bL_{\Pi^\pl}^\vee[1], \bL_{\Pi^\pl}) = 0$; the map $\mu$ is
  unique up to homotopy in $\Hom(\bL_{\Pi^\pl}^\vee[1],
  \bL_{\Pi^\pl}[-1]) = 0$ and therefore is genuinely unique; the map
  $\zeta$ is unique up to homotopy in $\Hom(\tilde\bE_{\fX^\pl/\fB},
  \bL_{\Pi^\pl}[-1]) = 0$ and therefore is also genuinely unique. The
  latter two observations imply that the maps in the third row are
  also uniquely determined to be $\kappa \otimes \mu^\vee[1]$ and
  $\kappa \otimes \zeta^\vee[1]$, respectively, since otherwise their
  duals would make $\mu$ and $\zeta$ non-unique. So the entire $3
  \times 3$ diagram is self-dual under $\kappa \otimes (-)^\vee[1]$,
  and thus the $\kappa$-symmetry of $\bE_{\fX/\fB}$ lifts to make
  $\bE_{\fX^\pl/\fB}$ also $\kappa$-symmetric. Finally, we claim that
  $\tilde\varphi^\pl$ factors through $\zeta$ to yield the desired map
  $\varphi^\pl$. Since the left-most column is an exact triangle, this
  is equivalent to the vanishing of the composition
  \begin{equation} \label{eq:traceless-Ext3-condition}
    \kappa \otimes \bL_{\Pi^\pl}^\vee[1] \xrightarrow{\mu} \tilde\bE_{\fX^\pl/\fB} \xrightarrow{\tilde\varphi^\pl} (\Pi^\pl)^*\bL_{\fX^\pl/\fB}.
  \end{equation}
  By hypothesis, this vanishes. The induced $\varphi^\pl$ is still an
  obstruction theory by long exact sequence. It is unique up to
  isomorphism because both $\bE_{\fX^\pl/\fB}$ and $\zeta$ are.
\end{proof}

\subsubsection{}

Given a smooth morphism $f\colon \fX \to \fY$ over $\fB$, the proof of
Lemma~\ref{lem:obstruction-theory-pl} suggests a procedure to induce
($\kappa$-symmetrically) compatible obstruction theories on $\fY$
starting from one on $\fX$ (relative to $\fB$), and one may hope that
it works equally well ``in reverse'' to construct smooth or
symmetrized pullbacks along $f$. However, attempting to run this
procedure for general $f$ immediately runs into issues every time a
cohomological degree argument was used in the proof of
Lemma~\ref{lem:obstruction-theory-pl}. We forcibly remove such issues
by using the following \'etale-local generalization of obstruction
theories.

\subsubsection{}

\begin{definition} \label{bg:def:apot}
  Suppose $\fX$ is Deligne--Mumford and of finite presentation.
  Following \cite[Definition 5.1]{kiem_savvas_loc}, a
  $\sT$-equivariant {\it almost-perfect obstruction theory (APOT)} on
  $\fX$ relative to $\fB$ consists of:
  \begin{enumerate}[label=(\roman*)]
  \item a $\sT$-equivariant \'etale cover $\{U_i\to \fX\}_{i \in I}$
    over $\fB$ where the $U_i$ are schemes;\footnote{By \cite[Theorem
        4.3]{alper_hall_rydh_luna_stacks_20}, after possibly passing
      to a multiple cover of $\sT$, such a $\sT$-equivariant \'etale
      cover $\{U_i \to \fX\}_i$ always exists, and in fact the $U_i$
      can be taken to be affine.}
  \item for each $i$, a $\sT$-equivariant POT $\varphi_i\colon \bE_i
    \to \bL_{U_i/\fB}$.
  \end{enumerate}
  We often write $\varphi = (U_i, \varphi_i)_{i \in I}$ to denote this
  data and let $U_{ij} \coloneqq U_i \times_{\fX} U_j$. This data must
  satisfy the following conditions:
  \begin{enumerate}[label=(\alph*)]
  \item \label{spb:def:apot:i} there exist $\sT$-equivariant
    transition isomorphisms
    \[ \psi_{ij}\colon \cH^1(\bE_i^\vee)\Big|_{U_{ij}} \xrightarrow{\sim} \cH^1(\bE_j^\vee)\Big|_{U_{ij}} \]
    which give descent data for a sheaf $\cOb_\varphi$ on $\fX$ called
    the {\it obstruction sheaf}; \footnote{Here, we follow the
    presentation of Kiem--Savvas, although it is slightly inaccurate:
    the obstruction sheaf is really part of the \emph{data} defining
    the almost perfect obstruction theory.}
  \item \label{spb:def:apot:ii} for each pair $i,j$, there exists a
    $\sT$-equivariant \'etale covering $\{V_k\to U_{ij}\}_{k \in K}$,
    such that the isomorphism $\psi_{ij}$ of the obstruction sheaves
    lifts to a $\sT$-equivariant isomorphism of the POTs over each
    $V_k$.
  \end{enumerate}
  In addition, if the following conditions are satisfied for some
  $\sT$-weight $\kappa$, then the APOT is {\it $\kappa$-symmetric} if:
  \begin{enumerate}[resume,label=(\alph*)]
  \item \label{it:sym-apotiii} for each $i$, the POT $\varphi_i\colon
    \bE_i \to \bL_{U_i/\fB}$ is $\kappa$-symmetric; and
  \item \label{it:sym-apotiv} there exists a $\sT$-equivariant
    isomorphism $\kappa \otimes \cOb_\varphi \cong
    \cH^0(\bL_{\fX/\fB})$.
  \end{enumerate}
  We say an APOT $(U_i, \varphi_i)_{i \in I}$ {\it admits a global
    virtual tangent bundle} $\bT_{\fX}^\vir \in K_\sT^\circ(\fX)$ if
  $\bT_{\fX}^\vir$ restricts to $\bE_i^\vee$ on each $U_i$.

  In the setting of
  Definition~\ref{def:obstruction-theories-compatibility}, an APOT
  $\varphi_{\fX} = (U_i, \varphi_{\fX,i})_{i \in I}$ on $\fX$ and an
  obstruction theory $\varphi_{\fY}$ on $\fY$ are {\it (symmetrically)
    compatible} under $f$ if the diagrams
  \eqref{eq:obstruction-theories-compatibility} or
  \eqref{eq:obstruction-theories-symmetric-compatibility} exist
  locally, i.e. there exists a cover $\{V_{i,j} \to U_i\}_{i,j}$
  refining $\{U_i\}_{i \in I}$ such that $\varphi_{\fX,i}|_{V_{i,j}}$
  and $\varphi_\fY$ are (symmetrically) compatible for every $i, j$,
  and moreover these local morphisms are compatible with the
  isomorphisms in conditions~\ref{spb:def:apot:i} and
  \ref{spb:def:apot:ii}. (This definition may be altered in the
  obvious way if $\fY$ is an algebraic space with an APOT instead of
  an obstruction theory.)
\end{definition}

\subsubsection{}
\label{sec:APOTs}

\begin{theorem} \label{thm:APOTs}
  Let $X$ be an algebraic space, and $f\colon X \to \fY$ be a
  smooth morphism over $\fB$ in $\cat{Art}_\sT$. Suppose that $\fY$
  has a $\kappa$-symmetric obstruction theory $\varphi_{\fY/\fB}$
  relative to $\fB$, or that $\fY$ is also an algebraic space and has
  a $\kappa$-symmetric APOT $\varphi_{\fY/\fB}$ relative to $\fB$.
  \begin{enumerate}[label=(\roman*)]
  \item \label{it:APOT-sym-pullback} (Symmetrized pullback,
    \cite[Theorem 2.3.4]{klt_dtpt}) There exists a $\kappa$-symmetric
    APOT on $X$ relative to $\fB$ which is a symmetrized pullback of
    $\varphi_{\fY/\fB}$ along $f$.

  \item \label{it:APOT-functoriality} (Functoriality, \cite[Lemma
    2.3.12]{klt_dtpt}) The symmetrized pullback operation in
    \ref{it:APOT-sym-pullback} is functorial: if $g\colon X' \to X$ is
    a smooth morphism of algebraic spaces over $\fB$ in
    $\cat{Art}_\sT$, then the symmetrized pullback of
    $\varphi_{\fY/\fB}$ along $g \circ f$ agrees with the symmetrized
    pullback along $g$ of the symmetrized pullback along $f$ of
    $\varphi_{\fY/\fB}$.

  \item \label{it:APOT-virtual-cycle} (Virtual cycle, \cite[Def.
    4.1]{kiem_savvas_apot}) A $\sT$-equivariant APOT on $X$ relative
    to $\fB$ gives rise to an $\sT$-equivariant virtual structure
    sheaf $\cO_X^\vir\in K_\sT(X)$.

  \item \label{it:APOT-localization} (Localization, \cite[Theorem
    2.4.3]{klt_dtpt}) The symmetrized pullback APOT from
    \ref{it:APOT-sym-pullback} admits a global virtual tangent bundle
    \begin{equation} \label{eq:symmetrized-pullback-Tvir}
      \bT_{X/\fB}^\vir \coloneqq f^*\bT_{\fY/\fB}^\vir + \bT_f - \kappa^{-1} \bT_f^\vee
    \end{equation}
    where $\bT_f$ is the dual of the relative cotangent complex for
    $f$. If $X^{\sT}$ has the resolution property, the virtual
    structure sheaf $\cO_{X}^\vir$ from \ref{it:APOT-virtual-cycle}
    satisfies the $\sT$-equivariant localization formula
    \begin{equation} \label{eq:equivariant-localization}
      \cO_{X}^\vir = \iota_*\frac{\cO_{X^\sT}^\vir}{\se(\cN_\iota^\vir)}
    \end{equation}
    where the $\sT$-fixed locus $\iota\colon X^\sT \hookrightarrow X$
    inherits an APOT by restriction from $X$, and the virtual normal
    bundle
    \[ \cN_\iota^\vir \coloneqq \left(\iota^*\bT_{X/\fB}^\vir\right)^{\text{mov}} \]
    is the non-$\sT$-fixed part of $\iota^*\bT_{X/\fB}^\vir$.
  \end{enumerate}
\end{theorem}

It is convenient to write $\bT_f^\vir \coloneqq \bT_f - \kappa^{-1}
\bT_f^\vee$.

\begin{proof}[Proof sketch.]
  We summarize the content of \cite[\S 2]{klt_dtpt} in the case that
  $\fY$ has a $\kappa$-symmetric obstruction theory
  $\varphi_{\fY}\colon \bE_{\fY/\fB} \to \bL_{\fY/\fB}$. Suppose that
  the natural map $\bL_f[-1] \to f^*\bL_{\fY/\fB}$ lifts along
  $\varphi_{\fY}$ to a map $\delta\colon \bL_f[-1] \to
  f^*\bE_{\fY/\fB}$, and that $\delta^\vee[1]\circ \delta = 0$ (using
  the $\kappa$-symmetry of $\bE_{\fY/\fB}$ to form this composition).
  Then, the following solid square exists and can be completed into
  the commutative diagram
  \begin{equation} \label{eq:obstruction-theories-symmetric-pullback}
    \begin{tikzcd}
      \bL_f^\vee[-1] \ar[dashed]{d} \ar[dashed]{r} & \kappa \otimes \bF^\vee[1] \ar[dashed]{r}{\zeta} \ar[dashed]{d}{\eta^\vee[1]} & \bE_{X/\fB} \ar[dashed]{r}{+1} \ar[dashed]{d}{\zeta^\vee[1]} & {} \\
      \bL_f^\vee[-1] \ar{r}{\delta} \ar{d} & f^*\bE_{\fY/\fB} \ar{d}{\delta^\vee[1]} \ar[dashed]{r}{\eta} & \bF \ar[dashed]{d} \ar[dashed]{r}{+1} & {} \\
      0 \ar{r} \ar[dashed]{d}{+1} & \kappa \otimes \bL_f^\vee[2] \ar[dashed]{d}{+1} \ar[dashed]{r} & \kappa \otimes \bL_f^\vee[2] \ar[dashed]{d}{+1} \ar[dashed]{r}{+1} & {} \\
      {} & {} & {}
    \end{tikzcd}
  \end{equation}
  of exact triangles. If moreover this diagram is symmetric under
  $\kappa \otimes (-)^\vee[1]$, then the composition $\bE_{X/\fB} \to
  \bF \to \bL_{X/\fB}$ is easily checked to be the desired symmetrized
  pullback. Note that $\bF \to \bL_{X/\fB}$ would become a smooth
  pullback of $\varphi_{\fY/\fB}$ along $f$.

  Of course, these hypotheses are not satisfied in general. However,
  the obstructions to these hypotheses being satisfied all lie in
  certain higher Ext groups that automatically vanish if $X$ were an
  affine scheme. Hence symmetrized pullbacks exist {\it affine
    \'etale-locally}, and some diagram chasing shows they glue
  appropriately, forming the desired $\kappa$-symmetric APOT
  $\varphi_{X/\fB}$ on $X$ relative to $\fB$. This strategy clearly
  also works if $\fY$ is an algebraic space with a $\kappa$-symmetric
  APOT instead of an obstruction theory. This proves
  \ref{it:APOT-sym-pullback}, and the functoriality property
  \ref{it:APOT-functoriality} is straightforward.

  For the torus-localization formula \ref{it:APOT-localization},
  recall that Kiem and Savvas prove a torus-localization formula
  assuming that the $\cN^\vir$ of each affine \'etale-local POT glue
  to form a global $\cN^\vir$ (compatibly with the gluing to form
  $\cOb$) \cite[Theorem 5.13]{kiem_savvas_loc}. This assumption is not
  satisfied by our APOT $\varphi_{X/\fB}$. However, it can be made to
  hold by pulling back along a sufficiently large affine bundle
  $a\colon A \to X$ such that the finite number of classes which
  obstruct the existence of
  \eqref{eq:obstruction-theories-symmetric-pullback} vanish
  \cite[Lemma 2.4.6]{klt_dtpt}, yielding a $\kappa$-symmetric POT
  $\bE_{X/\fB}^A$ for $X$ {\it on $A$}. Affine \'etale-locally, it is
  easy to check that $\bE_{X/\fB}^A \oplus \Omega_a$ form an APOT on
  $A$, which by functoriality is identified with the smooth pullback
  of $\varphi_{X/\fB}$ along $a$. But the $\sT$-moving part of
  $(\bE_{X/\fB}^A \oplus \Omega_a)^\vee|_{A^\sT}$ is obviously a
  global $\cN^\vir$, so the torus-localization formula of Kiem--Savvas
  is applicable. Since smooth pullback preserves virtual cycles, and
  $a^*\colon K_\sT(X) \to K_\sT(A)$ is an isomorphism by the Thom
  isomorphism theorem, the torus-localization formula on $A$ implies
  the desired one on $X$.
\end{proof}

\subsubsection{}

\begin{remark}
  We think of an APOT as an \'etale-local collection of POTs where
  only the local obstruction sheaves, not the POTs themselves, glue to
  form a global object. On the other hand, the virtual structure sheaf
  produced by Theorem~\ref{thm:APOTs}\ref{it:APOT-virtual-cycle} is
  global. While Theorem~\ref{thm:APOTs}\ref{it:APOT-virtual-cycle}
  holds for any $\sT$-equivariant APOT on $X$,
  Theorem~\ref{thm:APOTs}\ref{it:APOT-localization} genuinely requires
  the existence of the global virtual tangent bundle $\bT_{X}^\vir$,
  which is a non-trivial condition satisfied by APOTs arising from
  symmetrized pullback. Later, in
  Definition~\ref{def:symmetrized-virtual-structure-sheaf},
  $\bT_{X}^\vir$ will also be genuinely needed to construct {\it
    twisted} virtual structure sheaves.

  To be clear, even if we assume the resolution property for $X$
  (Remark~\ref{rem:localization-resolution-property}), the existence
  of a global virtual tangent bundle $\bT_X^\vir$ is still a
  non-trivial condition for a general APOT on $X$.
\end{remark}

\subsubsection{}

\begin{remark}[Resolution property] \label{rem:localization-resolution-property}
  In Theorem~\ref{thm:APOTs}\ref{it:APOT-localization}, the resolution
  property on $X^{\sT}$ is used in exactly one way: to resolve the
  obstruction theory $\iota^*\bE_{\fY/\fB}$ as a complex of vector
  bundles (necessarily two-term). This locally free resolution is then
  used to
  \begin{itemize}
  \item apply \cite[Lemma 2.4.6]{klt_dtpt} to make certain extension
    classes vanish, and
  \item define the Euler class $\se(\cN_\iota^\vir)$ for use in
    $\sT$-equivariant localization.
  \end{itemize}
  Note that the latter usage is superficial, because one may actually
  define K-theoretic Euler classes of perfect complexes in general and
  prove that \eqref{eq:equivariant-localization} still holds, as
  explained by \cite{Aranha2024} for Chow homology. So, if it is known
  by other means that the former usage is unnecessary, e.g. \cite[Ex.
    2.6, Rem. 2.7]{liu_ss_vw}, then the assumption that $X^{\sT}$ has
  the resolution property may be entirely removed from this paper.

  In all applications of symmetrized pullback in this paper, the
  source of the smooth morphism will be a semistable locus in an
  auxiliary framed stack for a moduli stack of some complexes or
  sheaves on a smooth quasi-projective variety. Thus, it will admit a
  GIT construction via Quot schemes, and is therefore quasi-projective
  and automatically has the resolution property. So, we do not worry
  about the resolution property throughout this paper.
\end{remark}

\subsubsection{}

\begin{theorem}[Comparison of virtual cycles, {\cite[Prop. 2.5.2]{klt_dtpt}}] \label{thm:APOT-comparison}
  Consider the following objects and morphisms in $\cat{Art}_\sT$:
  \[ \begin{tikzcd}
    Z \ar{d}{f} \ar[hookrightarrow]{r} & \bM \ar[loop right]{r}{\bC^\times} \ar{d}{g} \\
    \fY & \fX, 
  \end{tikzcd} \]
  where $f$ and $g$ are smooth, $\bM$ is an algebraic space with
  $\bC^\times$-action compatible with the $\sT$-action, and $Z \subset
  \bM$ is a $\bC^\times$-fixed component. Suppose both $\fX$ and $\fY$
  have symmetric obstruction theories. By
  Theorem~\ref{thm:APOTs}\ref{it:APOT-sym-pullback}, $Z$ has two
  symmetric APOTs:
  \begin{itemize}
  \item the one obtained by symmetrized pullback along $f$;
  \item the $\bC^\times$-fixed part of the restriction of the
    symmetric APOT on $\bM$ obtained by symmetrized pullback along
    $g$.
  \end{itemize}
  Let $\bT^\vir_Z$ and $\bT^\vir_{\bM}\big|_Z^{\text{fix}}$ denote
  their global virtual tangent bundles, respectively, from
  Theorem~\ref{thm:APOTs}\ref{it:APOT-localization}. If $\bM$ has the
  $(\sT \times \bC^\times)$-equivariant resolution property, and
  \begin{equation} \label{eq:APOT-comparison-Tvir}
    \bT^\vir_Z = \bT^\vir_{\bM}\Big|_Z^{\text{fix}},
  \end{equation}
  then the (symmetrized) virtual structure sheaves of these two APOTs
  on $Z$ coincide.
\end{theorem}

\begin{proof}[Proof sketch.]
  Denote the two virtual structure sheaves by $\cO_Z^{\vir}$ and
  $\cO_{Z \subset \bM}^\vir$ respectively.
  
  A {\it (equivariant) Jouanolou device} for $\bM$ is an (equivariant)
  affine bundle $b\colon B \to \bM$ such that $B$ is an affine scheme
  \cite{Jouanolou1973}. Since $\bM$ has the $(\sT \times
  \bC^\times)$-equivariant resolution property, there exists a $(\sT
  \times \bC^\times)$-equivariant Jouanolou device for $\bM$ \cite[\S
    2.5.3]{klt_dtpt}. In this stronger setting, all obstructions to
  the existence of the diagram
  \eqref{eq:obstruction-theories-symmetric-pullback} vanish after
  pullback along $b$, because $B$ is affine, yielding a global element
  and morphism $\hat\bE_{\bM}^B \to b^*\bL_\bM$ which would be a
  symmetrized pullback POT for $\bM$ if it existed on $\bM$. Moreover,
  $\Omega_b[-1] \to b^*\bL_\bM$ is the zero map because $B$ is affine,
  so
  \[ \bE_{\bM}^B \coloneqq \hat\bE_{\bM}^B \oplus \Omega_b \]
  is easily checked to be a POT for $B$. It is straightforward to
  check that it agrees with the smooth pullback along $b$ of the
  symmetrized pullback along $g$ of the symmetric obstruction theory
  on $\fX$.

  This procedure may be repeated for the $\bC^\times$-fixed part of
  the induced $(\sT \times \bC^\times)$-equivariant Jouanolou device
  $B|_Z \to Z$, which is a $\sT$-equivariant Jouanolou device $a\colon
  A \to Z$, to produce a POT
  \[ \bE_Z^A \coloneqq \hat\bE_Z^A \oplus \Omega_a \]
  on $A$ which agrees with the smooth pullback along $a$ of the
  symmetrized pullback along $f$ of the symmetric obstruction theory
  on $\fY$.

  The hypothesis \eqref{eq:APOT-comparison-Tvir} is equivalent to
  $(\bE_{\bM}^B|_A)^{\text{fix}} = \bE_Z^A \in K_\sT^\circ(A)$. Thus
  their induced virtual cycles agree \cite{Thomas2022}. Because smooth
  pullback preserves virtual cycles, this means $a^*\cO_Z^{\vir} =
  a^*\cO_{Z \subset \bM}^\vir$. We conclude because $a^*\colon
  K_\sT(Z) \to K_\sT(A)$ is an isomorphism by the Thom isomorphism
  theorem.
\end{proof}
  
\subsubsection{}

\begin{remark}
  Throughout this paper, we only use APOTs in order to construct and
  localize virtual cycles (on semistable=stable loci of auxiliary
  framed stacks) corresponding to the symmetrized pullback of a
  symmetric obstruction theory. Using this Jouanolou trick, this usage
  of APOTs may be completely avoided. First, let $a\colon A \to X$ be
  a $\sT$-equivariant Jouanolou device and construct the POT $\bE_X^A
  \coloneqq \hat\bE_X^A \oplus \Omega_a$ as above, let
  $\cO_X^{\vir,A}$ be its virtual structure sheaf, and then {\it
    define}
  \[ \cO_X^{\vir} \coloneqq (a^*)^{-1}\cO_X^{\vir,A}. \]
  In other words, instead of constructing a symmetrized pullback APOT
  on $X$ and then using the Kiem--Savvas machinery, we work throughout
  with a POT on $A$ (corresponding to the smooth pullback of the APOT
  on $X$ along $a$), do all K-theoretic constructions on $A$, and then
  apply the isomorphism $(a^*)^{-1}$ afterward. Then an analogue of
  Theorems~\ref{thm:APOTs} and \ref{thm:APOT-comparison} continues to
  hold. In particular, the proof of
  Theorem~\ref{thm:APOTs}\ref{it:APOT-localization} now only requires
  the torus localization formula for POTs and not APOTs.

  Note that the construction of a Jouanolou device is more involved
  than the construction \cite[Lemma 2.4.6]{klt_dtpt} of an affine
  bundle, whose total space may not be affine nor a scheme, such that
  certain extension classes vanish. Indeed, the latter suffices for
  all our actual applications of Theorem~\ref{thm:APOT-comparison}: by
  directly comparing the two obstruction theories --- see the argument
  in \S\ref{rc:sec:master-space-argument} --- we can avoid any use of
  \cite{Thomas2022}, the only place in the proof where we used that
  the total space of the affine bundle is an affine scheme.
\end{remark}

\subsection{Symmetrized enumerative invariants}

\subsubsection{}

\begin{definition} \label{def:symmetrized-virtual-structure-sheaf}
  Let $X \in \cat{Art}_\sT$ be an algebraic space with a
  $\kappa$-symmetric $\sT$-equivariant APOT obtained by symmetrized
  pullback (Theorem~\ref{thm:APOTs}\ref{it:APOT-sym-pullback}) along a
  smooth morphism (which may be the identity). By
  Theorem~\ref{thm:APOTs}\ref{it:APOT-localization}, the APOT has a
  global virtual tangent bundle $\bT_X^\vir$. Let
  \[ \cK_X^\vir \coloneqq \det(\bT_X^\vir)^\vee \]
  be the {\it virtual canonical bundle}. Assuming that $X^{\sT}$ has
  the resolution property, and under the conditions of
  Lemma~\ref{lem:virtual-canonical-half} below, define the {\it
    symmetrized virtual structure sheaf} (cf.
  \eqref{eq:equivariant-localization})
  \begin{equation} \label{eq:symmetrized-virtual-sheaf}
    \hat\cO_X^\vir \coloneqq \iota_*\left(\frac{\cO_{X^\sT}^\vir}{\se(N_\iota^\vir)} \otimes \left(\iota^*\cK_X^\vir\right)^{\frac{1}{2}}\right) \in K_{\tilde\sT}(X)_{\loc}
  \end{equation}
  where $\tilde\sT \twoheadrightarrow \sT$ is any cover where
  $\kappa^{1/2} \in \bk_{\tilde\sT}$ exists.
\end{definition}

\subsubsection{}

\begin{lemma}[cf. {\cite[Proposition 2.6]{Thomas2020}}] \label{lem:virtual-canonical-half}
  Let $\kappa = t^\theta$ be a $\sT$-weight and suppose that $E =
  \sum_\mu t^\mu E_\mu$ is $\kappa$-symmetric with $E_{\theta/2} = 0$.
  Pick any $\sT$-cocharacter $\sigma$, generic in the sense that
  \[ s(\mu) \coloneqq \inner{\sigma, \mu - \theta/2} \neq 0 \]
  for all $\mu$ such that $E_\mu \neq 0$. Then $\det(E)$ has a square
  root given by
  \[ \det(E)^{\frac{1}{2}} \coloneqq \kappa^{-\frac{1}{2} \rank E_\sigma^{>0}} \det(E_\sigma^{>0}), \]
  where $E^{>0}_\sigma \coloneqq \sum_{\mu :\, s(\mu) > 0} t^\mu E_\mu$, and
  this square root is independent of $\sigma$.
\end{lemma}

In practice, $\kappa^{1/2}$ will not even be an integral $\sT$-weight,
i.e. $\theta$ is not divisible by two in the character lattice of
$\sT$, so $E_{\theta/2} = 0$ automatically. We assume this is the
case, from now on.

\begin{proof}
  The $\kappa$-symmetry $E = -\kappa E^\vee$ identifies $E_\mu =
  -E_{\theta-\mu}^\vee$, and by hypothesis $\mu \mapsto \theta-\mu$ is
  a fixed-point-free involution on $\{\mu : E_\mu \neq 0\}$ which
  flips the sign of $s(-)$. Hence
  \begin{align*}
    \det(E)
    &= \bigotimes_{\mu :\, s(\mu) > 0} \det (t^\mu E_\mu) \otimes \bigotimes_{\mu :\, s(\mu) < 0} \det (t^\mu E_\mu) \\
    &= \bigotimes_{\mu :\, s(\mu) > 0} \det (t^\mu E_\mu) \otimes \bigotimes_{\mu :\, s(\mu) < 0} \det (-t^{\mu} E_{\theta-\mu}^\vee) \\
    &= \bigotimes_{\mu :\, s(\mu) > 0} \det (t^\mu E_\mu) \otimes \bigotimes_{\mu :\, s(\mu) > 0} \det (-t^{\theta-\mu} E_{\mu}^\vee) \\
    &= \bigotimes_{\mu :\, s(\mu) > 0} \det (t^\mu E_\mu) \otimes \bigotimes_{\mu :\, s(\mu) > 0} \det (t^{-\theta+\mu} E_{\mu}) \\
    &= \bigg(\bigotimes_{\mu :\, s(\mu) > 0} \det (t^\mu E_\mu)\bigg)^{\otimes 2} \kappa^{-\sum_{\mu :\, s(\mu) > 0} \rank E_\mu}.
  \end{align*}
  If $\sigma'$ is another choice of cocharacter, then
  \[ E_\sigma^{> 0} - E_{\sigma'}^{> 0} = F + \kappa F^\vee \]
  for some K-theory class $F$. The computation
  \[ \det(F + \kappa F^\vee) = \kappa^{\rank F} = \kappa^{\frac{1}{2} (\rank E_\sigma^{> 0} - \rank E_{\sigma'}^{> 0})} \]
  shows that the square roots defined by $\sigma$ and $\sigma'$ are
  the same.
\end{proof}

\subsubsection{}

\begin{definition}\label{bg:def:univ-enum-inv}
  Let $X$ be an algebraic space with $\sT$-action and a symmetrized
  virtual structure sheaf $\hat\cO_X^\vir$ as defined in
  \eqref{eq:symmetrized-virtual-sheaf}, and furthermore suppose that
  $X^\sT$ is a proper algebraic space with the resolution property.
  Define the {\it universal enumerative invariant}
  \begin{equation} \label{eq:univ-enum-inv}
    \begin{aligned}
      \sZ_X
      &\coloneqq \chi\left(X, \hat\cO_X^\vir \otimes -\right) \\
      &\coloneqq \chi\left(X^{\sT}, \frac{\cO^{\vir}_{X^{\sT}}}{\se(N_\iota^\vir)} \otimes \left(\iota^*\cK_X^\vir\right)^{\frac{1}{2}} \otimes \iota^*(-)\right) \in K_\circ^{\tilde\sT}(X)_\loc
    \end{aligned}
  \end{equation}
  where the first line is an abbreviation for the second (see
  \eqref{eq:symmetrized-virtual-sheaf}), and we are using the
  shorthand notation of \S\ref{sec:universal-invariants-shorthand} for
  elements of ``concrete'' K-homology.

  The typical setting is that $j\colon X \hookrightarrow \fX$ is a
  $\sT$-invariant open locus in a moduli stack $\fX \in \cat{Art}_\sT$
  with a $\sT$-equivariant $\kappa$-symmetric obstruction theory, and
  we will be interested in $j_*\sZ_X \in
  K_\circ^{\tilde\sT}(\fX)_{\loc}$.
\end{definition}

\subsubsection{}

\begin{lemma}[Pole cancellation] \label{lem:master-space-no-poles}
  Let $\sT$ and $\sS$ be tori with coordinates denoted $t$ and $s$
  respectively. Consider symmetrized pullback (as in
  Theorem~\ref{thm:APOTs}\ref{it:APOT-sym-pullback}) along a smooth
  $(\sT \times \sS)$-equivariant morphism $f\colon M \to \fX$. Suppose
  that:
  \begin{enumerate}[label = (\alph*)]
  \item the $(\sT \times \sS)$-equivariant obstruction theory on $\fX$
    is $\kappa$-symmetric for a $\sT$-weight $\kappa$;
  \item \label{it:no-poles-properness} the fixed locus $M^{\sT_w}$ is
    proper for all $(\sT \times \sS)$-weights $w = t^\mu s^\nu$ with
    $\nu \neq 0$, where
    \[ \sT_w \subset \ker(w) \]
    is the maximal torus in $\sT \times \sS$ where $w = 1$.
  \end{enumerate}
  Then, for the $(\sT \times \sS)$-equivariant APOT on $M$ obtained by
  symmetrized pullback,
  \[ \chi(M, \hat\cO_M^\vir \otimes \cE) \in \bk_{\sT,\loc}[\kappa^{\pm\frac{1}{2}}] \otimes_\bZ \bk_{\sS} \]
  for any $\cE \in K_{\sT \times \sS}^\circ(M) \otimes_{\bk_\sT}
  \bk_{\sT,\loc}$.
\end{lemma}

\begin{proof}
  For any $w$ appearing in condition~\ref{it:no-poles-properness},
  consider the obvious closed embeddings
  \[ \iota\colon M^{\sT \times \sS} \hookrightarrow M^{\sT_w} \xrightarrow{\iota_w} M. \]
  Since $M^{\sT_w}$ is proper by hypothesis, so is $M^{\sT \times
    \sS}$. So,
  \[ \chi(M, \hat\cO_M^\vir \otimes \cE) = \chi\left(M^{\sT \times \sS}, \frac{\cO_{M^{\sT \times \sS}}^\vir}{\se(N_\iota^\vir)} \otimes (\iota^*\cK_M^\vir)^{\frac{1}{2}} \otimes \iota^*\cE\right) \in \bk_{\sT \times \sS, \loc}[\kappa^{\pm\frac{1}{2}}] \]
  is well-defined, using Lemma~\ref{lem:virtual-canonical-half} to
  define $(\iota^*\cK_M^\vir)^{1/2}$. To show that it actually lies in
  $\bk_{\sT,\loc}[\kappa^{\pm 1/2}] \otimes_\bZ \bk_\sS$, we claim that
  \[ (\iota_w^*\cK_M^{\vir})^{\frac{1}{2}} \in K_{\sT \times \sS}(M^{\sT_w})[\kappa^{\pm \frac{1}{2}}] \]
  exists, and therefore, by localization with respect to a
  $\bC^\times$ (non-canonical) such that $\sT_w \times \bC^\times =
  \sT \times \sS$,
  \[ \chi(M, \hat\cO_M^\vir \otimes \cE) = \chi\left(M^{\sT_w}, \frac{\cO_{M^{\sT_w}}^\vir}{\se(N_{\iota_w}^\vir)} \otimes (\iota_w^*\cK_M^\vir)^{\frac{1}{2}} \otimes \iota_w^*\cE\right). \]
  The right hand side is clearly well-defined $\sT_w$-equivariantly,
  i.e. after specialization to $w = 1$, and therefore has no poles $(1
  - w)^{-1}$. Since this holds for all $w$ appearing in
  condition~\ref{it:no-poles-properness}, we are done.

  It remains to explain why a square root of $\iota_w^*\cK_M^\vir$
  exists. In the formula \eqref{eq:symmetrized-pullback-Tvir} for
  $\bT_M^\vir$, clearly
  \[ \det(\bT_f - \kappa^{-1} \bT_f^\vee) = \kappa^{\rank \bT_f} \det(\bT_f)^{\otimes 2} \]
  already has a square root, so it suffices to show that the remaining
  piece $\iota_w^*\det(f^*\bT_{\fX}^\vir)$ has a square root. Apply
  Lemma~\ref{lem:virtual-canonical-half} to the splitting of
  $\iota_w^*(f^*\bT_{\fX}^\vir)$ into its $\sT_w$-weight pieces.
  Because $\kappa$ is still non-trivial as a $\sT_w$-weight (here we
  use that $\nu \neq 0$), the lemma provides the desired square root.
\end{proof}

\subsubsection{}

\begin{remark}
  In the special case where a square root of $\cK_X^\vir$ actually
  exists, $(\iota^*\cK_X^\vir)^{1/2} = \iota^*(\cK_X^\vir)^{1/2}$ and
  therefore
  \[ \hat\cO_X^\vir = \cO_X^\vir \otimes (\cK_X^\vir)^{\frac{1}{2}} \]
  by the projection formula. For the rest of this paper, the following
  abuses of notation will be convenient.
  \begin{itemize}
  \item We will write everything {\it as if} such a square root of
    $\cK_X^\vir$ exists. By appropriately pushing and pulling from the
    fixed locus, all formulas and proofs may be modified to work in
    the general case.
  \item Expressions of the form $\hat\cO_X^\vir \otimes \cE$ will
    always be shorthand for $\iota_*(\cdots \otimes \iota^*\cE)$ when
    $\iota^*\cE$ is well-defined but $\cE$ is not. This is consistent
    with the conventions in
    \S\ref{sec:monoidal-stack-vertex-product-well-defined} and
    Proposition~\ref{prop:rigidity}.
  \end{itemize}
  Both are well-illustrated by the statement and proof of
  Lemma~\ref{lem:symmetrized-projective-bundle-formula} below.
\end{remark}

\subsubsection{}

\begin{definition}\label{def:symm-pullback-homology}
  Given a non-trivial $\sT$-weight $\kappa$ and a smooth proper
  $\sT$-equivariant map $\pi\colon \fX \to \fY$ of Artin stacks with
  $\sT$-action, the {\it $\kappa$-symmetrized pullback on K-homology}
  is
  \begin{align*}
    \hat\pi^*\colon K_\circ^\sT(\fY) &\to K_\circ^\sT(\fX) \\
    \phi &\mapsto \pi^*\phi \cap \left(\se(\kappa^{-1} \otimes \bL_\pi) \otimes \kappa^{-\frac{\dim \pi}{2}} \cK_\pi\right).
  \end{align*}
  Here $\dim \pi$ is the relative dimension of $\pi$, and $\cK_\pi
  \coloneqq \det(\bL_\pi)$ is the relative canonical.
\end{definition}

\subsubsection{}
\label{bg:rem:symm-pb-formula-alg-sp}

In the set-up of Definition~\ref{def:symm-pullback-homology}, suppose
further that $\fX = X$ and $\fY = Y$ are algebraic spaces, and that
$Y$ carries a $\sT$-equivariant $\kappa$-symmetric APOT. Endow $X$
with the symmetrized pullback APOT. Then both $X$ and $Y$ admit
symmetrized virtual structure sheaves, related by the formula
\begin{equation} \label{eq:symm-pullback-Ovir}
  \hat\cO_X^{\vir} = \pi^*\hat\cO_Y^{\vir} \otimes \se(\kappa^{-1} \otimes \bL_{\pi}) \otimes \kappa^{-\frac{\dim \pi}{2}}\cK_{\pi}.
\end{equation}
If, additionally, $Y^{\sT}$ is proper, then $X^{\sT}$ is proper as
well. Assume $X^{\sT}, Y^{\sT}$ are proper algebraic spaces with the
resolution property, so that there are universal enumerative
invariants
\begin{align*}
  \sZ_X &\coloneqq \chi(\hat\cO_X^{\vir} \otimes -) \in K_\circ^{\tilde\sT}(X)_{\loc}, \\
  \sZ_Y &\coloneqq \chi(\hat\cO_Y^{\vir} \otimes -) \in K_\circ^{\tilde\sT}(Y)_{\loc}.
\end{align*}
By the projection formula, \eqref{eq:symm-pullback-Ovir} becomes
\[ \sZ_X = \hat\pi^* \sZ_Y. \]
This motivates the definition of symmetrized pullback on K-homology.
We will use it to prove a K-homology version
(Lemma~\ref{lem:homology-projective-bundle-formula}) of the following
lemma.

\subsubsection{}

\begin{lemma}[Virtual projective bundle formula] \label{lem:symmetrized-projective-bundle-formula}
  Let $X$ be an algebraic space with $\sT$-action and a
  $\sT$-equivariant $\kappa$-symmetric APOT for a $\sT$-weight
  $\kappa$. Let $\cV \in \cat{Vect}_\sT(X)$, and equip $\bP(\cV)$ with the
  APOT obtained by symmetrized pullback along the natural projection
  \[ \pi\colon \bP(\cV) \to X. \]
  Let $s = \cO_{\bP(\cV)}(1)$ be the dual of the universal line bundle
  on $\bP(\cV)$. If $\sT$ acts trivially on $X$, then for any $f(s) \in
  K_\sT^\circ(X)[s^\pm]$,
  \begin{equation} \label{eq:symmetrized-projective-bundle-formula}
    (\kappa^{-\frac{1}{2}} - \kappa^{\frac{1}{2}}) \cdot \pi_*\left(f(s) \otimes \hat\cO_{\bP(\cV)}^{\vir}\right) = \rho_{K,z} \left(f(z) \frac{\hat\se_{z^{-1}}(\kappa^{-1} \cV^\vee)}{\hat\se_z(\cV)}\right) \otimes \hat\cO_X^{\vir}.
  \end{equation}
\end{lemma}

\begin{proof}
  Use the formula \eqref{eq:symm-pullback-Ovir} for
  $\hat\cO_{\bP(\cV)}^\vir$ and the formula
  \[ \bT_\pi = s \otimes (\pi^* \cV - s^{-1}) = s\pi^* \cV - 1. \]
  The left hand side of
  \eqref{eq:symmetrized-projective-bundle-formula} becomes
  \begin{align*}
    &(\kappa^{-\frac{1}{2}} - \kappa^{\frac{1}{2}}) \cdot \pi_*\left(f(s) \otimes \pi^*\hat\cO_X^\vir \otimes \se(\kappa^{-1} s^{-1} \pi^* \cV^\vee - \kappa^{-1}) \otimes \kappa^{-\frac{\rank \cV-1}{2}} \det(s \pi^* \cV)^\vee \right) \\
    &= \pi_*\left(f(s) \otimes \se_{s^{-1}}(\kappa^{-1} \pi^* \cV^\vee) \otimes \kappa^{-\frac{\rank \cV}{2}} \det(s \pi^* \cV)^\vee\right) \otimes \hat\cO_X^\vir \\
    &= \rho_{K,z} \left(f(z) \frac{\se_{z^{-1}}(\kappa^{-1} \cV^\vee)}{\se_z(\cV)} \kappa^{-\frac{\rank \cV}{2}} \det(z \cV)^\vee\right) \otimes \hat\cO_X^\vir,
  \end{align*}
  where the first equality is the projection formula and the second
  equality is the projective bundle formula
  \eqref{eq:projective-bundle-formula}. This is equal to
  the right hand side by the definition of $\hat\se$.
\end{proof}

\subsubsection{}

\begin{corollary} \label{cor:symmetrized-projective-bundle-formula}
  In the setting of
  Lemma~\ref{lem:symmetrized-projective-bundle-formula},
  \[ \pi_* \hat\cO_{\bP(\cV)}^{\vir} = [\rank \cV]_\kappa \cdot \hat\cO_X^\vir \in K_{\tilde\sT}(X). \]
\end{corollary}

\begin{proof}
  We may do this computation analytically, i.e. treating the Chern
  roots $\cL$ of $\cV$ as generic non-zero complex numbers close to $1$.
  Then
  \begin{equation} \label{eq:k-theoretic-residue-rigidity}
    \rho_{K,z} \frac{\hat\se_{z^{-1}}(\kappa^{-1} \cV^\vee)}{\hat\se_z(\cV)} = \left(\lim_{z \to \infty} - \lim_{z \to 0}\right) \frac{\hat\se_{z^{-1}}(\kappa^{-1} \cV^\vee)}{\hat\se_z(\cV)}
  \end{equation}
  because both limits exist, namely:
  \[ \lim_{z \to \infty} = (-\kappa^{\frac{1}{2}})^{\rank \cV}, \qquad \lim_{z \to 0} = (-\kappa^{-\frac{1}{2}})^{\rank \cV} \]
  by writing
  \[ \frac{\hat\se_{z^{-1}}(\kappa^{-1} \cV^\vee)}{\hat\se_z(\cV)} = \prod_\cL \frac{(z\kappa \cL)^{-\frac{1}{2}} - (z\kappa\cL)^{\frac{1}{2}}}{(z\cL)^{\frac{1}{2}} - (z\cL)^{-\frac{1}{2}}}. \qedhere \]
\end{proof}

This computation motivates the sign $(-1)^{N-1}$ in the definition
\eqref{eq:quantum-integer} of $[N]_\kappa$.

\subsubsection{}

\begin{remark}
  The remarkably explicit form of the residue
  \eqref{eq:k-theoretic-residue-rigidity} as a quantum integer is the
  crucial ``rigidity'' property that enables us to obtain explicit
  wall-crossing formulas in practical applications later on. The
  existence of the limits really requires the $\kappa$-symmetry of the
  APOTs, as well as the usage of the symmetrized $\hat\se$ and
  $\hat\cO^\vir$ instead of $\se$ and $\cO^\vir$.
\end{remark}

\subsubsection{}

\begin{lemma}[Homology projective bundle formula] \label{lem:homology-projective-bundle-formula}
  Let $\Pi_X\colon \fX \to X$ be a $\sT$-equivariant
  $\bC^\times$-gerbe over an Artin stack $X$ with $\sT$-action, and
  $\cV \in \cat{Vect}_\sT(\fX)$ have weight $1$ with respect to
  $\Pi_X$. Let
  \[ P_{\fX}(\cV) \coloneqq \tot_{\fX}(\cV) \setminus \{0\} \]
  and consider the commutative diagram
  \[ \begin{tikzcd}
    P_{\fX}(\cV) \ar{drr}{\pi_X} \ar[hookrightarrow]{d}{j} \\
    \tot_{\fX}(\cV) \ar[shift left]{r} & \fX \ar[shift left]{l}{i} \ar{r}[swap]{\Pi_X} & X
  \end{tikzcd} \]
  where $j$ and $\pi_X$ are the natural inclusion and projection
  respectively, and $i$ is the zero section. Given $\phi \in
  K_\circ^\sT(X)$,
  \[ (\kappa^{-\frac{1}{2}} - \kappa^{\frac{1}{2}}) \cdot j_*\hat\pi_X^*\phi = \rho_{K,z} \left(i_* D(z)\left(\tilde \phi \cap \frac{\hat\se_{z^{-1}}(\kappa^{-1} \cV^\vee)}{\hat\se_z(\cV)}\right)\right) \]
  where $\tilde\phi \in K_\circ^\sT(\fX)$ is an
  arbitrary lift of $\phi$ (i.e. $\phi = (\Pi_X)_*\tilde\phi$).
\end{lemma}

Throughout the following proof, since $\tilde\phi$ is ultimately
computed in $K_\sT(-)$ rather than $K_\sT^\circ(-)$, without loss of
generality we may let $K_\sT^\circ(-)$ denote its image inside
$K_\sT(-)$.

\begin{proof}
  Let $\cL$ be the weight-$1$ line bundle on $[\pt/\bC^\times]$. Since
  $\pi_X$ is not necessarily a projective bundle, we replace it by one
  using the diagram
  \[ \begin{tikzcd}
    \tot_{[\pt/\bC^\times] \times \fX}(\cL \boxtimes \cV) \ar{d}{\Pi_{\tot(\cV)}} & \bP_{\fX}(\cV) \ar{l}{\tilde j} \ar{r}{\pi_{\fX}} \ar{d}{\Pi_{P(\cV)}} & \fX \ar{d}{\Pi_X} \\
    \tot_{\fX}(\cV) & P_{\fX}(\cV) \ar{l}{j} \ar{r}{\pi_X} & X.
  \end{tikzcd} \]
  Both squares are Cartesian: the base change of $\pi_X$ along $\Pi_X$
  is naturally identified with the projection $\tot(\cL \boxtimes \cV)
  \setminus \{0\} \to \fX$, because $\cV$ has weight one, but there is
  a natural isomorphism $\bP(\cV) \xrightarrow{\sim} \tot(\cL
  \boxtimes \cV) \setminus \{0\}$ which identifies the pullback of
  $\cL^\vee$ with the universal line bundle on $\bP(\cV)$. By
  functoriality and base change,
  \begin{equation} \label{eq:P-to-projective-bundle}
    j_*\hat\pi_X^*\phi = j_*\hat\pi_X^*(\Pi_X)_*\tilde\phi = j_*(\Pi_{P(\cV)})_* \hat\pi_{\fX}^*\tilde\phi = (\Pi_{\tot(\cV)})_*\tilde j_* \hat\pi_{\fX}^*\tilde\phi.
  \end{equation}

  The virtual projective bundle formula
  (Lemma~\ref{lem:symmetrized-projective-bundle-formula}) is now
  applicable to $\pi_{\fX}\colon \bP(\cV) \to \fX$, as follows.
  By the equivariant localization axiom for $\tilde\phi$, we may
  assume without loss of generality that $\fX$ is an algebraic space
  with trivial $\sT$-action. Then the proof of
  Lemma~\ref{lem:symmetrized-projective-bundle-formula}, applied to
  $\pi_{\fX}\colon \bP(\cV) \to \fX$ with $s = \cO_{\bP(\cV)}(1)$,
  shows that
  \begin{align} 
    &(\kappa^{-\frac{1}{2}} - \kappa^{\frac{1}{2}}) \cdot (\hat\pi_{\fX}^*\tilde\phi)(f(s)) \nonumber \\
    &=(\kappa^{-\frac{1}{2}} - \kappa^{\frac{1}{2}}) \cdot \tilde\phi\left((\pi_{\fX})_*\left(f(s) \otimes \se(\kappa^{-1} \bL_{\pi_{\fX}}) \otimes \kappa^{-\frac{\rank \cV-1}{2}} \cK_{\pi_{\fX}}\right)\right) \nonumber \\
    &= \rho_{K,z} \tilde\phi\left(f(z) \frac{\hat\se_{z^{-1}}(\kappa^{-1} \cV^\vee)}{\hat\se_z(\cV)}\right). \label{eq:homology-symmetrized-pullback-projective-bundle}
  \end{align}
  
  It remains to express the operation $\tilde
  j^*(\Pi_{\tot(\cV)})^*\colon K^\circ_\sT(\tot(\cV)) \to
  K^\circ_\sT(\bP(\cV))$ as an explicit polynomial in $s$. The square
  (of solid arrows)
  \[ \begin{tikzcd}
    \tot(\cL \boxtimes \cV) \ar{r}{\tilde p} \ar{d}{\Pi_{\tot(\cV)}} & {[\pt/\bC^\times]} \times \fX \ar{d}{\Psi} \\
    \tot(\cV) \ar[shift left]{r}{p} & \fX \ar[dotted,shift left]{l}{i}
  \end{tikzcd} \]
  is Cartesian, and the dotted arrow $i$ is the zero section of $\cV$.
  By the Thom isomorphism theorem, $i^*$ and $p^*$ are isomorphisms on
  $K_\sT(-)$, so
  \[ \tilde j^*(\Pi_{\tot(\cV)})^* = \tilde j^* \tilde p^* \Psi^* i^*. \]
  If $\cG$ is a $\sT$-equivariant sheaf on $\fX$ with weight
  decomposition $\cG = \bigoplus_{i \in \bZ} \cG_i$ with respect to
  $\Pi_X$, then
  \[ \tilde j^* \tilde p^* \Psi^* \cG = \bigoplus_{i \in \bZ} \tilde j^* \tilde p^*(\cL^i \boxtimes \cG_i) = \bigoplus_{i \in \bZ} s^i \otimes \pi_{\fX}^*\tilde \cG_i, \]
  since $s = \tilde j^* \tilde p^*(\cL \boxtimes \cO)$. It follows
  that
  \begin{equation} \label{eq:homology-symmetrized-pullback-projective-bundle-2}
    \tilde j^*(\Pi_{\tot(\cV)})^* = s^{\deg} i^*.
  \end{equation}
  Combining \eqref{eq:P-to-projective-bundle},
  \eqref{eq:homology-symmetrized-pullback-projective-bundle}, and
  \eqref{eq:homology-symmetrized-pullback-projective-bundle-2}, we
  obtain the desired result.
\end{proof}

\section{Semistable invariants}
\label{sec:semistable-invariants}

\subsection{Construction and main result}

\subsubsection{}

The goal of this section is to construct semistable invariants
satisfying Theorem~\ref{thm:sst-invariants}. Throughout this section,
we fix the following data:
\begin{itemize}
\item a weak stability condition $\tau$ on $\cat{A}$ satisfying
  Assumption~\ref{assump:semistable-invariants};
\item the class $\alpha \in C(\cat{A})$ for which we want to construct
  the semistable invariant $\sz_\alpha(\tau)$
  (Theorem~\ref{thm:sst-invariants});
\item a set $\Frs$ of framing functors satisfying
  Assumption~\ref{assump:semistable-invariants}\ref{assump:it:framing-functor},
  such that $\fM_\alpha^{\sst}(\tau) \subset \fM_\alpha^{\Fr,\pl}$.
\end{itemize}
The strategy is to define $\sz_\alpha(\tau) \coloneqq 0$ if
$\fM_\alpha^{\sst}(\tau) = \emptyset$, and otherwise to induct on
$r(\alpha) > 0$. It turns out it is enough to consider
$\tau$-semistable objects $E \in \cat{A}$ whose classes belong to the
finite set
\[ R_\alpha \coloneqq \{\alpha\} \cup \{\beta \in C(\cat{A}) : \alpha-\beta\in C(\cat{A}), \, \tau(\beta) = \tau(\alpha-\beta), \, \fM_{\beta}^{\sst}(\tau), \fM_{\alpha-\beta}^{\sst}(\tau)\neq \emptyset\} \]
from Definition~\ref{def:dominates-at}.

\subsubsection{}

\begin{definition} \label{def:pair-invariant}
  Let $Q$ denote the quiver
  \[ \begin{tikzcd}
    \overset{V}{\blacksquare} \ar{r}{\rho} & \overset{\Fr(E)}{\blackbullet}
  \end{tikzcd}. \]
  Given a framing functor $\Fr \in \Frs$, consider the auxiliary exact
  category $\tilde{\cat{A}}^{Q(\Fr)}$ parameterizing triples $(E, V,
  \rho)$ as labeled above. Define
   \begin{equation} \label{eq:pair-stability}
    \tau^Q(\beta, d) \coloneqq \begin{cases}
      \left(\tau(\beta), \frac{d}{r(\beta)}\right) & \beta \neq 0, \, \tau(\beta) = \tau(\alpha-\beta), \\
      \left(\tau(\beta), \infty\right), & \beta \neq 0, \, \tau(\beta)> \tau(\alpha - \beta), \\
      \left(\tau(\beta), -\infty\right), & \beta \neq 0, \, \tau(\beta)< \tau(\alpha - \beta), \\
      \left(\infty, 1\right) & \beta = 0,
    \end{cases}
  \end{equation}
  where pairs $(a, b)$ are ordered lexicographically, i.e. $(a, b) <
  (a', b')$ means either $a < a'$, or $a = a'$ and $b < b'$. This is a
  special case of Definition~\ref{def:flag-invariant} (with $N = 1$,
  $\vec\mu = 1$, $\lambda = 0$, and $x = 0$), thus it is effectively a
  weak stability condition on $\tilde{\cat{A}}^{Q(\Fr)}$ by
  Lemma~\ref{lem:joyce-framed-stack-stability}. By
  Lemma~\ref{lem:pairs-stack-sst=st} below, there are no strictly
  semistable objects when $d = 1$. By
  Assumption~\ref{assump:semistable-invariants}\ref{assump:it:properness},
  the open locus
  \[ \tilde\fM_{\alpha,1}^{Q(\Fr),\sst}(\tau^Q)^{\sT} \subset (\tilde\fM_{\alpha,1}^{Q(\Fr),\pl})^{\sT} \]
  is therefore an algebraic space and may be equipped with the
  $\kappa$-symmetric APOT obtained by symmetrized pullback
  (Theorem~\ref{thm:APOTs}) along the smooth forgetful map
  $\pi_{\fM_\alpha^{\Fr,\pl}}$, and thus the universal enumerative
  invariant
  \begin{equation} \label{eq:pairs-stack-enumerative-invariant}
    \tilde\sZ_{\alpha,1}^{\Fr}(\tau^Q) \coloneqq \chi\left(\tilde\fM_{\alpha,1}^{Q(\Fr), \sst}(\tau^Q), \hat{\cO}^{\vir} \otimes - \right) \in K_\circ^{\tilde\sT}(\tilde\fM_{\alpha,1}^{Q(\Fr),\pl})_{\loc}
  \end{equation}
  is well-defined (Definition~\ref{bg:def:univ-enum-inv}). Define also
  the K-homology element
  \[ \partial \coloneqq \chi\left(\tilde\fM_{0,1}^{Q(\Fr),\pl}, -\right) = \id \in K_\circ^{\tilde\sT}(\tilde\fM_{0,1}^{Q(\Fr),\pl})_\loc. \]
  Both $\tilde\fM_{0,1}^{Q(\Fr),\pl} = \pt$ and $\partial$ are
  independent of $\Fr$, but Lie brackets with $\partial$ depend on
  $\fr$.
\end{definition}

\subsubsection{}

\begin{lemma} \label{lem:pairs-stack-sst=st}
  Take $(E, V, \rho) \in \tilde{\cat{A}}^{Q(\Fr)}$ such that $E$ is
  $\tau$-semistable of class $\beta \in R_\alpha$ and $\dim V = 1$.
  Then $(E, V, \rho)$ is $\tau^Q$-stable if and only if it is
  $\tau^Q$-semistable, if and only if:
  \begin{itemize}
  \item $\rho \neq 0$;
  \item there does not exist $0 \neq E' \subsetneq E$ with $\tau(E') =
    \tau(E/E')$ and $\rho(V) \subset \Fr(E') \subset \Fr(E)$.
  \end{itemize}
  In particular, if $E$ is $\tau$-stable, then $(E, V, \rho)$ is
  $\tau^Q$-(semi)stable if and only if $\rho \neq 0$.
\end{lemma}

\begin{proof}
  If $\rho = 0$, then $(0, V, 0) \subset (E, V, \rho)$ is a
  destabilizing sub-object, and if such a $E' \subset E$ exists then
  $(E',V,\rho) \subset (E,V,\rho)$ is a destabilizing sub-object.

  Conversely, suppose $(E',V',\rho') \subset (E,V,\rho)$ is a
  destabilizing sub-object. Suppose $E' = 0$. Since $(E',V',\rho')$ is
  not the zero object, $\dim V' > 0$, thus $\dim V' = \dim V$ and
  $\rho' = \rho = 0$. Otherwise, if $E' \neq 0$, since $E$ is
  $\tau$-semistable it must be that $\tau(E') = \tau(E/E')$. By
  Assumption~\ref{assump:semistable-invariants}\ref{assump:it:rank-function},
  $r(E) = r(E') + r(E/E')$ and $r(E'), r(E/E') > 0$. Since $\dim V =
  0$ cannot be destabilizing, $\dim V = 1$ and thus $\rho$ factors
  through $E'$.

  Finally, $\tau^Q(E',V',\rho') \neq \tau^Q(E/E',V/V',\rho/\rho')$ in
  every case, so $\tau^Q$-semistable is equivalent to $\tau^Q$-stable.
\end{proof}

\subsubsection{}
\label{sec:pairs-maps}

Throughout this section, we will frequently use the following maps and
commutative square:
\begin{equation} \label{eq:pairs-forgetful-rigidification-maps}
  \begin{tikzcd}
    \tilde\fM_{\alpha,0}^{Q(\Fr)} & \tilde\fM_{\alpha,1}^{Q(\Fr)} \ar[shift left]{r}{\Pi_{\alpha,1}^\pl} \ar{d}{\pi_{\fM_\alpha^{\Fr}}} & \tilde\fM_{\alpha,1}^{Q(\Fr),\pl} \ar{d}{\pi_{\fM_\alpha^{\Fr,\pl}}} \ar[shift left]{l}{I_{\alpha,1}} \\
    \fM_\alpha^{\Fr} \ar{u}{\iota^Q_\alpha} & \fM_\alpha^{\Fr} \ar{r}{\Pi_\alpha^\pl} & \fM_\alpha^{\Fr,\pl}.
  \end{tikzcd}
\end{equation}
Here $\iota^Q_\alpha$ are the natural isomorphisms given by $E \mapsto
(E, 0, 0)$; in particular, $\tilde\fM_{0,1}^{Q(\Fr),\pl} = \pt$. Both
$\iota^Q$ and the forgetful maps $\pi_{\fM^{\Fr}}$ induce morphisms of
graded monoidal $\sT$-stacks. The $I_{\alpha,1}$ are the canonical
de-rigidification maps (Definition~\ref{def:de-rigidification}), using
which we obtain elements
\[ I_*\tilde \sZ_{\alpha,1}^{\Fr}(\tau^Q) \in K_\circ^{\tilde\sT}(\tilde\fM_{\alpha,1}^{Q(\Fr)})_\loc, \qquad I_*\partial \in K_\circ^{\tilde\sT}(\tilde\fM_{0,1}^{Q(\Fr)})_\loc. \]
We continue to use the same symbols to denote their images in
$K_\circ^{\tilde\sT}(-)^\pl_{\loc}$.

\subsubsection{}

\begin{theorem}[Semistable invariants] \label{thm:sst-invariants-construction}
  Suppose $\tau$ is a weak stability condition on $\cat{A}$ for which
  Assumption~\ref{assump:semistable-invariants} holds. For classes
  $\alpha \in C(\cat A)$ and framing functors $\Fr \in \Frs$ such that
  $\fM_{\alpha}^{\sst}(\tau) \subset \fM_{\alpha}^{\Fr,\pl}$, the
  equations
  \begin{equation}\label{eq:sstable-def}
    I_*\tilde{\sZ}_{\alpha,1}^{\Fr}(\tau^Q) = \sum_{\substack{n>0 \\ \alpha = \alpha_1+\cdots+\alpha_n\\ \forall i: \,\tau(\alpha_i) = \tau(\alpha)\\ \;\;\fM_{\alpha_i}^{\sst}(\tau) \neq \emptyset}} \frac{1}{n!} \left[\iota^Q_*\sz^{\Fr}_{\alpha_n}(\tau), \left[\cdots,\left[\iota^Q_*\sz^{\Fr}_{\alpha_2}(\tau), \left[\iota^Q_*\sz^{\Fr}_{\alpha_1}(\tau), I_*\partial\right]\right]\cdots\right]\right],
  \end{equation}
  in the Lie algebra
  $K_\circ^{\tilde\sT}(\tilde\fM^{Q(\Fr)})^\pl_{\loc,\bQ}$
  (Theorem~\ref{thm:auxiliary-stack-vertex-algebra}), uniquely define
  elements
  \[ \sz_\alpha^{\Fr}(\tau) \in K_\circ^\sT(\fM_\alpha)^\pl_{\loc,\bQ} \]
  which are independent of $\Fr$ and satisfy all the properties listed
  in Theorem~\ref{thm:sst-invariants}.
\end{theorem}

The goal of this section is to prove
Theorem~\ref{thm:sst-invariants-construction}. In the remainder of
this subsection, we construct the invariants $\sz_\alpha^{\Fr}(\tau)$
and prove that they satisfy properties~\ref{item:vss-support},
\ref{item:vss-no-strictly-semistables} and
\ref{item:vss-isomorphic-moduli} in Theorem~\ref{thm:sst-invariants}.
The remaining property~\ref{item:vss-pairs-relation} and framing
independence require a master space relation
(Theorem~\ref{thm:pairs-master-relation}) and are proved in
\S\ref{sec:vss-framing-independence}.

\subsubsection{}

\begin{lemma} \label{lem:pushforward-of-bracket-partial}
  The operator
  \[ [\iota^Q_*(-), I_*\partial]\colon K_\circ^{\tilde\sT}(\fM_\alpha^{\Fr})^\pl_{\loc} \to K_\circ^{\tilde\sT}(\tilde\fM_{\alpha,1}^{Q(\Fr)})^\pl_{\loc} \]
  has a left inverse. Explicitly,
  \[ (\pi_{\fM_\alpha^{\Fr}})_* [\iota^Q_*\phi, I_*\partial] = [\fr(\alpha)]_\kappa \cdot \phi. \]
\end{lemma}

\begin{proof}
  By the definition of the Lie bracket
  (Theorem~\ref{thm:auxiliary-stack-vertex-algebra}),
  \begin{align*}
    (\kappa^{-\frac{1}{2}} - \kappa^{\frac{1}{2}}) \cdot &\left((\pi_{\fM_\alpha^{\Fr}})_* [\iota^Q_*\phi, I_*\partial]\right)(\cE) \\
    &= \rho_K (\phi \boxtimes I_*\partial)(\iota^Q \times \id)^*\left(\hat\Theta_{(\alpha,0),(0,1)}(z) \otimes z^{\deg_1} \Phi_{(\alpha,0),(0,1)}^* \pi_{\fM_\alpha^{\Fr}}^* \cE\right) \\
    &= (\phi \boxtimes I_*\partial) \rho_K\left(\hat\Theta_{\alpha,(0,1)}(z) \otimes z^{\deg_1} \pr_1^* \cE\right)
  \end{align*}
  where $\hat\Theta_{\alpha,(0,1)}(z) \coloneqq (\iota^Q \times
  \id)^*\hat\Theta_{(\alpha,0),(0,1)}(z)$ and the second equality
  follows from Lemma~\ref{lem:pl-group-functoriality} and the
  commutative triangle
  \[ \begin{tikzcd}[column sep=large]
    \fM_\alpha^{\Fr} \times \tilde\fM_{0,1}^{Q(\Fr)} \ar{r}{\Phi \circ (\iota^Q \times \id)} \ar{dr}[swap]{\pr_1} & \tilde\fM_{\alpha,1}^{Q(\Fr)} \ar{d}{\pi_{\fM_\alpha^{\Fr}}} \\
    {} & \fM_\alpha^{\Fr}
  \end{tikzcd} \]
  where $\pr_1$ is projection to the first factor. In
  $K_\circ^{\tilde\sT}(-)^\pl$, by definition $z^{\deg_1} = \id$, and
  since $\cE$ itself is independent of $z$, it remains to compute that
  \[ \rho_K \hat\Theta_{\alpha,(0,1)}(z) = (-1)^{-\fr(\alpha)} \left(\kappa^{\frac{1}{2} \fr(\alpha)} - \kappa^{-\frac{1}{2} \fr(\alpha)}\right) \]
  using the same steps as in the proof of
  Proposition~\ref{prop:rigidity} and that $\rank
  \tilde\scE^{Q(\Fr)}_{(\alpha,0),(0,1)} = \fr(\alpha)$. Putting it all
  together, we get
  \[ \left(\pi^\pl_* [\iota^Q_*\phi, \partial]\right)(\cE) = [\fr(\alpha)]_\kappa \cdot (\phi \boxtimes I_*\partial)(\pr_1^* \cE) = [\fr(\alpha)]_\kappa \cdot \phi(\cE) \]
  as claimed. In the second equality we used that $(I_*\phi)(\cO) =
  \phi(\cO) = 1$.
\end{proof}

\subsubsection{}

\begin{definition}[Construction of semistable invariants] \label{def:sstable-explicit}
  If $\fM_\alpha^\sst(\tau) = \emptyset$, we define
  $\sz_\alpha^{\Fr}(\tau) \coloneqq 0$. Otherwise, the implicit
  characterization \eqref{eq:sstable-def} of $\sz_\alpha^{\Fr}$ can be
  written as
  \[ I_*\tilde\sZ_{\alpha,1}^{\Fr}(\tau^Q) = \left[\iota^Q_*\sz_\alpha^{\Fr}(\tau), I_*\partial\right] + \sum_{n>1} \cdots, \]
  where the Lie bracket is the $n=1$ term. Using
  Lemma~\ref{lem:pushforward-of-bracket-partial}, this can be turned
  into an explicit definition by applying $(\pi_{\fM_\alpha^{\Fr}})_*$
  to both sides and dividing by $[\fr(\alpha)]_\kappa$, namely:
  \begin{equation} \label{eq:sstable-explicit}
    \sz_\alpha^{\Fr}(\tau) = \frac{1}{[\fr(\alpha)]_\kappa} (\pi_{\fM_\alpha^{\Fr}})_* \bigg(I_*\tilde\sZ_{\alpha,1}^{\Fr}(\tau^Q) - \sum_{n>1} \cdots\bigg)
  \end{equation}
  where the omitted terms $\cdots$ only involve
  $\sz_\beta^{\Fr}(\tau)$ for $\beta \in R_\alpha$, so in particular
  $r(\beta) < r(\alpha)$. Since $\fM_\alpha^{\sst}(\tau) \subset
  \fM_\alpha^{\Fr,\pl}$, if $\beta$ appears in the decompositions in
  \eqref{eq:sstable-def} then $\fM_{\beta}^{\sst}(\tau) \subset
  \fM_{\beta}^{\Fr,\pl}$ as well because $\cat{A}^{\Fr}$ is closed
  under direct summands. Hence \eqref{eq:sstable-explicit} is a valid
  inductive definition of $\sz_\alpha^{\Fr}(\tau)$.
\end{definition}

\subsubsection{}

\begin{proposition}[Theorem~\ref{thm:sst-invariants}\ref{item:vss-support}] \label{prop:vss-support}
  $\sz_\alpha^{\Fr}(\tau)$ is supported on
  $(\Pi^\pl_\alpha)^{-1}(\fM_\alpha^{\sst}(\tau))$.
\end{proposition}

\begin{proof}
  The result of \eqref{eq:sstable-explicit} is clearly supported on
  $\fM_\alpha^{\Fr} \subset \fM_\alpha$. Since all steps in the
  construction of Definition~\ref{def:sstable-explicit} hold equally
  well after shrinking the domain of definition of $\Fr$ from
  $\cat{A}^{\Fr}$ to the full exact $\sT$-invariant subcategory of
  $\tau$-semistable objects, we may without loss of generality assume
  that $\fM_{\beta}^{\Fr} = (\Pi^\pl_\beta)^{-1}(\fM_{\beta}^{\sst})$
  for all classes $\beta$. This shows that $\sz_{\alpha}^{\Fr}(\tau)$
  is in fact supported on
  $(\Pi^\pl_\alpha)^{-1}(\fM_{\alpha}^{\sst})$.
\end{proof}

\subsubsection{}

\begin{proposition}[Theorem~\ref{thm:sst-invariants}\ref{item:vss-no-strictly-semistables}] \label{prop:vss-no-strictly-semistables}
  If there are no strictly $\tau$-semistable objects of class
  $\alpha$, then
  \[ (\Pi^\pl_\alpha)_*\sz_{\alpha}^{\Fr}(\tau) = \chi\left(\fM^{\sst}_{\alpha}(\tau), \hat\cO^{\vir}_{\fM_{\alpha}^{\sst}(\tau)}\otimes -\right). \]
\end{proposition}

\begin{proof}
  If terms with $n > 1$ exist in the sum in
  \eqref{eq:sstable-explicit}, then a choice of $[E_i] \in
  \fM_{\alpha_i}^\sst(\tau)$ for $i = 1, \ldots, n$ gives $[E_1 \oplus
    \cdots \oplus E_n] \in \fM_\alpha^\sst(\tau)$ which is strictly
  $\tau$-semistable. So only the $n = 1$ term contributes. Apply
  $(\Pi^\pl_\alpha)_*$ to both sides and use the commutative square
  \eqref{eq:pairs-forgetful-rigidification-maps} to obtain
  \[ (\Pi^\pl_\alpha)_*\sz_\alpha^{\Fr}(\tau) = \frac{1}{[\fr(\alpha)]_\kappa} (\pi_{\fM_\alpha^{\Fr,\pl}})_*\tilde{\sZ}_{\alpha,1}(\tau^Q). \]
  By
  Assumption~\ref{assump:semistable-invariants}\ref{assump:it:properness},
  $\fM_\alpha^{\sst}(\tau)^{\sT}$ is a proper algebraic space. By
  Lemma~\ref{lem:pairs-stack-sst=st}, the restriction
  \[ \pi_{\fM_\alpha^{\Fr,\pl}}\colon \tilde\fM_{\alpha,1}^{Q(\Fr),\sst}(\tau^Q) \to \fM_\alpha^{\sst}(\tau) = \fM_\alpha^{\st}(\tau) \]
  is a $\bP^{\fr(\alpha)-1}$-bundle, so by the projection formula and
  the virtual projective bundle formula
  (Corollary~\ref{cor:symmetrized-projective-bundle-formula}),
  \[ \frac{1}{[\fr(\alpha)]_\kappa} \chi\left(\tilde\fM_{\alpha,1}^{Q(\Fr),\sst}, \hat\cO^\vir \otimes (\pi_{\fM_\alpha^{\Fr,\pl}})^*(-)\right) = \chi\left(\fM_\alpha^\sst(\tau), \hat\cO^\vir \otimes -\right). \]
  Note that the right hand side is well-defined by
  Assumption~\ref{assump:semistable-invariants}\ref{assump:it:properness}.
\end{proof}

\subsubsection{}

\begin{proposition}[Theorem~\ref{thm:sst-invariants}\ref{item:vss-isomorphic-moduli}]
  If $\fM_\alpha^{\sst}(\tau) = \fM_\alpha^{\sst}(\tau')$ for all
  $\alpha$, then
  \[ \sz_\alpha^{\Fr}(\tau) = \sz_\alpha^{\Fr}(\tau') \]
  for all $\alpha$.
\end{proposition}

\begin{proof}
  This is clear from the inductive definition
  \eqref{eq:sstable-explicit}: the hypothesis implies
  \[ \tilde\fM_{\alpha,1}^{\sst}(\tau^Q) = \tilde\fM_{\alpha,1}^{\sst}((\tau')^Q), \]
  which implies $\tilde\sZ_{\alpha,1}^{\Fr}(\tau^Q) =
  \tilde\sZ_{\alpha,1}^{\Fr}((\tau')^Q)$, and, by induction on
  $r(\alpha)$, it follows that $\sz_\alpha^{\Fr}(\tau) =
  \sz_\alpha^{\Fr}(\tau')$.
\end{proof}

\subsection{The master Lie algebra}
\label{sec:semistable-invariants-master-space}

\subsubsection{}

\begin{definition}
  Let $Q \wedge Q$ denote the quiver
  \[ \begin{tikzcd}[row sep=tiny]
    & \overset{V_1}{\blacksquare} \ar{r}{\rho_1} & \overset{\Fr_1(E)}{\blackbullet} \\
    \overset{V_3}{\blacksquare} \ar{ur}{\rho_3} \ar{dr}[swap]{\rho_4} \\
    & \overset{V_2}{\blacksquare} \ar{r}{\rho_2} & \overset{\Fr_2(E)}{\blackbullet}.
  \end{tikzcd} \]
  Given framing functors $\Fr_1, \Fr_2 \in \Frs$ such that
  \[ \fM_\alpha^\sst(\tau) \subset \fM_\alpha^{\Fr_1,\pl} \cap \fM_\alpha^{\Fr_2,\pl}, \]
  consider the auxiliary exact category $\tilde{\cat{A}}^{(Q \wedge
    Q)(\vec\Fr)}$ (Definition~\ref{def:auxiliary-stack})
  parameterizing triples $(E, \vec V, \vec \rho)$ as labeled above.
  For clarity, we will write $Q(\Fr_1) \wedge Q(\Fr_2)$ instead of $(Q
  \wedge Q)(\vec\Fr)$, and $\cFr_i(\cE_j)$ for the universal bundle of
  $\Fr_i(E_j)$. Note the natural isomorphisms
  \begin{align*}
    \iota^{Q \wedge Q}_1\colon \tilde\fM_{\alpha,d}^{Q(\Fr_1)} \cap \pi_{\fM_\alpha^{\Fr_1}}^{-1}(\fM_\alpha^{\vec\Fr}) &\xrightarrow{\sim} \tilde\fM_{\alpha,(d,0,0)}^{Q(\Fr_1)\wedge Q(\Fr_2)}, \\
    \iota^{Q \wedge Q}_2\colon \tilde\fM_{\alpha,d}^{Q(\Fr_2)} \cap \pi_{\fM_\alpha^{\Fr_2}}^{-1}(\fM_\alpha^{\vec\Fr}) &\xrightarrow{\sim} \tilde\fM_{\alpha,(0,d,0)}^{Q(\Fr_1)\wedge Q(\Fr_2)},
  \end{align*}
  and their rigidified versions $\iota_i^{Q\wedge Q,\pl}$. These are
  morphisms of graded monoidal $\sT$-stacks; we will repeatedly use
  that $\iota_1^{Q \wedge Q} \circ I_{\alpha,d} = I_{\alpha,(d,0,0)}
  \circ \iota_1^{Q \wedge Q,\pl}$ and similarly for $\iota_2^{Q \wedge
    Q}$. Define 
  \begin{alignat*}{2}
    \partial_1 &\coloneqq (\iota^{Q \wedge Q,\pl}_1)_* \partial, \qquad & \partial_2 &\coloneqq (\iota^{Q \wedge Q,\pl}_2)_* \partial, \\
    \tilde\sZ_{\alpha,(1,0,0)}^{\Fr_1}(\tau^Q) &\coloneqq (\iota^{Q \wedge Q,\pl}_1)_* \tilde\sZ_{\alpha,1}^{\Fr_1}(\tau^Q), \qquad & \tilde\sZ_{\alpha,(0,1,0)}^{\Fr_2}(\tau^Q) &\coloneqq (\iota^{Q \wedge Q,\pl}_2)_* \tilde\sZ_{\alpha,1}^{\Fr_2}(\tau^Q).
  \end{alignat*}
  The latter are well-defined: by Lemma~\ref{lem:pairs-stack-sst=st},
  $\tilde\sZ_{\alpha,1}^{\Fr_i}(\tau^Q)$ is supported on the
  domain of $\iota^{Q \wedge Q,\pl}_i$.
\end{definition}

\subsubsection{}

\begin{theorem} \label{thm:pairs-master-relation}
  Using the Lie bracket in
  $K_\circ^{\tilde\sT}(\tilde\fM^{Q(\Fr_1)\wedge Q(\Fr_2)})^\pl_{\loc}$
  (Theorem~\ref{thm:auxiliary-stack-vertex-algebra}),
  \begin{equation} \label{eq:pairs-master-relation}
    0 = \left[I_*\tilde\sZ_{\alpha,(1,0,0)}^{\Fr_1}(\tau^Q), I_*\partial_2\right] + \left[I_*\partial_1, I_*\tilde\sZ_{\alpha,(0,1,0)}^{\Fr_2}(\tau^Q)\right] + \sum_{\substack{\alpha = \alpha_1+\alpha_2\\\forall i: \tau(\alpha_i) = \tau(\alpha)\\ \;\;\fM_{\alpha_i}^{\sst}(\tau) \neq \emptyset}} \left[I_*\tilde\sZ_{\alpha_1,(1,0,0)}^{\Fr_1}(\tau^Q), I_*\tilde\sZ_{\alpha_2,(0,1,0)}^{\Fr_2}(\tau^Q)\right]
  \end{equation}
  is a relation in degree $(\alpha, (1,1,0))$.
\end{theorem}

The pushforward of this relation along $\pi_{\fM^{\vec\Fr,\pl}_\alpha}$ is related to the master space relation used in \cite[Theorem
  4.2.7]{liu_eq_k_theoretic_va_and_wc}, or, in homology, \cite[Corollary 9.10]{Joyce2021}. We will give the same
geometric, wall-crossing proof. This will occupy the remainder of this
subsection. From now on, the class $\alpha$ is fixed once and for all.

\subsubsection{}

\begin{proof}[Proof of Theorem~\ref{thm:pairs-master-relation}.] 
On objects of the category $\tilde{\cat{A}}^{Q(\Fr_1)\wedge Q(\Fr_2)}$
with $0 \le \beta \le \alpha$, define the weak stability condition
(cf. \eqref{eq:pair-stability})
\begin{equation} \label{eq:master-pair-stability}
  \tau^{Q \wedge Q}(\beta, (d_1, d_2, d_3)) \coloneqq \begin{cases} \left(\tau(\beta), \frac{\epsilon d_1 + \epsilon d_2 + d_3}{r(\beta)}\right) & \beta \neq 0, \, \tau(\beta) = \tau(\alpha-\beta) \\ \left(\tau(\beta), \infty\right) & \beta \neq 0, \, \tau(\beta) > \tau(\alpha-\beta) \\ \left(\tau(\beta), -\infty\right) & \beta \neq 0, \, \tau(\beta) < \tau(\alpha-\beta) \\ \left(\infty, \frac{\epsilon d_1 + \epsilon d_2 + d_3}{d_1 + d_2 + d_3}\right) & \beta = 0 \end{cases}
\end{equation}
for a choice of $0 < \epsilon < 1/(2r(\alpha))$. The upper bound
ensures that there are no strictly $\tau^{Q \wedge Q}$-semistable
objects of class $(\alpha, (1, 1, 1))$. Comparing
\eqref{eq:master-pair-stability} with \eqref{eq:pair-stability}, it is
clear that $\iota_i^{Q \wedge Q}$ restrict to isomorphisms of
(semi)stable loci, for $i = 1, 2$.

Let $\sS \coloneqq \bC^\times$, with coordinate denoted $z$, act on
$\tilde\fM^{Q(\Fr_1)\wedge Q(\Fr_2)}$ by scaling $\rho_4$ with weight
$z$. This clearly commutes with the $\sT$-action and scaling
automorphisms and descends to an action on the rigidified stack, and
the forgetful maps are $\sS$-equivariant for the trivial
$\sS$-action on their targets.

\subsubsection{}

\begin{proposition}[{\cite[Propositions 9.5, 9.6]{Joyce2021}}] \label{prop:pairs-master-space-fixed-loci}
  The \emph{master space}
  \[ \bM_\alpha \coloneqq \tilde\fM^{Q(\Fr_1)\wedge Q(\Fr_2),\sst}_{\alpha,(1,1,1)}(\tau^{Q \wedge Q}) \]
  is $\sS$-invariant. Its $\sS$-fixed locus is a
  disjoint union of the following pieces.
  \begin{enumerate}[label = (\roman*)]
  \item Let $\iota_{\rho_4=0}\colon Z_{\rho_4=0} \coloneqq
    \{\rho_4=0\} \hookrightarrow \bM_\alpha$, with virtual normal
    bundle $z \cV_3^\vee \otimes \cV_2 - \kappa^{-1}z^{-1} \cV_3
    \otimes \cV_2^\vee$. By $\tau^{Q \wedge Q}$-stability, $\rho_1,
    \rho_2, \rho_3 \neq 0$. The forgetful map
    \[ \pi_{\tilde\fM^{Q(\vec\Fr),\pl}_{\alpha,(1,0,0)}}\colon Z_{\rho_4=0} \to \tilde\fM_{\alpha,(1,0,0)}^{Q(\Fr_1)\wedge Q(\Fr_2),\sst}(\tau^{Q \wedge Q}), \]
    which remembers only $\rho_1\colon V_1 \to \Fr_1(E)$, is a
    $\bP^{\fr_2(\alpha)-1}$-bundle.

  \item Let $\iota_{\rho_3=0}\colon Z_{\rho_3=0} \coloneqq
    \{\rho_3=0\} \hookrightarrow \bM_\alpha$, with virtual normal
    bundle $z^{-1} \cV_3^\vee \otimes \cV_1 - \kappa^{-1} z \cV_3
    \otimes \cV_1^\vee$. By $\tau^{Q \wedge Q}$-stability, $\rho_1,
    \rho_2, \rho_4 \neq 0$. The forgetful map
    \[ \pi_{\tilde\fM^{Q(\vec\Fr),\pl}_{\alpha,(0,1,0)}}\colon Z_{\rho_3=0} \to \tilde\fM_{\alpha,(0,1,0)}^{Q(\Fr_1)\wedge Q(\Fr_2),\sst}(\tau^{Q \wedge Q}), \]
    which remembers only $\rho_2\colon V_2 \to \Fr_2(E)$, is a
    $\bP^{\fr_1(\alpha)-1}$-bundle.

  \item \label{sst:it:complicated-locus} For each splitting $\alpha =
    \alpha_1 + \alpha_2$ with $\tau(\alpha_1) = \tau(\alpha_2)$ and
    $\fM_{\alpha_i}^{\sst}(\tau) \neq 0$ for $i = 1, 2$, let
    \[ \iota_{\alpha_1,\alpha_2}\colon Z_{\alpha_1,\alpha_2} \coloneqq \{E = E_1 \oplus E_2, \; \rho_i\colon V_i \to \Fr_i(E_i) \subset \Fr_i(E) \text{ for } i = 1, 2\} \hookrightarrow \bM_\alpha, \]
    where $\sS$ scales $V_3$, $V_1$, and $E_1$ with weight $z$.
    The virtual normal bundle is
    \begin{equation} \label{eq:pairs-master-space-interaction-Nvir}
      \begin{aligned}
        \cN_{\iota_{\alpha_1,\alpha_2}}^{\vir} = &\left(z^{-1} \cV_1^\vee \otimes \cFr_1(\cE_2) + z \cV_2^\vee \otimes \cFr_2(\cE_1) \right) - \kappa^{-1} \otimes (\cdots)^\vee \\
        &-(\pi_{\fM_{\alpha_1}^{\Fr_1}} \times \pi_{\fM_{\alpha_2}^{\Fr_2}})^*\left(z^{-1} \scE_{\alpha_1,\alpha_2} + z (12)^* \scE_{\alpha_2,\alpha_1}\right)
      \end{aligned}
    \end{equation}
    where $(\cdots)^\vee$ denotes the dual of the preceding terms, and
    $\scE_{\alpha_i,\alpha_j}$ are the bilinear elements of $\fM$. By
    $\tau^{Q \wedge Q}$-stability, $\rho_1, \ldots, \rho_4 \neq 0$.
    The forgetful map
    \begin{equation} \label{eq:pairs-master-space-complicated-locus}
      Z_{\alpha_1,\alpha_2} \xrightarrow{\sim} \tilde\fM_{\alpha_1,(1,0,0)}^{Q(\Fr_1)\wedge Q(\Fr_2),\sst}(\tau^{Q \wedge Q}) \times \tilde\fM_{\alpha_2,(0,1,0)}^{Q(\Fr_1)\wedge Q(\Fr_2),\sst}(\tau^{Q \wedge Q}),
    \end{equation}
    whose $i$-th factor is given by remembering only $\rho_i\colon V_i \to \Fr_i(E_i)$, is an isomorphism.
  \end{enumerate}
\end{proposition}

Our convention is to write $\tilde\sT \times \sS$-equivariant sheaves
on $\sS$-fixed loci as a product of an explicit $\sS$-weight and a
$\tilde\sT$-equivariant sheaf. Obvious pullbacks are omitted.

\subsubsection{}
\label{sec:pairs-master-space-localization}

Assumption~\ref{assump:semistable-invariants}\ref{assump:it:properness}
implies $\bM_\alpha^{\sT \times \sS}$ is a proper algebraic space with
the resolution property, so by symmetrized pullback
(Theorem~\ref{thm:APOTs}) along $\pi_{\fM_\alpha^\pl}$, the universal
enumerative invariant
\[ \chi\left(\bM_\alpha, \hat\cO^{\vir}_{\bM_\alpha} \otimes -\right) \in K_\circ^{\tilde\sT \times \sS}(\tilde\fM^{Q(\Fr_1)\wedge Q(\Fr_2),\pl}_{\alpha,(1,1,1)})_\loc \]
is well-defined. Define the forgetful map
\[ \pi_{\tilde\fM^{Q(\vec\Fr),\pl}_{\alpha,(1,1,0)}}\colon \tilde\fM_{\alpha,(1,1,1)}^{Q(\Fr_1)\wedge Q(\Fr_2),\pl} \to \tilde\fM_{\alpha,(1,1,0)}^{Q(\Fr_1)\wedge Q(\Fr_2),\pl}. \]
Using the description of $\sS$-fixed loci in
Proposition~\ref{prop:pairs-master-space-fixed-loci}, $(\tilde\sT
\times \sS)$-equivariant localization says
\begin{align}
  &\chi\left(\bM_\alpha, \hat\cO_{\bM_\alpha}^{\vir} \otimes (\pi_{\tilde\fM^{Q(\vec\Fr),\pl}_{\alpha,(1,1,0)}})^*(-)\right) \label{eq:pairs-master-lhs} \\
  &= \chi\left(Z_{\rho_4=0}, \hat\cO_{Z_{\rho_4=0}}^{\vir} \otimes \frac{\hat\se(z^{-1}\kappa^{-1} \cV_3 \otimes \cV_2^\vee)}{\hat\se(z\cV_3^\vee \otimes \cV_2)} \otimes (\pi_{\tilde\fM^{Q(\vec\Fr),\pl}_{\alpha,(1,1,0)}} \circ \iota_{\rho_4=0})^*(-)\right) \label{eq:pairs-master-term-1} \\
  &\quad +\chi\left(Z_{\rho_3=0}, \hat\cO_{Z_{\rho_3=0}}^{\vir} \otimes \frac{\hat\se(z\kappa^{-1} \cV_3 \otimes \cV_1^\vee)}{\hat\se(z^{-1}\cV_3^\vee \otimes \cV_1)} \otimes (\pi_{\tilde\fM^{Q(\vec\Fr),\pl}_{\alpha,(1,1,0)}} \circ \iota_{\rho_3=0})^*(-)\right) \label{eq:pairs-master-term-2} \\
  &\quad +\sum_{\substack{\alpha=\alpha_1+\alpha_2\\\forall i: \, \tau(\alpha_i)=\tau(\alpha)\\\;\;\fM_{\alpha_i}^{\sst}(\tau) \neq \emptyset}} \chi\left(Z_{\alpha_1,\alpha_2}, \hat\cO_{Z_{\alpha_1,\alpha_2}}^\vir \otimes \frac{1}{\hat\se(\cN_{\iota_{\alpha_1,\alpha_2}}^{\vir})} \otimes (\pi_{\tilde\fM^{Q(\vec\Fr),\pl}_{\alpha,(1,1,0)}} \circ \iota_{\alpha_1,\alpha_2})^*(-) \right) \label{eq:pairs-master-term-3}.
\end{align}
Each term of the form $\chi(X, \hat\cO_X^\vir \otimes \cF)$ on the
right hand side is {\it notational shorthand} for
\[ \chi\bigg(X^{\sT}, \frac{\hat\cO_{X^{\sT}}^\vir \otimes \cF|_{X^{\sT}}}{\hat\se(\cN_{X^{\sT}/X}^\vir)^\vee}\bigg). \]
This is important because $(\tilde\sT \times \sS)$-equivariant classes
like $\hat\se(z\cV_3^\vee \otimes \cV_2)$ on $\sS$-fixed loci must be
further restricted to $(\sT \times \sS)$-fixed loci before their
inverses exist.

\subsubsection{}
\label{sec:pairs-master-pole-cancellation}

We will obtain the desired relation \eqref{eq:pairs-master-relation}
by applying the K-theoretic residue map $\rho_K$
(Definition~\ref{def:k-theoretic-residue-map}) in the variable $z$,
dividing both sides by $\kappa^{-1/2} - \kappa^{1/2}$, and applying
the de-rigidification $(I_{\alpha,(1,1,0)})_*$. To be clear, we will
only evaluate at elements in
\[ K_{\tilde\sT}^\circ(\tilde\fM_{\alpha,(1,1,0)}^{Q(\Fr_1)\wedge Q(\Fr_2),\pl})_\loc \subset K_{\tilde\sT \times \sS}^\circ(\tilde\fM_{\alpha,(1,1,0)}^{Q(\Fr_1)\wedge Q(\Fr_2),\pl})_{\loc}, \]
noting that the $\sS$-action is trivial so indeed this is an inclusion
of $\bk_{\tilde\sT,\loc}$-modules. By
Assumption~\ref{assump:semistable-invariants} and
Lemma~\ref{lem:master-space-no-poles}, the left hand side
\eqref{eq:pairs-master-lhs} takes values in
\[ \bk_{\tilde\sT,\loc} \otimes_\bZ \bk_\sS \subset \bk_{\tilde\sT \times \sS,\loc}. \]
Therefore it has no poles in $z$ except at $z=0$ and $z=\infty$. By
the description \eqref{eq:k-theoretic-residue-map-finite-poles} of
$\rho_K$, it follows that
\[ \rho_K \text{ of }\eqref{eq:pairs-master-lhs} = 0. \]

\subsubsection{}
\label{sec:pairs-master-term-1}

Next, consider \eqref{eq:pairs-master-term-1}. Applying $\rho_K$, the
only contribution is
\[ \rho_K \frac{\hat\se(z^{-1}\kappa^{-1} \cV_3 \otimes \cV_2^\vee)}{\hat\se(z\cV_3^\vee \otimes \cV_2)} = \kappa^{-\frac{1}{2}} - \kappa^{\frac{1}{2}} \]
because all other terms are sheaves on a $\sS$-fixed component.
Hence
\begin{align*}
  \rho_K \text{ of } \eqref{eq:pairs-master-term-1}
  &= (\kappa^{-\frac{1}{2}} - \kappa^{\frac{1}{2}}) \cdot \chi\left(Z_{\rho_4=0}, \hat\cO_{Z_{\rho_4=0}}^{\vir} \otimes (\pi_{\tilde\fM^{Q(\vec\Fr),\pl}_{\alpha,(1,1,0)}} \circ \iota_{\rho_4=0})^*(-)\right) \\
  &= (\kappa^{-\frac{1}{2}} - \kappa^{\frac{1}{2}}) \cdot (\pi_{\tilde\fM^{Q(\vec\Fr),\pl}_{\alpha,(1,1,0)}} \circ \iota_{\rho_4=0})_* \sZ_{Z_{\rho_4=0}} 
\end{align*}
using the notation of Definition~\ref{bg:def:univ-enum-inv} for
universal invariants. We will apply the homology projective bundle
formula (Lemma~\ref{lem:homology-projective-bundle-formula}), using
the notation there, as follows. By
Proposition~\ref{prop:pairs-master-space-fixed-loci},
\begin{equation} \label{eq:pairs-master-space-fixed-loci-1-bundle}
  \pi_{\tilde\fM_{\alpha,(1,0,0)}^{Q(\vec\Fr),\pl}}\colon Z_{\rho_4=0} \to X \coloneqq \tilde\fM^{Q(\Fr_1)\wedge Q(\Fr_2),\sst}_{\alpha,(1,0,0)}(\tau^{Q \wedge Q})
\end{equation}
identifies $Z_{\rho_4=0} = P_{\fX}(\cFr_2(\cE))$, where $\cFr_2(\cE)$
denotes the universal bundle of $\Fr_2(E)$ on the non-rigidified stack
$\fX$ of $X$. The APOT on $Z_{\rho_4=0}^{\sT}$ arises by restriction
of the APOT on $\bM_\alpha^{\sT \times \sS}$, which, a priori, might
not be the APOT obtained by symmetrized pullback along the $\sT$-fixed
part of \eqref{eq:pairs-master-space-fixed-loci-1-bundle}, but
Theorem~\ref{thm:APOT-comparison} assures us that the symmetrized
virtual structure sheaves of these two APOTs agree. Hence
\[ \sZ_{Z_{\rho_4=0}} = (\hat\pi_{\tilde\fM_{\alpha,(1,0,0)}^{Q(\vec\Fr),\pl}})^* \tilde\sZ_{\alpha,(1,0,0)}^{\Fr_1}(\tau^Q). \]
Next, when $\rank \cV_2 = 1$, it is identified with the weight-$1$
line bundle on $[\pt/\bC^\times]$, and therefore
$\tot_{[\pt/\bC^\times] \times \fX}(\cV_2^\vee \boxtimes \cFr_2(\cE))$
is identified with the product of forgetful maps
\[ \pi_{\tilde\fM_{\alpha,(1,0,0)}^{Q(\vec\Fr)}} \times \pi_{\tilde\fM_{0,(0,1,0)}^{Q(\vec\Fr)}} \colon \tilde\fM_{\alpha,(1,1,0)}^{Q(\Fr_1)\wedge Q(\Fr_2)} \to \tilde\fM_{\alpha,(1,0,0)}^{Q(\Fr_1)\wedge Q(\Fr_2)} \times \tilde\fM_{0,(0,1,0)}^{Q(\Fr_1)\wedge Q(\Fr_2)} \]
restricted to the pre-image of $\fX$. Its zero section is the direct
sum map $\Phi \coloneqq \Phi_{(\alpha,(1,0,0)),(0,(0,1,0))}$, and the
zero section of its rigidification $\tot_{\fX}(\cFr_2(\cE))$ is
therefore the composition $\Pi_{\alpha,(1,1,0)}^\pl \circ \Phi$. The
open immersion $j\colon Z_{\rho_4=0} \hookrightarrow
\tot_{\fX}(\cFr_2(\cE))$ is identified with
$\pi_{\tilde\fM^{Q(\vec\Fr),\pl}_{\alpha,(1,1,0)}} \circ
\iota_{\rho_4=0}$. Hence
\begin{align}
  &(\kappa^{-\frac{1}{2}} - \kappa^{\frac{1}{2}}) \cdot (\pi_{\tilde\fM^{Q(\vec\Fr),\pl}_{\alpha,(1,1,0)}} \circ \iota_{\rho_4=0})_* (\hat\pi_{\tilde\fM^{Q(\vec\Fr),\pl}_{\alpha,(1,0,0)}})^* \tilde\sZ_{\alpha,(1,0,0)}^{\Fr_1}(\tau^Q) \nonumber \\
  &= \rho_K (\Pi_{\alpha,(1,1,0)}^\pl \circ \Phi)_* (D(z) \times \id) \left((I_*\tilde\sZ_{\alpha,(1,0,0)}^{\Fr_1}(\tau^Q) \boxtimes I_*\partial_2) \cap \frac{\hat\se_{z^{-1}}(\kappa^{-1} \cFr_2(\cE)^\vee \boxtimes \cV_2)}{\hat\se_z(\cFr_2(\cE) \boxtimes \cV_2^\vee)}\right), \label{eq:pairs-master-term-1-explicit}
\end{align}
where we used that $(I_*\partial_2)(\cV_2^k) = 1$ for any $k \in \bZ$
in order to insert $\boxtimes \cV_2^\vee$ and $\boxtimes \cV_2$ into
the fraction on the right hand side. We do this because, comparing
with the bilinear elements $\tilde\scE^{Q(\Fr_1)\wedge
  Q(\Fr_2)}_{(\alpha,(1,0,0)),(0,(0,1,0))}$
(Theorem~\ref{thm:auxiliary-stack-vertex-algebra}),
\[ \frac{\hat\se_{z^{-1}}(\kappa^{-1} \cFr_2(\cE)^\vee \boxtimes \cV_2)}{\hat\se_z(\cFr_2(\cE) \boxtimes \cV_2^\vee)} = \hat\Theta(z) \coloneqq \hat\Theta_{(\alpha,(1,0,0)),(0,(0,1,0))}(z). \]
Plugging this into \eqref{eq:pairs-master-term-1-explicit}, it becomes
\begin{align*}
  &\rho_K (\Pi_{\alpha,(1,1,0)}^\pl)_* \Phi_* (D(z) \times \id)\left((I_*\tilde\sZ_{\alpha,(1,0,0)}^{\Fr_1}(\tau^Q) \boxtimes I_*\partial_2) \cap \hat\Theta(z)\right) \\
  &= (\kappa^{-\frac{1}{2}} - \kappa^{\frac{1}{2}}) \cdot (\Pi_{\alpha,(1,1,0)}^\pl)_* [I_*\tilde\sZ_{\alpha,(1,0,0)}^{\Fr_1}(\tau^Q), I_*\partial_2]
\end{align*}
by the definition (Theorem~\ref{thm:auxiliary-stack-vertex-algebra})
of the Lie bracket. This produces the first term in the desired
relation \eqref{eq:pairs-master-relation}.

\subsubsection{}

By the exact same reasoning,
\begin{align*}
  \rho_K \text{ of }\eqref{eq:pairs-master-term-2}
  &= -(\kappa^{-\frac{1}{2}} - \kappa^{\frac{1}{2}}) \cdot \chi\left(Z_{\rho_3=0}, \hat\cO_{Z_{\rho_3=0}}^{\vir} \otimes (\pi_{\tilde\fM^{Q(\vec\Fr),\pl}_{\alpha,(1,1,0)}} \circ \iota_{\rho_3=0})^*(-)\right) \\
  &= -(\kappa^{-\frac{1}{2}} - \kappa^{\frac{1}{2}}) \cdot (\Pi_{\alpha,(1,1,0)}^\pl)_* \left[I_*\tilde\sZ_{\alpha,(0,1,0)}^{\Fr_2}(\tau^Q), I_*\partial_1\right].
\end{align*}
This produces the second term in the desired relation
\eqref{eq:pairs-master-relation}.

\subsubsection{}
\label{sec:pairs-master-term-3}

Finally, consider \eqref{eq:pairs-master-term-3}. A priori, the source
and target of the $\sT$-fixed part of the isomorphism
$\pi_{\alpha_1,\alpha_2}$ may have different APOTs, coming from
restriction from $\bM_\alpha$ and symmetrized pullback from
$\fM_{\alpha_1}^\pl \times \fM_{\alpha_2}^\pl$ respectively, but
Theorem~\ref{thm:APOT-comparison} assures us that the symmetrized
virtual structure sheaves of these two APOTs agree. Now use the
$(\tilde\sT \times \sS)$-equivariantly commutative diagram
\[ \begin{tikzcd}[column sep=huge]
  \tilde\fM_{\alpha_1,1}^{\Fr_1,\sst}(\tau^Q) \times \tilde\fM_{\alpha_2,1}^{\Fr_2,\sst}(\tau^Q) \ar[hookrightarrow]{r}{\iota_{\alpha_1,\alpha_2}} \ar{d}[swap]{(I_{\alpha_1,(1,0,0)} \times I_{\alpha_2,(0,1,0)}) \circ (\iota_1^{Q \wedge Q,\pl} \times \iota_2^{Q \wedge Q,\pl})} & \bM_\alpha \ar{d}{I_{\alpha,(1,1,0)} \circ \pi_{\tilde\fM^{Q(\vec\Fr),\pl}_{\alpha,(1,1,0)}}} \\
  \tilde\fM_{\alpha_1,(1,0,0)}^{Q(\Fr_1)\wedge Q(\Fr_2)} \times \tilde\fM_{\alpha_2,(0,1,0)}^{Q(\Fr_1)\wedge Q(\Fr_2)} \ar{r}{\Phi} & \tilde\fM_{\alpha,(1,1,0)}^{Q(\Fr_1)\wedge Q(\Fr_2)}
\end{tikzcd} \]
where $\Phi \coloneqq \Phi_{(\alpha_1,(1,0,0)),(\alpha_2,(0,1,0))}$.
Then the terms in \eqref{eq:pairs-master-term-3} become
\begin{equation} \label{eq:pairs-master-term-3-explicit}
  \begin{aligned}
    \rho_K (\Pi_{\alpha,(1,1,0)}^\pl)_* \chi\bigg(&\tilde\fM_{\alpha_1,1}^{Q(\Fr_1),\sst}(\tau^Q) \times \tilde\fM_{\alpha_2,1}^{Q(\Fr_2),\sst}(\tau^Q), \\
    &\frac{\hat\cO^\vir \boxtimes \hat\cO^\vir}{\hat\se(\cN_{\iota_{\alpha_1,\alpha_2}}^\vir)} \otimes (I \times I)^*(\iota_1^{Q \wedge Q} \times \iota_2^{Q \wedge Q})^* z^{\deg_1} \Phi^*(-)\bigg),
  \end{aligned}
\end{equation}
where now all objects are only $\tilde\sT$-equivariant and we have
written $\sS$-weights explicitly; this is why $z^{\deg_1}$ appears in
front of $\Phi^*$. Furthermore, comparing the bilinear elements
$\tilde\scE^{Q(\Fr_1)\wedge
  Q(\Fr_2)}_{(\alpha_1,(1,0,0)),(\alpha_2,(0,1,0))}$
(Theorem~\ref{thm:auxiliary-stack-vertex-algebra}) with the formula
\eqref{eq:pairs-master-space-interaction-Nvir} for
$\cN_{\iota_{\alpha_1,\alpha_2}}^\vir$,
\[ \frac{1}{\hat\se(\cN_{\iota_{\alpha_1,\alpha_2}}^\vir)} = (I \times I)^*(\iota_1^{Q \wedge Q} \times \iota_2^{Q \wedge Q})^* \hat\Theta_{(\alpha_1,(1,0,0)),(\alpha_2,(0,1,0))}(z). \]
Plugging this into \eqref{eq:pairs-master-term-3-explicit}, it becomes
\begin{align*}
  &\rho_K (\Pi_{\alpha,(1,1,0)}^\pl)_* \left((\iota_1^{Q \wedge Q})_* I_*\tilde\sZ_{\alpha_1,1}^{\Fr_1}(\tau^Q) \boxtimes (\iota_2^{Q \wedge Q})_* I_*\tilde\sZ_{\alpha_2,1}^{\Fr_2}(\tau^Q)\right)\left(\hat\Theta(z) \otimes z^{\deg_1} \Phi^*(-)\right) \\
  &= (\kappa^{-\frac{1}{2}} - \kappa^{\frac{1}{2}}) \cdot \left[I_*\tilde\sZ_{\alpha_1,(1,0,0)}^{\Fr_1}(\tau^Q), I_*\tilde\sZ_{\alpha_2,(0,1,0)}^{\Fr_2}(\tau^Q)\right]
\end{align*}
by the definition (Theorem~\ref{thm:auxiliary-stack-vertex-algebra})
of the Lie bracket. This produces the last term in the desired
relation \eqref{eq:pairs-master-relation}, and concludes the proof of
Theorem~\ref{thm:pairs-master-relation}.
\end{proof}

\subsubsection{}

\begin{remark} \label{rem:joyce-lie-bracket-vs-ours-1}
  In \eqref{eq:pairs-master-term-3-explicit}, the virtual normal
  bundle $\cN^{\vir}_{\iota_{\alpha_1,\alpha_2}}$ depends
  non-trivially on framing data and is not a quantity which lives on
  the original moduli stack $\fM$. This means
  \eqref{eq:pairs-master-term-3-explicit} has no hope of being written
  in terms of the Lie bracket on $\fM$.

  Experts may object that, at this step in Joyce's machinery, his
  analogue of \eqref{eq:pairs-master-term-3-explicit} {\it is} written
  in terms of the Lie bracket on $\fM$ \cite[Prop. 9.9]{Joyce2021}.
  The crucial difference is that, in Joyce's setting, the piece of
  $\cN^{\vir}_{\iota_{\alpha_1,\alpha_2}}$ which involves framing data
  is a vector bundle, denoted $\bF$ (cf.
  \eqref{eq:framed-stack-forgetful-map-cotangent}), while in our
  setting this same piece is the symmetrized version $\bF -
  \kappa^{-1} \bF^\vee$. Hence, Joyce may cancel the contribution
  $1/\hat\se(\bF)$ in \eqref{eq:pairs-master-term-3-explicit} by
  inserting the class $\hat\se(\bF)$ into the integrand in
  \eqref{eq:pairs-master-lhs}, but we cannot similarly insert
  $\hat\se(\bF)/\hat\se(\kappa^{-1}\bF)$ because this is a localized
  class --- in particular, it has poles in $z$ --- which invalidates
  the pole-cancellation argument in
  \S\ref{sec:pairs-master-pole-cancellation}.
\end{remark}

\subsection{Proof of well-definedness and framing independence}
\label{sec:vss-framing-independence}

\subsubsection{}

\begin{lemma} \label{lem:wall-crossing-total-term}
  Let $A$ be a monoid consisting of classes $\alpha$ and
  \[ L = \bigoplus_{(\alpha,k_1,k_2) \in A \times \bZ_{\ge 0}^2} L_{\alpha,k_1,k_2}, \qquad [L_{\alpha,k_1,k_2}, L_{\beta,\ell_1,\ell_2}] \subset L_{\alpha+\beta, k_1+\ell_1, k_2+\ell_2} \]
  be an $(A \times \bZ_{\ge 0}^2)$-graded Lie algebra. Let:
  \begin{itemize}
  \item $\vec z \coloneqq \{z_\alpha \in L_{\alpha,0,0}\}_{\alpha \in A}$
    be an arbitrary collection of elements;
  \item $\delta_1 \in L_{0,1,0}$ and $\delta_2 \in L_{0,0,1}$ be
    elements such that $[\delta_1, \delta_2] = 0$.
  \end{itemize}
  Define $C_{\alpha_1,\ldots,\alpha_n}(\vec z, \delta_1,
  \delta_2)$ to be
  \begin{equation} \label{eq:wall-crossing-total-term}
    \sum_{m=0}^n \bigg[\frac{1}{m!} \Big[z_{\alpha_m}, \big[\cdots[z_{\alpha_2}, [z_{\alpha_1}, \delta_1]]\cdots\big]\Big], \frac{1}{(n-m)!} \Big[z_{\alpha_n}, \big[\cdots[z_{\alpha_{m+2}}, [z_{\alpha_{m+1}}, \delta_2]]\cdots\big]\Big]\bigg].
  \end{equation}
  Then
  \[ \sum_{\substack{n>1 \\ \alpha = \alpha_1+\cdots+\alpha_n\\ \forall i: \,\tau(\alpha_i) = \tau(\alpha)\\ \;\;\fM_{\alpha_i}^{\sst}(\tau) \neq \emptyset}} C_{\alpha_1,\ldots,\alpha_n}(\vec z, \delta_1, \delta_2) = 0. \]
\end{lemma}

\begin{proof}
  In the completion of the universal enveloping algebra of $L$ with
  respect to the grading, consider the operator
  \[ \ad_z(-) \coloneqq [z, -], \qquad z \coloneqq \sum_{\substack{\beta:\,\tau(\beta)=\tau(\alpha),\\\;\;\;\fM_\beta^\sst(\tau) \neq \emptyset}} z_\beta. \]
  Then
  \[ \sum_{\substack{n\ge 0\\\alpha_1+\cdots+\alpha_n=\alpha\\\forall i: \tau(\alpha_i)=\tau(\alpha)\\\;\;\fM_{\alpha_i}^\sst(\tau) \neq \emptyset}} C_{\alpha_1,\ldots,\alpha_n}(\vec z, \delta_1, \delta_2) = (\alpha, 1,1)\text{-weight piece of } \left[e^{\ad_z} \delta_1, e^{\ad_z} \delta_2\right]. \]
  A standard combinatorial result in Lie theory says $e^{\ad_u} v =
  e^u v e^{-u}$ for any $u, v$, so
  \[ \left[e^{\ad_z} \delta_1, e^{\ad_z} \delta_2\right] = [e^z \delta_1 e^{-z}, e^z \delta_2 e^{-z}] = e^z [\delta_1, \delta_2] e^{-z} = 0. \]
  The last equality is by hypothesis. We are done since $C_\emptyset =
  [\delta_1, \delta_2] = 0$ and $C_{\alpha_1} = 0$ for all
  $\alpha_1 \in A$ by the Jacobi identity.
\end{proof}

\subsubsection{}

\begin{proposition} \label{prop:vss-framing-independence}
  Let $\Fr_1, \Fr_2 \in \Frs$ be two framing functors with
  $\fM_\alpha^\sst(\tau) \subset \fM_\alpha^{\vec\Fr,\pl}$. Then the
  elements $\sz_\alpha^{\Fr_i}(\tau)$, defined explicitly by
  \eqref{eq:sstable-explicit}, are independent of $i = 1,2$:
  \[ \sz_\alpha^{\Fr_1}(\tau) = \sz_\alpha^{\Fr_2}(\tau). \]
\end{proposition}

This proposition crucially requires the relation
\eqref{eq:pairs-master-relation} proved by
Theorem~\ref{thm:pairs-master-relation}.

\begin{proof}
  Let $l \ge 0$ be given, and suppose for induction that the
  proposition holds for all $r(\alpha) < l$. The base case $l = 0$ is
  vacuous since $r(\alpha) > 0$ for all $\alpha \neq 0$. Suppose now
  that $r(\alpha) = l$. Applying the induction hypothesis to
  \eqref{eq:sstable-def} yields
  \[ I_*\tilde{\sZ}_{\alpha,1}^{\Fr_i}(\tau^Q) = \left[\iota^Q_*\sz^{\Fr_i}_\alpha(\tau), I_*\partial\right] + \sum_{\substack{n>1 \\ \alpha = \alpha_1+\cdots+\alpha_n\\ \forall i: \,\tau(\alpha_i) = \tau(\alpha)\\ \;\;\fM_{\alpha_i}^{\sst}(\tau) \neq \emptyset}} \frac{1}{n!} \left[\iota^Q_*\sz^{\Fr_1}_{\alpha_n}(\tau), \left[\cdots,\left[\iota^Q_*\sz^{\Fr_1}_{\alpha_2}(\tau), \left[\iota^Q_*\sz^{\Fr_1}_{\alpha_1}(\tau), I_*\partial\right]\right]\cdots\right]\right] \]
  for $i = 1, 2$, since $r(\alpha_i) < r(\alpha)$ for all $\alpha_i$
  in the sum by the additivity of $r$. Apply $(\iota_i^{Q \wedge
    Q})_*$ to this and plug it into the relation
  \eqref{eq:pairs-master-relation} to get
  \begin{equation} \label{eq:sst-invariants-comparison}
    0 = (\pi_{\fM_\alpha^{\vec\Fr,\pl}})_*\bigg(\left[\left[z_\alpha^{\Fr_1}, I_*\partial_1\right], I_*\partial_2\right] + \left[I_*\partial_1, \left[z_\alpha^{\Fr_2}, I_*\partial_2\right]\right] + \sum_{\substack{n>1 \\ \alpha = \alpha_1+\cdots+\alpha_n\\ \forall i: \,\tau(\alpha_i) = \tau(\alpha)\\ \;\;\fM_{\alpha_i}^{\sst}(\tau) \neq \emptyset}} C_{\alpha_1,\ldots,\alpha_n}(\vec z^{\Fr_1}, I_*\partial_1, I_*\partial_2)\bigg),
  \end{equation}
  where $z_\beta^{\Fr_i} \coloneqq (\iota^{Q\wedge Q}_1 \circ
  \iota^Q)_* \sz_\beta^{\Fr_i}(\tau) = (\iota^{Q\wedge Q}_2 \circ
  \iota^Q)_* \sz_\beta^{\Fr_i}(\tau)$, and
  $C_{\alpha_1,\ldots,\alpha_n}$ is defined by
  \eqref{eq:wall-crossing-total-term}. Applying
  Lemma~\ref{lem:pushforward-of-bracket-partial} to the first two
  terms of \eqref{eq:sst-invariants-comparison} gives
  \[ [\fr_1(\alpha)]_\kappa \cdot [\fr_2(\alpha)]_\kappa \cdot \left(\sz_\alpha^{\Fr_1}(\tau) - \sz_\alpha^{\Fr_2}(\tau)\right). \]
  Since $[I_*\partial_1, I_*\partial_2] = 0$ by simple computation,
  Lemma~\ref{lem:wall-crossing-total-term} shows that the sum in
  \eqref{eq:sst-invariants-comparison} vanishes. Hence
  $\sz_\alpha^{\Fr_1}(\tau) = \sz_\alpha^{\Fr_2}(\tau)$. This
  concludes the inductive step.
\end{proof}

\subsubsection{}

\begin{proposition}[Theorem~\ref{thm:sst-invariants}\ref{item:vss-pairs-relation}]
  Let $\Fr \in \Frs$ be a framing functor. Then the elements
  $\sz_\alpha^{\Fr}(\tau)$, defined explicitly by
  \eqref{eq:sstable-explicit}, satisfy \eqref{eq:sstable-def}.
\end{proposition}

\begin{proof}
  Let $l \ge 0$ be given, and suppose for induction that
  \eqref{eq:sstable-def} holds for all $r(\alpha) < l$. The base case
  $l = 0$ is vacuous since $r(\alpha) > 0$ for all $\alpha \neq 0$.
  Suppose now that $r(\alpha) = l$. Consider the relation
  \eqref{eq:pairs-master-relation} for $\Fr_1 = \Fr_2 = \Fr$. The
  induction hypothesis is applicable to the terms
  $I_*\tilde\sZ_{\alpha_1,(1,0,0)}^{\Fr}(\tau^Q)$ and
  $I_*\tilde\sZ_{\alpha_2,(0,1,0)}^{\Fr}(\tau^Q)$ because
  $r(\alpha_1) + r(\alpha_2) = r(\alpha)$ and therefore $r(\alpha_i) <
  r(\alpha)$ for $i = 1, 2$. The result is
  \begin{align*}
    0 = \Big[I_*\tilde\sZ_{\alpha,(1,0,0)}^{\Fr}(\tau^Q), I_*\partial_2\Big] + \Big[I_*\partial_1, I_*\tilde\sZ_{\alpha,(0,1,0)}^{\Fr}&{}(\tau^Q)\Big] \\
    +\sum_{\substack{n>1\\\alpha = \alpha_1+\cdots+\alpha_n\\\forall i: \tau(\alpha_i) = \tau(\alpha)\\ \;\;\fM_{\alpha_i}^{\sst}(\tau) \neq \emptyset}} \bigg(C_{\alpha_1,\ldots,\alpha_n}(\vec z^{\Fr}, I_*\partial_1, I_*\partial_2)
    &- \frac{1}{n!} \Big[\Big[z_{\alpha_n}, \big[\cdots[z_{\alpha_2}^{\Fr}, [z_{\alpha_1}^{\Fr}, I_*\partial_1]]\cdots\big]\Big], I_*\partial_2\Big] \\[-3em]
    &- \frac{1}{n!} \Big[I_*\partial_1, \Big[z_{\alpha_n}^{\Fr}, \big[\cdots[z_{\alpha_2}^{\Fr}, [z_{\alpha_1}^{\Fr}, I_*\partial_2]]\cdots\big]\Big]\Big]\bigg).
  \end{align*}
  where $z_\beta^{\Fr} \coloneqq (\iota_1^{Q\wedge Q}\circ
  \iota^Q)_*\sz_\alpha^{\Fr}(\tau) = (\iota_2^{Q\wedge Q}\circ
  \iota^Q)_*\sz_\alpha^{\Fr}(\tau)$ and $C_{\alpha_1,\ldots,\alpha_n}$
  is defined by \eqref{eq:wall-crossing-total-term}.
  Lemma~\ref{lem:wall-crossing-total-term} shows that the sum of
  $C_{\alpha_1,\ldots,\alpha_n}$ vanishes. Apply pushforward along the
  forgetful map
  \[ \pi_{\tilde\fM_{\alpha,1}^{Q(\Fr)}}\colon \tilde\fM_{\alpha,(1,1,0)}^{Q(\Fr)\wedge Q(\Fr)} \to \tilde\fM_{\alpha,(1,0,0)}^{Q(\Fr)\wedge Q(\Fr)} \xrightarrow[\sim]{(\iota_1^{Q\wedge Q})^{-1}} \tilde\fM_{\alpha,1}^{Q(\Fr)} \]
  and use Lemma~\ref{lem:pushforward-of-bracket-partial} to conclude
  that
  \begin{align*}
    0 = [\fr(\alpha)]_\kappa \cdot I_*\tilde\sZ_{\alpha,1}^{\Fr}(\tau^Q) &+ (\pi_{\tilde\fM_{\alpha,1}^{Q(\Fr)}})_*\left[I_*\partial_1, (\iota^{Q\wedge Q}_2)_*I_*\tilde\sZ_{\alpha,1}^{\Fr}(\tau^Q)\right] \\
    -\sum_{\substack{n>1\\\alpha = \alpha_1+\cdots+\alpha_n\\\forall i: \tau(\alpha_i) = \tau(\alpha)\\ \;\;\fM_{\alpha_i}^{\sst}(\tau) \neq \emptyset}} \frac{1}{n!} \bigg(
    &[\fr(\alpha)]_\kappa \cdot \Big[\iota^Q_*\sz_{\alpha_n}^{\Fr}(\tau), \big[\cdots[\iota^Q_*\sz_{\alpha_2}^{\Fr}(\tau), [\iota^Q_*\sz_{\alpha_1}^{\Fr}(\tau), I_*\partial]]\cdots\big]\Big] \\[-3em]
    &+ (\pi_{\tilde\fM_{\alpha,1}^{Q(\Fr)}})_*\Big[I_*\partial_1, (\iota^{Q\wedge Q}_2)_*\Big[\iota^Q_*\sz_{\alpha_n}^{\Fr}(\tau), \big[\cdots[\iota^Q_*\sz_{\alpha_2}^{\Fr}(\tau), [\iota^Q_*\sz_{\alpha_1}^{\Fr}(\tau), I_*\partial]]\cdots\big]\Big]\Big]\bigg).
  \end{align*}
  In the fourth term, we used that $(\iota^{Q\wedge Q}_2)_*$ is a Lie
  algebra homomorphism. By
  Lemma~\ref{lem:pushforward-of-partial-bracket} below, the second and
  fourth terms become
  \[ \bigg[I_*\partial, \iota^Q_* (\pi_{\fM_\alpha^{\Fr}})_* \bigg(I_*\tilde\sZ_{\alpha,1}^{\Fr}(\tau^Q)\; - \!\!\!\!\!\sum_{\substack{n>1\\\alpha = \alpha_1+\cdots+\alpha_n\\\forall i: \tau(\alpha_i) = \tau(\alpha)\\ \;\;\fM_{\alpha_i}^{\sst}(\tau) \neq \emptyset}} \frac{1}{n!} \Big[\iota^Q_*\sz_{\alpha_n}^{\Fr}(\tau), \big[\cdots[\iota^Q_*\sz_{\alpha_2}^{\Fr}(\tau), [\iota^Q_*\sz_{\alpha_1}^{\Fr}(\tau), \partial]]\cdots\big]\Big]\bigg)\bigg]. \]
  We don't know yet that the expression in round brackets equals
  $[\iota_*^Q \sz_\alpha^{\Fr}(\tau), I_*\partial]$ --- this is what
  we want to prove --- but upon applying $(\pi_{\fM_\alpha^{\Fr}})_*$,
  the expression becomes $[\fr(\alpha)]_\kappa \cdot
  \sz_\alpha^{\Fr}(\tau)$ by the definition
  \eqref{eq:sstable-explicit}. This produces the ``missing'' $n=1$
  piece in the third term. The resulting first and third terms,
  divided through by $[\fr(\alpha)]_\kappa$, is precisely the desired
  relation \eqref{eq:sstable-def}.
\end{proof}

\subsubsection{}

\begin{lemma} \label{lem:pushforward-of-partial-bracket}
  For any $\phi \in K_\circ^\sT(\tilde\fM_{\alpha,1}^{Q(\Fr)})^\pl_{\loc}$,
  \[ (\pi_{\tilde\fM_{\alpha,1}^{Q(\Fr)}})_*\left[(\iota_1^{Q \wedge Q})_*I_*\partial, (\iota_2^{Q \wedge Q})_*\phi\right] = \left[I_*\partial, \iota^Q_* (\pi_{\fM_\alpha^{\Fr}})_* \phi\right] \in K_\circ^\sT(\tilde\fM_{\alpha,1}^{Q(\Fr)})^\pl_{\loc}. \]
\end{lemma}

\begin{proof}
  Writing out the definition of the Lie brackets on both sides, it
  suffices to verify that
  \begin{align*}
    &\rho_K (\iota_1^{Q \wedge Q} \times \iota_2^{Q \wedge Q})^* \hat\Theta_{(0,(1,0,0)),(\alpha,(0,1,0))}(z) \otimes z^{\deg_1} \Phi_{(0,(1,0,0)),(\alpha,(0,1,0))}^* (\pi_{\tilde\fM_{\alpha,1}^{Q(\Fr)}})^* \\
    &\stackrel{?}{=} \rho_K (\id \times \pi_{\fM_\alpha^{\Fr}})^*(\id \times \iota^Q)^* \hat\Theta_{(0,1),(\alpha,0)}(z) \otimes z^{\deg_1} \Phi_{(0,1),(\alpha,0)}^*
  \end{align*}
  This equality holds because
  \[ (\iota_1^{Q \wedge Q} \times \iota_2^{Q \wedge Q})^* \hat\Theta_{(0,(1,0,0)),(\alpha,(0,1,0))}(z) = (\id \times \pi_{\fM_\alpha^{\Fr}})^*(\id \times \iota^Q)^* \hat\Theta_{(0,1),(\alpha,0)}(z), \]
  as both only involve $(\iota_1^{Q \wedge Q} \times \iota_2^{Q \wedge
    Q})^* (\cV_1^\vee \boxtimes \cFr(\cE)) = \cV^\vee \boxtimes
  \cFr(\cE) = (\id \times \pi_{\fM_\alpha^{\Fr}})^*(\id \times
  \iota^Q)^* (\cV^\vee \boxtimes \cFr(\cE))$, and there is the
  commutative diagram
  \[ \begin{tikzcd}
    {} & \tilde\fM_{0,(1,0,0)}^{Q(\Fr)\wedge Q(\Fr)} \times \tilde\fM_{\alpha,(0,1,0)}^{Q(\Fr)\wedge Q(\Fr)} \ar{d} \ar{r}{\Phi} & \tilde\fM_{\alpha,(1,1,0)}^{Q(\Fr)\wedge Q(\Fr)} \ar{d}{\pi_{\tilde\fM_{\alpha,1}^{Q(\Fr)}}} \\
    \tilde\fM_{0,1}^{Q(\Fr)} \times \tilde\fM_{\alpha,1}^{Q(\Fr)} \ar{ur}{\iota_1^{Q\wedge Q} \times \iota_2^{Q \wedge Q}} \ar{r}[swap]{\id \times (\iota^Q \circ \pi_{\fM_\alpha^{\Fr}})} & \tilde\fM_{0,1}^{Q(\Fr)} \times \tilde\fM_{\alpha,0}^{Q(\Fr)} \ar{r}[swap]{\Phi} & \tilde\fM_{\alpha,1}^{Q(\Fr)},
  \end{tikzcd} \]
  and clearly $z^{\deg_1}$ commutes with pullback along any morphism
  in the triangle on the left.
\end{proof}

\section{Wall-crossing}
\label{sec:wall-crossing}

\subsection{Setup and proof strategy}
\label{wc:sec:wc-formula}

\subsubsection{}
\label{wc:sec:dominant-wcf-proof-strategy}

The goal of this section is to prove the dominant wall-crossing
formula (Theorem~\ref{thm:wcf}). We first outline the strategy, following
ideas of \cite[\S 10]{Joyce2021}.

A priori, there are two main issues. First, the semistable invariants
$\sz_\alpha(\tau)$ constructed by Theorem~\ref{thm:sst-invariants} are
obtained from the (geometric) auxiliary invariants
$I_*\tilde\sZ^{\Fr}_{\alpha,1}(\tau^Q)$ in
\eqref{eq:sstable-def-intro} via a complicated inversion process,
making it difficult to directly compare $\sz_\alpha(\tau)$ and
$\sz_\alpha(\mathring\tau)$. Second, moving from $\tau$-semistable to
$\mathring\tau$-semistable objects, classes which destabilize can do
so in uncontrollably complicated ways.

To fix these problems, we pass to an auxiliary framed stack like in
\S\ref{sec:semistable-invariants}, but with a more complicated quiver,
and, on it, construct a family of stability conditions and auxiliary
enumerative invariants $\tilde\sz_{\beta,\vec e}^{s,x}$ depending on
parameters $(s,x) \in [0,1] \times [-1,0]$. The rough strategy is
shown in Figure~\ref{fig:dominant-wc-strategy-1}.

\begin{figure}[!ht]
  \centering
  \begin{tikzpicture}
    \draw[->,thick] (-0.5,0)--(9.5,0) node[right]{$s$};
    \draw (1.5,0)--(1.5,.15);
    \draw (3,0)--(3,.15);
    \draw (7.5,0)--(7.5,.15);
    
    \node[above] at (1.5,0) {$s_1$};
    \node[above] at (3,0) {$s_2$};
    \node[above] at (5.25,0) {$\cdots$};
    \node[above] at (7.5,0) {$s_p$};
    
    \draw[->,thick] (0,0.5)--(0,-4.5) node[below]{$x$};
    
    \node[vertex] at (0,0) {};
    \node[below] at (0.5,0) {$(0,0)$};
    \node[vertex] at (9,0) {};
    \node[below] at (8.5,0) {$(1,0)$};
    \node[vertex] at (0,-4) {};
    \node[below] at (0.65,-3.4) {$(0,-1)$};
    \node[vertex] at (0,-3) {};
    \node[below] at (0.5,-2.4) {$(0,x_0)$};
    \node[vertex] at (9,-4) {};
    \node[below] at (8.35,-3.4) {$(1,-1)$};
    \node[vertex] at (9,-3) {};
    \node[below] at (8.5,-2.4) {$(1,x_0)$};
    
    \draw[thin, draw=gray, fill=gray, opacity=0.15]
    (0,0) -- (0,-4) -- (9,-4) -- (9,0) -- cycle;
    
    \node[above] at (-1,0) {$\tilde\sz_{\beta,\mathbf{e}}^{0,0}$};
    \node[] at (-1,-3) {$\tilde\sz_{\beta,\mathbf{1}}^{0,x_0}$};
    \node[above] at (10,0) {$\tilde\sz^{1,0}_{\beta,\mathbf{e}}$};
    \node[] at (10,-3) {$\tilde\sz_{\beta,\vec 1}^{1,x_0}$};
    \node[] at (-1,-5.5) {$\sz_{\beta}(\tau)$};
    \node[] at (10,-5.5) {$\sz_{\beta}(\mathring \tau)$};
    
    \draw[<->, thick, decorate, decoration=snake] (-1,-3.3)--node[left]{\S\ref{wc:sec:sst-inv-and-pair-inv}}(-1,-5.2);
    \draw[<->, thick, decorate, decoration=snake] (10,-3.3)--node[right]{\S\ref{wc:sec:sst-inv-and-pair-inv}}(10,-5.2);
    \draw[<->, thick, red, text=darkgray] (-0.4,-5.5)--node[above]{desired {\it dominant wall-crossing}} node[below]{Theorem~\ref{thm:wcf}}(9.4,-5.5);
    
    \draw[<->, thick, orange, text=darkgray] (-1,-2.7)--node[left,align=center]{vertical\\``wall''-crossing}(-1,0.1);
    \draw[<->, thick, orange, text=darkgray] (10,-2.7)--node[right, align=center]{vertical\\``wall''-crossing}(10,0.1);
    
    \node at (1.1,0.1) (sm1) {};
    \node at (1.9,0.1) (sp1) {};
    \draw[<->, thick, orange, text=darkgray, bend left=90, distance=22] (sm1) to (sp1);
    \node at (2.6,0.1) (sm2) {};
    \node at (3.4,0.1) (sp2) {};
    \draw[<->, thick, orange, text=darkgray, bend left=90, distance=22] (sm2) to (sp2);
    \node at (7.1,0.1) (sm3) {};
    \node at (7.9,0.1) (sp3) {};
    \draw[<->, thick, orange, text=darkgray, bend left=90, distance=22] (sm3) to (sp3);
    
    \node[above, text=darkgray] at (4.5,0.8) {horizontal wall-crossing};
    
    \draw[<->, thick, blue] (9.5,-5) to (9.5,-.5) --node[below]{our path} (-.5,-.5) -- (-.5,-5);
    
  \end{tikzpicture}
  \caption{Proof strategy for dominant wall-crossing formula.}
  \label{fig:dominant-wc-strategy-1}
\end{figure}
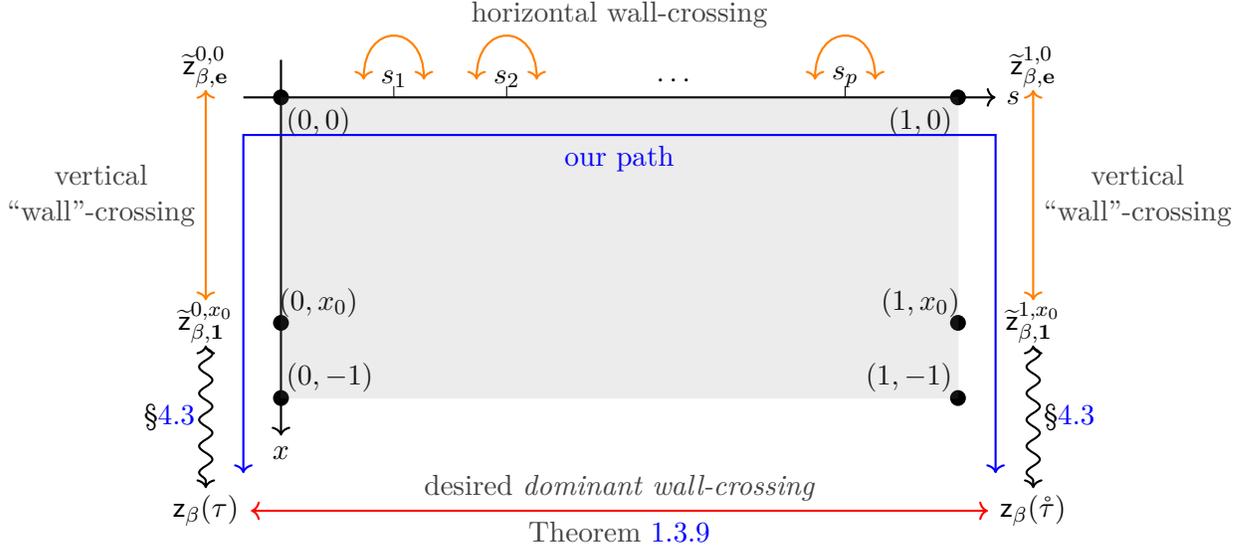

Specifically, in \S\ref{wc:sec:sst-inv-and-pair-inv}, we reformulate
the relation \eqref{eq:sstable-def-intro} which characterizes the
semistable invariants in terms of our new auxiliary invariants with
dimension vector $\vec 1 \coloneqq (1,\dots,1)$. Then, in
\S\ref{wc:sec:pair-inv-and-flag-inv}, we perform a series of
``wall-crossings'' as $x$ increases, to increase the increase the
framing dimensions until the framing is a full flag. Consequently, the
``horizontal'' wall-crossing in $s$ has only simple walls, and then
master space techniques become applicable; this takes place in
\S\ref{wc:sec:horizontal-wc}. Finally, these steps are put together in
\S\ref{wc:sec:putting-things-together} to conclude the proof.

\subsubsection{}

We now proceed with the setup necessary to explain the content of
\S\ref{wc:sec:aux-wc-strategy}. Throughout this section, we fix the
following data:
\begin{itemize}
\item stability conditions $\tau$ and $\mathring\tau$ on $\cat A$
  satisfying Assumption~\ref{assump:wall-crossing};
\item the class $\alpha \in C(\cat A)$ for which we want to prove the
  dominant wall-crossing formula (Theorem~\ref{thm:wcf});
\item a framing functor $\Fr \in \Frs$ such that
  $\fM_\beta^{\sst}(\tau) \subset \fM_\beta^{\Fr,\pl}$ for all $\beta
  \in R_\alpha$ (this is possible by
  Assumption~\ref{assump:semistable-invariants}\ref{assump:it:framing-functor}
  and finiteness of $R_\alpha$).
\end{itemize}
Let
\[ N \coloneqq \fr(\alpha). \]
We fix two extra pieces of data, to be used later in the stability
condition in Definition~\ref{def:flag-invariant}:
\begin{itemize}
\item let $\lambda$ be the function from
  Assumption~\ref{assump:wall-crossing}\ref{assump:it:lambda} scaled
  by a large positive constant so that, without loss of generality,
  $|\lambda(\beta)| > \binom{N+1}{2} r(\alpha)$ for all $\beta \in
  R_\alpha$ with $\lambda(\beta) \neq 0$;
\item choose $\vec \mu = (\mu_1,\dots,\mu_N) \in \bR^N$ with
  $1>\mu_1\gg \mu_2\gg\cdots \gg\mu_N>0$, generic in the sense that
  $\vec f \cdot \vec\mu = 0$ only has the solution $\vec f =\vec 0$ for
  finitely many equations. More concretely, by finiteness of $S_\alpha$ below, there are only finitely many possible equations \eqref{wc:eq:gen-artifical-generalities}, \eqref{wc:eq:uniqueness-genericity}, \eqref{wc:eq:ab-uniqueness-genericity}, and RHSs of \eqref{wc:eq:s-eq-genericity-nonzero}, which are the ones we assume admit only the trivial solution for our chosen $\vec \mu$\footnote{In fact, if the $\gg$-inequalities for $\vec \mu$ are sufficiently widened so that in each of these finitely many equalities any sum of the following $\vec \mu$ terms with any of the finite coefficients can never cancel the previous terms, this is the case without assuming genericity in the above form.}. (For the $\gg$-inequalities, it concretely suffices to take
  $\mu_{i+1}<\frac{\mu_i}{3RN^3}$ for all $i$, where
  $R=\prod_{\beta\in R_\alpha} r(\beta)$.)
\end{itemize}

\subsubsection{}
\label{wc:sec:S-and-R-sets}

We are mostly only interested in classes of strictly-semistabilizing
subobjects of objects of class $\alpha$. Recall the sets $R_\alpha$
and $\mathring R_\alpha$ of such classes for $\tau$ and $\mathring
\tau$, from Definition~\ref{def:dominates-at}, and furthermore define
$\hat R_\alpha$:
\begin{align*}
  R_{\alpha} &= \{\alpha\} \cup \{\beta \in C(\cat{A}) : \alpha-\beta\in C(\cat{A}), \, \tau(\beta) = \tau(\alpha-\beta), \, \fM_{\beta}^{\sst}(\tau), \fM_{\alpha-\beta}^{\sst}(\tau)\neq \emptyset\}, \\
  \hat{R}_{\alpha} &\coloneqq \{\alpha\} \cup \{\beta \in C(\cat{A}) : \alpha-\beta\in C(\cat{A}), \, \mathring\tau(\beta) = \mathring\tau(\alpha-\beta), \, \fM_{\beta}^{\sst}(\tau), \fM_{\alpha-\beta}^{\sst}(\tau)\neq \emptyset\}, \\
  \mathring R_{\alpha} &= \{\alpha\} \cup \{\beta \in C(\cat{A}) : \alpha-\beta\in C(\cat{A}), \, \mathring\tau(\beta) = \mathring\tau(\alpha-\beta), \, \fM_{\beta}^{\sst}(\mathring\tau), \fM_{\alpha-\beta}^{\sst}(\mathring\tau)\neq \emptyset\}.
\end{align*}
Throughout this section, we work only with sub- and quotient objects
of $\tau$-semistable objects in classes in $R_\alpha \sqcup \{0\}$.
For certain steps involving ``artificial'' invariants (see
Definition~\ref{wc:def:artificial-invs}), we restrict to
$\hat{R}_\alpha$. For certain steps involving $\mathring
\tau$-semistable invariants, we restrict to $\mathring R_\alpha$.

Passing to auxiliary stacks, we similarly consider sets of classes of
objects which appear in effective decompositions of semistable objects
of class $(\alpha,(1,2,\ldots,N))$:
\begin{align*}
  S_{\alpha} &\coloneqq \{(\beta,\vec{e}) : (\beta,\vec{e})\neq 0, \, \beta\in R_{\alpha} \sqcup \{0\},\, 0 \le e_i\le i \mbox{ for } 1\leq i\leq N\} \\
  \hat S_{\alpha} &\coloneqq \{(\beta,\vec{e}) : (\beta,\vec{e})\neq 0, \, \beta\in \hat R_{\alpha}\sqcup \{0\},\, 0 \le e_i \le i \mbox{ for } 1\leq i\leq N\} \\
  \mathring S_{\alpha} &\coloneqq \{(\beta,\vec{e}) : (\beta,\vec{e})\neq 0, \, \beta\in \mathring R_{\alpha} \sqcup \{0\},\, 0 \le e_i \le i \mbox{ for } 1\leq i\leq N\}
\end{align*}

Since
Assumptions~\ref{es:assump:semistable-invariants}\ref{es:assump:it:semistable-loci},
\ref{es:assump:it:rank-function}, and
\ref{es:assump:it:Bpe-sst-summands} are satisfied by hypothesis,
$R_\alpha$ is a finite set \cite[Lemma 9.1]{Joyce2021}. By
Lemma~\ref{wc:lem:R-sets}\ref{it:R-sets-i} below, therefore so are
$\hat R_\alpha$, $\mathring R_\alpha$, $S_\alpha$, $\hat S_\alpha$,
and $\mathring S_\alpha$.

\subsubsection{}

\begin{lemma} \label{wc:lem:R-sets}
  \begin{enumerate}[label = (\roman*)]
  \item \label{it:R-sets-i} $\mathring R_\alpha \subseteq
    \hat R_\alpha = \left\{\beta\in R_\alpha :
    \lambda(\beta)=0\right\} \subseteq R_\alpha$.
  \item \label{it:R-sets-ii} Suppose $E$ is $\tau$-semistable of class
    $\beta \in R_\alpha$. If $\beta'$ is the class of a sub-object $0
    \neq E' \subsetneq E$, then
    \[ \tau(\beta') < \tau(\beta - \beta') \iff \tau(\beta') < \tau(\alpha - \beta'), \]
    and the same for $=$ and $>$. If additionally $\beta \in \hat
    R_\alpha$, the same holds for $\mathring\tau$.
  \end{enumerate}
\end{lemma}

To emphasize,
Assumption~\ref{assump:semistable-invariants}\ref{assump:it:semi-weak-stability}
enters only in the proof of \ref{it:R-sets-ii}, which is later used to
prove that $\tau_x^s$ (Definition~\ref{def:flag-invariant}) is indeed
a weak stability condition on the relevant auxiliary stacks.

\begin{proof}
  For \ref{it:R-sets-i}, we follow ideas of \cite[Prop.
    10.2]{Joyce2021}. The proof there shows that
  $\fM_\alpha^{\sst}(\mathring \tau) \subset \fM_\alpha^{\sst}(\tau)$,
  so $\mathring R_\alpha \subset \hat R_\alpha$ is obvious. To show
  $\hat R_\alpha \subset R_\alpha$, take $\beta \in \hat R_\alpha$
  with $\beta \neq \alpha$. Pick elements $[E] \in
  \fM_\beta^{\sst}(\tau)$ and $[F] \in
  \fM_{\alpha-\beta}^{\sst}(\tau)$. Then $E \oplus F$ is $\mathring
  \tau$-semistable because $\mathring \tau(E) = \mathring \tau(F)$ by
  hypothesis. So $[E \oplus F] \in \fM_\alpha^{\sst}(\mathring \tau)
  \subset \fM_\alpha^{\sst}(\tau)$. But for $E, F, E \oplus F$ to all
  be $\tau$-semistable, it must be that $\tau(E) = \tau(E \oplus F) =
  \tau(F)$. Hence $\beta \in R_\alpha$ and thus $\hat R_\alpha \subset
  R_\alpha$. Finally, the claim that $\hat R_\alpha \subset R_\alpha$
  is exactly where $\lambda$ vanishes is the content of
  Assumption~\ref{assump:wall-crossing}\ref{assump:it:lambda}.

  For \ref{it:R-sets-ii}, note that $\alpha - \beta' = (\alpha -
  \beta) + (\beta - \beta')$. Suppose $\tau(\beta') <
  \tau(\beta-\beta')$. By applying the weak see-saw property twice,
  \begin{equation} \label{eq:semi-weak-stability-comparison}
    \tau(\beta') \le \tau(\alpha-(\beta-\beta')) \le \tau(\alpha-\beta) \le \tau(\alpha-\beta') \le \tau(\beta-\beta'),
  \end{equation}
  and at least one $\le$ must be strict. By
  Assumption~\ref{assump:semistable-invariants}\ref{assump:it:semi-weak-stability},
  it cannot be that only the first inequality is strict, and it also
  cannot be that only the last inequality is strict. Hence
  $\tau(\beta') < \tau(\alpha-\beta')$. Similarly one proves that
  $\tau(\beta') > \tau(\beta-\beta')$ implies $\tau(\beta') >
  \tau(\alpha-\beta')$ and similarly for $=$. Then these implications
  automatically become equivalences.
\end{proof}

\subsubsection{}

\begin{definition} \label{def:flag-invariant}
  Let $\vec Q$ denote the quiver
  \[ \begin{tikzcd}
    \overset{V_1}{\blacksquare} \ar{r}{\rho_1} & \overset{V_2}{\blacksquare} \ar{r}{\rho_2} & \cdots \ar{r}{\rho_{N-1}} & \overset{V_N}{\blacksquare} \ar{r}{\rho_N} & \overset{V_{N+1}\coloneqq \Fr(E)}{\blackbullet},
  \end{tikzcd} \]
  and consider the auxiliary exact category $\tilde{\cat A}^{\vec
    Q(\Fr)}$ parameterizing triples $(E, \vec V, \vec \rho)$ as
  labeled above. We put a two-parameter family of stability conditions
  on it. For $s \in [0,1]$ and $x \in [-1, 0]$, define
  \begin{equation}\label{wc:eq:joyce-framed-stack-stability}
    \tau_x^s\colon (\beta,\vec e) \mapsto \begin{cases}
      \left(\tau(\beta), \frac{s\lambda(\beta)+(\vec\mu+x\vec 1)\cdot \vec e}{r(\beta)}\right), & \beta \neq 0, \, \tau(\beta) = \tau(\alpha-\beta) \text{ or } \beta = \alpha, \\
      \left(\tau(\beta), \infty\right), & \beta \neq 0, \, \tau(\beta)> \tau(\alpha - \beta),\\
      \left(\tau(\beta), -\infty\right), & \beta \neq 0, \, \tau(\beta)< \tau(\alpha - \beta), \\
      \left(\infty, \frac{(\vec\mu+x\vec 1)\cdot \vec e}{\vec 1\cdot \vec e}\right), & \beta=0,\, (\vec\mu+x\vec 1)\cdot \vec e>0,\\
      \left(-\infty, \frac{(\vec\mu+x\vec 1)\cdot \vec e}{\vec 1\cdot \vec e}\right), &\beta=0,\, (\vec\mu+x\vec 1)\cdot \vec e\leq 0,
    \end{cases}
  \end{equation}
  where pairs $(a,b)$ are ordered lexicographically, i.e. $(a, b) <
  (a', b')$ means either $a < a'$, or $a = a'$ and $b < b'$. By
  Lemma~\ref{lem:joyce-framed-stack-stability}, this is effectively a
  stability condition on $\tilde{\cat{A}}^{\vec
    Q(\Fr)}$. \footnote{The extra
    Assumption~\ref{assump:semistable-invariants}\ref{assump:it:semi-weak-stability},
    not present in \cite[Assumption 5.2]{Joyce2021}, and our
    complicated definition of $\tau_x^s$, together fix an issue with
    \cite[Eq. (5.21)]{Joyce2021}, which is not a weak stability
    condition, i.e. does not satisfy \eqref{eq:stability-condition},
    if $\tau$ is a weak stability condition instead of a stability
    condition. See also Remark~\ref{rem:joyce-framed-stack-stability}. \label{footnote:joyce-framed-stack-stability}}

  Whenever there are no strictly $\tau_{x}^s$-semistable objects of
  class $(\beta, \vec{e})$, by
  Assumption~\ref{assump:wall-crossing}\ref{assump:it:properness-wcf},
  the $\sT$-fixed part of the open locus
  \[ \tilde\fM^{\vec Q(\Fr),\sst}_{\beta,\vec e}(\tau_x^s) \subset \tilde\fM^{\vec Q(\Fr),\pl}_{\beta,\vec e} \]
  is therefore an algebraic space and may be equipped with the
  $\kappa$-symmetric APOT obtained by symmetrized pullback
  (Theorem~\ref{thm:APOTs}) along the smooth forgetful
  map $\pi_{\fM_\alpha^\pl}$. Thus the universal enumerative invariant
  \[ \tilde\sZ_{\beta,\vec e}^{s,x} \coloneqq \chi\left(\tilde\fM_{\beta,\vec e}^{\vec Q(\Fr), \sst}(\tau_x^s), \hat{\cO}^{\vir} \otimes - \right) \in K_\circ^{\tilde\sT}(\tilde\fM_{\beta,\vec e}^{\vec Q(\Fr),\pl})_{\loc} \]
  is well-defined (Definition~\ref{bg:def:univ-enum-inv}).
\end{definition}

\subsubsection{}

\begin{lemma} \label{lem:joyce-framed-stack-stability}
  Let $(E, \vec V, \vec\rho) \in \cat{A}^{\vec Q(\Fr)}$ have class
  $(\beta, \vec e)$ such that $\beta \in R_\alpha \sqcup \{0\}$. Then
  for any non-zero sub-object $(E', \vec V', \vec\rho') \subsetneq (E,
  \vec V, \vec\rho)$, of class $(\beta', \vec e')$, either
  \begin{equation} \label{eq:joyce-framed-stack-is-a-stability-condition}
    \begin{aligned}
      &\tau_x^s(\beta', \vec e') \le \tau_x^s(\beta, \vec e) \le \tau_x^s(\beta-\beta', \vec e-\vec e') \text{ or }  \\
      &\tau_x^s(\beta', \vec e') \ge \tau_x^s(\beta, \vec e) \ge \tau_x^s(\beta-\beta', \vec e-\vec e').
    \end{aligned}
  \end{equation}
\end{lemma}

We say $\tau_x^s$ is {\it effectively a weak stability condition} on
$\tilde{\cat{A}}^{\vec Q(\Fr)}$ because, throughout this section (and
also in \S\ref{sec:semistable-invariants} with
$\tilde{\cat{A}}^{Q(\Fr)}$), we will only need to determine the
$\tau_x^s$-(semi)stability of objects $(E, \vec V, \vec \rho)$ with
$\beta \in R_\alpha \sqcup \{0\}$. Hence we may treat $\tau_x^s$ as a
genuine weak stability condition.

\begin{proof}
  We proceed by casework. Let $\tau_x^s(\beta, \vec e) = (\tau(\beta),
  \psi)$. Note that $\psi \in \bR$ because $\tau(\beta) = \tau(\alpha
  - \beta)$ by hypothesis.

  Suppose $\beta' = 0$. Then we must compare
  \[ (\pm \infty, *) \text{ and } (\tau(\beta), \psi) \text{ and } (\tau(\beta), \psi - (\vec\mu + x\vec 1)\cdot \vec e'). \]
  If $(\vec\mu + x\vec 1)\cdot \vec e' > 0$, then the sign is $+$ and
  the ordering is $>$ and $>$; otherwise the sign is $-$ and the
  ordering is $<$ and $\le$. In both cases,
  \eqref{eq:joyce-framed-stack-is-a-stability-condition} holds. By
  symmetry, this also holds if $\beta - \beta' = 0$.

  Suppose $\beta', \beta-\beta' \neq 0$ and $\tau(\beta') <
  \tau(\beta-\beta')$. By Lemma~\ref{wc:lem:R-sets}\ref{it:R-sets-ii},
  $\tau(\beta') < \tau(\alpha-\beta')$ and $\tau(\beta-\beta') <
  \tau(\alpha-(\beta-\beta'))$ (this genuinely requires
  Assumption~\ref{assump:semistable-invariants}\ref{assump:it:semi-weak-stability}),
  and thus we must compare
  \[ (\tau(\beta'), -\infty) \text{ and } (\tau(\beta), \psi) \text{ and } (\tau(\beta-\beta'), \infty). \]
  Regardless of whether $\tau(\beta') = \tau(\beta) <
  \tau(\beta-\beta')$ or $\tau(\beta') < \tau(\beta) =
  \tau(\beta-\beta')$, the ordering is therefore $<$ and $<$. So
  \eqref{eq:joyce-framed-stack-is-a-stability-condition} holds. By
  symmetry, this also holds if $\tau(\beta') > \tau(\beta-\beta')$.

  Suppose $\beta', \beta-\beta' \neq 0$ and $\tau(\beta') =
  \tau(\beta-\beta')$. Again by
  Lemma~\ref{wc:lem:R-sets}\ref{it:R-sets-ii}, $\tau(\beta') =
  \tau(\alpha-\beta')$ and $\tau(\beta-\beta') =
  \tau(\alpha-(\beta-\beta'))$ (this does not genuinely require
  Assumption~\ref{assump:semistable-invariants}\ref{assump:it:semi-weak-stability}),
  and thus we must compare the second entries in the first case of
  \eqref{wc:eq:joyce-framed-stack-stability}. By
  Assumption~\ref{assump:semistable-invariants}\ref{assump:it:rank-function},
  the denominator is positive and additive, i.e. $r(\beta'),
  r(\beta-\beta') > 0$ and $r(\beta) = r(\beta') + r(\beta-\beta')$.
  By Assumption~\ref{assump:wall-crossing}\ref{assump:it:lambda}, the
  numerator is also additive. Since $(a+b)/(c+d)$ always sits between
  $a/c$ and $b/d$ if $c, d > 0$, we conclude that
  \eqref{eq:joyce-framed-stack-is-a-stability-condition} holds.

  This exhausts all possible cases.
\end{proof}

\subsubsection{}

\begin{remark} \label{rem:joyce-framed-stack-stability}
  We will actually only use the following properties of $\tau^s_x$:
  \[ \tau^s_x\colon (\beta, \vec e) \mapsto \begin{cases}
      \left(\tau(\beta), \frac{s\lambda(\beta)+(\vec\mu+x\vec 1)\cdot \vec e}{r(\beta)}\right), & \beta \neq 0, \, \tau(\beta) = \tau(\alpha-\beta) \text{ or } \beta = \alpha, \\
      \left(\infty, \frac{(\vec\mu+x\vec 1)\cdot \vec e}{\vec 1\cdot \vec e}\right), & \beta=0,\, (\vec\mu+x\vec 1)\cdot \vec e>0,\\
      \left(-\infty, \frac{(\vec\mu+x\vec 1)\cdot \vec e}{\vec 1\cdot \vec e}\right), &\beta=0,\, (\vec\mu+x\vec 1)\cdot \vec e\leq 0,
    \end{cases} \]
  and that $\tau^s_x$ is effectively a weak stability condition. The
  second and third lines in the definition
  \eqref{wc:eq:joyce-framed-stack-stability}, along with
  Assumption~\ref{assump:semistable-invariants}\ref{assump:it:semi-weak-stability},
  are our {\it choice} of a construction of such a weak stability
  condition. Other constructions may be useful and/or necessary for
  applications outside the scope of this paper.
\end{remark}

\subsubsection{}

\begin{definition} \label{def:full-flags}
  A class $(\beta,\vec{e}) \in S_\alpha$ is a {\it flag} if $\beta
  \neq 0$ and
  \[ e_1 \le 1, \quad e_i \le e_{i+1} \le e_i+1 \text{ for } 1 \le i < N, \quad e_N \le \fr(\beta) \]
  and is a {\it full flag} if in addition $e_N = \fr(\beta)$. Note
  that, for flags, if $e_N \ge 1$ then there exists $1 \le j \le N$
  such that $e_j = 1$, and therefore there is a ``de-rigidification''
  map $I_{\beta,\vec e}$ (Definition~\ref{def:de-rigidification})
  which is inverse to the isomorphism $\Pi^\pl_{\beta,\vec e}$. This
  condition is automatically satisfied for full flags because $\beta
  \neq 0$ implies $\fr(\beta) > 0$.
  
  Let $\vec 0$ be the zero vector. For integers $a$ and $b$, let $\vec
  1_{[a,b]} \coloneqq (0, \ldots, 0, 1, \ldots, 1, 0, \ldots, 0)$
  where the $1$'s appear exactly in the positions in the interval
  $[a,b]$. In particular, if $a > b$ then $\vec 1_{[a,b]} = \vec 0$ by
  definition.
\end{definition}

\subsubsection{}
\label{wc:sec:va-homomorphisms}

To summarize, we will work with: the original moduli stack $\fM$ of
the abelian category $\cat{A}$, the auxiliary moduli stack
$\tilde\fM^{Q(\Fr)}$ (Definition~\ref{def:pair-invariant}) used to
define semistable invariants, and the auxiliary moduli stack
$\tilde\fM^{\vec Q(\Fr)}$ (Definition~\ref{def:flag-invariant}). Their
K-homologies are equipped with vertex algebra structures by
Theorems~\ref{thm:mVOA-monoidal-stack} and
\ref{thm:auxiliary-stack-vertex-algebra}. There are natural
isomorphisms of moduli stacks
\begin{equation} \label{eq:auxiliary-stack-zero-framing-isomorphism}
  \iota^{\vec Q}\colon \fM^{\Fr}_\beta \xrightarrow{\sim} \tilde\fM^{\vec Q(\Fr)}_{\beta,\vec 0},\quad [E] \mapsto [E, \vec 0, \vec 0]
\end{equation}
and, for any $1 \le a \le N$, natural morphisms of moduli stacks
\[ \iota^{\vec Q}_{[a,N]}\colon \tilde\fM_{\beta,e}^{Q(\Fr)} \to \tilde\fM_{\beta, e \cdot \vec 1_{[a,N]}}^{\vec Q(\Fr)}, \quad [E, V, \rho] \mapsto [E, \vec V, \vec\rho] \]
where $V_i = 0$ and $\rho_i = 0$ for $i<a$, $V_i=V$ for $a\leq i\leq
N$, $\rho_i=\id_V$ for $a\leq i< N$, and $\rho_N=\rho$.

It is easy to verify that both $\iota^{\vec Q}$ and $\iota^{\vec
  Q}_{[a,N]}$ are morphisms of graded monoidal $\sT$-stacks: the only
non-trivial part is to check that the contributions
\eqref{eq:framed-stack-forgetful-map-cotangent}, to the bilinear
elements of Theorem~\ref{thm:auxiliary-stack-vertex-algebra}, pull
back correctly. Hence the morphisms $\iota_*^{\vec Q}$ and
$\iota_{[a,N]*}^{\vec Q}$ are homomorphisms of vertex and Lie algebras.

\subsubsection{}

\begin{definition}[Auxiliary and artificial invariants] \label{wc:def:artificial-invs}
  Let $(\beta, \vec e) \in S_\alpha$ and $(s, x) \in [0,1] \in [-1,
    0]$. Define the {\it auxiliary invariants}
  \[ \tilde\sz^{s,x}_{\beta,\vec e} \in K^{\tilde{\sT}}_\circ(\tilde\fM^{\vec Q(\Fr)}_{\beta,\vec e})_{\loc,\bQ}^\pl \]
  in the following (disjoint) cases:
  \begin{enumerate}[label = (\roman*)]
  \item \label{wc:def:aux-inv-i} if there are no strictly
    $\tau_{x}^s$-semistable objects of class $(\beta, \vec{e})$ and
    $(\beta, \vec e)$ is a flag with $e_N \ge 1$, then
    \[ \tilde\sz^{s,x}_{\beta,\vec{e}} \coloneqq I_* \tilde \sZ^{s,x}_{\beta,\vec e} \]
    is the ``de-rigidified'' universal enumerative invariant
    of Definition~\ref{def:flag-invariant};
  \item \label{wc:def:aux-inv-ii} if $s=0$ and $\beta\in R_{\alpha}$,
    then $\tilde\sz_{\beta,\vec{0}}^{0,x} \coloneqq \iota_*^{\vec
      Q}(\sz_{\beta}(\tau))$ is a semistable invariant
    (Theorem~\ref{thm:sst-invariants});
  \item \label{wc:def:aux-inv-iii} if $s = 1$ and $\beta\in \mathring
    R_{\alpha}$, then $\tilde\sz_{\beta,\vec{0}}^{1,x} \coloneqq
    \iota_*^{\vec Q}(\sz_{\beta}(\mathring\tau))$ is a semistable
    invariant (Theorem~\ref{thm:sst-invariants});
  \item \label{wc:def:aux-inv-iv} if $\beta = 0$, then
    \[ \tilde\sz_{0,\vec e}^{s,x} \coloneqq \begin{cases} I_*\partial_{[a,b]} & \vec e = \vec 1_{[a,b]} \\ 0 & \text{otherwise}, \end{cases} \]
    where
    \[ \partial_{[a,b]} \coloneqq \chi\left(\tilde\fM_{0,\vec 1_{[a,b]}}^{\vec Q(\Fr),\sst}(\tau^s_x), -\right) = \id \in K_\circ^{\tilde\sT}(\tilde\fM_{0,\vec 1_{[a,b]}}^{\vec Q(\Fr),\pl})_{\loc}. \]
    By
    Lemma~\ref{wc:lem:general-semistable-loci}\ref{wc:lem:general-semistable-loci-i}
    below, $\tilde\fM_{0,\vec 1_{[a,b]}}^{\vec Q(\Fr),\sst}(\tau^s_x) =
    \pt$ is independent of $(s,x)$, justifying the notation
    $\partial_{[a,b]}$.
  \end{enumerate}
  In all the above cases, $\tilde\sz_{\beta,\vec e}^{s,x}$ is indeed
  supported on the (un-rigidified) $\tau_x^s$-semistable locus. Hence
  these cases are compatible with the following last case:
  \begin{enumerate}[resume, label = (\roman*)]
  \item \label{wc:def:aux-inv-v} if $\tilde\fM_{\beta,\vec e}^{\vec
    Q(\Fr),\sst}(\tau^s_x) = \emptyset$, then $\tilde\sz_{\beta,\vec
    e} \coloneqq 0$.
  \end{enumerate}
  Using these auxiliary invariants, define the \textit{artificial
    invariants} \footnote{\label{footnote:artificial-invariants} In
  contrast to \cite[\S 10.3]{Joyce2021}, these are defined using the
  honest Lie bracket on the K-homology of the auxiliary stack
  $\tilde\fM^{\vec Q(\Fr)}$. This complicates some aspects of the
  proof, in particular Lemma~\ref{wc:lem:vanishing-basic}, but we
  believe it is less ad-hoc.}
  \begin{equation} \label{eq:artificial-invariant}
    \hat\sz^{s,x}_{\beta,\vec e} \coloneqq \!\!\!\!\!\!\sum_{\substack{n\geq 1, \, (\beta_i,\vec{e}_i) \in S_{\alpha}\\ (\beta, \vec{e}) = (\beta_1, \vec{e}_1) + \cdots + (\beta_n, \vec{e}_n)}} \!\!\!\!\!\! \tilde{U}\left((\beta_1, \vec e_1),\ldots,(\beta_n, \vec e_n); \tau_{-1}^0, \tau_{x}^s\right) \cdot \left[\left[\cdots\left[\tilde\sz_{\beta_1, \vec{e}_1}^{0,-1}, \tilde\sz_{\beta_2, \vec{e}_2}^{0,-1} \right], \cdots\right],\tilde\sz_{\beta_n, \vec{e}_n}^{0,-1}\right].
  \end{equation}
  This expression is well-defined: by
  Lemma~\ref{wc:lem:general-semistable-loci} below, if
  $\tilde\fM_{\beta, \vec{e}}^{\vec Q(\Fr),\sst}(\tau_{-1}^0)\neq
  \emptyset$, then either $\beta=0$ or $\vec e=\vec 0$, and so the
  invariants $\tilde\sz_{\beta_i,\vec e_i}^{0,-1}$ are defined by
  cases \ref{wc:def:aux-inv-iv} and \ref{wc:def:aux-inv-ii} above,
  respectively.
\end{definition}

\subsubsection{}
\label{wc:sec:def-x0a}

\begin{definition}
  For $1 \le a \le N$, define
  \[ x_0(a) \coloneqq -\frac{1}{N-a+1}\sum_{i=a}^N \mu_i, \]
  and set $x_0(N+1) \coloneqq 0$. This is the solution to
  \begin{equation} \label{wc:eq:x0a}
    \left(\vec \mu + x_0(a)\vec 1\right)\cdot \vec 1_{[a,N]}=0,
  \end{equation}
  i.e. a quotient object of class $(0, \vec 1_{[a,N]})$ is
  $\tau_x^s$-destabilizing for an object of class $(\beta, \vec e)$ if
  and only if $x \le x_0(a)$. The condition $1 > \mu_1 \gg \cdots \gg
  \mu_N > 0$ ensures that
  \begin{equation} \label{wc:eq:x0a-monotonicity}
    -1 < x_0(1) < x_0(2) < \cdots < x_0(N) < 0.
  \end{equation}
\end{definition}

\subsubsection{}

\begin{lemma}[{\cite[Prop. 10.3]{Joyce2021}}]\label{wc:lem:general-semistable-loci}
  Suppose $s \in [0,1]$ and $x \in [-1, 0]$ and $(\beta,\vec{e})\in
  S_{\alpha}$, such that $\tilde\fM_{\beta, \vec{e}}^{\vec
    Q(\Fr),\sst}(\tau_x^s)\neq \emptyset$. Then
  \begin{enumerate}[label=(\roman*)]
  \item \label{wc:lem:general-semistable-loci-i} If $\beta = 0$, then
    $\vec{e} = k\vec{1}_{[a,b]}$ for some integers $1\leq a\leq b\leq
    N$ and $k > 0$, and $\tilde\fM_{0,\vec{1}_{[a,b]}}^{\vec
      Q(\Fr),\sst}(\tau_x^s)$ is the locus where $\rho_i$ are
    isomorphisms for $a\leq i<b$.
  \item \label{wc:lem:general-semistable-loci-ii} If $\beta\neq 0$,
    then $e_1\leq e_2\leq \cdots \leq e_N\leq \fr(\beta)$ and
    $\tilde\fM_{\beta,\vec e}^{\vec Q(\Fr),\sst}(\tau_x^s)$ is
    contained in the locus where all $\rho_i$ are injective and $E$ is
    $\tau$-semistable. Furthermore, if $s = 1$ and $(\beta, \vec e)
    \in \mathring S_\alpha$, then $E$ is also
    $\mathring\tau$-semistable.
  \item \label{wc:lem:general-semistable-loci-iii} If $\beta\neq 0$
    and $x\leq x_0(a)$ for $2\leq a \leq N$, then $e_{a-1} = e_a =
    \cdots = e_N$.
  \item \label{wc:lem:general-semistable-loci-iv} If $\beta \neq 0$
    and $x\leq x_0(1)$, then $\vec{e}=0$.
  \end{enumerate}
\end{lemma}

\begin{proof}
  The same proof as for \cite[Prop. 10.3]{Joyce2021} holds, despite
  our more complicated definition
  \eqref{wc:eq:joyce-framed-stack-stability} of $\tau_x^s$.
\end{proof}

\subsubsection{}
\label{wc:sec:aux-wc-strategy}

The goal of this section is to prove the following formula relating
auxiliary and artificial invariants, for various $(\beta,\vec e) \in
S_\alpha$ and $(s,x) \in [0,1] \times [-1, 0]$:
\begin{equation}\label{wc:eq:aux-wc-formula}
  \tilde\sz^{s,x}_{\beta,\vec e} \labeleq{?} \hat\sz^{s,x}_{\beta,\vec e}.
\end{equation}
Specifically, by Lemma~\ref{wc:lemma:reduction-wc-to-aux-wc} below,
the final goal is to prove \eqref{wc:eq:aux-wc-formula} for
$(\beta,\vec e)=(\alpha,\vec 0)$ and $(s,x)=(1,-1)$. To do so, we
follow the path in Figure~\ref{fig:dominant-wc-strategy-2} for {\it
  both} the auxiliary and the artificial invariants.

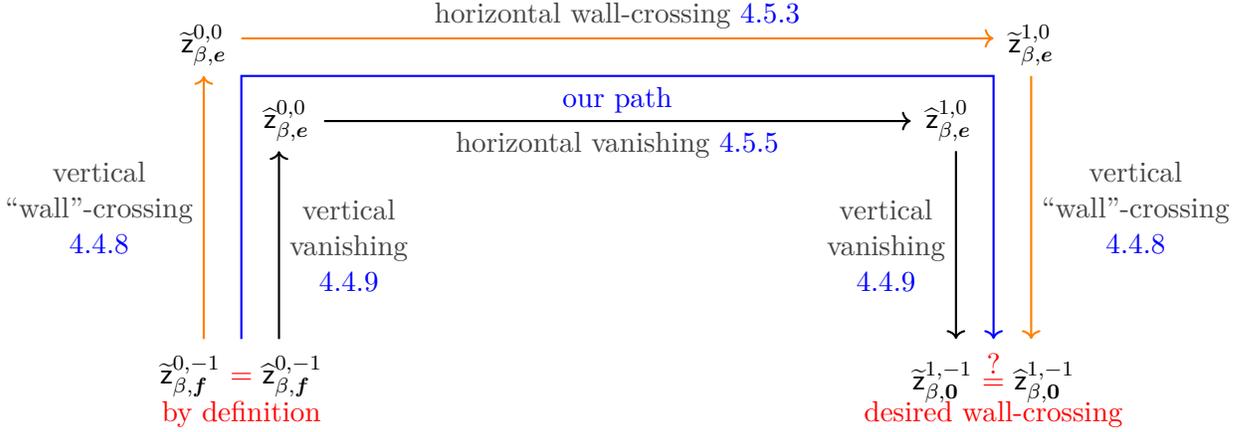
\begin{figure}[!ht]
  \centering
  \begin{tikzpicture}
    \node[above] at (-1,-.5) {$\tilde\sz_{\beta,\vec{e}}^{0,0}$};
    \node[] at (.1,-1.1) {$\hat\sz_{\beta,\vec{e}}^{0,0}$};

    \node[above] at (10,-.5) {$\tilde\sz^{1,0}_{\beta,\vec{e}}$};
    \node[] at (8.9,-1.1) {$\hat\sz_{\beta,\vec e}^{1,0}$};

    \node[] at (-.5,-4.5) {$\tilde\sz_{\beta,\vec f}^{0,-1}\ {\color{red}=}\ \hat\sz_{\beta,\vec f}^{0,-1}$};
    \node[] at (9.5,-4.5) {$\tilde\sz_{\beta,\vec 0}^{1,-1}\ {\color{red}\labeleq{?}}\ \hat\sz_{\beta,\vec 0}^{1,-1}$};

    \node[text=red] at (-.5,-5) {by definition};
    \node[text=red] at (9.5,-5) {desired wall-crossing};

    \draw[->, thick, orange, text=darkgray] (-1,-4)--node[left,align=center]{vertical\\``wall''-crossing\\ \ref{wc:lem:pair-flag}}(-1,-.5);
    \draw[->, thick, text=darkgray] (0,-4)--node[right,align=center]{vertical\\vanishing\\ \ref{wc:lem:pair-flag-vanishing}}(0,-1.5);

    \draw[->, thick, orange, text=darkgray] (10,-.5)--node[right,align=center]{vertical\\``wall''-crossing\\ \ref{wc:lem:pair-flag}}(10,-4);
    \draw[->, thick, text=darkgray] (9,-1.5)--node[left,align=center]{vertical\\vanishing\\ \ref{wc:lem:pair-flag-vanishing}}(9,-4);

    \draw[->, thick, orange, text=darkgray] (-.5,0)--node[above,align=center]{horizontal wall-crossing \ref{wc:prop:horizontal-wc}}(9.5,0);
    \draw[->, thick, text=darkgray] (0.6,-1.1)--node[below,align=center]{horizontal vanishing \ref{wc:prop:horizontal-vanishing}}(8.4,-1.1);
    
    \draw[<-, thick, blue] (9.5,-4) to (9.5,-.5) --node[below]{our path} (-.5,-.5) -- (-.5,-4);
    
  \end{tikzpicture}
  \caption{Proof strategy for \eqref{wc:eq:aux-wc-formula} for $(\beta,\vec e)=(\alpha,\vec 0)$ and $(s,x)=(1,-1)$.}
  \label{fig:dominant-wc-strategy-2}
\end{figure}

When crossing consecutive walls, wall-crossing terms from one wall
will themselves exhibit wall-crossing behavior at further walls. The
artificial invariants $\hat\sz$ simplify the bookkeeping of such
iterated wall-crossings. While the behavior of the auxiliary
invariants $\tilde\sz$ may be studied by geometric means, to prove
analogous results for the artificial invariants $\hat\sz$ requires a
careful analysis of the universal coefficients $\tilde U$ and the Lie
bracket. For this, we first prove several technical vanishing results
in \S\ref{wc:sec:gen-sst-inv}.

\subsubsection{}

\begin{lemma}\label{wc:lemma:reduction-wc-to-aux-wc}
  The dominant wall-crossing formula (Theorem~\ref{thm:wcf}) holds if
  and only if \eqref{wc:eq:aux-wc-formula} holds for $(\beta, \vec e)
  = (\alpha, \vec 0)$ and $(s,x) = (1,-1)$.
\end{lemma}

\begin{proof}
  Consider \eqref{wc:eq:aux-wc-formula} for $(\beta, \vec e) =
  (\alpha, \vec 0)$ and $(s,x) = (1,-1)$. We implicitly use the
  isomorphism $\iota^{\vec Q}_*$ of Lie algebras, induced from
  \eqref{eq:auxiliary-stack-zero-framing-isomorphism}, throughout.
  
  On the left hand side, the invariant $\tilde\sz^{1,-1}_{\alpha,\vec
    0}$ equals $\sz_{\alpha}(\mathring\tau)$ by definition.

  On the right hand side, all terms in the sum
  \eqref{eq:artificial-invariant} must be of the form
  $\tilde\sz_{\beta, \vec 0}^{0,-1}$, which equals $\sz_\beta(\tau)$
  by definition. Note that $(\beta,\vec 0)$ is in $S_\alpha$ if and
  only if $\beta$ is in $R_\alpha$. Finally, we claim
  \[ \tilde U\left((\beta_1,\vec 0),\ldots,(\beta_1,\vec 0); \tau_{-1}^0, \tau_{-1}^1\right) = \tilde U\left(\beta_1,\ldots,\beta_1; \tau, \mathring\tau\right). \]
  This follows from \cite[Prop. 10.5(b)]{Joyce2021}, whose proof only
  uses the combinatorics of the $\tilde{U}$-coefficients, which is
  unchanged for us, and also \cite[Eq. (10.15)]{Joyce2021}, which we
  prove below in \S\ref{wc:sec:tilde-tau-tau1-equivalence} with our
  weakened assumptions on $\lambda$ (see
  footnote~\ref{footnote:lambda}). Putting all this together, we get
  \[ \sum_{\substack{n\geq 1, \beta_i \in R_{\alpha}\\ \alpha = \beta_1 + \cdots + \beta_n}} \tilde{U}\left(\beta_1,\ldots,\beta_n; \tau, \mathring\tau\right) \cdot \left[\left[\cdots\left[\sz_{\beta_1}(\tau), \sz_{\beta_2}(\tau) \right], \cdots\right],\sz_{\beta_n}(\tau)\right]. \]

  These are exactly the left and right hand sides of the desired
  dominant wall-crossing formula.

\subsubsection{}
\label{wc:sec:tilde-tau-tau1-equivalence}

  We prove \cite[Eq. (10.15)]{Joyce2021}, which is the claim that,
  given $\gamma, \delta \in R_\alpha$ with $\beta\coloneqq
  \gamma+\delta\in \mathring R_\alpha$,
  \[ \mathring\tau(\gamma) \le \mathring\tau(\delta) \iff \tau_{-1}^1(\gamma,\vec 0)\leq \tau_{-1}^1(\delta,\vec 0). \]
  First, by Lemma~\ref{wc:lem:R-sets}\ref{it:R-sets-ii},
  \[ \mathring\tau(\gamma) \le \mathring\tau(\delta) \iff \mathring\tau(\gamma) \le \mathring\tau(\alpha-\gamma). \]
  This allows us to use
  Assumption~\ref{assump:wall-crossing}\ref{assump:it:lambda} for
  $\lambda$ (which is weaker than Joyce's version) to conclude that
  \[ \mathring\tau(\gamma) \le \mathring\tau(\delta) \iff \lambda(\gamma) \le 0 \le \lambda(\delta). \]
  The remainder of the argument is identical to Joyce's. Namely,
  \[ \lambda(\gamma) \le 0 \le \lambda(\delta) \iff \frac{\lambda(\gamma)}{r(\gamma)} \le \frac{\lambda(\delta)}{r(\delta)} \iff \tau_{-1}^1(\gamma,\vec 0) \le \tau_{-1}^1(\delta,\vec 0) \]
  where the first equivalence holds because $r(-) > 0$ and
  $\lambda(\gamma) + \lambda(\delta) = \lambda(\beta) = 0$, and the
  second holds by the definition of $\tau_x^s$ as $\tau(\gamma) =
  \tau(\alpha) = \tau(\delta)$.
\end{proof}

\subsection{Generalities on universal coefficients and artificial invariants}
\label{wc:sec:gen-sst-inv}

\subsubsection{}
\label{wc:sec:artificial-wcf}

In this subsection, we provide some general wall-crossing and
vanishing results for artificial invariants. In contrast to the
auxiliary invariants, whose wall-crossing behavior is studied using
geometric techniques, results about artificial invariants are obtained
purely using the combinatorics of the universal coefficients
$\tilde{U}$. We list the main results.
\begin{itemize}
\item (Lemma~\ref{wc:lem:U-properties}) The universal coefficients $U$
  satisfy certain combinatorial identities. These are used repeatedly
  throughout this section.
\item (Lemma~\ref{wc:lem:artificial-base-wcf}) $\hat\sz^{s,x}$ may be
  expressed in terms of $\hat\sz^{s',x'}$ instead of
  $\tilde\sz^{0,-1}$, yielding an ``artificial wall-crossing
  formula''.
\item (Lemma~\ref{wc:lem:vanishing-basic}) $\hat\sz$ satisfies an
  analogue of Lemma~\ref{wc:lem:general-semistable-loci}. We will use
  this later to simplify the artificial wall-crossing formula along
  various parts of our path in
  Figure~\ref{fig:dominant-wc-strategy-2}.
\end{itemize}

\subsubsection{}

\begin{lemma} \label{wc:lem:U-properties}
  Let $\cat{A}$ be an abelian category and $\alpha_1, \ldots, \alpha_n
  \in C(\cat A)$. Let $\tau$, $\tau'$, $\tau''$ be stability
  conditions.
  \begin{enumerate}[label = (\roman*)]
  \item \label{it:U-identity} (\cite[Thm. 3.11, (3.4)]{Joyce2021})
    \[ U(\alpha_1, \ldots, \alpha_n; \tau, \tau) = \begin{cases} 1 & n = 1 \\ 0 & \text{otherwise}.\end{cases} \]
  \item \label{it:U-composition} (\cite[Thm. 3.11, (3.5)]{Joyce2021})
    \begin{align*}
      U(\alpha_1, \ldots, \alpha_n; \tau, \tau'') = \!\!\!\!\sum_{\substack{m\ge 1\\0<a_1<\cdots<a_m=n\\\beta_i\coloneqq \alpha_{a_{j-1}+1}+\cdots+\alpha_{a_j}}} &\!\!\!\!U(\beta_1, \ldots, \beta_m; \tau', \tau'') \\[-2.5em]
      &\qquad \cdot \prod_{j=1}^m U(\alpha_{a_{j-1}+1}, \alpha_{a_{j-1}+2}, \ldots, \alpha_{a_j}; \tau, \tau').
    \end{align*}
  \item \label{it:U-vanishing} (\cite[Prop. 3.14]{Joyce2021}) Suppose
    $n \ge 2$ and for some $\emptyset \neq I \subsetneq \{1, \ldots,
    n\}$ we have $\tau(\alpha_i) < \tau(\alpha_j)$ for all $i \in I$
    and $j \notin I$, and $\tau'(\alpha_i) < \tau'(\alpha_1 + \cdots +
    \alpha_n)$ for all $i \in I$. Then
    \[ S(\alpha_1, \ldots, \alpha_n); \tau, \tau') = U(\alpha_1, \ldots, \alpha_n; \tau, \tau') = \tilde U(\alpha_1, \ldots, \alpha_n; \tau, \tau') = 0. \]
    The same holds if instead $\tau(\alpha_i) > \tau(\alpha_j)$ and
    $\tau'(\alpha_i) > \tau'(\alpha_1 + \cdots + \alpha_n)$ above.
  \end{enumerate}
\end{lemma}

\subsubsection{}

\begin{lemma}[Artificial wall-crossing formula] \label{wc:lem:artificial-base-wcf}
  \leavevmode
  \begin{enumerate}[label = (\roman*)]
  \item \label{it:artificial-wcf-i} For any $(\beta,\vec e)\in
    S_\alpha$ and any $(s_1,x_1),(s_2,x_2) \in [0,1] \times [-1,0]$,
    \begin{equation} \label{wc:eq:artificial-base-wcf}
      \hat\sz_{\beta,\vec{e}}^{s_1,x_1} = \sum_{\substack{n\geq 1, (\beta_i,\vec{e}_i)\in S_{\alpha}\\(\beta, \vec{e}) = (\beta_1, \vec{e}_1) + \cdots + (\beta_n, \vec{e}_n)}}\tilde{U}\left((\beta_1,\vec e_1),\ldots,(\beta_n,\vec e_n); \tau_{x_2}^{s_2}, \tau_{x_1}^{s_1}\right) \left[\left[\cdots\left[\hat\sz_{\beta_1,\vec e_1}^{s_2,x_2}, \hat\sz_{\beta_2, \vec e_2}^{s_2,x_2} \right], \cdots\right],\hat\sz_{\beta_n,\vec e_n}^{s_2,x_2}\right].
    \end{equation} 
  \item \label{it:artificial-wcf-ii} For any $(\beta,\vec e)\in
    \hat{S}_\alpha$ and $-1 \le x_2\le x_1 \le 0$,
    \begin{equation}\label{wc:eq:artificial-base-wcf-s1}
      \hat\sz_{\beta,\vec{e}}^{1,x_1} = \sum_{\substack{n\geq 1, (\beta_i,\vec{e}_i)\in \hat{S}_{\alpha}\\(\beta, \vec{e}) = (\beta_1, \vec{e}_1) + \cdots + (\beta_n, \vec{e}_n)}}\tilde{U}\left((\beta_1,\vec e_1),\ldots,(\beta_n,\vec e_n); \tau_{x_2}^{1}, \tau_{x_1}^{1}\right) \left[\left[\cdots\left[\hat\sz_{\beta_1,\vec e_1}^{1,x_2}, \hat\sz_{\beta_2,\vec e_2}^{1,x_2} \right], \cdots\right],\hat\sz_{\beta_n,\vec e_n}^{1,x_2}\right].
    \end{equation}
  \end{enumerate}
\end{lemma}

\begin{proof}
  For \ref{it:artificial-wcf-i}, insert
  Definition~\ref{wc:def:artificial-invs} for the $\hat\sz^{s_2,x_2}$
  on the right hand side, and use the identity
  Lemma~\ref{wc:lem:U-properties}\ref{it:U-composition}
  \begin{align*}
    U(\alpha_1, \ldots, \alpha_n; \tau_{-1}^0, \tau_{x_1}^{s_1}) = \!\!\!\!\sum_{\substack{m\ge 1\\0<a_1<\cdots<a_m=n\\\beta_i\coloneqq \alpha_{a_{j-1}+1}+\cdots+\alpha_{a_j}}} &\!\!\!\!U(\beta_1, \ldots, \beta_m; \tau_{x_2}^{s_2}, \tau_{x_1}^{s_1}) \\[-2.5em]
    &\qquad \cdot \prod_{j=1}^m U(\alpha_{a_{j-1}+1}, \alpha_{a_{j-1}+2}, \ldots, \alpha_{a_j}; \tau_{-1}^0, \tau_{x_2}^{s_2})
  \end{align*}
  along with Lemma~\ref{lem:Utilde-definition} defining $\tilde U$.
  This immediately produces the definition of
  $\hat\sz_{\beta,\vec{e}}^{s_1,x_1}$.

  For \ref{it:artificial-wcf-ii}, we want to check that the formula in
  \ref{it:artificial-wcf-i} reduces to a sum over decompositions with
  pieces in $\hat{S}_\alpha \subset S_\alpha$. If $\beta=0$, there is
  nothing to show, so assume $\beta\neq 0$.

  We first claim that every class $(\beta_i,\vec{e}_i)$ appearing in a
  nonzero term of the sum in \eqref{wc:eq:artificial-base-wcf} for
  which $\beta_i=0$ satisfies $\tau_{x_2}^1(\beta_i,\vec{e}_i) =
  (-\infty, *)$ where $*$ denotes an arbitrary quantity. Otherwise,
  let $I$ be the collection of $i\in \{1,\ldots,n\}$ satisfying
  $\tau_{x_2}^1(\beta_i,\vec{e}_i) = (\infty,*)$. We know $I \neq
  \emptyset$ by hypothesis, and $I \subsetneq \{1, \ldots, n\}$
  because $\tau_{x_2}^1(\beta,\vec e) < (\infty, *)$. By construction,
  $\tau_{x_2}^1(\beta_i,\vec{e}_i)> \tau_{x_2}^1(\beta_j, \vec{e}_j)$
  for any $i\in I$ and $j\notin I$, and
  $\tau_{x_1}^1(\beta_i,\vec{e}_i) = (\infty, *) >
  \tau_{x_1}^1(\beta,\vec{e}) = (\tau(\beta), *)$ for all $i \in I$.
  Applying Lemma~\ref{wc:lem:U-properties}\ref{it:U-vanishing}, the
  coefficient $\tilde{U}$ in the term vanishes, which contradicts our
  assumption.
	
  Now suppose there is a non-zero term in the right hand side of
  \eqref{wc:eq:artificial-base-wcf} involving some $(\beta_i,
  \vec{e}_i)\in S_{\alpha}\setminus\hat{S}_{\alpha}$. By the choice of
  $\lambda$, we have $\lambda(\beta_i)\neq 0$. Since $\lambda(\beta) =
  0$, there must in fact be $i_1,i_2$ with $\lambda(\beta_{i_1})<0
  <\lambda(\beta_{i_2})$. Hence $I \coloneqq \{i\in \{1,\ldots,n\} :
  \lambda(\beta_i)>0\}$ is neither empty nor $\{1,\ldots,n\}$. Since
  $\lambda$ dominates\footnote{This is where we use that
  $\abs{\lambda(\beta)}>\frac{1}{2}N(N+1)r(\alpha)$ for $\beta\in
  R_\alpha\setminus \hat{R}_\alpha$.} the other terms in the stability
  condition for non-zero $\beta_i$, we have $\tau_{x_1}^1(\beta_i,
  \vec{e}_i) > \tau_{x_1}^1(\beta, \vec{e}) $ for $i\in I$. Similarly,
  $\tau_{x_2}^1(\beta_i, \vec{e}_i) > \tau_{x_2}^1(\beta_j,
  \vec{e}_j)$ for $i\in I$ and $j\not\in I$, since if $\beta_j\neq 0$,
  then the term involving $\lambda$ dominates, and if $\beta_j = 0$,
  then the right-hand side is $(-\infty, *)$. So, again, by
  Lemma~\ref{wc:lem:U-properties}\ref{it:U-vanishing} the
  corresponding $\tilde{U}$ vanishes, contradicting our assumption.
\end{proof}

\subsubsection{}
\label{wc:sec:general-vanishing-result}

\begin{lemma}[Artificial vanishings] \label{wc:lem:vanishing-basic}
  Let $(\beta, \vec e) \in S_\alpha$ and $(s,x)\in [0,1]\times
  [-1,0]$. Suppose $\hat\sz_{\beta,\vec{e}}^{s,x} \neq 0$.
  \begin{enumerate}[label = (\roman*)]
  \item \label{it:vanishing-basic-i} If $\beta = 0$, then $\vec{e} =
    \vec{1}_{[a,b]}$ for some integers $1\leq a\leq b\leq N$, and the
    invariant $\hat\sz_{0,\vec{e}}^{s,x}$ is supported on the locus
    where $\rho_i$ are isomorphisms for $a\leq i<b$.
  \item \label{it:vanishing-basic-ii} If $\beta \neq 0$, then $e_1\leq
    e_2\leq \cdots \leq e_N$ and the invariant
    $\hat\sz_{\beta,\vec{e}}^{s,x}$ is supported on the locus where
    $\rho_i$ is injective for $i<N$. Moreover, if $e_{i+1}\leq e_{i} +
    1$ for all $i=1,\ldots,N-1$, then $e_N\leq \fr(\beta)$ and the
    invariant $\hat\sz_{(\beta,\vec{e})}^{s,x}$ is supported on the
    locus where all $\rho_i$ are injective.
  \item \label{it:vanishing-basic-iii} If $\beta\neq 0$ and $x\leq
    x_0(a)$ for $2\leq a \leq r$, then $e_{a-1} = e_a = \cdots = e_N$.
  \item \label{it:vanishing-basic-iv} If $\beta \neq 0$ and $x\leq
    x_0(1)$, then $\vec{e}=0$.
  \end{enumerate}
\end{lemma}

The proof will occupy the remainder of this subsection, and roughly
follows the proof of \cite[Prop. 10.11]{Joyce2021}. The claims in
\ref{it:vanishing-basic-ii} about the support locus are genuinely new.
They are required by the proof itself, for the vanishings in
\S\ref{wc:sec:gen-van-res-proof-ii-iv-set-up-cases}, due to our
definition of artificial invariants; see
footnote~\ref{footnote:artificial-invariants}.

\subsubsection{}
\label{wc:sec:gen-van-res-proof-i}

\begin{proof}[Proof of Lemma~\ref{wc:lem:vanishing-basic}.]
  We prove case~\ref{it:vanishing-basic-i} by explicit computation.
  First, note that $\tau^s_x(0,\vec e)\leq \tau^s_x(0,\vec e')$ if and
  only if $\tau^0_{-1}(0,\vec e)\leq \tau^0_{-1}(0,\vec e')$, since
  the definition \eqref{wc:eq:joyce-framed-stack-stability} of
  $\tau^s_x$ on such classes has entries which are independent of $s$
  and only depend on a linear equation in $x$. Hence, in
  Definition~\ref{def:universal-coefficients} for $S(\cdots;
  \tau_{-1}^0, \tau_x^s)$ and $U(\cdots; \tau_{-1}^0, \tau_x^s)$, we
  may replace all occurrences of $\tau_x^s$ with $\tau_{-1}^0$. Then
  \[ U\left((0,\vec e_1),\ldots,(0,\vec e_n); \tau_{-1}^{0}, \tau_{x}^{s}\right) = U\left((0,\vec e_1),\ldots,(0,\vec e_n); \tau_{-1}^{0}, \tau_{-1}^{0}\right) = \begin{cases} 1 & n=1,\\ 0 & n>1, \end{cases} \]
  where the second equality is
  Lemma~\ref{wc:lem:U-properties}\ref{it:U-identity}. Thus the
  definition \eqref{eq:artificial-invariant} of $\hat\sz^{s,x}_{0,\vec
    e}$ becomes
  \begin{equation} \label{eq:artificial-invariant-pure-framing}
    \hat\sz^{s,x}_{0,\vec e} = \tilde \sz^{0,-1}_{0,\vec e}
  \end{equation}
  by relating $U$ and $\tilde U$ via
  Lemma~\ref{lem:Utilde-definition}, and the desired result follows
  from Definition~\ref{wc:def:artificial-invs}\ref{wc:def:aux-inv-iv}
  and Lemma~\ref{wc:lem:general-semistable-loci}.

\subsubsection{}
\label{wc:sec:gen-van-res-proof-ii-iv-set-up-induction}

  We treat the remaining cases~\ref{it:vanishing-basic-ii},
  \ref{it:vanishing-basic-iii}, and \ref{it:vanishing-basic-iv}
  together. Note that if $\beta = 0$ or $\vec e = 0$ then there is
  nothing to show, so from now on we only consider $(\beta, \vec e)$
  with $\beta \neq 0$ and $\vec e \neq 0$.

  We proceed by contradiction. The strategy goes as follows. Assume
  that the set
  \[ P_{\text{bad}} \coloneqq \left\{\begin{array}{c} (s,x) \in [0,1] \times [-1,0] \\ (\beta, \vec e) \in S_\alpha \end{array} : \begin{array}{c} \hat\sz^{s,x}_{\beta,\vec e} \neq 0 \\ \text{not all of \ref{it:vanishing-basic-ii}--\ref{it:vanishing-basic-iv} hold at } x \end{array}\right\} \]
  is non-empty. Observe that at $(s,x)=(0,-1)$,
  Definition~\ref{wc:def:artificial-invs} for
  $\hat\sz^{s,x}_{\beta,\vec e}$ becomes $\hat\sz^{s,x}_{\beta,\vec e}
  = \tilde \sz^{0,-1}_{\beta,\vec e}$ by
  Lemma~\ref{wc:lem:U-properties}\ref{it:U-identity}, and we inherit
  all desired results from Lemma~\ref{wc:lem:general-semistable-loci}
  about $\tilde \sz^{0,-1}_{\beta,\vec e}$. Hence
  \begin{equation} \label{eq:vanishing-basic-bad-set-minimal-case}
    ((0,-1),(\beta,\vec e)) \notin P_{\text{bad}}.
  \end{equation}
  So we will take a certain kind of minimal choice of
  $((s,x),(\beta,\vec e)) \in P_{\text{bad}}$, and then use the
  wall-crossing formula \eqref{wc:eq:artificial-base-wcf} for
  artifical invariants to express $\hat\sz_{\beta,\vec e}^{s,x}$ in
  terms of artificial invariants of ``smaller'' classes and $(s,x)$ to
  obtain a contradiction to minimality.

  Throughout this proof, given a K-homology class $\phi$ of class
  $(\gamma, \vec f)$, we say that ``$\phi$ satisfies
  \ref{it:vanishing-basic-ii}--\ref{it:vanishing-basic-iv} at $x'$''
  to mean that
  \ref{it:vanishing-basic-ii}--\ref{it:vanishing-basic-iv} are true
  with $(\gamma, \vec f)$, $x'$, and $\phi$ in place of $(\beta, \vec
  e)$, $x$, and $\hat\sz_{\beta,\vec e}^{s,x}$ respectively. Note that
  if $\phi$ satisfies
  \ref{it:vanishing-basic-ii}--\ref{it:vanishing-basic-iv} at $x'$, it
  automatically satisfies
  \ref{it:vanishing-basic-ii}--\ref{it:vanishing-basic-iv} for any
  $x'' \ge x'$.

\subsubsection{}

  We make the strategy in
  \S\ref{wc:sec:gen-van-res-proof-ii-iv-set-up-induction} precise.
  Since $S_\alpha$ is a finite set, choose any $((s,x), (\beta, \vec
  e)) \in P_{\text{bad}}$ with lexicographically minimal $(\abs{\vec
    e}, r(\beta))$, where $\abs{\vec e} \coloneqq e_1 + \cdots + e_N$.
  Consider the line $\{(us, u(x+1)-1)\}_{u \in [0,1]}$ connecting
  $(-1,0)$ with $(s,x)$ and let $u_0$ be the infimum of all $u$ such
  that $((us, u(x+1)-1), (\beta, \vec e)) \in P_{\text{bad}}$. Pick
  $u_+ \ge u_0 \ge u_-$, letting $(s_\star, x_\star) \coloneqq
  (u_\star s, u_\star(x+1)-1)$ for $\star \in \{0, +, -\}$, such that
  \begin{align}
    ((s_-, x_-), (\beta, \vec e)) &\notin P_{\text{bad}}, \label{eq:vanishing-basic-minus-point} \\
    ((s_+, x_+), (\beta, \vec e)) &\in P_{\text{bad}}, \label{eq:vanishing-basic-plus-point}
  \end{align}
  and such that $u_\pm$ are very close to $u_0$. This is possible
  because there can be only finitely many points along $u \in [0,1]$
  where the value of $\hat\sz^{us,u(x+1)-1}_{\beta,\vec e}$ changes
  --- including genuine walls as well as discontinuities in the family
  of stability conditions $(\tau^{us}_{u(x+1)-1})_{u \in [0,1]}$ ---
  and moreover \eqref{eq:vanishing-basic-bad-set-minimal-case}
  guarantees that $u_-$ exists.
  
  Consider the non-vanishing terms on the right hand side of the
  artificial wall-crossing formula \eqref{wc:eq:artificial-base-wcf}
  from $(s_-,x_-)$ to $(s_+,x_+)$:
  \begin{equation}\label{wc:eq:contradiction-small-wc-formula}
    \hat\sz_{\beta,\vec{e}}^{s_+,x_+} = \sum_{\substack{n\geq 1, (\beta_i,\vec{e}_i)\in S_{\alpha}\\(\beta, \vec{e}) = (\beta_1, \vec{e}_1) + \cdots + (\beta_n, \vec{e}_n)}}\tilde{U}\left((\beta_1,\vec e_1),\ldots,(\beta_n,\vec e_n); \tau_{x_-}^{s_-}, \tau_{x_+}^{s_+}\right) \left[\left[\cdots\left[\hat\sz_{\beta_1,\vec e_1}^{s_-,x_-}, \hat\sz_{\beta_2, \vec e_2}^{s_-,x_-} \right], \cdots\right],\hat\sz_{\beta_n,\vec e_n}^{s_-,x_-}\right].
  \end{equation}
  If the $n=1$ term is non-vanishing, it automatically satisfies the
  claims \ref{it:vanishing-basic-ii}--\ref{it:vanishing-basic-iv} at
  $x_-$, and therefore at $x_+$, due to
  \eqref{eq:vanishing-basic-minus-point}. So there must exist a
  non-vanishing term with $n>1$, otherwise $\hat\sz_{\beta,\vec
    e}^{s_+,x_+} = \hat\sz_{\beta,\vec e}^{s_-,x_-}$ contradicts
  \eqref{eq:vanishing-basic-plus-point}. Every non-zero artificial
  invariant in such non-vanishing term with $n>1$ satisfies the claims
  \ref{it:vanishing-basic-ii}--\ref{it:vanishing-basic-iv} at $x_-$ by
  minimality of $(\abs{\vec e}, r(\beta))$. We will show that this
  implies every non-vanishing term with $n>1$ satisfies the claims
  \ref{it:vanishing-basic-ii}--\ref{it:vanishing-basic-iv} at $x_+$,
  and therefore so does the left hand side, again contradicting
  \eqref{eq:vanishing-basic-plus-point}.

\subsubsection{}
\label{wc:sec:gen-van-res-proof-ii-iv-possible-classes}

  Pick, for the sake of contradiction, a term
  \begin{equation} \label{eq:vanishing-basic-bad-term}
    0 \neq \tilde{U}\left((\beta_1,\vec e_1),\ldots,(\beta_n,\vec e_n); \tau_{x_-}^{s_-}, \tau_{x_+}^{s_+}\right) \left[\left[\cdots\left[\hat\sz_{\beta_1,\vec e_1}^{s_-,x_-}, \hat\sz_{\beta_2, \vec e_2}^{s_-,x_-} \right], \cdots\right],\hat\sz_{\beta_n,\vec e_n}^{s_-,x_-}\right],
  \end{equation}
  with $n > 1$, which does not satisfy one or more of the claims
  \ref{it:vanishing-basic-ii}--\ref{it:vanishing-basic-iv} at $x_+$.
  We first find some restrictions on what the classes $\{(\beta_i,
  \vec e_i)\}_{i=1}^n$ and the values $x_+ \ge x_0 \ge x_-$ can be.

  First, if $\beta_i \neq 0$ for all $i$, then the properties
  \ref{it:vanishing-basic-ii}--\ref{it:vanishing-basic-iv} for the
  classes $(\beta_i,\vec e_i)$ are preserved under addition of those
  classes, and, additionally, the Lie bracket
  (Theorem~\ref{thm:auxiliary-stack-vertex-algebra}) preserves the
  support properties in \ref{it:vanishing-basic-ii}. So $\beta_i = 0$
  for some $1 \le i \le n$.

  Second, assume there exists an $i$ with $\beta_i=0$ and $(\vec \mu +
  x_0 \vec 1) \cdot \vec e_i < 0$, so that the set
  \[ I \coloneqq \left\{i\in \{1,\ldots,n\} : \beta_i=0 \text{ and } \left(\vec \mu + x_0 \vec 1\right)\cdot \vec e_i \text{ minimal among }i\text{ with }\beta_i=0\right\} \]
  is non-empty. Since $\beta \neq 0$, we also know $I \neq \{1,
  \ldots, n\}$. By continuity, since $x_+$ and $x_-$ are very close to
  $x_0$, we have also $(\vec \mu + x_\pm \vec 1) \cdot \vec e_i<0$ and
  that these values for any $i\in I$ are strictly smaller than those
  for any $i \notin I$.
  Lemma~\ref{wc:lem:U-properties}\ref{it:U-vanishing} and the
  definition of $\tau_{x_\pm}^{s_\pm}$ then imply that the coefficient
  $\tilde U$ in \eqref{eq:vanishing-basic-bad-term} is zero,
  contradicting the non-vanishing of this term. Hence, any $i$ with
  $\beta_i=0$ must have $(\vec \mu + x_0 \vec 1)\cdot \vec e_i\geq 0$.
  Analogously, any $i$ with $\beta_i=0$ must also have $(\vec \mu +
  x_0 \vec 1)\cdot \vec e_i\leq 0$. So,
  \[ \beta_i=0 \implies \left(\vec \mu + x_0 \vec 1\right) \cdot \vec e_i= 0. \]

  Third, if $\beta_i = 0$, then case~\ref{it:vanishing-basic-i},
  already proved in \S\ref{wc:sec:gen-van-res-proof-i}, implies $\vec
  e_i = \vec 1_{[a,b]}$ for some $1\leq a\leq b\leq N$. By the
  genericity of $\vec \mu$, at most one such pair $a,b$ satisfies
  \begin{equation} \label{wc:eq:vertical-wall-ei}
    \left(\vec \mu + x_0 \vec 1\right)\cdot \vec 1_{[a,b]} = 0,
  \end{equation}
  so $a$ and $b$ are really independent of $i$. More precisely, if another such pair $(a',b')$ exists, then we get the equation
  \begin{equation}\label{wc:eq:gen-artifical-generalities}
    \vec \mu \cdot \left(\frac{\vec 1_{[a,b]}}{-x_0(b-a+1)}-\frac{\vec 1_{[a',b']}}{-x_0(b'-a'+1)}\right)=0,
  \end{equation}
  which yields $\frac{\vec 1_{[a,b]}}{-x_0(b-a+1)}=\frac{\vec 1_{[a',b']}}{-x_0(b'-a'+1)}$ by genericity of $\vec \mu$. This in turn implies $a=a'$ and $b=b'$. Fix this $a$ and $b$
  for the remainder of this proof. Then \eqref{wc:eq:vertical-wall-ei}
  becomes a constraint on $x_0$; in particular, $x_0 \le x_0(a)$, and
  hence $x_- \le x_0(a)$ as well. Thus, depending on $a$,
  \ref{it:vanishing-basic-iii} or \ref{it:vanishing-basic-iv} apply
  for $x_-$ to any $(\beta_i,\vec e_i)$ with $\beta_i\neq 0$ by our
  minimality assumption on $(\beta,\vec e)$. In particular, for any $1
  \le i \le n$ with $\beta_i\neq 0$,
  \begin{equation} \label{wc:eq:gen-van-res-nonzero-beta}
    \begin{aligned}
      a > 1 &\implies e_{i,a-1}=\cdots = e_{i,N}, \\
      a = 1 &\implies \vec e_i= \vec 0.
    \end{aligned}
  \end{equation}

  Finally, note that \eqref{wc:eq:vertical-wall-ei} implies
  $-1<x_0<0$, and we know $x_+ \ge x_0$ by construction. Suppose $x_+
  = x_0$. Take
  \[ I \coloneqq \left\{i : \beta_i=0\right\}. \]
  Then $I \neq \{1, \ldots, n\}$ since $\beta \neq 0$, and we showed
  $I \neq \emptyset$ earlier. Also $\vec e_i = \vec 1_{[a,b]}$ for $i
  \in I$. So \eqref{wc:eq:vertical-wall-ei} implies that
  $\tau^{s_\pm}_{x_\pm}(\beta_i,\vec e_i) = (-\infty, *)$ for any
  $i\in I$. Again, Lemma~\ref{wc:lem:U-properties}\ref{it:U-vanishing}
  applies, a contradiction. So $x_+>x_0$.

\subsubsection{}
\label{wc:sec:gen-van-res-proof-ii-iv-set-up-cases}

  With these restrictions in mind, consider the following different
  cases:
  \begin{enumerate}[label = (\alph*)]
  \item \label{wc:it:gen-van-res-proof-1} $\beta_1=\beta_2=0$;
  \item \label{wc:it:gen-van-res-proof-2} $\beta_1=0$, $\beta_2\neq
    0$, and $b<N$;
  \item \label{wc:it:gen-van-res-proof-3} $\beta_i\neq 0$ for $i\leq
    m$, but $\beta_{m+1}=0$, for some $m<n$, and $b<N$;
  \item \label{wc:it:gen-van-res-proof-4} $b=N$.
  \end{enumerate}
  Recall from \S\ref{wc:sec:gen-van-res-proof-ii-iv-possible-classes}
  that there exists $1 \le i \le n$ with $\beta_i = 0$, so this list
  is exhaustive. We will derive a contradiction in each case. The
  first three cases are relatively straightforward: we will show that
  some Lie bracket in \eqref{eq:vanishing-basic-bad-term} vanishes, a
  contradiction.

  In case~\ref{wc:it:gen-van-res-proof-1}, we have shown that $\vec
  e_1=\vec e_2=\vec 1_{[a,b]}$, and in particular $(\beta_1,\vec e_1)
  = (\beta_2,\vec e_2)$, so the Lie bracket $[\hat\sz_{\beta_1,\vec
      e_1}^{s_-,x_-}, \hat\sz_{\beta_2,\vec e_2}^{s_-,x_-}]$ vanishes,
  a contradiction.

  In case~\ref{wc:it:gen-van-res-proof-2}, by our minimality
  assumption and \ref{it:vanishing-basic-ii}, $\vec e_1 = \vec
  1_{[a,b]}$ and $\hat\sz_{\beta_1,\vec 1_{[a,b]}}^{s_-,x_-}$ is
  supported on the locus where $\rho_j$ are isomorphisms for $a \le j
  < b$, and $\hat\sz_{\beta_2,\vec e_2}^{s_-,x_-}$ is supported on the
  locus where $\rho_j$ is injective for all $j < N$. Moreover,
  $e_{2,a-1} = \cdots = e_{2,N}$ by
  \eqref{wc:eq:gen-van-res-nonzero-beta}. The Lie bracket of two such
  classes vanishes by direct computation, because the K-theory class
  \eqref{eq:framed-stack-forgetful-map-cotangent} vanishes on the
  product of these two support loci. This is a contradiction.

  In case~\ref{wc:it:gen-van-res-proof-3}, we apply the argument of
  case~\ref{wc:it:gen-van-res-proof-2} to the Lie bracket of the first
  $m$ terms and the $(m+1)$-th term. All the necessary properties are
  preserved by taking Lie brackets of the first $m$ terms
  $\hat\sz_{\beta_i,\vec e_i}^{s_-,x_-}$ with $\beta_i \neq 0$, so the
  Lie bracket vanishes, a contradiction.

\subsubsection{}

  In case~\ref{wc:it:gen-van-res-proof-4}, we will show that the term
  \eqref{eq:vanishing-basic-bad-term} in question actually satisfies
  \ref{it:vanishing-basic-ii}--\ref{it:vanishing-basic-iv} at $x_+$, a
  contradiction.

  To start, recall from
  \S\ref{wc:sec:gen-van-res-proof-ii-iv-possible-classes} that $x_+ >
  x_0$, and then observe that $x_0 = x_0(a)$ by comparing
  \eqref{wc:eq:vertical-wall-ei} with \eqref{wc:eq:x0a} when $b = N$.
  Hence if $\beta_i = 0$ then $\vec e_i = \vec 1_{[a,N]}$, otherwise
  if $\beta_i \neq 0$ then \eqref{wc:eq:gen-van-res-nonzero-beta}
  holds. Both cases satisfy \ref{it:vanishing-basic-iii} and
  \ref{it:vanishing-basic-iv} at $x_+$, therefore $(\beta, \vec e)$
  satisfies \ref{it:vanishing-basic-iii} and
  \ref{it:vanishing-basic-iv} at $x_+$ as well. Similarly, both cases
  satisfy the first part of \ref{it:vanishing-basic-ii}, either by
  case~\ref{it:vanishing-basic-i} (because $b = N$) or
  case~\ref{it:vanishing-basic-ii}, hence so does the entire term
  \eqref{eq:vanishing-basic-bad-term}.

  It remains to prove the second part of \ref{it:vanishing-basic-ii}.
  So, assume additionally that $e_{j+1} \le e_j+1$ for all $j<N$. We
  want to show that $e_N \le \fr(\beta)$ and that the term
  \eqref{eq:vanishing-basic-bad-term} is supported on the locus where
  all $\rho_j$ are injective. By the analysis of
  \S\ref{wc:sec:gen-van-res-proof-ii-iv-possible-classes}, there can
  be at most one index $l$ with $\beta_l = 0$, since otherwise $\vec
  e$ would increase by at least two at the $a$-th step, violating this
  assumption. By \ref{it:vanishing-basic-i} and $b=N$, we know $\vec
  e_l = \vec 1_{[a,N]}$. From
  \eqref{eq:artificial-invariant-pure-framing} and
  Definition~\ref{wc:def:artificial-invs}, $\hat\sz_{0,\vec
    1_{[a,N]}}^{s_-,x_-} = I_*\partial_{[a,N]}$ and is therefore
  supported on the locus where the $\rho_j$ are isomorphisms for $a
  \le j < N$. The Lie bracket preserves the property of being
  supported on the locus where all $\rho_j$ are injective, and when
  $\beta_i \neq 0$ it also preserves the property $e_N \le
  \fr(\beta)$, so the only interesting Lie bracket in
  \eqref{eq:vanishing-basic-bad-term} is $[\phi, I_*\partial_{[a,N]}]$
  where
  \[ \phi \coloneqq \left[\left[\cdots\left[\hat\sz_{\beta_1,\vec e_1}^{s_-,x_-}, \hat\sz_{\beta_2,\vec e_2}^{s_-,x_-} \right], \cdots\right],\hat\sz_{\beta_{l-1}, \vec e_{l-1}}^{s_-,x_-}\right] \in K_\circ^{\tilde\sT}(\tilde\fM^{\vec Q(\Fr)}_{\beta_{[l-1]},\vec e_{[l-1]}})^\pl_{\loc,\bQ} \]
  is the output of the first $l-1$ Lie brackets and $(\beta_{[l-1]},
  \vec e_{[l-1]}) \coloneqq \sum_{i=1}^{l-1} (\beta_i, \vec e_i)$. By
  Lemma~\ref{lem:flag-forget-projective}\ref{lem:flag-forgetii} below,
  $[\phi, I_*\partial_{[a,N]}]$ is supported on the locus where all
  $\rho_j$ are injective. In particular, if $e_{[l-1],N} =
  \fr(\beta_{[l-1]})$ then this Lie bracket vanishes, a contradiction.
  So $e_{[l-1],N} < \fr(\beta_{[l-1]})$ and therefore $e_{[l],N} \le
  \fr(\beta_{[l]})$. Hence the term
  \eqref{eq:vanishing-basic-bad-term} satisfies the second part of
  \ref{it:vanishing-basic-ii} as well, a contradiction.

  We have ruled out all cases
  \ref{wc:it:gen-van-res-proof-1}--\ref{wc:it:gen-van-res-proof-4}. So
  the term \eqref{eq:vanishing-basic-bad-term} cannot exist. This
  concludes the proof of Lemma~\ref{wc:lem:vanishing-basic}.
\end{proof}

\subsection{Semistable and pair invariants}
\label{wc:sec:sst-inv-and-pair-inv}

\subsubsection{}
\label{wc:sec:sst-pair}

In this subsection, we begin the proof of Theorem~\ref{thm:wcf} with
the wiggly arrows in Figure~\ref{fig:dominant-wc-strategy-1}. We list
the main results.
\begin{itemize}
\item (Lemma~\ref{lem:sst-pair}) We rewrite the characterization
  (Theorem~\ref{thm:sst-invariants}\ref{item:vss-pairs-relation}) of
  the semistable invariants $\sz_\beta(\tau)$ and
  $\sz_\beta(\mathring\tau)$ using our auxiliary invariants
  $\tilde\sz^{s,x}$.
\item (Lemma~\ref{wc:lem:vanishing-sst-pair}) The artificial
  invariants $\hat\sz^{s,x}$ satisfy the {\it same} formula.
\end{itemize}

\subsubsection{}

\begin{lemma} \label{lem:sst-pair}
  \begin{enumerate}[label = (\roman*)]
  \item \label{it:sst-pair-a} Let $\beta \in R_\alpha$ and $\vec e =
    \vec{1}_{[a,N]}$ for some $1 \le a \le N$. Then there are
    isomorphisms
    \[ \tilde\fM^{\vec Q(\Fr),\sst}_{\beta,\vec{1}_{[a,N]}}\left(\tau^0_x\right) \cong \begin{cases} \emptyset & x \le x_0(a), \\ \tilde\fM^{Q(\Fr),\sst}_{\beta,1}(\tau^Q) & x > x_0(a), \end{cases} \]
    preserving symmetrized virtual structure sheaves. In particular,
    \begin{equation}\label{wc:eq:sst-pair-s0-inv-identity}
      \tilde\sz_{\beta,\vec{1}_{[a,N]}}^{0,x} = \begin{cases} 0 & x \le x_0(a), \\ \iota^{\vec Q}_{[a,N]*}I_*\tilde\sZ_{\beta,1}^{\Fr}(\tau^Q) & x > x_0(a) \end{cases}
    \end{equation}
    for the invariant \eqref{eq:pairs-stack-enumerative-invariant} of
    $\tilde\fM^{Q(\Fr),\sst}_{\beta,1}(\tau^Q)$. Moreover, for
    $x_0(a)<x<0$,
    \begin{equation}\label{wc:eq:sst-pair-s0-wc}
      \tilde\sz_{\beta,\vec{1}_{[a,N]}}^{0,x} = \sum_{\substack{n\geq 1, \beta_i\in R_{\alpha}\\\beta = \beta_1 + \cdots + \beta_n}} \frac{1}{n!} \left[\tilde\sz^{0,-1}_{\beta_n,\vec 0},\left[\cdots,\left[\tilde\sz_{\beta_2,\vec 0}^{0,-1}, \left[\tilde\sz_{\beta_1,\vec 0}^{0,-1},\tilde\sz^{0,-1}_{0,\vec{1}_{[a,N]}}\right]\right]\cdots\right]\right].
    \end{equation}

  \item \label{it:sst-pair-b} The results of \ref{it:sst-pair-a} also
    hold with $\mathring R_\alpha$, $\mathring\tau$, and $s=1$ in
    place of $R_\alpha$, $\tau$, and $s=0$ respectively. Namely,
    \eqref{wc:eq:sst-pair-s0-inv-identity} becomes
    \begin{equation}\label{wc:eq:sst-pair-s1-inv-identity}
      \tilde\sz_{\beta,\vec{1}_{[a,N]}}^{1,x} = \begin{cases} 0 & x \le x_0(a), \\ \iota^{\vec Q}_{[a,N]*}I_*\tilde\sZ_{\beta,1}^{\Fr}(\mathring\tau^Q) & x > x_0(a) \end{cases}
    \end{equation}
    and, for $x_0(a) < x < 0$, \eqref{wc:eq:sst-pair-s0-wc} becomes
    \begin{equation}\label{wc:eq:sst-pair-s1-wc}
      \begin{aligned}
        \tilde\sz_{\beta,\vec{1}_{[a,N]}}^{1,x} &= \sum_{\substack{n\geq 1, \beta_i\in \mathring R_{\alpha}\\\beta = \beta_1 + \cdots + \beta_n}} \frac{1}{n!} \left[\tilde\sz^{1,-1}_{\beta_n,\vec 0},\left[\cdots,\left[\tilde\sz_{\beta_2,\vec 0}^{1,-1}, \left[\tilde\sz_{\beta_1,\vec 0}^{1,-1},\tilde\sz^{1,-1}_{0,\vec{1}_{[a,N]}}\right]\right]\cdots\right]\right]\\
        &= \sum_{\substack{n\geq 1, \beta_i\in \hat{R}_{\alpha}\\\beta = \beta_1 + \cdots + \beta_n}} \frac{1}{n!} \left[\tilde\sz^{1,-1}_{\beta_n,\vec 0},\left[\cdots,\left[\tilde\sz_{\beta_2,\vec 0}^{1,-1}, \left[\tilde\sz_{\beta_1,\vec 0}^{1,-1},\tilde\sz^{1,-1}_{0,\vec{1}_{[a,N]}}\right]\right]\cdots\right]\right].
      \end{aligned}
    \end{equation}
  \end{enumerate}
\end{lemma}

\begin{proof}
  First, \eqref{wc:eq:sst-pair-s0-wc} and the first equality in
  \eqref{wc:eq:sst-pair-s1-wc} follow from
  \eqref{wc:eq:sst-pair-s0-inv-identity} and
  \eqref{wc:eq:sst-pair-s1-inv-identity} respectively, directly by
  applying the characterization
  Theorem~\ref{thm:sst-invariants}\ref{item:vss-pairs-relation} of
  semistable invariants and using that $\iota_{[a,N]*}^{\vec Q}$ is a
  Lie algebra homomorphism.

  For the second line of \eqref{wc:eq:sst-pair-s1-wc}, consider a term
  in the sum corresponding to $\beta = \beta_1 + \cdots + \beta_n$
  such that $\beta_i\in \hat{R}_\alpha\setminus \mathring R_\alpha$
  for some $1 \le i \le n$. Since $\beta \in \mathring R_\alpha$, we
  know $\fM_{\alpha-\beta}^{\sst}(\mathring\tau) \neq \emptyset$.
  Suppose for the sake of contradiction that
  $\fM_{\beta_j}^\sst(\mathring\tau) \neq \emptyset$ for all $1 \le j
  \le n$ as well. Since $\beta_j \in \hat R_\alpha$, we know
  $\mathring\tau(\beta_j) = \mathring\tau(\alpha) =
  \mathring\tau(\alpha - \beta)$ for all $j$. So, picking
  $\mathring\tau$-semistable objects $E_j$ and $E'$ of classes
  $\beta_j$ and $\beta$ respectively, $E' \oplus \bigoplus_{j \neq i}
  E_i$ is $\mathring\tau$-semistable of class $\alpha - \beta_i$, i.e.
  $\fM_{\alpha-\beta_i}^{\sst}(\mathring\tau) \neq \emptyset$. Then
  $\beta_i \notin \mathring R_\alpha$ implies
  $\fM_{\beta_i}^{\sst}(\mathring\tau) = \emptyset$, a contradiction.
  Hence $\fM_{\beta_j}^{\sst}(\mathring\tau) = \emptyset$ for some
  $j$, making the entire term vanish. This shows that the two lines of
  \eqref{wc:eq:sst-pair-s1-wc} are in fact equal.

  So it remains to prove \eqref{wc:eq:sst-pair-s0-inv-identity} and
  \eqref{wc:eq:sst-pair-s1-inv-identity}. Note that any object of
  class $(\beta,\vec{1}_{[a,N]})$ has a unique quotient object of
  class $(0,\vec{1}_{[a,N]})$. For $x\leq x_0(a)$ this quotient has
  slope $(-\infty, *)$, and is therefore destabilizing, which shows
  the first claims of both \ref{it:sst-pair-a} and
  \ref{it:sst-pair-b}. We now prove the second claims.

\subsubsection{}\label{wc:sec:s0-sst-pair-proof}

  We prove \eqref{wc:eq:sst-pair-s0-inv-identity} in
  \ref{it:sst-pair-a} for $x>x_0(a)$. Let $(E,\vec V,\vec \rho)$ be an
  object of class $(\beta,\vec 1_{[a,N]})$.

  Suppose it is $\tau^0_x$-semistable. By
  Lemma~\ref{wc:lem:general-semistable-loci}\ref{wc:lem:general-semistable-loci-ii},
  all framing maps $\rho_j$ are injective and $E$ is
  $\tau$-semistable. For our choice of dimension vector $\vec
  1_{[a,N]}$, this implies that $(E,\vec V,\vec \rho)$ is the image of
  a pair $(E,V,\rho)$ under $\iota_{[a,N]}^{\vec Q}$. We check that a
  $\tau^Q$-destabilizing sub-object $(E',V',\rho')$ of $(E,V,\rho)$
  induces a $\tau^0_x$-destabilizing sub-object $\iota_{[a,N]}^{\vec
    Q}(E', V', \rho')$ of $(E,\vec V, \vec\rho)$. Let $\beta'$ be the
  class of $E'$. Since $E$ is $\tau$-semistable, $\tau(\beta') =
  \tau(\beta - \beta')$ and $V' = V$ and $\rho'= \rho$. Then $r(\beta)
  = r(\beta-\beta') + r(\beta') > r(\beta')$. Hence
  \[ \tau^0_x(\beta, \vec{1}_{[a,N]})  = \left(\tau(\beta), \frac{(\vec{\mu}+ x\vec{1})\cdot \vec{1}_{[a,N]}}{r(\beta)}\right) < \left(\tau(\beta), \frac{(\vec{\mu}+ x\vec{1})\cdot \vec{1}_{[a,N]}}{r(\beta')}\right) = \tau^0_x(\beta',\vec{1}_{[a,N]}), \]
  which shows that $(E', \vec V, \vec\rho)$ is a
  $\tau^0_x$-destabilizing sub-object, as desired.

  Conversely, suppose $(E, \vec V, \vec \rho) = \iota_{[a,N]}(E, V,
  \rho)$ for some $\tau^Q$-semistable pair $(E, V, \rho)$. Then $V =
  V_N$ and $\rho = \rho_N$. Assume, for the sake of contradiction,
  that $(E,\vec V,\vec \rho)$ has a $\tau^0_x$-destabilizing
  sub-object $(E',\vec V',\vec \rho')$. If this sub-object has class
  $(\beta', \vec 0)$ for some $\beta' < \beta$, then, since $x >
  x_0(a)$, and hence $(\vec\mu + x\vec 1)\cdot \vec 1_{[a,N]} > 0$,
  the underlying object $E'$ must $\tau$-destabilize $E$,
  contradicting $\tau^Q$-semistability of $(E, V, \rho)$ and
  Lemma~\ref{lem:pairs-stack-sst=st}. So the sub-object must have
  class $(\beta', \vec 1_{[b,N]})$ for some $\beta' \le \beta$ and $a
  \le b \le N$. If $\beta' = \beta$, then
  \[ \vec \mu\cdot \vec{1}_{[a,N]} + x(N-a+1) \leq \vec \mu \cdot \vec{1}_{[b,N]} +x(N-b+1),\]
  from the second entry of the stability condition \eqref{wc:eq:joyce-framed-stack-stability}, which implies
  \[x \leq - \frac{1}{b-a} \sum_{i = a}^{b-1} \mu_i< x_0(a),\]
  contradicting our assumption $x>x_0(a)$. Hence $\beta' < \beta$. But
  then $(E', \vec V', \vec\rho')$ is a $\tau^Q$-destabilizing
  sub-object of $(E, V, \rho)$, a contradiction. We have ruled out all
  cases, so $(E, \vec V, \vec \rho)$ must be $\tau^0_x$-semistable.

\subsubsection{}\label{wc:sec:s1-sst-pair-proof}

  We prove \eqref{wc:eq:sst-pair-s1-inv-identity} in
  \ref{it:sst-pair-b} for $x>x_0(a)$ in a similar fashion.

  Suppose $(E,\vec V,\vec \rho)$ is $\tau_x^1$-semistable. By the same
  reasoning as in \S\ref{wc:sec:s0-sst-pair-proof}, it is the image of
  a pair $(E,V,\rho)$ under $\iota_{[a,N]}^{\vec Q}$, and $E$ is
  $\mathring\tau$-semistable. Again, we check that a
  $\tau^Q$-destabilizing sub-object $(E',V,\rho)$ of $(E,V,\rho)$
  induces a $\tau^1_x$-destabilizing sub-object $(E', \vec V, \vec
  \rho)$ of $(E,\vec V, \vec\rho)$. The same reasoning as in
  \S\ref{wc:sec:s0-sst-pair-proof} shows that $\mathring\tau(\beta') =
  \mathring\tau(\beta - \beta')$ and therefore $\mathring\tau(\beta')
  = \mathring\tau(\alpha-\beta')$. Thus $\beta' \in \hat R_\alpha$ and
  so $\lambda(\beta') = 0$. We also have $\lambda(\beta) = 0$ by
  assumption. Since $\hat R_\alpha \subset R_\alpha$, we have
  $\tau(\beta') = \tau(\beta-\beta')$ as well, thus $r(\beta) >
  r(\beta')$. Putting all this together, we get that $(E',\vec V, \vec
  \rho)$ is a $\tau^1_x$-destabilizing sub-object.

  Conversely, suppose $(E, \vec V, \vec \rho) = \iota_{[a,N]}(E, V,
  \rho)$ for some $\mathring \tau^Q$-semistable pair $(E, V, \rho)$.
  Since $E$ is $\mathring\tau$-semistable, any $\tau^1_x$-destabilizing
  sub-object $(E', \vec V', \vec \rho')$ must have $\mathring\tau(\beta') =
  \mathring\tau(\beta-\beta')$ where $\beta'$ is the class of $E'$. As before,
  this implies $\lambda(\beta') = 0$, and $\lambda(\beta) = 0$ by
  assumption. The rest of the argument is the same as in
  \S\ref{wc:sec:s0-sst-pair-proof}; we conclude that $(E, \vec V, \vec
  \rho)$ is $\tau^1_x$-semistable.
\end{proof}

\subsubsection{}

\begin{lemma} \label{wc:lem:vanishing-sst-pair}
  Let $\beta \neq 0$, and $1 \le a \le N$ and $x$ be such that $x_0(a)
  < x \le x_0(a+1)$.
  \begin{enumerate}[label = (\roman*)]
  \item \label{it:vanishing-sst-pair-a} If $(\beta,\vec 1_{[a,N]}) \in
    S_\alpha$, then
    \begin{equation}\label{wc:eq:van-sst-pair-s0-wc}
      \hat\sz_{\beta,\vec{1}_{[a,N]}}^{0,x} = \sum_{\substack{n\geq 1, \beta_i\in R_{\alpha}\\\beta = \beta_1 + \cdots + \beta_n}} \frac{1}{n!} \left[\hat\sz^{0,-1}_{\beta_n,\vec 0},\left[\cdots,\left[\hat\sz_{\beta_2,\vec 0}^{0,-1}, \left[\hat\sz_{\beta_1,\vec 0}^{0,-1},\hat\sz^{0,-1}_{0,\vec{1}_{[a,N]}}\right]\right]\cdots\right]\right].
    \end{equation}
    
  \item \label{it:vanishing-sst-pair-b} If $(\beta,\vec 1_{[a,N]})\in
    \hat S_\alpha$, then
    \begin{equation}\label{wc:eq:van-sst-pair-s1-wc}
      \hat\sz_{\beta,\vec{1}_{[a,N]}}^{1,x} = \sum_{\substack{n\geq 1, \beta_i\in \hat{R}_{\alpha}\\\beta = \beta_1 + \cdots + \beta_n}} \frac{1}{n!} \left[\hat\sz^{1,-1}_{\beta_n,\vec 0},\left[\cdots,\left[\hat\sz_{\beta_2,\vec 0}^{1,-1}, \left[\hat\sz_{\beta_1,\vec 0}^{1,-1},\hat\sz^{1,-1}_{0,\vec{1}_{[a,N]}}\right]\right]\cdots\right]\right].
    \end{equation}
  \end{enumerate}
\end{lemma}

\begin{proof}
  We follow the proof of \cite[Prop. 10.13]{Joyce2021} directly. The
  strategy is to reduce the artificial wall-crossing formulas
  (Lemma~\ref{wc:lem:artificial-base-wcf}) to the claimed formula
  using the vanishings in Lemma~\ref{wc:lem:vanishing-basic}.

  We prove \ref{it:vanishing-sst-pair-a}. Consider the artificial
  wall-crossing formula \eqref{wc:eq:artificial-base-wcf} for
  parameters $(s_1,x_1)=(0,x)$, $(s_2,x_2) = (0,-1)$ and with class
  $(\beta,\vec 1_{[a,N]})$. Take a non-zero term on the right hand
  side corresponding to the decomposition
  $(\beta,\vec{1}_{[a,N]})=(\gamma_1,\vec f_1)+\cdots +(\gamma_m,\vec
  f_m)$. A fortiori, all $\hat\sz^{0,-1}_{\gamma_i,\vec f_i}$ are
  non-zero. Then, by Lemma~\ref{wc:lem:vanishing-basic}, each $i$ must
  satisfy either $\gamma_i=0$ or $\vec f_i=0$. In particular, we must
  have $m \ge 2$.
  
  Now assume for contradiction that none of the $(\gamma_i,\vec f_i)$
  are equal to $(0,\vec 1_{[a,N]})$. By
  Lemma~\ref{wc:lem:vanishing-basic}\ref{it:vanishing-basic-i}, all
  $i$ with $\gamma_i=0$ must have $\vec f_i=\vec 1_{[a_i,b_i]}$ for
  some $a\leq a_i\leq b_i\leq N$. By our assumption, at least two such
  indices exist. In particular, this implies $a<N$ and that there must
  be a unique index $i_0$ for which $\vec f_{i_0} = \vec 1_{[b,N]}$
  for some $b>a$. Take $I \coloneqq \{i_0\}$. We have that
  $\tau_{-1}^0(\gamma_{i_0}, \vec f_{i_0}) < \tau_{-1}^0(\gamma_j,\vec
  f_j)$ for any $j\neq i_0$, and, using that $x \leq x_0(a+1)$, also
  that $\tau_x^0(\gamma_{i_0}, \vec f_{i_0}) = (-\infty, *) <
  \tau_x^0(\beta,\vec e)$. By
  Lemma~\ref{wc:lem:U-properties}\ref{it:U-vanishing}, the
  corresponding coefficient $\tilde{U}$ vanishes, contradicting our
  initial assumption that we chose a non-zero term in the sum.

  Now we have the correct decompositions in the wall-crossing formula
  and it only remains to compute that
  \[ \tilde U((\gamma_1,0), \ldots, (\gamma_k,0), (0, \vec 1_{[a,N]}), (\gamma_{k+1},0), \ldots, (\gamma_n, 0); \tau^0_{-1}, \tau^0_x) = \begin{cases} (-1)^n/n! & k=0, \\ 0 & k > 0. \end{cases} \]
  This follows from either \cite[(10.33)]{Joyce2021} or
  Proposition~\ref{prop:Utilde}, using the facts that
  $\tau^0_x(\gamma_i, 0) = \tau^0_x(\gamma_j, 0)$ for all $i, j$ and
  $x \in [-1,0]$ and
  \[ \tau^0_{-1}(\gamma_i,0) > \tau^0_{-1}(0, \vec 1_{[a,N]}), \quad \tau^0_x(\gamma_i, 0) < \tau^0_x(0, \vec 1_{[a,N]}). \]
  Note that the last inequality requires $x > x_0(a)$. We absorb the
  $(-1)^n$ by re-ordering the Lie brackets, by skew symmetry.

  Part~\ref{it:vanishing-sst-pair-b} follows analogously using
  \eqref{wc:eq:artificial-base-wcf-s1} instead of
  \eqref{wc:eq:artificial-base-wcf}.
\end{proof}

\subsection{Pair and flag invariants}
\label{wc:sec:pair-inv-and-flag-inv}

\subsubsection{}

In this subsection, we follow the vertical arrows in
Figure~\ref{fig:dominant-wc-strategy-2}. We find ``vertical''
wall-crossing formulas for the invariants
$\tilde\sz_{\beta,\vec{e}}^{s,x}$ and $\hat\sz_{\beta,\vec{e}}^{s,x}$
where $s\in \{0,1\}$ and $x$ decreases from $0$ down to $-1$.
Typically, in these vertical wall-crossing steps, the framing
dimension vectors $\vec e$ get decomposed into smaller parts and the
class $\beta$ remains unchanged. Iterating this will yield classes of
the form $(\beta,\vec 1_{[a,N]})$ for various $1 \le a \le N$, which
we studied previously in \S\ref{wc:sec:sst-inv-and-pair-inv}.

Throughout, we assume $(s,x)\in [0,1]\times [-1,0]$, and, for the
class $(\beta, \vec e)$ of interest, we let $1 \le a \le N$ denote the
minimal index such that $e_a = e_N$. We list the main results.
\begin{itemize}
\item (Lemma~\ref{wc:lemma:pair-to-flag-moduli-setup}) For $s = 0, 1$
  and varying $x \in [-1,0]$, the only wall for the semistable loci
  $\tilde\fM^{\vec Q(\Fr),\sst}_{\beta,\vec e}(\tau_x^s)$ is at
  $x=x_0(a)$. There is a projective bundle relating classes $(\beta,
  \vec e)$ and $(\beta, \vec e - \vec 1_{[a,N]})$.
\item (Lemma~\ref{wc:lem:pair-flag}) For $s = 0,1$, there is a
  ``vertical'' wall-crossing formula for the auxiliary invariants
  $\tilde\sz^{s,x}$ relating $x > x_0(a)$ to $x = x_0(a)$ and the
  class $(\beta, \vec e)$ to the class $(\beta, \vec e-\vec
  1_{[a,N]})$. This uses the aforementioned projective bundle.
\item (Lemma~\ref{wc:lem:pair-flag-vanishing}) For $s = 0,1$, the
  artificial invariants $\hat\sz^{s,x}$ satisfy the {\it same}
  ``vertical'' wall-crossing formula. 
\end{itemize}

\subsubsection{}

\begin{definition} \label{wc:def:flag-stripping}
  Let $(\beta, \vec e) \in S_\alpha$ be a flag and assume that $e_N
  \ge 1$ so that $\vec e - \vec 1_{[a,N]} \ge \vec 0$. Define the
  morphism of moduli stacks
  \begin{align*}
    \pi\colon \tilde\fM^{\vec Q(\Fr)}_{\beta,\vec e} &\to \tilde\fM^{\vec Q(\Fr)}_{\beta,\vec e - \vec 1_{[a,N]}} \\
    (E, \vec V, \vec \rho) &\mapsto (E, \vec V', \vec \rho')
  \end{align*}
  where $\vec V'$ and $\vec\rho'$ are defined by
  \[ V_i' \coloneqq \begin{cases} V_i & 1\leq i\leq a-1, \\ V_{a-1} & a \le i \le N, \end{cases} \qquad \rho_i' \coloneqq \begin{cases} \rho_i & 1 \le i < a-1, \\ \id_{V_{a-1}} & a-1 \le i < N, \\ \rho_{N}\circ\cdots\circ \rho_{a-1} & i = N. \end{cases} \]
  Let $\pi^\pl$ denote the induced morphism of rigidified stacks.
  Moreover, for any class $(\gamma,\vec f)\in S_{\alpha}$, define the
  open locus
  \[ \tilde\fM^{\vec Q(\Fr),\inj}_{\gamma, \vec f} \coloneqq \{\rho_i \text{ injective } \forall 1 \le i \le N\} \subset \tilde\fM^{\vec Q(\Fr)}_{\gamma,\vec f} \]
  For $1 \le a \le N$, define also the open locus
  \[ \tilde\fM^{\vec Q(\Fr),\inj(<a)}_{\gamma,\vec{f}} \coloneqq \left\{\begin{array}{c} \rho_i \text{ injective } \forall 1 \le i < N\\ \rho_{N}\circ \cdots \circ \rho_{a-1} \text{ is injective}\end{array}\right\} \subset \tilde\fM^{\vec Q(\Fr)}_{\gamma,\vec f}. \]
  Clearly $j_a\colon \tilde\fM^{\vec Q(\Fr),\inj}_{\gamma, \vec f}
  \hookrightarrow \tilde\fM^{\vec Q(\Fr),\inj(<a)}_{\gamma, \vec f}$ is
  an open immersion for any $a$.
\end{definition}

\subsubsection{}

\begin{lemma} \label{lem:flag-forget-projective}
  Let $(\beta, \vec e) \in S_\alpha$ be a flag and assume that $e_N
  \ge 1$.
  \begin{enumerate}[label = (\roman*)]
  \item \label{lem:flag-forgeti} The restriction of $\pi$ to
    $\tilde\fM^{\vec Q(\Fr),\inj}_{\beta,\vec e}$ is a morphism
    \begin{equation} \label{eq:flag-stripping-projective-bundle}
      \pi\colon \tilde\fM^{\vec Q(\Fr),\inj}_{\beta,\vec e} \to \tilde\fM^{\vec Q(\Fr),\inj}_{\beta,\vec e - \vec 1_{[a,N]}}
    \end{equation}
    which is naturally a projective bundle for the vector bundle
    $\cFr(\cE)/\cV_{a-1}$ on
    $\tilde\fM_{\beta,\vec{e}-\vec{1}_{[a,N]}}^{\vec Q(\Fr),\inj}$.

  \item \label{lem:flag-forgetii} For any $\phi \in
    K_\circ^{\tilde\sT}(\tilde\fM_{\beta,\vec e-\vec 1_{[a,N]}}^{\vec
      Q(\Fr),\inj})^\pl_{\loc}$,
    \[ \left[\phi, I_*\partial_{[a,N]}\right] = j_{a*} \hat\pi^*\phi \in K_\circ^{\tilde\sT}\left(\tilde\fM_{\beta,\vec e}^{\vec Q(\Fr),\inj(<a)}\right)^\pl_{\loc} \]
    where $\hat\pi^*$ denotes symmetrized pullback in K-homology
    (Definition~\ref{def:symm-pullback-homology}) along
    \eqref{eq:flag-stripping-projective-bundle}.
  \end{enumerate}
\end{lemma}

To be clear, whenever the composition $\rho_N \circ \cdots \circ
\rho_{a-1}$ is injective, it identifies $V_{a-1}$ as a subspace of
$\Fr(E)$.

\begin{proof}
  To show \ref{lem:flag-forgeti}, let $\cL$ denote the universal line
  bundle on $[\pt/\bC^\times]$, let $\fX \coloneqq
  \tilde\fM_{\beta,\vec{e}-\vec{1}_{[a,N]}}^{\vec Q(\Fr),\inj}$ for
  short, and consider the map
  \begin{equation} \label{eq:flag-stripping-projective-bundle-base}
    \tilde\fM_{\beta,\vec{e}}^{\vec Q(\Fr),\inj(<a)} \to \fX \times [\pt/\bC^\times]
  \end{equation}
  whose first coordinate is the restriction of $\pi$ and whose second
  coordinate is given by $(E, \vec V, \vec \rho) \mapsto
  (V_a/V_{a-1})^\vee$. Then
  \eqref{eq:flag-stripping-projective-bundle-base} lifts to an
  isomorphism
  \begin{equation} \label{eq:flag-stripping-as-projective-bundle-ambient}
    \tilde\fM_{\beta,\vec{e}}^{\vec Q(\Fr),\inj(<a)} \xrightarrow{\sim} \tot_{\fX \times [\pt/\bC^\times]}(\cV \boxtimes \cL)
  \end{equation}
  sending $(E, \vec V, \vec\rho)$ to the induced map $\zeta\colon
  \cL^\vee \cong V_a/V_{a-1} \to \Fr(E)/V_{a-1}$. This is an
  isomorphism, since the map $\zeta$ recovers $V_a$ upon pullback
  along $\Fr(E) \to \Fr(E)/V_{a-1}$. The open inclusion $j_a$ is the
  locus where $\zeta \neq 0$. Thus this isomorphism restricts to an
  isomorphism
  \begin{equation} \label{eq:flag-stripping-as-projective-bundle}
    \tilde\fM_{\beta,\vec{e}}^{\vec Q(\Fr),\inj} \xrightarrow{\sim} \bP_{\fX}(\cV)
  \end{equation}
  as claimed. (Note that the $[\pt/\bC^\times]$ action on $\fX \times
  [\pt/\bC^*]$ is by weight $-1$ on the second factor.)

\subsubsection{}

  To show \ref{lem:flag-forgetii}, we apply the homology projective
  bundle formula (Lemma \ref{lem:homology-projective-bundle-formula}),
  using the notation there, as follows. This will essentially be the
  same argument as in \S\ref{sec:pairs-master-term-1}. Explicitly,
  take
  \[ X = \fM_{\beta, \vec{e}-\vec{1}_{[a,N]}}^{\vec Q(\Fr),\inj,\pl}, \quad \fX = \fM_{\beta, \vec{e}-\vec{1}_{[a,N]}}^{\vec Q(\Fr),\inj}, \quad \cV = \cFr(\cE) / \cV_{a-1}, \]
  and $\pi_{\fX}$ and $\pi_X$ are our $\pi$ and $\pi^\pl$, and
  \eqref{eq:flag-stripping-as-projective-bundle} identifies $j$ as our
  $j_a^\pl$. We identify the $[\pt/\bC^\times]$ in
  \eqref{eq:flag-stripping-projective-bundle-base} with
  $\tilde\fM_{0,\vec 1_{[a,N]}}^{\vec Q(\Fr),\inj}$. Combined with the
  isomorphism \eqref{eq:flag-stripping-as-projective-bundle-ambient},
  the zero section of $\tot_{\fX \times [\pt/\bC^\times]}(\cV
  \boxtimes \cL)$ is the map that adds a trivial summand $\bC$ to each
  of $V_a, \ldots, V_N$ and extends $\rho_N$ to be zero on this
  summand, i.e. it is identified with the direct sum map
  \[ \Phi \coloneqq \Phi_{(\beta,\vec e-\vec 1_{[a,N]}),(0,\vec 1_{[a,N]})}\colon \tilde\fM_{\beta,\vec{e}-\vec 1_{[a,N]}}^{\vec Q(\Fr),\inj} \times \tilde\fM_{0,\vec 1_{[a,N]}}^{\vec Q(\Fr),\inj} \to \tilde\fM_{\beta,\vec{e}}^{\vec Q(\Fr),\inj(<a)}. \]
  The zero section of its rigidification $\tot_{\fX}(\cV)$ is
  therefore the composition $\Pi_{\beta,\vec e}^\pl \circ \Phi$. Hence
  \begin{equation} \label{eq:flag-stripping-homology-projective-bundle-formula}
    \begin{aligned}
      &(\kappa^{-\frac{1}{2}} - \kappa^{\frac{1}{2}}) \cdot (j_a^\pl)_*\hat{\pi^\pl}^*(\Pi^\pl_{\beta,\vec e-\vec 1_{[a,N]}})_*\phi \\
      &= \rho_{K,z} (\Pi_{\beta,\vec e}^\pl \circ \Phi)_* (D(z) \times \id) \left((\phi \boxtimes I_*\partial_{[a,N]}) \cap \frac{\hat\se_{z^{-1}}(\kappa^{-1} \cV^\vee \boxtimes \cL)}{\hat\se_z(\cV \boxtimes \cL^\vee)}\right)
    \end{aligned}
  \end{equation}
  where we used that $(I_*\partial_{[a,N]})(\cL^k) = 1$ for any $k \in
  \bZ$ in order to insert $\boxtimes \cL^\vee$ and $\boxtimes \cL$
  into the fraction on the right hand side. We do this because, upon
  restriction to the support of $\phi$, the bilinear elements
  $\tilde\scE^{\vec Q(\Fr)}$
  (Theorem~\ref{thm:auxiliary-stack-vertex-algebra}) become
  \[ \tilde\scE^{\vec Q(\Fr)}_{(\beta,\vec e-\vec 1_{[a,N]}),(0,\vec 1_{[a,N]})} = -\cV_{a-1}^\vee \boxtimes \cL + \kappa^{-1} \cdot \cFr(\cE)^\vee \boxtimes \cL \]
  due to the injectivity conditions on the $\rho_i$, and therefore
  \[ \frac{\hat\se_{z^{-1}}(\kappa^{-1} \cV^\vee \boxtimes \cL)}{\hat\se_z(\cV \boxtimes \cL^\vee)} = \hat\Theta(z) \coloneqq \hat\Theta_{(\beta,\vec e-\vec 1_{[a,N]}),(0,\vec 1_{[a,N]})}(z) \]
  Finally, applying base change as in
  \eqref{eq:P-to-projective-bundle}, the left hand side of
  \eqref{eq:flag-stripping-homology-projective-bundle-formula} becomes
  $(\kappa^{-\frac{1}{2}} - \kappa^{\frac{1}{2}}) \cdot
  (\Pi_{\beta,\vec e}^\pl)_* j_{a*}\hat\pi^*\phi$. Plugging all this
  back into
  \eqref{eq:flag-stripping-homology-projective-bundle-formula}, we get
  \begin{align*}
    (\kappa^{-\frac{1}{2}} - \kappa^{\frac{1}{2}}) \cdot (\Pi_{\beta,\vec e}^\pl)_* j_{a*}\hat\pi^*\phi
    &= \rho_{K,z} (\Pi_{\beta,\vec e}^\pl)_*\Phi_* (D(z) \times \id) \left((\phi \boxtimes I_*\partial_{[a,N]}) \cap \hat\Theta(z)\right) \\
    &= (\kappa^{-\frac{1}{2}} - \kappa^{\frac{1}{2}}) \cdot (\Pi_{\beta,\vec e}^\pl)_*\left[\phi, I_*\partial_{[a,N]}\right]
  \end{align*}
  by the definition (Theorem~\ref{thm:auxiliary-stack-vertex-algebra})
  of the Lie bracket. We conclude by applying the de-rigidification
  map $(I_{\beta,\vec e})_*$ to both sides.
\end{proof}

\subsubsection{}

\begin{lemma}\label{wc:lemma:pair-to-flag-moduli-setup}
  Let $(\beta,\vec{e}) \in S_\alpha$ be a flag and assume that $e_N
  \ge 2$.
  \begin{enumerate}[label = (\roman*)]
  \item \label{wc:lemma:pair-to-flag-moduli-setupi} If $x \le x_0(a)$,
    then $\tilde\fM_{\beta,\vec{e}}^{\vec Q(\Fr),\sst}(\tau_x^s) =
    \emptyset$;
  \item \label{wc:lemma:pair-to-flag-moduli-setupii} If $x > x_0(a)$
    and $s=0$, then $\tilde\fM_{\beta, \vec{e}}^{\vec
      Q(\Fr),\sst}(\tau_x^s)$ is independent of $x$ and contains no
    strictly $\tau_x^s$-semistable objects. An object $(E,\vec V,\vec
    \rho)$ belongs to $\tilde\fM_{\beta, \vec{e}}^{\vec
      Q(\Fr),\sst}(\tau_x^s)$ if and only if
    \begin{itemize}
    \item $\rho_i$ is an isomorphism for $a\leq i \leq N$, and
    \item $(E,\vec V',\vec \rho')$, from
      Definition~\ref{wc:def:flag-stripping}, is
      $\tau_x^s$-semistable.
    \end{itemize}
  \item \label{wc:lemma:pair-to-flag-moduli-setupiii} If $x>x_0(a)$
    and $s=0$, then the virtual structure sheaf for the APOT on
    $\tilde\fM_{\beta,\vec{e}}^{\vec Q(\Fr),\sst}(\tau_x^s)$ obtained
    by symmetrized pullback from $\fM_{\beta}^{\pl}$ coincides with
    the one for the APOT obtained by symmetrized pullback along
    \[ \pi^\pl\colon \tilde\fM_{\beta,\vec{e}}^{\vec Q(\Fr),\sst}(\tau_x^s) \to \tilde\fM_{\beta,\vec{e}-\vec{1}_{[a,N]}}^{\vec Q(\Fr),\sst}(\tau_x^s), \]
    the restriction of the projective bundle
    \eqref{eq:flag-stripping-projective-bundle}.
  \end{enumerate}
  If additionally $(\beta,\vec{e})\in \mathring S_{\alpha}$, then
  \ref{wc:lemma:pair-to-flag-moduli-setupii} and
  \ref{wc:lemma:pair-to-flag-moduli-setupiii} also hold for $s=1$.
\end{lemma}

\begin{proof}
  The claim \ref{wc:lemma:pair-to-flag-moduli-setupi} follows
  immediately from
  Lemma~\ref{wc:lem:general-semistable-loci}\ref{wc:lem:general-semistable-loci-iii}--\ref{wc:lem:general-semistable-loci-iv},
  since their conclusions cannot hold under our assumptions here.

  For \ref{wc:lemma:pair-to-flag-moduli-setupii}, we follow the proof
  of \cite[Prop. 10.14]{Joyce2021}. The statement that
  $\tilde\fM_{\beta, \vec{e}}^{\vec Q(\Fr),\sst}(\tau_x^s)$ is
  independent of $x$ and contains no strictly $\tau_x^s$-semistable
  objects, for both the $s=0$ and $s=1$ cases, holds by the same proof
  as in \cite[Proof of Prop. 10.14, (iii)]{Joyce2021}. So it remains
  to prove the equivalent characterization of semistable objects,
  which we do in \S\ref{wc:lemma:pair-to-flag-moduli-setup-proof-ii-1}
  and \S\ref{wc:lemma:pair-to-flag-moduli-setup-proof-ii-2} below.

  For \ref{wc:lemma:pair-to-flag-moduli-setupiii}, by
  Lemma~\ref{wc:lem:general-semistable-loci}\ref{wc:lem:general-semistable-loci-ii}
  and the results of \ref{wc:lemma:pair-to-flag-moduli-setupii},
  indeed \eqref{eq:flag-stripping-projective-bundle} restricts to
  $\pi^\pl$ as claimed, yielding a commutative triangle
  \[ \begin{tikzcd}
    \tilde\fM_{\beta, \vec{e}}^{\vec Q(\Fr),\sst}(\tau_x^s) \ar{dr}[swap]{\pi_{\fM_\beta^{\Fr}}} \ar{rr}{\pi^\pl} && \tilde\fM_{\beta,\vec{e}-\vec{1}_{[a,N]}}^{\vec Q(\Fr),\sst}(\tau_x^s) \ar{dl}{\pi_{\fM_\beta^{\Fr}}} \\
    & \fM_\beta^{\Fr} & 
  \end{tikzcd} \]
  where the vertical maps are the natural forgetful morphisms used to
  define APOTs on the auxiliary framed stacks. Both semistable loci in
  the upper row are algebraic spaces, so we conclude by the
  functoriality of APOTS
  (Theorem~\ref{thm:APOTs}\ref{it:APOT-functoriality}).

\subsubsection{}
\label{wc:lemma:pair-to-flag-moduli-setup-proof-ii-1}

  Suppose the element $(E, \vec V, \vec \rho)$ in $\tilde\fM_{\beta,
    \vec{e}}^{\vec Q(\Fr)}$ is $\tau^s_x$-semistable. Then by
  Lemma~\ref{wc:lem:general-semistable-loci}\ref{wc:lem:general-semistable-loci-ii}
  all $\rho_i$ must be injective and hence $\rho_i$ must be
  isomorphisms for $a \le i < N$. Furthermore, by the injectivity,
  $(E, \vec V', \vec\rho')$ becomes a subobject of $(E, \vec V,
  \vec\rho)$. We claim that any $\tau^s_x$-destabilizing sub-object
  $(E'', \vec V'', \vec\rho'')$ of $(E, \vec V', \vec\rho')$ is also a
  $\tau^s_x$-destabilizing sub-object of $(E, \vec V, \vec \rho)$. It
  suffices to show this for $x > x_0(a)$ very close to $x_0(a)$
  because $\tilde\fM_{\beta,\vec e}^{\vec Q(\Fr),\sst}(\tau^s_x)$ is
  independent of $x > x_0(a)$, and $\tilde\fM_{\beta,\vec e'}^{\vec
    Q(\Fr),\sst}(\tau^s_x)$ is also independent of $x > x_0(a')$,
  where $\vec e' \coloneqq \vec e - \vec 1_{[a,N]}$ and $a' < a$ is
  the minimal index with $e'_{a'} = e'_N$. Note that $a' < a$ implies
  $x_0(a') < x_0(a)$.

  Let $(\beta'', \vec e'')$ be the class of a destabilizing sub-object
  $(E'', \vec V'', \vec\rho'')$ of $(E, \vec V', \vec \rho')$. If
  $\beta'' = 0$ or $\tau(\beta'') > \tau(\beta - \beta'')$, then
  clearly $(E'', \vec V'', \vec\rho'')$ also destabilizes $(E, \vec V,
  \vec\rho)$. Otherwise, $\beta'' \neq 0$ and $\tau(\beta'') =
  \tau(\beta - \beta'')$, and we have to compare the quantities
  \begin{equation} \label{eq:flag-stripping-slopes}
    \begin{aligned}
      \tau^s_x(E/E'',\vec V'/\vec V'', \vec\rho'/\vec\rho'') &= \left(\tau(\beta-\beta''), \frac{s\lambda(\beta-\beta'') + (\vec \mu + x\vec 1) \cdot (\vec e - \vec 1_{[a,N]}-\vec e'')}{r(\beta-\beta'')}\right),\\
      \tau^s_x(E/E'', \vec V/\vec V'', \vec\rho/\vec \rho'') &= \left(\tau(\beta-\beta''), \frac{s\lambda(\beta-\beta'') + (\vec \mu + x\vec 1) \cdot (\vec e - \vec e'')}{r(\beta-\beta'')}\right).
    \end{aligned}
  \end{equation}
  These are equal at $x = x_0(a)$ because $(\vec\mu + x_0(a) \vec 1)
  \cdot \vec 1_{[a,N]} = 0$ by the definition of $x_0(a)$. By
  continuity in $x$, it follows that
  \[ \tau^s_x(E'', \vec V'', \vec \rho'') > \tau^s_x(E/E'',\vec V'/\vec V'', \vec\rho'/\vec\rho'') \implies \tau^s_x(E'', \vec V'', \vec \rho'') > \tau^s_x(E/E'',\vec V/\vec V'', \vec\rho/\vec\rho'') \]
  for $x > x_0(a)$ very close to $x_0(a)$. This is the desired claim.
  Hence $(E, \vec V', \vec \rho')$ is $\tau^s_x$-semistable, and the
  morphism $\pi$ is well-defined.

  If $s=1$ and $(\beta, \vec e) \in \mathring S_\alpha$, we modify the
  argument using properties of $\lambda$. Since $\beta \in \mathring
  R_\alpha$, we have $\lambda(\beta) = 0$. Given a
  $\tau^s_x$-destabilizing sub-object $(E'',\vec V'', \vec\rho'')$ of
  $(E, \vec V', \vec\rho')$, from $\tau(\beta'') = \tau(\beta-\beta'')$
  we conclude that also $\beta'',\beta-\beta'' \in \mathring
  R_{\alpha}$, so $\lambda(\beta'') = 0 = \lambda(\beta-\beta'')$.
  With this, the rest of the argument proceeds exactly as in the $s=0$
  case.

\subsubsection{}
\label{wc:lemma:pair-to-flag-moduli-setup-proof-ii-2}

  Conversely, suppose that $\rho_i$ is injective for $a\leq i \leq N$
  and that $(E, \vec V', \vec\rho')$ is $\tau^s_x$-semistable. We
  claim that any $\tau^s_x$-destabilizing sub-object $(E'', \vec V'',
  \vec\rho'')$ of $(E, \vec V, \vec \rho)$, with class denoted
  $(\beta'', \vec e'')$, induces a $\tau^s_x$-destabilizing sub-object
  of $(E, \vec V', \vec \rho')$. Namely, since all $\rho_i$ are
  injective, so are all $\rho''_i$ and thus $E'' \neq 0$ and we obtain
  a non-zero sub-object $(E'', \vec V''', \vec \rho''') \subset (E,
  \vec V', \vec \rho')$ from $(E'', \vec V'', \vec \rho'')$ by
  intersecting the two inside $(E, \vec V, \vec \rho)$. This
  sub-object has class $(\beta'', \vec e'' - \vec 1_{[b,N]})$ for some
  $b \ge a$. Moreover $(E'', \vec V''', \vec \rho''') \neq (E, \vec
  V', \vec \rho')$ because $(E, \vec V', \vec \rho')$ cannot
  $\tau^s_x$-destabilize $(E, \vec V, \vec\rho)$ for $x > x_0(a)$ by
  inspecting \eqref{eq:flag-stripping-slopes}. If $\tau(\beta'') >
  \tau(\beta - \beta'')$ then clearly $(E'', \vec V''', \vec \rho''')$
  also destabilizes $(E, \vec V', \vec \rho')$. Otherwise
  $\tau(\beta'') = \tau(\beta-\beta'')$ and we claim that
  \begin{equation} \label{eq:flag-stripping-slopes-2}
    \begin{aligned}
      \frac{s\lambda(\beta'') + (\vec\mu + x\vec 1) \cdot \vec e''}{r(\beta'')} &> \frac{s\lambda(\beta-\beta'') + (\vec\mu + x\vec 1) \cdot (\vec e - \vec e'')}{r(\beta - \beta'')} \\
      \implies \frac{s\lambda(\beta'') + (\vec\mu + x\vec 1) \cdot (\vec e'' - \vec 1_{[b,N]})}{r(\beta'')} &> \frac{s\lambda(\beta-\beta'') + (\vec\mu + x\vec 1) \cdot (\vec e - \vec e'' - \vec 1_{[a,b-1]})}{r(\beta - \beta'')},
    \end{aligned}
  \end{equation}
  i.e. $(E'', \vec V''', \vec\rho''')$ destabilizes $(E, \vec V', \vec
  \rho')$ anyways. To show this claim, like in
  \S\ref{wc:lemma:pair-to-flag-moduli-setup-proof-ii-1} we may assume
  $x > x_0(a)$ is very close to $x_0(a)$, in which case $(\mu + x\vec
  1) \cdot 1_{[a,N]} > 0$. If $b > a$ then $(\mu + x\vec 1) \cdot
  1_{[b,N]} < 0$ and we are done. Else if $b = a$, then the right hand
  sides of \eqref{eq:flag-stripping-slopes-2} are equal and the left
  hand sides of \eqref{eq:flag-stripping-slopes-2} are equal at $x =
  x_0(a)$, and therefore the desired implication holds by continuity
  in $x$.

  If $s=1$ and $(\beta, \vec e) \in \mathring R_\alpha$, again we
  conclude from $\tau(\beta'') = \tau(\beta-\beta'')$ that $\beta'',
  \beta - \beta'' \in \mathring R_{\alpha}$, so $\lambda(\beta) =
  \lambda(\beta'') = 0 = \lambda(\beta - \beta'')$. With this, the
  rest of the argument proceeds exactly as in the $s = 0$ case.
\end{proof}

\subsubsection{}

\begin{lemma} \label{wc:lem:pair-flag}
  Let $(\beta,\vec{e})\in S_{\alpha}$ be a flag and assume that $e_N
  \ge 2$. Then, for $s=0$,
  \begin{equation} \label{wc:eq:pair-flag-wc}
    \tilde\sz_{\beta,\vec{e}}^{s,x} = \begin{cases} 0 & x \leq x_0(a), \\ \left[\tilde\sz_{\beta, \vec{e}-\vec{1}_{[a,N]}}^{s,x'}, \tilde\sz_{0,\vec{1}_{[a,N]}}^{s,x'} \right] & x > x' \coloneqq x_0(a). \end{cases}
  \end{equation}
  If additionally $(\beta,\vec{e})\in \mathring S_{\alpha}$, then this
  also holds for $s=1$.
\end{lemma}

\begin{proof}
  The claim for $x\leq x_0(a)$ follows from
  Lemma~\ref{wc:lemma:pair-to-flag-moduli-setup}\ref{wc:lemma:pair-to-flag-moduli-setupi}.
  Hence, assume $x_0(a)<x \leq 1$ and continue with the notation of
  Lemma~\ref{wc:lemma:pair-to-flag-moduli-setup}. We may assume that
  $\tilde\fM_{\beta,\vec e}^{\vec Q(\Fr),\sst}(\tau^s_x)\neq
  \emptyset$, otherwise there is nothing to prove. There are no
  strictly $\tau_{x'}^s$-semistable objects of class $(\beta, \vec e)$
  or $(\beta, \vec e-\vec 1_{[a,N]})$ by
  Lemma~\ref{wc:lemma:pair-to-flag-moduli-setup}\ref{wc:lemma:pair-to-flag-moduli-setupii}.
  We claim that the hypotheses of
  \S\ref{bg:rem:symm-pb-formula-alg-sp} hold and therefore
  \begin{equation} \label{wc:eq:pair-flag-wc-1}
    \tilde\sz_{\beta,\vec{e}}^{s,x} = \hat\pi^* \tilde\sz_{\beta,\vec{e}-\vec{1}_{[a,N]}}^{s,x}
  \end{equation}
  where $\hat\pi^*$ is symmetrized pullback on K-homology. (Here we
  have implicitly used the base change formula $I_*\hat{\pi^\pl}^* =
  \hat\pi^* I_*$.) Namely, by
  Assumption~\ref{assump:wall-crossing}\ref{assump:it:properness-wcf},
  the source and target of $\pi$ are both algebraic spaces with proper
  $\sT$-fixed loci with the resolution property, and, by
  Lemma~\ref{wc:lemma:pair-to-flag-moduli-setup}\ref{wc:lemma:pair-to-flag-moduli-setupiii},
  they carry APOTs related by symmetrized pullback. Moreover, by
  Lemma~\ref{wc:lem:general-semistable-loci}\ref{wc:lem:general-semistable-loci-ii},
  the $\tau_x^s$-semistable locus is contained in the locus where all
  $\rho_j$ are injective, so
  Lemma~\ref{lem:flag-forget-projective}\ref{lem:flag-forgetii} yields
  \begin{equation} \label{wc:eq:pair-flag-wc-2}
    \hat\pi^* \tilde\sz_{\beta,\vec{e}-\vec{1}_{[a,N]}}^{s,x} = \left[\tilde\sz_{\beta,\vec{e}-\vec{1}_{[a,N]}}^{s,x}, I_*\partial_{[a,N]}\right].
  \end{equation}
  We conclude by combining \eqref{wc:eq:pair-flag-wc-1} and
  \eqref{wc:eq:pair-flag-wc-2} with the
  Definition~\ref{wc:def:artificial-invs}\ref{wc:def:aux-inv-iv} that
  $\tilde\sz^{s,x'}_{0,\vec 1_{[a,N]}} = I_*\partial_{[a,N]}$. Note
  that
  Lemma~\ref{wc:lemma:pair-to-flag-moduli-setup}\ref{wc:lemma:pair-to-flag-moduli-setupii}
  implies $\tilde\sz_{\beta, \vec{e}-\vec{1}_{[a,N]}}^{s,x} =
  \tilde\sz_{\beta, \vec{e}-\vec{1}_{[a,N]}}^{s,x'}$.
\end{proof}

\subsubsection{}

\begin{lemma} \label{wc:lem:pair-flag-vanishing}
  Let $(\beta, \vec e) \in S_\alpha$ be a flag and assume that $e_N
  \ge 1$. Then, for $s = 0$,
  \begin{enumerate}[label = (\roman*)]
  \item \label{it:pair-flag-vanishing-i} if $x \ge x' > x_0(a)$, then
    $\hat\sz^{s,x}_{\beta, \vec e} = \hat\sz^{s,x'}_{\beta, \vec e}$;
  \item \label{it:pair-flag-vanishing-ii} if $x > x' = x_0(a)$ and
    $e_N \ge 2$, then
    \begin{equation} \label{wc:eq:pair-flag-wc-artificial}
      \hat\sz^{s,x}_{\beta, \vec e} = \left[\hat\sz^{s,x'}_{\beta, \vec e-\vec 1_{[a,N]}}, \hat\sz^{s,x'}_{0,\vec 1_{[a,N]}}\right].
    \end{equation}
  \end{enumerate}
  If additionally $(\beta,\vec e)\in \mathring S_\alpha$, then these
  also hold for $s = 1$.
\end{lemma}

\begin{proof}
  First, for either \ref{it:pair-flag-vanishing-i} or
  \ref{it:pair-flag-vanishing-ii}, use the artificial wall-crossing
  formula \eqref{wc:eq:artificial-base-wcf} with $(s_1, x_1) = (0, x)$
  and $(s_2, x_2) = (0, x')$ to express
  $\hat\sz_{\beta,\vec{e}}^{0,x}$ in terms of invariants
  $\hat\sz^{0,x'}_{\beta_i,\vec{e}_i}$. Consider a non-vanishing term
  in the sum, corresponding to a decomposition $(\beta,\vec{e}) =
  (\beta_1,\vec{e}_1) + \cdots + (\beta_n,\vec{e}_n)$. By
  Lemma~\ref{wc:lem:vanishing-basic}\ref{it:vanishing-basic-i}, the
  only classes of the form $(0, \vec e_i)$ which appear must have
  $\vec e_i = \vec 1_{[a_i,b_i]}$ for some $1 \le a_i \le b_i \le N$.

  Suppose for the sake of contradiction that $\tau^0_{x'}(0, \vec e_i)
  = (\infty, *)$ for some $i$. This cannot happen if $n=1$ as $\beta
  \neq 0$, so we may assume $n \ge 2$. Since $x' \le x$, by the
  definition of $\tau^s_x$ we also have $\tau_x^0(0,\vec{e}_i) =
  (\infty,*)$. Let $I \coloneqq \{i :
  \tau_{x'}^0(\beta_i,\vec{e}_i)=(\infty, *)\}$. This can not be all
  of $\{1, \ldots, n\}$ as $\beta\neq 0$, and is non-empty by
  assumption. Applying
  Lemma~\ref{wc:lem:U-properties}\ref{it:U-vanishing}, the coefficient
  $\tilde U$ in the term vanishes, which contradicts our assumption.
  Hence if $(\beta_i,\vec{e}_i) = (0,\vec{1}_{[a',b]})$ for some $i$,
  then necessarily $(\vec{\mu} + x'\vec{1})\cdot \vec{1}_{[a',b]} \le
  0$, or, equivalently,
  \begin{equation}\label{wc:eq:xprime-ineq}
    x' \le -\frac{\mu_{a'} + \cdots + \mu_b}{b-a'+1} \le x_0(a').
  \end{equation}

\subsubsection{}
\label{wc:sec:pair-flag-vanishing-i-proof}

  We prove \ref{it:pair-flag-vanishing-i}. In this case, $x_0(a) < x'$
  implies that, for \eqref{wc:eq:xprime-ineq} to hold, any term of the
  form $(0,\vec{1}_{[a',b]})$ must have $a'>a$. If such a term occurs,
  then there must also be a term of the form $(\beta_i,\vec{e}_i)$
  with $\beta_i\neq 0$ and with $e_{i,\ell} > e_{i,\ell+1}$ for some
  $a \le \ell < N$, which would contradict
  Lemma~\ref{wc:lem:vanishing-basic}\ref{it:vanishing-basic-ii}.
  Therefore $\beta_i\neq 0$ for all $i$. Applying
  Lemma~\ref{wc:lem:vanishing-basic}\ref{it:vanishing-basic-ii}, it
  follows that each $\vec{e}_i$ is monotonically increasing, i.e.
  $e_{i,\ell} \le e_{i,\ell+1}$ for all $1 \le i \le n$ and $1 \le
  \ell < N$. In fact, each $\vec e_i$ is monotonically increasing in
  steps of at most one, i.e. $e_{i,\ell+1} \le e_{i,\ell} + 1$ as
  well, otherwise the condition $e_{\ell+1} \le e_\ell+1$ cannot be
  satisfied.

  Now assume $n>1$. Since $(\beta, \vec e)$ is a flag and $e_N \ge 1$,
  there is a minimal index $i_1(\vec{e})$ such that
  $e_{i_1(\vec{e})}=1$. This implies that, within the decomposition
  $\vec e = \vec e_1 + \cdots + \vec e_n$, there is a unique index $i$
  with $e_{i,i_1(\vec{e})} = 1$. Let $I \coloneqq \{i\}$. For $j \neq
  i$,
  \[ \tau_{x'}^0(\beta_i,\vec{e}_i) = \left(\tau(\beta), \frac{(\vec{\mu}+{x'}\vec{1})\cdot\vec{e}_i}{r(\beta_i)}\right) > \left(\tau(\beta), \frac{(\vec{\mu} + x'\vec{1})\cdot \vec{e}_j}{r(\beta_j)}\right) =  \tau_{x'}^0(\beta_j,\vec{e}_j) \]
  by using that $\mu_{i_1(\vec{e})}\gg \mu_{i'}$ for $i'>i_1(\vec{e})$
  and that
  \[ \abs{x'}\leq - x_0(a)= \frac{\mu_{a} + \cdots +\mu_N}{N-a+1} \ll \mu_{i_1(\vec{e})}. \]
  Similarly, $\tau_x^0(\beta_i,\vec e_i) > \tau_x^0(\beta,\vec e)$
  since the $\mu_{i_1(\vec{e})}$ contribution dominates and
  $r(\beta_i)< r(\beta)$. Applying
  Lemma~\ref{wc:lem:U-properties}\ref{it:U-vanishing}, the coefficient
  $\tilde{U}$ in the term vanishes, which contradicts our assumption.
  Hence, no terms with $n>1$ can contribute to the artificial
  wall-crossing formula, proving \ref{it:pair-flag-vanishing-i}.

  If $s=1$ and $(\beta, \vec e) \in \mathring S_\alpha$, we use the
  artificial wall-crossing formula
  \eqref{wc:eq:artificial-base-wcf-s1} instead of
  \eqref{wc:eq:artificial-base-wcf}, and use that $\lambda(\beta_i) =
  0$ for classes $\beta_i \in \mathring R_\alpha$. With this, the rest
  of the argument proceeds exactly as in the $s=0$ case.

\subsubsection{}\label{wc:lem:pair-flag-vanishing-ii}

  We prove \ref{it:pair-flag-vanishing-ii}, where $x' = x_0(a)$. In
  this case, by \eqref{wc:eq:xprime-ineq}, any class of the form
  $(0,\vec{1}_{[a',b]})$ must have $a'\geq a$. The argument in
  \S\ref{wc:sec:pair-flag-vanishing-i-proof} rules out the case $a' >
  a$, so $a' = a$. Then the inequalities in \eqref{wc:eq:xprime-ineq}
  are all equalities, which implies $b=N$. Hence any class in the
  decomposition with $\beta_i = 0$ must have $\vec e_i = \vec
  1_{[a,N]}$. Since $(\beta, \vec e)$ is a flag, at most one such
  class may exist. If no such classes exist in the decomposition, the
  argument in \S\ref{wc:sec:pair-flag-vanishing-i-proof} applies to
  prove that $n=1$, but then our assumptions on $\vec e$ violate
  Lemma~\ref{wc:lem:vanishing-basic}\ref{wc:lem:general-semistable-loci-iii}
  and therefore $\hat\sz_{\beta,\vec{e}}^{0,x'}$ vanishes, a
  contradiction. So, every non-vanishing term has $n \ge 2$ and there
  is exactly one $i$ such that $\beta_i = 0$.

  Assume for the sake of contradiction that some non-vanishing term in
  \eqref{wc:eq:artificial-base-wcf} has $n \ge 3$. Then, as in
  \S\ref{wc:sec:pair-flag-vanishing-i-proof}, take $I \coloneqq \{i\}$
  where $\vec{e}_{i, i_1(\vec{e})} = 1$. Since $a$ is the minimal
  index with $e_a = e_N$ and $e_N \ge 2$ by hypothesis, $i$ cannot be
  the index of the class $(0,\vec 1_{[a,N]})$ in the decomposition.
  Then the argument of \S\ref{wc:sec:pair-flag-vanishing-i-proof}
  shows that the coefficient $\tilde U$ in the term vanishes, a
  contradiction. Here we use that $(\vec \mu+x'\vec 1)\cdot \vec
  1_{[a,N]}=0$ so that $\tau_{x'}^0(\beta_i, \vec e_i) >
  \tau_{x'}^0(0, \vec 1_{[a,N]})$, and the assumption $n \ge 3$ is
  necessary so that there is at least one other term with $\beta_j
  \neq 0$ for $j \neq i$, in order to conclude that $r(\beta_i) <
  r(\beta)$.

  Thus, the only non-vanishing terms have $n=2$ and correspond to the
  decomposition of $(\beta, \vec e)$ into
  $(\beta,\vec{e}-\vec{1}_{[a,N]})$ and $(0,\vec{1}_{[a,N]})$ in some
  order. In this case, one can compute the coefficient $\tilde U$
  directly from Definition~\ref{def:universal-coefficients}: 
  \begin{align*}
    U((0,\vec{1}_{[a,N]}),(\beta,\vec{e}-\vec{1}_{[a,N]});\tau_{x_0(a)}^s,\tau_x^s) &=  S((0,\vec{1}_{[a,N]}),(\beta,\vec{e}-\vec{1}_{[a,N]});\tau_{x_0(a)}^s,\tau_x^s) = -1 \\
    U((\beta,\vec{e}-\vec{1}_{[a,N]}), (0,\vec{1}_{[a,N]});\tau_{x_0(a)}^s,\tau_x^s) &=  S((\beta,\vec{e}-\vec{1}_{[a,N]}, (0,\vec{1}_{[a,N]}));\tau_{x_0(a)}^s,\tau_x^s) = 1.
  \end{align*}
  By Lemma~\ref{lem:Utilde-definition}, this shows that the only
  non-vanishing term is $[\hat\sz^{s,x'}_{\beta, \vec e-\vec
      1_{[a,N]}}, \hat\sz^{s,x'}_{0,\vec 1_{[a,N]}}]$.

  If $s=1$ and $(\beta, \vec e) \in \mathring S_\alpha$, again, we use
  the artificial wall-crossing formula
  \eqref{wc:eq:artificial-base-wcf-s1} instead of
  \eqref{wc:eq:artificial-base-wcf}, and follow the same argument as
  in the $s=0$ case.
\end{proof}

\subsection{The ``horizontal'' wall-crossing}
\label{wc:sec:horizontal-wc}

\subsubsection{}

In this subsection, we follow the horizontal direction in
Figure~\ref{fig:dominant-wc-strategy-2}. Namely, we consider how the
invariants $\tilde\sz_{\beta,\vec{e}}^{s,0}$ and
$\hat\sz_{\beta,\vec{e}}^{s,0}$ change as $s$ ranges from $0$ to $1$.
By Lemma~\ref{wc:lem:wall-crossing-horizontal-setup}, there are a
finite number of walls in $s$, whose locations depend on the class
$(\beta, \vec e)$, and each wall is {\it simple} in the sense that
each decomposition of a polystable object (i.e. a semi-stable object
whose destabilizing sub-objects are direct summands) has at most two
parts. Thus, unlike at $x=-1$, master space techniques are applicable.
We list the main reuslts.
\begin{itemize}
\item (Proposition~\ref{wc:prop:horizontal-wc}) There is a
  ``horizontal'' wall-crossing formula for the invariants
  $\tilde\sz_{\beta,\vec e}^{s,0}$ at each simple wall in $s$.
\item (Proposition~\ref{wc:prop:horizontal-vanishing}) The artificial
  invariants $\hat\sz_{\beta,\vec e}^{s,0}$ satisfy the {\it same}
  ``horizontal'' wall-crossing formula.
\end{itemize}

\subsubsection{}

\begin{lemma}\label{wc:lem:wall-crossing-horizontal-setup}
  Let $(\beta, \vec e) \in S_\alpha$ be a full flag. Suppose that
  there is a decomposition $(\beta, \vec e) = (\beta_1, \vec e_1) +
  \cdots + (\beta_n, \vec e_n)$ with $n \ge 2$, such that for some $s
  \in [0,1]$,
  \begin{equation} \label{eq:wall-crossing-horizontal-walls}
    \tau_0^s(\beta_1, \vec e_1) = \cdots = \tau_0^s(\beta_n, \vec e_n),
  \end{equation}
  and $\tilde\fM^{\vec Q(\Fr),\sst}_{\beta_i,\vec e_i}(\tau_0^s) \neq
  \emptyset$ for all $1 \le i \le n$. Then $n = 2$ and $s \in (0,1)$,
  each $(\beta_i, \vec e_i)$ is a full flag, and there are no strictly
  $\tau_0^s$-semistable objects of class $(\beta_i, \vec e_i)$. For a
  given $(\beta, \vec e)$, there are only finitely many possibilities
  for $s$, which we denote
  \[ 0 \eqqcolon s_0 < s_1 < \cdots < s_K < s_{K+1} \coloneqq 1. \]
\end{lemma}

\begin{proof}
  This is the same proof as for \cite[Lemma 6.2.3(i)]{klt_dtpt}, which
  we summarize here. By the definition
  \eqref{wc:eq:joyce-framed-stack-stability} of $\tau_x^s$, evidently
  $\beta_i \neq 0$ for all $i$. If $\vec e_i = \vec 0$ for some $i$,
  then $e_{j,N} > \fr(\beta_j)$ for some $j$, contradicting
  Lemma~\ref{wc:lem:general-semistable-loci} in class $(\beta_j, \vec
  e_j)$. Similarly Lemma~\ref{wc:lem:general-semistable-loci} implies
  each $(\beta_i, \vec e_i)$ must be a full flag. Since $\tau(\beta_i)
  = \tau(\beta)$ for all $i$, we have $\tau(\beta_i) = \tau(\alpha -
  \beta_i)$ for all $i$ and therefore only the first case of
  \eqref{wc:eq:joyce-framed-stack-stability} is relevant. No two $\vec
  e_i$ can be proportional, otherwise $\vec e$ cannot satisfy the
  condition that $e_{j+1} \le e_j+1$ for all $j$. 
  
  The condition \eqref{eq:wall-crossing-horizontal-walls} yields a collection of
  $n-1$ linear equations
  \begin{equation}\label{wc:eq:s-eq-genericity-nonzero}
    s\left(\frac{\lambda(\beta_i)}{r(\beta_i)}-\frac{\lambda(\beta_{i+1})}{r(\beta_{i+1})}\right) = \vec \mu \cdot \left(\frac{\vec e_{i+1}}{r(\beta_{i+1})}-\frac{\vec e_{i}}{r(\beta_{i})}\right)
  \end{equation}
  for the variable $s$. Now by genericity of $\vec \mu$ and since no two $\vec e_i$ can be proportional, neither side of the equation can be zero, so that there is a unique $s$ solving each equation. Assume by contradiction that $n>2$ for such a given $s$. Then we can solve for $s$ and insert into another equation to get
  \begin{equation}\label{wc:eq:uniqueness-genericity}
    \vec \mu \cdot \left(\frac{\vec e_{i+1}}{r(\beta_{i+1})}-\frac{\vec e_{i}}{r(\beta_{i})}\right) = \eta \vec \mu \cdot \left(\frac{\vec e_{i+2}}{r(\beta_{i+2})}-\frac{\vec e_{i+1}}{r(\beta_{i+1})}\right),
  \end{equation}
  for some non-zero constant $\eta$. By genericity of $\vec \mu$, this equation can only be satisfied if
  \begin{equation*}
    \eta \frac{\vec e_{i+2}}{r(\beta_{i+2})} - (1+\eta) \frac{\vec e_{i+1}}{r(\beta_{i+1})} + \frac{\vec e_{i}}{r(\beta_{i})} = \vec 0,
  \end{equation*}
  which contradicts the assumption that all $\vec e_i$ sum up to the full flag $\vec e$, since there is at least one negative and one positive coefficient in the equation.
  
  Finally, if there were
  a $\tau_0^s$-semistable object, in class $(\beta_1, \vec e_1)$
  without loss of generality, then there would be a further splitting
  $(\beta_1, \vec e_1) = (\beta_1', \vec e_1') + (\beta_1'', \vec
  e_1'')$ such that the splitting $(\beta, \vec e) = (\beta_1', \vec
  e_1') + (\beta_1'', \vec e_1'') + (\beta_2, \vec e_2)$ also
  satisfies the conditions of this lemma, contradicting $n \le 2$.
\end{proof}

\subsubsection{}
\label{sec:horizontal-wc}

\begin{proposition}\label{wc:prop:horizontal-wc}
  Let $(\beta, \vec e) \in S_\alpha$ be a full flag. Let $0 = s_0 <
  s_1 < \cdots < s_K < s_{K+1} = 1$ be as in
  Lemma~\ref{wc:lem:wall-crossing-horizontal-setup}.
  \begin{enumerate}[label = (\roman*)]
  \item \label{prop:hor-wc-arti} The semistable locus $\tilde\fM^{\vec
    Q(\Fr),\sst}_{\beta,\vec{e}}(\tau_0^s)$ depends only on the
    connected component of $s$ in $[0,1] \setminus \{s_1,\ldots,
    s_K\}$. Thus the same is true for the invariant
    $\tilde\sz^{s,0}_{\beta,\vec{e}}$.
  \item \label{prop:hor-wc-artii} Fix some $s=s_k\in\{s_1,\ldots,
    s_K\}$ and let $s_+\in (s_k,s_k+1)$ and $s_- \in (s_{k-1}, s_k)$.
    Then
    \begin{equation}\label{wc:eq:simple-hor-wc-formula}
      \tilde\sz^{s_+,0}_{\beta,\vec{e}} = \tilde\sz^{s_-,0}_{\beta,\vec{e}} + \sum_i \left[\tilde\sz^{s_-,0}_{\gamma_i,\vec{f}_i}, \tilde\sz^{s_-,0}_{\delta_i,\vec{g}_i}\right]
    \end{equation}
    where the sum ranges over all splittings $(\beta,\vec{e}) =
    (\gamma_i,\vec{f}_i) + (\delta, \vec{g}_i)$ appearing in
    Lemma~\ref{wc:lem:wall-crossing-horizontal-setup} for $s=s_k$ such
    that $\vec f_i > \vec g_i$. In fact $\vec f_i \eqqcolon \vec f$
    and $\vec g_i \eqqcolon \vec g$ are independent of $i$.
  \end{enumerate}
\end{proposition}

\begin{proof}
  The claim \ref{prop:hor-wc-arti} follows from
  Lemma~\ref{wc:lem:wall-crossing-horizontal-setup} and the continuity
  of the function $\tau_0^s$ for $s \in [0,1]$.

  For \ref{prop:hor-wc-artii}, the proof in \cite[\S 6.2]{klt_dtpt}
  applies with minor adjustments which we summarize here. First, we claim that, once
  $s$ is fixed, genericity of $\vec\mu$ implies that the integer
  vectors $(\vec f, \vec g)$ solving
  \eqref{eq:wall-crossing-horizontal-walls} are unique, and so $\vec f_i=\vec f$ and $\vec g_i=\vec g$ are indeed independent of $i$. 
  
  Assume that there are two pairs $((\gamma_i,\vec f_i), (\delta_i,\vec g_i))_{i=1,2}$ that decompose the $(\beta,\vec e)$ and satisfy \eqref{eq:wall-crossing-horizontal-walls} for a fixed $s$. Assume they're ordered such that $\vec f_>\vec g_i$. We can set up equation \eqref{wc:eq:s-eq-genericity-nonzero} for each pair, and as above genericity of $\vec \mu$ and the full flag conditions imply that neither side vanishes. Since $s$ is fixed, we can solve for it to obtain
  \begin{equation}\label{wc:eq:ab-uniqueness-genericity}
    \vec \mu \cdot \left(\frac{\vec g_{1}}{r(\delta_{1})}-\frac{\vec f_{1}}{r(\gamma_{1})}\right) = \eta \vec \mu \cdot \left(\frac{\vec g_{2}}{r(\delta_{2})}-\frac{\vec f_{2}}{r(\gamma_{2})}\right),
  \end{equation}
  for some non-zero constant $\eta$. By genericity of $\vec \mu$, this equation can only be satisfied if
  \begin{equation*}
    \frac{\vec g_{1}}{r(\delta_{1})}-\frac{\vec f_{1}}{r(\gamma_{1})} = \eta \left(\frac{\vec g_{2}}{r(\delta_{2})}-\frac{\vec f_{2}}{r(\gamma_{2})}\right).
  \end{equation*}
  Take $a$ to be the smallest index where $e_a>0$. At this index $a$, we must have $f_{1,a}=f_{2,a}=1$ and $g_{1,a}=g_{2,a}=0$, since $\vec f_i>\vec g_i$. Evaluating the above equation at index $a$, we get $\eta=\frac{r(\gamma_2)}{r(\gamma_1)}$. After inserting this coefficient, multiplying the equation by $r(\gamma_1)>0$, and adding $\vec e$, we get
  \begin{equation*}
    \left(1+\frac{r(\gamma_1)}{r(\delta_1)}\right)\vec g_1 = \left(1+\frac{r(\gamma_2)}{r(\delta_2)}\right)\vec g_2,  
  \end{equation*}
  where we used $\vec e -\vec f_i=\vec g_i$. But this means the two $\vec g_i$ are proportional with a positive constant, which implies $\vec g_1=\vec g_2$ as they're both full flags. This also implies $\vec f_1=\vec f_2$.
  
  Since $\vec f > \vec g$ by hypothesis, let $a < b$ be the smallest
  indices such that $f_a > 0$ and $g_b > 0$, and consider the quiver
  \[ \begin{tikzcd}[column sep=1em]
    \overset{\hat V_1}{\blacksquare} \ar{r} & \blacksquare \ar{r} & \cdots \ar{r} & \blacksquare \ar{r} & \overset{\hat V_a}{\blacksquare} \ar{r} \ar{drr}[swap]{\hat\rho_{-1}} & \blacksquare \ar{r} & \cdots \ar{r} & \blacksquare \ar{r} & \blacksquare \ar{r}{\hat\rho_{b-1}} \ar[dotted]{dll}[swap]{0} & \overset{\hat V_b}{\blacksquare} \ar{r} \ar{dlll}{\hat\rho_0} & \blacksquare \ar{r} & \cdots \ar{r} & \overset{\hat V_r}{\blacksquare} \ar{r} & \overset{\Fr(E)}{\bigbullet} \\
            {} & & & & & & \overset{\hat V_0}{\blacksquare}
  \end{tikzcd} \]
  where the dotted arrow is not an edge but rather the relation
  $\hat\rho_0 \circ \hat\rho_{b-1} = 0$. On the $(\beta,(1,\vec e))$
  component of the resulting auxiliary moduli stack, consider the
  stability condition defined using
  \eqref{wc:eq:joyce-framed-stack-stability} but with the parameter
  $(-\epsilon, \vec\mu)$ in place of $\vec\mu$, for very small
  $\epsilon > 0$. Let $\bM_{\beta,\vec e}$ denote the semistable
  locus; this is the {\it master space}. From it, we obtain the
  desired wall-crossing formula \eqref{wc:eq:simple-hor-wc-formula}
  following the same strategy as in
  \S\ref{sec:semistable-invariants-master-space}, as follows. Let $\sS
  \coloneqq \bC^\times$, with coordinate denoted $z$, act on
  $\bM_{\beta, \vec e}$ by scaling the map $\hat\rho_{-1}$ with weight
  $z$. The components of the $\sS$-fixed locus may be identified with
  $\tilde\fM_{\beta,\vec e}^{\vec Q(\Fr),\sst}(\tau_0^{s_-})$ (the
  divisor where $\hat\rho_{-1}=0$), $\tilde\fM_{\beta,\vec e}^{\vec
    Q(\Fr),\sst}(\tau_0^{s_+})$ (the divisor where $\hat\rho_0=0$),
  and products
  \begin{equation} \label{wc:eq:master-space-complicated-locus}
    Z_{(\gamma,\vec f), (\delta, \vec g)} \cong \tilde\fM^{\vec Q(\Fr),\sst}_{\gamma,\vec f}(\tau_0^s) \times \tilde\fM^{\vec Q(\Fr),\sst}_{\delta,\vec g}(\tau_0^s)
  \end{equation}
  for every splitting $(\beta, \vec e) = (\gamma, \vec f) + (\delta,
  \vec g)$ appearing in
  Lemma~\ref{wc:lem:wall-crossing-horizontal-setup} (embedded in
  $\bM_{\beta, \vec e}$ via the direct sum map). Under these
  identifications, the natural symmetrized virtual structure sheaves
  on these algebraic spaces match the induced symmetrized virtual
  structure sheaf on $\bM_{\beta,\vec e}^{\sS}$ by applying
  Theorem~\ref{thm:APOT-comparison}. We conclude using the same
  arguments as in the proof of
  Theorem~\ref{thm:pairs-master-relation}, using
  Assumption~\ref{assump:wall-crossing}\ref{assump:it:properness-wcf}.
\end{proof}

\subsubsection{}

\begin{remark} \label{rem:joyce-lie-bracket-vs-ours-2}
  In \eqref{wc:eq:simple-hor-wc-formula}, for the exact same reason as
  in Remark~\ref{rem:joyce-lie-bracket-vs-ours-1}, the Lie bracket
  must be the one on the auxiliary stacks $\tilde\fM^{\vec Q(\Fr)}$;
  there is no clever choice of integrand in the master space argument
  which produces a useful formula involving only the Lie bracket on
  $\fM$.
\end{remark}

\subsubsection{}

\begin{proposition} \label{wc:prop:horizontal-vanishing}
  Let $(\beta, \vec e) \in S_\alpha$ be a full flag. Let $0 = s_0 <
  s_1 < \cdots < s_K < s_{K+1} = 1$ be as in Lemma
  \ref{wc:lem:wall-crossing-horizontal-setup}.
  \begin{enumerate}[label = (\roman*)]
  \item \label{prop:hor-van-arti} Suppose there is a decomposition
    $(\beta, \vec e) = (\beta_1, \vec e_1) + \cdots + (\beta_n, \vec
    e_n)$ with $n \ge 2$, such that for some $s \in [0,1]$,
    \[ \tau_0^s(\beta_1, \vec e_1) = \cdots = \tau_0^s(\beta_n, \vec e_n), \]
    and $\hat\sz^{s,0}_{\beta_i,\vec{e}_i}\neq 0$ for all $1 \le i \le
    n$. Then $n = 2$ and $s \in \{s_1,\ldots,s_K\}$ and each
    $(\beta_i, \vec e_i)$ is a full flag.

  \item \label{prop:hor-van-artii} The invariant $\hat\sz_{\beta,\vec
    e}^{s,0}$ depends only on the connected component of $s$ in $[0,1]
    \setminus \{s_1, \ldots, s_K\}$.

  \item \label{prop:hor-van-artiii} Fix some $s=s_k\in\{s_1,\ldots,
    s_K\}$ and let $s_+\in (s_k,s_k+1)$ and $s_- \in (s_{k-1}, s_k)$.
    Then
    \begin{equation} \label{wc:eq:simple-hor-wc-formula-art}
      \hat\sz^{s_+,0}_{\beta,\vec{e}} = \hat\sz^{s_-,0}_{\beta,\vec{e}} + \sum_i \left[\hat\sz^{s_-,0}_{\gamma_i,\vec{f}_i}, \hat\sz^{s_-,0}_{\delta_i,\vec{g}_i}\right]
    \end{equation}
    where the sum ranges over all splittings $(\beta,\vec{e}) =
    (\gamma_i,\vec{f}_i) + (\delta, \vec{g}_i)$ appearing in
    \ref{prop:hor-van-arti} for $s=s_k$ such that $\vec f_i > \vec g_i$.
  \end{enumerate}
\end{proposition}

\begin{proof}
  The proof of \ref{prop:hor-van-arti} is exactly the proof of
  Lemma~\ref{wc:lem:wall-crossing-horizontal-setup}, but using
  Lemma~\ref{wc:lem:vanishing-basic} in place of
  Lemma~\ref{wc:lem:general-semistable-loci}.
  
  The remainder of the proof closely follows part of the proof of
  \cite[Prop. 10.17]{Joyce2021}.

  For \ref{prop:hor-van-artii}, let $s_-, s_+ \in [0, 1] \setminus
  \{s_1, \ldots, s_K\}$ lie in the same connected component, i.e. the
  interval $[s_-, s_+]$ does not contain any of the $s_k$. We will
  show $\hat\sz_{\beta,\vec e}^{s_+,0} = \hat\sz_{\beta,\vec
    e}^{s_-,0}$. Suppose there exists a non-vanishing term in the
  artificial wall-crossing formula \eqref{wc:eq:artificial-base-wcf}
  corresponding to a splitting $(\beta,\vec{e}) = (\beta_1,\vec{e}_1)
  + \cdots + (\beta_n,\vec{e}_n)$ with $n \ge 2$. By \cite[Prop.
    3.16]{Joyce2021} applied to the continuous family of stability
  conditions $\{\tau_0^s\}_{s\in[s_-,s_+]}$, there exists $s' \in
  (s_-,s_+)$ such that $\tau_0^{s'}(\beta_1,\vec e_1) = \cdots =
  \tau_0^{s'}(\beta_n, \vec e_n)$. By \ref{prop:hor-van-arti}, $s' \in
  \{s_1, \ldots, s_K\}$, a contradiction. Hence only the $n=1$ term is
  non-vanishing.

  We argue similarly for \ref{prop:hor-van-artiii}. For $s_\pm$
  sufficiently close to $s_k$, combining \cite[Prop. 3.16]{Joyce2021}
  with \ref{prop:hor-van-arti} implies all terms in the artificial
  wall-crossing formula \eqref{wc:eq:artificial-base-wcf} with $n \ge
  2$ vanish except those corresponding to the splittings $(\beta,\vec
  e) = (\gamma_i, \vec f_i) + (\delta_i, \vec g_i)$ appearing in
  \ref{prop:hor-van-arti} for $s=s_k$. For these non-vanishing terms, we
  may compute the coefficient $\tilde U$ directly from
  Definition~\ref{def:universal-coefficients}: ordering the pairs such
  that $\vec f_i > \vec g_i$, using Lemma~\ref{wc:lem:hor-wc-ordering}
  below, we get
  \[ \tau^{s_+}_0(\gamma,\vec f) < \tau^{s_+}_0(\delta,\vec g), \quad \tau^{s_-}_0(\gamma,\vec f) > \tau^{s_-}_0(\delta,\vec g), \]
  and therefore
  \begin{align*}
    U\left((\gamma_i,\vec f_i), (\delta_i, \vec g_i); \tau^{s_-}_0,\tau^{s_+}_0\right) = S\left((\gamma_i,\vec f_i), (\delta_i, \vec g_i); \tau^{s_-}_0,\tau^{s_+}_0\right) = 1, \\
    U\left((\delta_i, \vec g_i), (\gamma_i, \vec f_i); \tau^{s_-}_0,\tau^{s_+}_0\right) = S\left((\delta_i, \vec g_i), (\gamma_i, \vec f_i); \tau^{s_-}_0,\tau^{s_+}_0\right) = -1.
  \end{align*}
  By Lemma~\ref{lem:Utilde-definition}, this shows that the only
  non-vanishing term with $n \ge 2$ is
  $[\hat\sz^{s_-,0}_{\gamma_i,\vec{f}_i},
    \hat\sz^{s_-,0}_{\delta_i,\vec{g}_i}]$.
\end{proof}

\subsubsection{}

\begin{lemma} \label{wc:lem:hor-wc-ordering}
  Let $(\gamma, \vec{f}), (\delta,\vec{g}) \in S_\alpha$ be full flags
  such that their sum is a flag. Suppose that $\tau_0^s(\gamma, \vec
  f) = \tau_0^s(\delta, \vec g)$ for some $s > 0$. Then $\vec f > \vec
  g$ in the lexicographical ordering if and only if
  \[ \frac{\lambda(\gamma)}{r(\gamma)} < \frac{\lambda(\delta)}{r(\delta)}. \]
\end{lemma}

\begin{proof}
  Writing out $\tau_0^s(\gamma,\vec f)=\tau_0^s(\delta,\vec g)$
  explicitly and simplifying, we get
  \begin{equation}\label{eq:explicit-joyce-full-flag-splitting-condition}
    s\left(\frac{\lambda(\gamma)}{r(\gamma)}-\frac{\lambda(\delta)}{r(\delta)}\right)= \vec \mu\cdot \left(\frac{\vec g}{r(\delta)}-\frac{\vec f}{r(\gamma)}\right).
  \end{equation}
  Since $s>0$, we have $\frac{\lambda(\gamma)}{r(\gamma)}<
  \frac{\lambda(\delta)}{r(\delta)}$ if and only if $\vec \mu\cdot
  \left(\frac{\vec g}{r(\delta)}-\frac{\vec f}{r(\gamma)}\right)$ is
  negative. Let $a$ be the largest index such that $f_j = g_j = 0$ for
  all $j < a$; as $(\gamma, \vec f)$ and $(\delta, \vec g)$ are full
  flags, $a \le N$. Since $(\gamma, \vec f) + (\delta, \vec g)$ is a
  flag, either $f_a = 1$ and $g_a = 0$ or $f_a = 0$ and $g_a = 1$.
  Then, since $\mu_1\gg \mu_2\gg\cdots \gg \mu_N$ and $r(-) >0$, the
  right hand side of
  \eqref{eq:explicit-joyce-full-flag-splitting-condition} is negative
  if and only if $\vec f > \vec g$.
\end{proof}

\subsection{Putting everything together}
\label{wc:sec:putting-things-together}

\subsubsection{}

In this subsection, we finally prove Theorem~\ref{thm:wcf} by putting
together the individual wall-crossing steps in
Figure~\ref{fig:dominant-wc-strategy-2}. Recall that, by
Lemma~\ref{wc:lemma:reduction-wc-to-aux-wc}, the goal is to prove
\[ \tilde\sz^{s,x}_{\beta,\vec e} \labeleq{?} \hat\sz^{s,x}_{\beta,\vec e} \]
for $(\beta, \vec e) = (\alpha, \vec 0)$ at $(s, x) = (1, -1)$. This
is done by Proposition~\ref{wc:prop:finally}.

\subsubsection{}

\begin{lemma} \label{lem:pure-framing-ident}
  Let $(s, x) \in [0, 1] \times [-1, 0]$ and $1 \le a \le N$. Then
  \[ \hat\sz^{s,x}_{0,\vec 1_{[a,N]}} = \tilde\sz^{0,-1}_{0,\vec 1_{[a,N]}} = I_*\partial_{[a,N]} = \tilde\sz^{s,x}_{0,\vec 1_{[a,N]}}. \]
\end{lemma}

\begin{proof}
  The first equality is \eqref{eq:artificial-invariant-pure-framing}
  in \S\ref{wc:sec:gen-van-res-proof-i}, and the second and third
  equalities follow from the Definition~\ref{wc:def:artificial-invs}
  of auxiliary invariants.
\end{proof}

\subsubsection{}

\begin{lemma} \label{lem:pair-equation}
  Let $\beta \in R_\alpha$. Then, for $1 \le a \le N$ and $x_0(a) < x
  \le 0$,
  \[ \tilde\sz_{\beta,\vec{1}_{[a,N]}}^{0,x} = \hat\sz_{\beta,\vec{1}_{[a,N]}}^{0,x}. \]
\end{lemma}

\begin{proof}
  At $(s,x)=(0,-1)$, Definition~\ref{wc:def:artificial-invs} for
  artificial invariants becomes $\hat\sz^{0,-1}_{\gamma,\vec f} =
  \tilde\sz^{0,-1}_{\gamma,\vec f}$, for any $(\gamma, \vec f) \in
  S_\alpha$, by Lemma~\ref{wc:lem:U-properties}\ref{it:U-identity}.
  Combining this with Lemma~\ref{lem:pure-framing-ident}, the desired
  statement when $x_0(a) < x \le x_0(a+1)$ then follows directly
  because the right hand sides of \eqref{wc:eq:sst-pair-s0-wc} in
  Lemma~\ref{wc:sec:sst-pair}\ref{it:sst-pair-a} and
  \eqref{wc:eq:van-sst-pair-s0-wc} in
  Lemma~\ref{wc:lem:vanishing-sst-pair}\ref{it:vanishing-sst-pair-a}
  are equal. By
  Lemma~\ref{wc:lemma:pair-to-flag-moduli-setup}\ref{wc:lemma:pair-to-flag-moduli-setupii}
  and
  Lemma~\ref{wc:lem:pair-flag-vanishing}\ref{it:pair-flag-vanishing-i},
  this extends to any $x_0(a) < x \le 0$.
\end{proof}

\subsubsection{}

\begin{lemma} \label{lem:left-wall-ident}
  Let $(\beta, \vec e) \in S_{\alpha}$ be a flag. Assume that $e_N \ge
  1$ and let $1 \le a \le N$ be the minimal index such that $e_a =
  e_N$. Then, for any $x_0(a) < x \le 0$, 
  \[ \tilde\sz_{\beta,\vec{e}}^{0,x} = \hat\sz_{\beta,\vec{e}}^{0,x}.\]
\end{lemma}

\begin{proof}
  Induct on $e_N \ge 1$. If $e_N = 1$, then the desired result is
  exactly Lemma~\ref{lem:pair-equation}. For $e_N > 1$, combining the
  induction hypothesis with Lemma~\ref{lem:pure-framing-ident}, the
  right hand sides of \eqref{wc:eq:pair-flag-wc} in
  Lemma~\ref{wc:lem:pair-flag} and
  \eqref{wc:eq:pair-flag-wc-artificial} in
  Lemma~\ref{wc:lem:pair-flag-vanishing}\ref{it:pair-flag-vanishing-ii}
  are equal. We are done by induction.
\end{proof}

\subsubsection{}

\begin{proposition} \label{wc:prop:ptt-horizontal-wc}
  Let $(\beta, \vec e) \in S_\alpha$ be a full flag. Let $s \in [0,1]
  \setminus \{s_1, \ldots, s_K\}$ where $0 < s_1 < \cdots < s_K < 1$
  are the walls in Lemma~\ref{wc:lem:wall-crossing-horizontal-setup}).
  Then
  \[ \tilde\sz^{s,0}_{\beta,\vec{e}} = \hat\sz^{s,0}_{\beta,\vec{e}}. \]
  In particular, $\tilde\sz^{1,0}_{\beta,\vec{e}} =
  \hat\sz^{1,0}_{\beta,\vec{e}}$.
\end{proposition}

\begin{proof}
  Induct on $r(\beta)$. Namely, we may assume that the claim holds for
  all classes $(\beta_i,\vec{e}_i) \in S_\alpha$ with $r(\beta_i) <
  r(\beta)$. This is vacuously true if $r(\beta) = 1$, which is the
  base case.

  Consider the claim for the given $(\beta, \vec e)$.
  Lemma~\ref{lem:left-wall-ident} implies it holds at $s=0$. Suppose
  it does not hold for some $s > 0$, and let $s'$ be the infimum over
  all $s\in [0,1] \setminus \{s_1,\ldots,s_K\}$ for which it does not
  hold. Then we may pick $s_- < s_+$, both arbitrarily close to $s'$,
  such that $\tilde\sz^{s_-,0}_{\beta,\vec{e}} =
  \hat\sz^{s_-,0}_{\beta,\vec{e}}$ but
  $\tilde\sz^{s_+,0}_{\beta,\vec{e}} \neq
  \hat\sz^{s_+,0}_{\beta,\vec{e}}$. If $s' \notin \{s_1, \ldots,
  s_K\}$, then we may choose $s_\pm$ such that $[s_-, s_+] \cap \{s_1,
  \ldots, s_K\} = \emptyset$, and then
  \[ \tilde\sz^{s_+, 0}_{\beta,\vec{e}} = \tilde\sz^{s_-, 0}_{\beta,\vec{e}} = \hat\sz^{s_-, 0}_{\beta,\vec{e}} = \hat\sz^{s_+, 0}_{\beta,\vec{e}} \]
  by Proposition~\ref{wc:prop:horizontal-wc}\ref{prop:hor-wc-arti},
  the assumption, and
  Proposition~\ref{wc:prop:horizontal-vanishing}\ref{prop:hor-van-artii}
  respectively. This is a contradiction. On the other hand, if $s' =
  s_k$ for some $1 \le k \le K$, choose $s_-$ and $s_+$ sufficiently
  close to $s'$ so that both
  Proposition~\ref{wc:prop:horizontal-wc}\ref{prop:hor-wc-artii} and
  Proposition~\ref{wc:prop:horizontal-vanishing}\ref{prop:hor-van-artiii}
  hold. Consider the right hand sides of
  \eqref{wc:eq:simple-hor-wc-formula} and
  \eqref{wc:eq:simple-hor-wc-formula-art}. By
  Lemma~\ref{wc:lem:wall-crossing-horizontal-setup} and
  Proposition~\ref{wc:prop:horizontal-vanishing}\ref{prop:hor-van-arti},
  $s'$ is not a wall for any of the classes $(\gamma_i, \vec f_i)$ and
  $(\delta_i, \vec g_i)$, therefore the induction hypothesis yields
  $\tilde\sz_{\gamma_i,\vec f_i}^{s,0} = \hat\sz_{\gamma_i,\vec
    f_i}^{s',0}$ and $\tilde\sz_{\delta_i,\vec g_i}^{s,0} =
  \hat\sz_{\delta_i,\vec g_i}^{s',0}$ and thus the sums in
  \eqref{wc:eq:simple-hor-wc-formula} and
  \eqref{wc:eq:simple-hor-wc-formula-art} range over the same
  splittings. Similarly, by choosing $s_-$ sufficiently close to $s'$,
  we may assume it is not a wall for any of the classes $(\gamma_i,
  \vec f_i)$ and $(\delta_i, \vec g_i)$, so, again by the induction
  hypothesis, the terms in the sums in
  \eqref{wc:eq:simple-hor-wc-formula} and
  \eqref{wc:eq:simple-hor-wc-formula-art} are equal. Hence the two
  right hand sides are equal. But then the left hand sides say
  $\tilde\sz^{s_+, 0}_{\beta,\vec{e}} = \hat\sz^{s_+,
    0}_{\beta,\vec{e}}$, a contradiction.

  Thus the desired claim holds for the given $(\beta, \vec e)$,
  completing the induction step.
\end{proof}

\subsubsection{}

\begin{proposition} \label{wc:prop:finally}
  For any $\beta \in \mathring R_\alpha$,
  \[ \tilde\sz_{\beta,\vec{0}}^{1,-1} = \hat\sz_{\beta,\vec{0}}^{1,-1}.\]
  In particular, $\tilde\sz_{\alpha,\vec{0}}^{1,-1} =
  \hat\sz_{\alpha,\vec{0}}^{1,-1}$.
\end{proposition}

\begin{proof}
  Induct on $r(\beta)$. Namely, we may assume that the claim holds for
  all classes $\beta_i \in \mathring R_\alpha$ with $r(\beta_i) <
  r(\beta)$. This is vacuously true if $r(\beta) = 1$, which is the
  base case.

  Choose $\vec e$ such that $(\beta, \vec e)$ is a full flag.
  Concretely, we choose $e_i \coloneqq \min\{i, \fr(\beta)\}$ for $1
  \le i \le N$. In particular, $(\beta, \vec e) \in \mathring
  S_\alpha$. Then
  \begin{equation} \label{wc:eq:ptt-right-vertical-top}
    \tilde\sz^{1,0}_{\beta,\vec e} = \hat\sz^{1,0}_{\beta, \vec e}
  \end{equation}
  by Proposition~\ref{wc:prop:ptt-horizontal-wc}. This implies the
  following.

\subsubsection{}

  \begin{lemma}
    \begin{equation} \label{wc:eq:ptt-right-vertical-pairs}
      \tilde\sz^{1,x_0(2)}_{\beta,\vec{1}_{[1,N]}} = \hat\sz^{1,x_0(2)}_{\beta,\vec{1}_{[1,N]}}.
    \end{equation}
  \end{lemma}

  \begin{proof}
    If $\fr(\beta) = 1$ then $\vec e = \vec 1_{[1,N]}$ and
    \eqref{wc:eq:ptt-right-vertical-top} becomes the desired claim
    upon using
    Lemmas~\ref{wc:lemma:pair-to-flag-moduli-setup}\ref{wc:lemma:pair-to-flag-moduli-setupii}
    and \ref{wc:lem:pair-flag-vanishing}\ref{it:pair-flag-vanishing-i}
    to move from $x=1$ to $x=x_0(2) > x_0(1)$.

    Else, $e_N = \fr(\beta) \ge 2$, and we use the following procedure
    to iteratively reduce to the case that $e_N = 1$. Applying
    Lemma~\ref{wc:lem:pair-flag} and
    Lemma~\ref{wc:lem:pair-flag-vanishing} for $s=1$ to the left and
    right hand sides of \eqref{wc:eq:ptt-right-vertical-top}
    respectively, along with Lemma~\ref{lem:pure-framing-ident},
    yields
    \[ \left[\tilde\sz^{1,x_0(a)}_{\beta,\vec{e}-\vec{1}_{[a,N]}}, I_*\partial_{[a,N]}\right] = \left[\widehat{\sZ}^{1,x_0(a)}_{\beta,\vec{e}-\vec{1}_{[a,N]}}, I_*\partial_{[a,N]}\right]. \]
    Note that, here, $a = \fr(\beta)$ by our specific choice of $\vec
    e$. Now use Lemma~\ref{wc:lemma:pushforward-vertical-wall} below
    to push the classes on both sides down to $\tilde\fM^{\vec
      Q(\Fr),\pl}_{\beta,\vec{e}-\vec{1}_{[a,N]}}$, to get
    \[ \tilde\sz^{1,x_0(a)}_{\beta,\vec{e}-\vec{1}_{[a,N]}} \cdot [1]_{\kappa} = \hat\sz^{1,x_0(a)}_{\beta,\vec{e}-\vec{1}_{[a,N]}} \cdot [1]_{\kappa}. \]
    Iterate this procedure $\fr(\beta) - 1$ times, where the $i$-th
    iteration subtracts $\vec 1_{[\fr(\beta)-i, N]}$ from the framing
    dimension, the result is
    \[ \tilde\sz^{1,x_0(2)}_{\beta,\vec{1}_{[1,N]}} \cdot [\fr(\beta)-1]_\kappa! = \hat\sz^{1,x_0(2)}_{\beta,\vec{1}_{[1,N]}} \cdot [\fr(\beta)-1]_\kappa! \]
    and we may cancel $[\fr(\beta)-1]_\kappa! \coloneqq
    \prod_{k=1}^{\fr(\beta)-1} [k]_\kappa$ from both sides.
  \end{proof}

\subsubsection{}

  Since $x_0(2) > x_0(1)$, we may substitute
  \eqref{wc:eq:sst-pair-s1-wc} from
  Lemma~\ref{lem:sst-pair}\ref{it:sst-pair-b} and
  \eqref{wc:eq:van-sst-pair-s1-wc} from
  Lemma~\ref{wc:lem:vanishing-sst-pair}\ref{it:vanishing-sst-pair-b}
  into the left and right hand sides of
  \eqref{wc:eq:ptt-right-vertical-pairs}. The terms of the sums which
  involve a non-trivial decomposition of $\beta$ cancel by the
  induction hypothesis. Applying Lemma~\ref{lem:pure-framing-ident},
  we are left with the equality
  \[ \left[\tilde\sz^{1,-1}_{\beta,\vec{0}}, I_*\partial_{[1,N]}\right] = \left[\hat\sz^{1,-1}_{(\beta,\vec{0})}, I_*\partial_{[1,N]}\right]. \]
  Use Lemma~\ref{wc:lemma:pushforward-vertical-wall} below to push the
  classes on both sides down to $\tilde\fM^{\vec
    Q(\Fr),\pl}_{\beta,\vec 0}$, to get
  \[ \tilde\sz^{1,-1}_{\beta,\vec{0}} \cdot [\fr(\beta)]_\kappa = \hat\sz^{1,-1}_{\beta,\vec{0}} \cdot [\fr(\beta)]_\kappa, \]
  which finishes the proof.
\end{proof}

\subsubsection{}

\begin{lemma}\label{wc:lemma:pushforward-vertical-wall}
  Let $(\beta, \vec e) \in S_\alpha$ be a flag. Assume that $e_N \ge
  1$ and let $1 \le a \le N$ be the minimal index such that $e_a =
  e_N$. Let
  \[ \pi\colon \tilde\fM^{\vec Q(\Fr)}_{\beta,\vec e} \to \tilde\fM^{\vec Q(\Fr)}_{\beta,\vec e-\vec 1_{[a,N]}} \]
  be the morphism of Definition~\ref{wc:def:flag-stripping}. Then
  \[ \pi_*\left[\phi, I_*\partial_{[a,N]}\right] = [\fr(\beta) - e_{a-1}]_\kappa \cdot \phi. \]
\end{lemma}

\begin{proof}
  The proof is completely analogous to that of
  Lemma~\ref{lem:pushforward-of-bracket-partial}, but using the
  commutative diagram
  \[ \begin{tikzcd}
    \tilde\fM^{\vec Q(\Fr)}_{\beta,\vec e-\vec 1_{[a,N]}} \times \tilde\fM^{\vec Q(\Fr)}_{0,\vec 1_{[a,N]}} \ar{r}{\Phi} \ar{dr}[swap]{\pr_1} & \tilde\fM^{\vec Q(\Fr)}_{\beta,\vec e} \ar{d}{\pi} \\
    {} & \tilde\fM^{\vec Q(\Fr)}_{\beta,\vec e-\vec 1_{[a,N]}}
  \end{tikzcd} \]
  and that $\rank \tilde \scE^{\vec Q(\Fr)}_{(\beta, \vec e-\vec
    1_{[a,N]}),(0,\vec 1_{[a,N]})} = \fr(\beta) - e_{a-1}$.
\end{proof}

\section{Additional features}
\label{sec:additional-features}

\subsection{(Co)homological version}
\label{sec:cohomological-version}

\subsubsection{}

In this subsection, we prove cohomological analogues of the main
Theorems~\ref{thm:sst-invariants} and \ref{thm:wcf}.

The main and only technical point is to provide a suitable {\it
  equivariant homology theory} $A_*^\sT(-)$ for $\fX \in
\cat{Art}_\sT$, i.e. Artin stacks locally of finite type (over any
base field $k$, not just $\bC$). By ``suitable'' we mean that all
results in \S\ref{sec:background} concerning the K-homology group
$K_\circ^\sT(-)$ must have analogues with $A_*^\sT(-)$. We give one
possibility for $A_*^\sT(-)$ in
Definition~\ref{def:homology-theory-operational} --- a certain dual of
equivariant operational Chow cohomology --- and state its relevant
analogous properties in
\S\ref{sec:homology-basic-properties}--\S\ref{sec:homology-enumerative-invariants}.
The main Theorems~\ref{thm:sst-invariants} and \ref{thm:wcf} are then
stated and proved by replacing all the K-theoretic pieces in
\S\ref{sec:semistable-invariants} and \S\ref{sec:wall-crossing},
respectively, by their (co)homological analogues. This adaptation is
done in Theorems~\ref{thm:sst-invariants-cohom} and
\ref{thm:wcf-cohom}.

\subsubsection{}

Throughout (cf. Definition~\ref{def:equivariant-k-theory}), let
\[ \bh_\sT^* \coloneqq \bZ[\Char\sT] = \bZ[\zeta_\mu : \mu \in \Char\sT] \]
be the group ring of the character lattice of $\sT$, where the
generators $\zeta_\mu$ satisfy $\zeta_{\mu} + \zeta_{\mu'} =
\zeta_{\mu+\mu'}$ and lie in degree $1$. Canonically, $\Spec
\bh_\sT^*$ is the Lie algebra of $\sT$ and therefore $\bh_\sT^*$ is
the $\sT$-equivariant cohomology of a point for any reasonable
equivariant cohomology theory. Let $\bh_{\sT,\loc}^* \coloneqq \Frac
\bh_\sT^*$.

We emphasize that, unlike in K-theory, in the cohomological version of
our machinery there is no symmetrization of characteristic classes
(like in \eqref{eq:k-theoretic-wedge-symmetrized}) or virtual cycles
(like in \eqref{eq:symmetrized-virtual-sheaf}). Consequently there is
no need to replace $\sT$ with a double cover $\tilde\sT
\twoheadrightarrow \sT$ where $\kappa^{1/2}$ exists.

\subsubsection{}

Within the framework of this paper, we believe the simplest way to
construct the desired functor $A_*^\sT(-)$ is to begin with an {\it
  equivariant Chow cohomology theory} $A_\sT^*(-)$ and then apply the
constructions of \S\ref{sec:equivariant-k-homology} using $A_\sT^*(-)$
in place of $K_\sT^\circ(-)$. By ``Chow cohomology'' we mean a functor
which is to singular cohomology as Chow homology $\CH_*^\sT(-)$
\cite{Kresch1999} is to Borel--Moore homology. Recall that
$\CH_{-p}^\sT(\pt) \cong \bh_\sT^p$.

For $\fX \in \cat{Art}_\sT$, take $A_\sT^*(\fX)$ to be the natural
$\sT$-equivariant version of {\it operational Chow cohomology}
\cite[Appendix C]{Bae2022}. \footnote{This is the natural
generalization to Artin stacks of the bivariant Chow groups of
\cite[\S 17.1]{Fulton1998} for the identity morphism $\id\colon \fX
\to \fX$.} Namely, an element $c \in A_\sT^p(\fX)$ is a collection of
homomorphisms
\[ c_g\colon \CH_*^\sT(\fY) \to \CH_{*-p}^\sT(\fY), \]
for all $m \in \bZ$ and all morphisms $g\colon \fY \to \fX$ in
$\cat{Art}_\sT$ where $\fY$ is of finite type and stratified by global
quotient stacks, and these homomorphisms must be compatible with
representable proper pushforwards, flat pullback, and refined Gysin
pullback along representable lci morphisms. We will sometimes use the
conventional notation
\[ c \cap \omega \coloneqq c_g(\omega) \qquad \text{for } \omega \in \CH_*^\sT(\fY). \]
Let $A_\sT^*(X)_\bQ$ and $A_\sT^*(X)_{\loc}$ denote the analogous
operational Chow cohomology groups after base change from $\bh_\sT^*$
to $\bh_\sT^* \otimes_{\bZ} \bQ$ and $\bh_{\sT,\loc}^*$ respectively.

We claim that $\CH_*^\sT(-)$ and $A_\sT^*(-)$ are appropriate
cohomological analogues of $K_\sT(-)$ and $K_\sT^\circ(-)$
respectively.

\subsubsection{}
\label{sec:homology-basic-properties}

We review some basic properties of $A_\sT^*(-)$ and $\CH_*^\sT(-)$.
This is the cohomological analogue of
\S\ref{sec:k-theory-basic-properties}. We will only consider
$\CH_*^\sT(-)$ for Artin stacks stratified by global quotient stacks.
\begin{itemize}
\item Composition naturally makes $A_\sT^*(\fX)$ into a graded
  $\bh_\sT^*$-algebra, and by definition $\CH_*^\sT(\fX)$ is a graded
  $A_\sT^*(\fX)$-module.

\item If $\fX$ is equidimensional and stratified by global quotient
  stacks, then there is a well-defined homomorphism $- \cap
  [\fX]\colon A_\sT^*(\fX) \to \CH_{\dim \fX-*}^\sT(\fX)$. If $\fX$ is
  smooth, then this induces an isomorphism $A_\sT^*(\fX)_{\bQ} \cong
  \CH_{\dim \fX-*}^\sT(\fX)_\bQ$. If moreover $\fX$ is a scheme, then
  the base change to $\bQ$ is unnecessary. So in particular
  $A_\sT^*(\pt) \cong \bh_\sT^*$.

\item A $\sT$-equivariant morphism $f\colon \fX \to \fY$ induces a
  functorial pullback $f^*\colon A_\sT^p(\fY) \to A_\sT^p(\fX)$. If
  $f$ is proper and representable, there is a functorial pushforward
  $f_*\colon \CH_*^\sT(\fX) \to \CH_*^\sT(\fY)$ satisfying the {\it
    projection formula}
  \[ c \cap f_*\omega = f_*(f^*c \cap \omega), \qquad \omega \in \CH_*^\sT(\fX), \, c \in A_\sT^*(\fY). \]
  If in addition $f$ is flat, there is a functorial pullback
  $f^*\colon \CH_*^\sT(\fY) \to \CH_{*+\dim f}^\sT(\fX)$ and a
  functorial pushforward $f_*\colon A_\sT^*(\fX) \to A_\sT^{*+\dim
    f}(\fY)$ satisfying the {\it projection formula}
  \[ f_*c \cap \omega = f_*(c \cap f^*\omega), \qquad \omega \in \CH_*^\sT(\fY), \, c \in A_\sT^*(\fX). \]
  In particular, both $f^*$ and $f_*$ are $\bh_\sT^*$-linear.
\end{itemize}

\subsubsection{}

We provide the cohomological analogues of characteristic classes, the
residue map, and the projective bundle formula. This is the
cohomological analogue of the remainder of \S\ref{sec:k-theory}.
\begin{itemize}
\item (Chern/Segre classes, cf.
  Definition~\ref{def:k-theoretic-wedge}) Given $\cE \in
  \cat{Vect}_\sT(\fX)$ of rank $r$ and a formal variable $u$, let
  \[ c_u(\cE) \coloneqq \sum_i u^i c_i(\cE) = \prod_\cL (1 + u c_1(\cL)) \in A_\sT^*(\fX)[u] \]
  be its {\it total Chern class} in operational Chow cohomology, where
  $c_i(\cE)_g\colon \CH_*^\sT(\fY) \to \CH_{*-i}^\sT(\fY)$ is the
  operator $c_i(g^*\cE) \cap -$. Its power series inverse is the {\it
    total Segre class}
  \[ s_u(\cE) \eqqcolon \sum_i u^i s_i(\cE) \in A_\sT^*(\fX)\pseries*{u}, \]
  and therefore $c_u(-)$ extends to $K_0(\cat{Vect}_\sT(\fX))$ by
  defining $c_u(\cE_1 - \cE_2) \coloneqq c_u(\cE_1) s_u(\cE_2)$.

\item (Euler class) In the setting of
  Definition~\ref{def:k-theoretic-euler-class}, the {\it cohomological
    Euler class} is
  \[ e_u(\cE) \coloneqq \prod_\mu (u + \zeta_\mu)^{\rank \cE_\mu} c_{-(u + \zeta_\mu)^{-1}}(\cE_\mu) \in A_\sT^*(\fX)[u]\pseries*{(u + \zeta_\mu)^{-1} : \mu \in \Char(\sT)}. \]
  The filtration by powers of $I^\circ(X)$ on $K(X)$ is replaced by
  the dimension grading on $\CH_*(X)$, so the analogue of
  Lemma~\ref{lem:k-theory-finiteness} is that $\CH_N(X) = 0$ for all
  $N < 0$ if $X$ is a finite-type algebraic space, and, consequently,
  for $\cE \in K_0(\cat{Vect}_\sT(X))$,
  \[ e_u(\cE) \in A_\sT^*(X)[u][(u + \zeta_\mu)^{-1} : \mu \in \Char(\sT)]. \]
  Its expansion as a formal power series around $u^{-1} = 0$ is
  \begin{equation} \label{eq:cohomological-euler-class-expansion}
    u^{\rank \cE_1 - \rank \cE_2} \sum_{i,j \ge 0} u^{-i-j} c_i(\cE_1) \cdot s_j(\cE_2) \in A_\sT^*(X)\lseries*{z^{-1}}.
  \end{equation}
\end{itemize}

\subsubsection{}

\begin{definition}[cf. Definition~\ref{def:k-theoretic-residue-map}]
  Given a rational function $f \in \bh_{\sT \times \bC^\times,\loc}^*$
  with the cohomological weight of the $\bC^\times$ denoted $u$, let
  $f_- \in \bh_{\sT,\loc}^*\lseries*{u^{-1}}$ be its formal series
  expansion around $u = \infty$. The {\it ($\sT$-equivariant)
    cohomological residue map} is the $\bh_{\sT,\loc}^*$-module
  homomorphism
  \begin{align*}
    \rho\colon \bh_{\sT \times \bC^\times,\loc}^* &\to \bh_{\sT,\loc} \\
    f &\mapsto u^{-1} \text{ term in } f_-.
  \end{align*}
  Treating cohomological $\sT$-weights as generic non-zero complex
  numbers, this is equivalent to
  \[ \rho(f) = \Res_{u=\infty}(f \, du) = \sum_{p \in \bC} \Res_{u=p} (f \, du) \]
  where the second equality is the residue theorem. We write $\rho_u$
  to emphasize that the residue is being taken in the variable $u$.
\end{definition}

\subsubsection{}

\begin{lemma}[Cohomological projective bundle formula] \label{lem:cohomological-projective-bundle-formula}
  Take the setting of Lemma~\ref{lem:projective-bundle-formula}, with
  $h \coloneqq c_1(s) \in A_\sT^*(\bP(\cV))$.
  \begin{enumerate}[label = (\roman*)]
  \item There is an isomorphism of $\CH_*^\sT(X)$-modules
    \[ \CH_*^\sT(\bP(\cV)) \cong \bigoplus_{k=0}^{\rank \cV - 1} (h^k \cap \pi^*\CH_*^\sT(X)); \]
  \item If $\sT$ acts trivially on $X$, then for any $f(h) \in
    A_\sT^*(X)[h]$,
    \begin{equation} \label{eq:cohomological-projective-bundle-formula}
      \pi_* f(h) = \rho_u \frac{f(u)}{e_u(\cV)}.
    \end{equation}
  \end{enumerate}
\end{lemma}

\begin{proof}
  Adapt the proof of Lemma~\ref{lem:projective-bundle-formula}. By
  linearity, it suffices to take $f(h) = h^k$ for $k \in \bZ_{\ge 0}$.
  The analogue of
  \eqref{eq:projective-bundle-tautological-pushforward} for the left
  hand side of \eqref{eq:cohomological-projective-bundle-formula} is
  \[ \pi_*(h^k) \cap \omega = \pi_*(h^k \cap \pi^*\omega) = s_{k - \rank \cV + 1}(\cV) \cap \omega, \]
  where the first equality is the projection formula and the second is
  the definition of the Segre class \cite[\S 3.1]{Fulton1998}. Since The
  analogue of \eqref{eq:projective-bundle-residue} for the right hand
  side of \eqref{eq:cohomological-projective-bundle-formula} is
  \[ u^{-1} \text{ term in } u^k \cdot u^{-\rank \cV} \sum_{i \ge 0} u^{-i} s_i(\cV) = s_{k - \rank \cV + 1}(\cV) \]
  using the expansion \eqref{eq:cohomological-euler-class-expansion}.
  We see that these match up exactly. 
\end{proof}

\subsubsection{}

\begin{definition} \label{def:homology-theory-operational}
  We provide an equivariant homology theory which requires minimal
  changes from the ``concrete'' K-homology theory of
  Definition~\ref{def:concrete-k-homology}. Let $A_*^\sT(\fX)_{\loc}$
  be the set of all triples $\phi = (Z_\phi, \fix_\phi, \omega_\phi)$
  where:
  \begin{itemize}
  \item $Z_\phi$ is a proper algebraic space with the resolution
    property;
  \item $\fix_\phi\colon Z_\phi \to \fX$ is a $\sT$-equivariant
    morphism for the trivial $\sT$-action on $Z_\phi$;
  \item $\omega_\phi \in \CH_*^\sT(Z_\phi)_{\loc}$ is an element in
    Chow homology.
  \end{itemize}
  Equip it with the group operation where $\fix_{\phi+\psi}$ is the
  obvious map from $Z_{\phi+\psi} \coloneqq Z_\phi \sqcup Z_\psi$ and
  $\omega_{\phi+\psi} \coloneqq \omega_\phi + \omega_\psi$, and the
  $\bh_\sT^*$-module structure given by $\zeta \cdot \phi \coloneqq
  (Z_\phi, \fix_\phi, \zeta \cdot \omega_\phi)$ for $\zeta \in
  \bh_\sT^*$. View $\phi \in A_p^\sT(\fX)_{\loc}$ as a
  $\bh_\sT^*$-linear homomorphism
  \begin{align*}
    \phi\colon A_\sT^*(\fX)_{\loc} &\to \bh_{\sT,\loc}^* \\
    c &\mapsto \int c \cap \omega_\phi \coloneqq \int c_{\fix_\phi}(\omega_\phi),
  \end{align*}
  where $\int$ denotes (proper) pushforward from $Z_\phi$ to a point,
  and identify two such triples if their associated homomorphisms are
  equal.
\end{definition}

\subsubsection{}

We verify that $A_*^\sT(-)_{\loc}$ is a {\it commutative
  $\sT$-equivariant operational homology theory}, meaning that it
satisfies the following properties. This is the analogue of
Definition~\ref{def:operational-k-homology} and
Proposition~\ref{prop:actual-k-homology}. Given $\phi \in
A_*^\sT(\fX)_{\loc}$ and a base $S \in \cat{Art}_\sT$, define (cf.
\eqref{eq:actual-k-homology-elements})
\begin{align*}
  \phi_S\colon A_\sT^*(\fX \times S)_{\loc} &\to A_\sT^*(S)_{\loc} \\
  c &\mapsto (\pi_S)_*(c_{\fix_\phi \times \id}(\pi_Z^*\omega_\phi))
\end{align*}
where $\pi_Z$ and $\pi_S$ are the projections from $Z_\phi \times S$
to $Z_\phi$ and $S$ respectively. Replace the $\bk_\sT$-algebra
$K_\sT^\circ(-)$ by the graded $\bh_\sT^*$-algebra $A_\sT^*(-)$, and
the $I^\circ$-adic filtration on $K^\circ(\fF \times S)$ by the
dimension grading on $A^*(-)$ (for the finiteness axiom). Then:
\begin{itemize}
\item the collection $\{\phi_S\}_S$, for all $S \in \cat{Art}_\sT$,
  satisfies all axioms in Definition~\ref{def:operational-k-homology};
  
\item $A_*^\sT(-)$ has the same structures as $\bK_\circ^\sT(-)$ in
  \S\ref{sec:k-homology-properties};

\item writing $A_\sT^*([\pt/\bC^\times]) \cong \bh_{\sT \times
  \bC^\times}^* = \bh_\sT^*[x]$, we have
  \begin{equation} \label{eq:homology-BGm}
    A_*^\sT([\pt/\bC^\times]) = \bZ[\zeta], \qquad \zeta^m(x^n) \coloneqq \delta_{m,n}
  \end{equation}
  by the same argument as in \cite[Theorem 2.3.2, Prop.
    2.3.5]{liu_eq_k_theoretic_va_and_wc} using that $\zeta^m = \int -
  \cap [\bP^k_\bC]$ for the map $\fix_{\zeta^m} \coloneqq \bP^k_\bC =
       [\bC^{m+1} \setminus \{0\}/\bC^\times] \hookrightarrow
       [\bC^{m+1}/\bC^\times] \to [\pt/\bC^\times]$.
\end{itemize}

\subsubsection{}

We provide the cohomological analogue of the vertex algebra of
\S\ref{sec:vertex-algebra}.
\begin{itemize}
\item (Series rings, cf. \S\ref{sec:series-rings}) For formal
  variables $u_1, \ldots, u_m$ and a $\bh_\sT^*$-module $M$, define
  \[ M\left[(u_1 + u_2 + \cdots + u_m)^{-1}\right]_\sT \coloneqq M\left[(i_1 u_1 + \cdots + i_m u_m + \zeta_\mu)^{-1} : \begin{array}{c} \mu \in \Char(\sT) \\ i_1,\ldots,i_m \in \bZ \setminus \{0\}\end{array}\right]. \]
  In the case of a single variable $u$, let $M\lseries*{u}_{\sT}
  \coloneqq M\pseries*{u}[u^{-1}]_\sT$.

\item (Equivariant expansion, cf. Definition~\ref{def:expansion}) For
  two formal variables $u$ and $v$, the {\it (additive) expansion}
  $\iota_{u,v}$ is the injective ring homomorphism
  \begin{align*}
    \iota_{u,v}\colon \bZ[(u - v)^{-1}] &\hookrightarrow \bZ[v^{-1}]\pseries*{u} \\
    (u - v)^k &\mapsto v^k (u - 1)^k.
  \end{align*}
  Since $1 - u \in \bZ\pseries*{u}$ is invertible, this is
  well-defined for all $k \in \bZ$. The {\it (additive) equivariant
    expansion}
  \[ \iota_{u,v}^\sT\colon \bh_\sT^*[(u - v)^{-1}]_\sT \to \bh_\sT^*[v^{-1}]_\sT\pseries*{u} \]
  is the $\bh_\sT^*$-algebra homomorphism given by applying the expansion
  \[ \iota_{iu, \zeta_\mu + jv}\colon \bh_\sT^*\left[(iu + jv + \zeta_\mu)^{-1}\right] \to \bh_\sT^*\left[(jv + \zeta_\mu)^{-1}\right]\pseries*{u^{-1}} \]
  to the monomial $(iu + jv + \zeta_\mu)^k$.
\end{itemize}

\subsubsection{}

\begin{definition}[cf. Definition~\ref{def:multiplicative-vertex-algebra}]
  A {\it $\sT$-equivariant additive vertex algebra} is the data of:
  \begin{enumerate}[label = (\roman*)]
  \item a $\bh_\sT^*$-module $V$ of {\it states} with a distinguished
    vacuum vector $\vac \in V$;
  \item an additive {\it translation operator} $D(u)
    \in \End(V)\pseries*{u}$, i.e. $D(u)D(v) = D(u+v)$;
  \item a {\it vertex product} $Y(-, u)\colon V \otimes V \to
    V\lseries{u}_\sT$.
  \end{enumerate}
  This data must satisfy the following axioms:
  \begin{enumerate}[label = (\alph*)]
  \item (vacuum) $Y(\vac, u) = \id$ and $Y(a, u)\vac \in V\pseries{u}$
    with $Y(a,0)\vac = a$;
  \item (skew symmetry) $Y(a, u)b = D(u) Y(b, -u) a$;
  \item (weak associativity) $Y(Y(a,u) b, v) \equiv Y(a, u+v) Y(b,
    v)$, where $\equiv$ means that when applied to any $c \in V$, both
    sides are additive equivariant expansions of the same element in
    \[ V\pseries*{u, v}\left[u^{-1}, v^{-1}, (u-v)^{-1}\right]_{\sT}. \]
  \end{enumerate}
\end{definition}

\subsubsection{}

\begin{theorem} \label{thm:cohVOA-monoidal-stack}
  In the setting of Theorem~\ref{thm:mVOA-monoidal-stack},
  \[ A_*^\sT(\fM) \coloneqq \bigoplus_\alpha A_*^\sT(\fM_\alpha) \]
  has the structure of a $\sT$-equivariant additive vertex algebra.
  \begin{itemize}
  \item The vacuum $\vac \in A_*^{\sT}(\fM_0)$ is given by
    the identity map $\bh_\sT^* \to \bh_\sT^*$.

  \item The translation operator is
    \[ D(u)\phi \coloneqq \sum_{k \ge 0} u^k \Psi_*\left(\zeta^k \boxtimes \phi\right) \in A_*^\sT(\fM)\pseries{u} \]
    where $A_*^\sT([\pt/\bC^\times]) = \bh_\sT^*[\zeta]$ as in
    \eqref{eq:homology-BGm}. Explicitly, $(D(u)\phi)(c) = \phi(\deg_u
    c)$ where $\deg_u$ is the \emph{degree operator}
    \[ \deg_u\colon A_\sT^*(\fM) \xrightarrow{\Psi^*} A_\sT^*([\pt/\bC^\times] \times \fM) \cong A_\sT^*(\fM)[u] \]
    associated to $\Psi$ (cf. Definition~\ref{def:degree-map}). Here
    we identify $A_\sT^*([\pt/\bC^\times]) = A_*^\sT([\pt/\bC^\times])
    = \bh_\sT^*[u]$.

  \item The vertex product is given on $\phi \in A_*^\sT(\fM_\alpha)$
    and $\psi \in A_*^\sT(\fM_\beta)$ by
    \[ \begin{aligned}
        Y(\phi, u) \psi
        &\coloneqq (\Phi_{\alpha,\beta})_* (D(u) \times \id) \left((\phi \boxtimes \psi) \cap \Theta_{\alpha,\beta}(u)\right) \\
        \CH_\sT^*(\fM_{\alpha+\beta}) \ni \zeta
        &\mapsto (\phi \boxtimes \psi)\left(\Theta_{\alpha,\beta}(u) \cup (\deg_u \times \id) \Phi_{\alpha,\beta}^* \zeta \right) \in \bh_\sT^*[u^\pm]_\sT
      \end{aligned} \]
    where (see \S\ref{sec:monoidal-stack-vertex-product-well-defined}
    for details)
    \[ \Theta_{\alpha,\beta}(u) \coloneqq e_{-u}\left(\scE_{\alpha,\beta}\right) \cup e_u\left((12)^*\scE_{\beta,\alpha}\right). \]
  \end{itemize}
\end{theorem}

This applies to the auxiliary stacks in exactly the same way as in
Theorem~\ref{thm:auxiliary-stack-vertex-algebra}.

\subsubsection{}
\label{sec:coh-lie-algebra}

We provide the cohomological analogue of the Lie algebra and rigidity
calculation of \S\ref{sec:vertex-algebra}.
\begin{itemize}
\item (Homology $\pl$ groups, cf. Definition~\ref{def:pl-groups})
  Given a $\sT$-equivariant additive vertex algebra $(V, \vac, D, Y)$,
  let
  \[ V^\pl \coloneqq V / \im(1 - D(u)). \]
  Lemmas~\ref{lem:k-homology-pl} and \ref{lem:pl-group-functoriality}
  still hold as they rely only on formal properties of $A_*^\sT(-)$
  and $A_\sT^*(-)$.

\item (Lie algebra, cf.
  Theorem~\ref{thm:mVOA-monoidal-stack-lie-algebra})
  Theorem~\ref{thm:cohVOA-monoidal-stack} induces a Lie algebra
  structure on $A_*^\sT(\fM)^\pl$, with Lie bracket given by
  \begin{align}
    -\hbar \cdot [\phi, \psi](c)
    &\coloneqq \rho_u\left(Y(\tilde\phi, u) \tilde\psi\right)(c) \nonumber \\
    &= \rho_u\left[ (\tilde\phi \boxtimes \tilde\psi)\left(\Theta_{\alpha,\beta}(u) \otimes (\deg_u \times \id) \Phi_{\alpha,\beta}^*c\right)\right] \label{eq:monoidal-stack-cohomological-lie-bracket}
  \end{align}
  for $\phi \in A_*^\sT(\fM_\alpha)^\pl$ and $\psi \in
  A_*^\sT(\fM_\beta)^\pl$, and $\tilde\phi \in A_*^\sT(\fM_\alpha)$
  and $\tilde\psi \in A_*^\sT(\fM_\beta)$ are any lifts of $\phi$ and
  $\psi$ respectively.
  
\item (Rigidity, cf. Proposition~\ref{prop:rigidity}) Let $1 \in
  A_\sT^*(\fX)$ denote the identity operator. For any $\phi \in
  A_*^\sT(\fM_\alpha)^\pl$ and $\psi \in A_*^\sT(\fM_\beta)^\pl$,
  \[ [\phi, \psi](1) = (\phi \boxtimes \psi)\left( (-1)^{\rank(\scE_{\alpha,\beta})-1} \rank(\scE_{\alpha,\beta}) \cdot (1 \boxtimes 1) \right). \]
  This follows from the same computation as in the proof of
  Lemma~\ref{lem:cohomological-projective-bundle-formulas}\ref{it:rigidity}.
\end{itemize}

\subsubsection{}
\label{sec:homology-enumerative-invariants}

We provide the cohomological analogues of virtual cycles, enumerative
invariants, and the homology projective bundle formula.
\begin{itemize}
\item (Virtual cycle) In the setting of
  Definition~\ref{def:symmetrized-virtual-structure-sheaf}, let
  $[X]^{\vir} \in \CH_{\vdim X}^\sT(X)$ be the {\it virtual
    fundamental class} induced by the (A)POT on $X$
  \cite{kiem_savvas_apot, kiem_savvas_loc}, \footnote{The equivariant
    K-theoretic machinery of Kiem and Savvas continues to work in
    equivariant Chow homology, see e.g. \cite[Theorem 1.1]{Chang2011}.
    Note that an APOT induces a semi-perfect obstruction theory \cite[Appendix
      A]{kiem_savvas_apot}.} satisfying the $\sT$-equivariant
  localization formula 
  \[ [X]^\vir = \iota_* \frac{[X^\sT]^\vir}{e(\cN_\iota^\vir)}. \]
  Here $\vdim X$ is the {\it virtual dimension} of (the (A)POT on)
  $X$. Note that there is no symmetrization in cohomology, i.e.
  $[X]^\vir$ and $e(-)$ are the direct analogues of the K-theoretic
  $\hat\cO_X^\vir$ and $\hat\se(-)$ respectively.

\item (Universal enumerative invariant) In the setting of
  Definition~\ref{bg:def:univ-enum-inv}, let
  \[ \sZ_X \coloneqq \int - \cap [X]^\vir \coloneqq \int \frac{\iota^*(-)}{e(N_\iota^\vir)} \cap [X^{\sT}]^{\vir} \in A_*^\sT(X)_{\loc} \]
  where $\iota\colon X^\sT \hookrightarrow X$ is the $\sT$-fixed
  locus. In general, all Euler characteristics $\chi(X, \hat\cO_X^\vir
  \otimes -)$ must be replaced by $\int - \cap [X]^\vir$. For
  instance, the pole cancellation
  Lemma~\ref{lem:master-space-no-poles} becomes
  \[ \int \omega \cap [M]^\vir \in \bh_{\sT,\loc}^* \otimes_{\bZ} \bh_{\sS}^* \]
  for any $\omega \in A_{\sT \times \sS}^*(M) \otimes_{\bh_\sT^*}
  \bh_{\sT,\loc}^*$.

\item (Symmetrized pullback on homology) In the setting of
  Definition~\ref{def:symm-pullback-homology}, the {\it
    $\kappa$-symmetrized pullback on homology} is
  \begin{align*}
    \hat\pi^*\colon A_*^\sT(\fY) &\to A_*^\sT(\fX) \\
    \phi &\mapsto \pi^*\phi \cap \left(e(\kappa^{-1} \otimes \bL_\pi)\right).
  \end{align*}
  In the setting of \S\ref{bg:rem:symm-pb-formula-alg-sp}, the
  analogue of \eqref{eq:symm-pullback-Ovir} is
  \begin{equation} \label{eq:symm-pullback-Xvir}
    [X]^\vir = e(\kappa^{-1} \otimes \bL_\pi) \cap \pi^*[Y]^\vir.
  \end{equation}
\end{itemize}

\subsubsection{}

\begin{lemma} \label{lem:cohomological-projective-bundle-formulas}
  \begin{enumerate}[label = (\roman*)]
  \item (Virtual projective bundle formula) In the setting of
    Lemma~\ref{lem:symmetrized-projective-bundle-formula}, with $h
    \coloneqq c_1(s) \in A_\sT^*(\bP(\cV))$,
    \begin{equation} \label{eq:cohomological-virtual-projective-bundle-formula}
      -\hbar \cdot \pi_*(f(h) \cap [\bP(\cV)]^\vir) = \rho_u\left(f(u) \frac{e_{-u}(\kappa^{-1} \cV^\vee)}{e_u(\cV)}\right) \cap [X]^\vir
    \end{equation}
    where $\hbar \coloneqq c_1(\kappa) \in \bh_\sT^*$.

  \item \label{it:rigidity} (cf.
    Corollary~\ref{cor:symmetrized-projective-bundle-formula}) In
    particular,
    \[ \pi_* [\bP(\cV)]^\vir = (-1)^{\rank(\cV)-1} \rank(\cV) \cdot [X]^\vir \in \CH_*^\sT(X). \]

  \item (Homology projective bundle formula) In the setting of
    Lemma~\ref{lem:homology-projective-bundle-formula}, given $\phi
    \in A_*^\sT(X)$,
    \[ -\hbar \cdot j_*\hat\pi_X^*\phi = \rho_u\left(i_*D(u) \left(\tilde\phi \cap \frac{e_{-u}(\kappa^{-1} \cV^\vee)}{e_u(\cV)}\right)\right) \]
    where $\tilde\phi \in A_*^\sT(\fX)$ is an arbitrary lift of $\phi$
    (i.e. $\phi = (\Pi_X)_* \tilde\phi$).
  \end{enumerate}
\end{lemma}

\begin{proof}
  The proof of Lemma~\ref{lem:symmetrized-projective-bundle-formula}
  continues to hold using \eqref{eq:symm-pullback-Xvir} and the
  cohomological projective bundle formula
  (Lemma~\ref{lem:cohomological-projective-bundle-formula}).

  Here is the cohomological analogue of the rigidity computation in
  Corollary~\ref{cor:symmetrized-projective-bundle-formula}: write
  \begin{align*}
    \frac{e_{-u}(\kappa^{-1} \cV^\vee)}{e_u(\cV)}
    = \prod_\omega \frac{-u - \hbar - \omega}{u + \omega}
    &= (-1)^{\rank \cV} \prod_\omega \left( 1 + \frac{\hbar}{u + \omega} \right) \\
    &= (-1)^{\rank \cV} \left(1 + \hbar \cdot \rank(\cV) \cdot u^{-1} + O(u^{-2})\right)
  \end{align*}
  where $\omega$ ranges over all (cohomological) equivariant Chern
  roots of $\cV$. The result follows by applying $\rho_u$ to both
  sides.

  The proof of Lemma~\ref{lem:homology-projective-bundle-formula}
  continues to hold, using
  \eqref{eq:cohomological-virtual-projective-bundle-formula} and that
  $A_\sT^*(-)$ has a Thom isomorphism theorem
  (Lemma~\ref{lem:operational-chow-cohomology-thom-isomorphism}). This
  was essentially already observed in \cite[Example
    17.5.1]{Fulton1998} but we sketch the proof below for
  completeness.
\end{proof}

\subsubsection{}

\begin{lemma}[Thom isomorphism for $A_\sT^*(-)$] \label{lem:operational-chow-cohomology-thom-isomorphism}
  If $\cV \in \cat{Vect}_\sT(\fX)$ with projection $\pi\colon
  \tot(\cV) \to \fX$, then there is an isomorphism
  \[ \pi^*\colon A_\sT^*(\fX) \xrightarrow{\sim} A_\sT^*(\tot(\cV)). \]
\end{lemma}

\begin{proof}
  Let $i\colon \fX \to \cV$ be the zero section. Then clearly $i^*
  \circ \pi^* = \id$ so $\pi^*$ is injective. Now we show $\pi^*$ is
  surjective. Given a morphism $g\colon \fY \to \fX$, form the
  Cartesian square
  \[ \begin{tikzcd}
    \tot_{\fY}(g^*\cV) \ar{r}{g'} \ar{d}{\pi''} & \tot_{\fX}(\cV) \ar{d}{\pi} \\
    \fY \ar{r}{g} & \fX
  \end{tikzcd} \]
  and, given $c \in A_\sT^*(\tot(\cV))$, define $\bar c \in
  A_\sT^*(\fX)$ by
  \[ \bar c_g(\omega) \coloneqq ((\pi'')^*)^{-1} c_{g'}((\pi'')^*\omega) \in \CH_*^\sT(\fY) \]
  for $\omega \in \CH_*^\sT(\fY)$. This is well-defined because
  $(\pi'')^*\colon \CH_*^\sT(\fY) \xrightarrow{\sim}
  \CH_*^\sT(\tot(g^*\cV))$ is an isomorphism by the usual Thom
  isomorphism theorem. We claim $c = \pi^* \bar c$. Unraveling
  notation, this boils down to the claim that
  \[ (\pi'')^* c_f(\omega) = c_{\pi' \circ f'}((\pi'')^*\omega) \]
  for a given morphism $f\colon \fY \to \tot(\cV)$ and the induced
  Cartesian squares
  \[ \begin{tikzcd}
    \tot(f^*\pi^*\cV) \ar{d}{\pi''} \ar{r}{f'} & \tot(\pi^*\cV) \ar{d}{\pi'} \ar{r}{\pi'} & \tot(\cV) \ar{d}{\pi} \\
    \fY \ar{r}{f} & \tot(\cV) \ar{r}{\pi} & \fX.
  \end{tikzcd} \]
  By definition the left hand side is $c_{\pi'' \circ
    f}((\pi'')^*\omega)$, so the claim holds by commutativity of the
  left-most square.
\end{proof}

\subsubsection{}

\begin{remark} \label{rem:khan-homology}
  As an alternative to our $A_*^\sT(-)$, the standard six-functor
  formalism for the derived $\infty$-category of sheaves on stacks may
  be used to construct an equivariant homology functor $H_*^\sT(-)$
  \cite{Khan2025}. For (topological) Artin stacks, this may be viewed
  as an equivariant version of singular homology. Although we have not
  checked this thoroughly, we believe it functions equally well as a
  cohomological analogue of our $K_\circ^\sT(-)$.

  Khan also explains how to refine this to a Chow-type functor ---
  ``equivariant motivic homology'' --- by upgrading to the six-functor
  formalism for motivic sheaves. However, for this paper, we have
  instead chosen to use our much more low-tech
  Definition~\ref{def:homology-theory-operational} for $A_*^\sT(-)$
  due to the relative familiarity and maturity of the theory of
  (operational) Chow groups.
\end{remark}

\subsubsection{}

\begin{theorem} \label{thm:sst-invariants-cohom}
  Suppose $\tau$ is a stability condition on $\cat{A}$ for which
  Assumption~\ref{assump:semistable-invariants} holds. Then there
  exists a unique collection
  \begin{equation}
    \left(\sz_{\alpha}(\tau)\in A_*^{\sT}(\fM_{\alpha})^\pl_{\loc,\bQ}\right)_{\alpha\in C(\cat{A})}
  \end{equation}
  of homology classes satisfying the same properties as in
  Theorem~\ref{thm:sst-invariants}, with the modifications:
  \begin{enumerate}[label = (\roman*')]
    \addtocounter{enumi}{1}

  \item for any $\alpha$ for which $\fM_\alpha^{\st}(\tau) =
    \fM_{\alpha}^{\sst}(\tau)$,
    \[ (\Pi^\pl_\alpha)_*\sz_{\alpha}(\tau) = \int - \cap [\fM^{\sst}_{\alpha}(\tau)]^{\vir}; \]

    \addtocounter{enumi}{1}

  \item for any framing functor $\Fr \in \Frs$, in the notation of
    Definition~\ref{def:pair-invariant} and \S\ref{sec:pairs-maps},
    \[ I_*\tilde\sZ_{\alpha,1}^{\Fr}(\tau^Q) = \sum_{\substack{n>0 \\ \alpha = \alpha_1+\cdots+\alpha_n\\ \forall i: \,\tau(\alpha_i) = \tau(\alpha)\\ \;\;\fM_{\alpha_i}^{\sst}(\tau) \neq \emptyset}} \frac{1}{n!} \left[\iota^Q_*\sz_{\alpha_n}(\tau), \left[\cdots,\left[\iota^Q_*\sz_{\alpha_2}(\tau), \left[\iota^Q_*\sz_{\alpha_1}(\tau), \partial\right]\right]\cdots\right]\right], \]
    in
    $A_*^{\sT}(\tilde\fM^{Q(\Fr)}_{\alpha,1})^\pl_{\loc,\bQ}$,
    with Lie bracket $[-, -]$ defined by
    Theorem~\ref{thm:cohVOA-monoidal-stack}. 
  \end{enumerate}
\end{theorem}

Here, $\tilde\sZ_{\alpha,1}^{\Fr}(\tau^Q)$ denotes the cohomological
version of the universal enumerative invariant in
Definition~\ref{def:pair-invariant}, i.e. as described in
\S\ref{sec:homology-enumerative-invariants}.

\begin{proof}
  All steps in \S\ref{sec:semistable-invariants} hold upon the
  replacements described in
  \S\ref{sec:homology-basic-properties}--\S\ref{sec:homology-enumerative-invariants}.
  In particular, the master space localization step in
  \S\ref{sec:pairs-master-space-localization} must be done in
  $\CH_*^\sT(-)_{\loc}$. Note that, in formulas, all quantum integers
  become ``un-quantized'', i.e. $[n]_\kappa \leadsto (-1)^{n-1} n$.
  For instance, Lemma~\ref{lem:pushforward-of-bracket-partial} becomes
  \[ (\pi_{\fM_\alpha^{\Fr}})_* [\iota^Q_*\phi, I_*\partial] = (-1)^{\fr(\alpha)-1} \fr(\alpha) \cdot \phi \]
  for $\phi \in A_*^{\sT}(\fM_\alpha^{\Fr})^\pl_{\loc}$, and a similar
  replacement happens in the explicit construction
  \eqref{eq:sstable-explicit} of semistable invariants.
\end{proof}

\subsubsection{}

\begin{theorem} \label{thm:wcf-cohom}
  Consider the situation of Theorem~\ref{thm:wcf}. Then the semistable
  invariants of Theorem~\ref{thm:sst-invariants-cohom} satisfy
  \[ \sz_\alpha(\mathring\tau) = \sum_{\substack{n>0\\\alpha = \alpha_1 + \cdots + \alpha_n\\\forall i:\, \tau(\alpha_i) = \tau(\alpha)\\\fM^{\sst}_{\alpha_i}(\tau)\neq \emptyset}}\tilde U\left(\alpha_1,\dots,\alpha_n;\tau,\mathring\tau\right)\left[\left[\cdots\left[\sz_{\alpha_1}(\tau),\sz_{\alpha_2}(\tau)\right],\cdots\right],\sz_{\alpha_n}(\tau)\right] \]
  in $A_*^{\sT}(\fM_\alpha)^\pl_{\loc,\bQ}$, with Lie
  bracket $[-, -]$ defined by
  Theorem~\ref{thm:cohVOA-monoidal-stack}.
\end{theorem}

\begin{proof}
  All steps in \S\ref{sec:wall-crossing} hold upon the replacements
  described in
  \S\ref{sec:homology-basic-properties}--\S\ref{sec:homology-enumerative-invariants}.
  In particular, the master space localization step in
  Proposition~\ref{wc:prop:horizontal-wc} must be done in
  $\CH_*^\sT(-)_{\loc}$. As in the proof of
  Theorem~\ref{thm:sst-invariants-cohom} above, note that
  Lemma~\ref{wc:lemma:pushforward-vertical-wall} becomes
  \[ \pi_*\left[\phi, I_*\partial_{[a,N]}\right] = (-1)^{\fr(\beta)-e_{a-1}} (\fr(\beta) - e_{a-1}) \cdot \phi. \qedhere \]
\end{proof}

\subsection{Restrictions of classes and moduli}
\label{sec:restrictions}

\subsubsection{}
\label{es:sec:description}

In this subsection, we explain how certain conditions in the setting
for the main Theorems~\ref{thm:sst-invariants} and \ref{thm:wcf} may
be weakened. We consider three situations.
\begin{itemize}
\item The abelian category $\cat{A}$ may be replaced by an exact
  subcategory $\cat{B} \subseteq \cat{A}$. Put differently, the
  underlying category is allowed to be exact and not necessarily
  abelian, as long as it sits inside an ambient abelian category. The
  existence of an ambient abelian category is needed to work with (weak)
  stability conditions.

\item We may work with only a given subset of classes $C(\cat{B})_\pe
  \subseteq C(\cat{B})$ which we call {\it permissible}. Consequently,
  all assumptions and data involved in constructing semistable
  invariants and wall-crossing need only be defined on permissible
  classes instead of all classes, and only permissible classes
  participate in wall-crossing.

\item The moduli stack $\fM$ (of $\cat{B}$) may be replaced by certain
  locally closed substacks $\fN \subseteq \fM$. The symmetric bilinear
  elements $\scE$ need only be defined for $\fN$, e.g. $\fM$ may not
  carry a symmetric obstruction theory at all. The setting of the
  equivariant vertices in \S\ref{sec:DT-PT} and \S\ref{sec:PT-BS} is a
  good example.
\end{itemize}
While we treat these three situations simultaneously in this
subsection, they can occur independently and it is not necessary to
use all three generalizations simultaneously. The first two situations
were already covered in \cite{Joyce2021}. The third situation requires
some care to show that the invariants that come up on the right hand
side of the wall-crossing formula are invariants of the substack and
not the original moduli stack.

\subsubsection{}
\label{es:sec:restricted-monoidal-stacks}

\begin{definition} \label{def:restricted-graded-monoidal-stack}
  A {\it (commutative) partial additive monoid} is a set $\sA$,
  together with a symmetric relation $R\subseteq \sA\times\sA$, an
  identity element $0\in\sA$, and a partial binary operation $+:R\to
  \sA$, such that:
  \begin{itemize}
  \item $(0,\alpha)\in R$ and $0+\alpha=\alpha+0=\alpha$ for every
    $\alpha\in\sA$;
  \item for every $\alpha,\beta,\gamma\in\sA$, we have
    $(\alpha,\beta)\in R$ and $(\alpha+\beta,\gamma)\in R$ if and only
    if $(\beta,\gamma)\in R$ and $(\alpha,\beta+\gamma)\in R$, and in
    both cases, we have $(\alpha+\beta)+\gamma=\alpha+(\beta+\gamma)$;
  \item $+$ is commutative where it is defined.
  \end{itemize}
  A {\it graded partially-monoidal $\sT$-stack with bilinear element}
  is the data of (cf. Definition~\ref{def:graded-monoidal-stack}):
  \begin{enumerate}[label = (\roman*)]
  \item an Artin stack $\fM = \bigsqcup_\alpha \fM_\alpha$, where
    $\alpha$ ranges over a {\it partial} additive monoid $\sA$ and the
    torus $\sT$ acts on each $\fM_\alpha$;
  \item for every $\alpha \in \sA$, and every $\beta \in \sA$ where
    $\alpha + \beta$ is defined, $\sT$-equivariant morphisms
    \begin{align*}
      &\Phi_{\alpha,\beta}\colon \fM_\alpha \times \fM_\beta \to \fM_{\alpha+\beta}, \\
      &\Psi_\alpha \colon [\pt/\bC^\times] \times \fM_\alpha \to \fM_\alpha,
    \end{align*}
    and elements $\scE_{\alpha,\beta} \in K_\sT^\circ(\fM_\alpha \times
    \fM_\beta)$.
  \end{enumerate}
  This data must satisfy the same axioms as in
  Definition~\ref{def:graded-monoidal-stack} whenever the necessary
  additions are defined. Finally, suppose $\fM = \bigsqcup_{\alpha \in
    \sA} \fM_\alpha$ is a graded partially-monoidal $\sT$-stack (not
  necessarily with bilinear element). A {\it restricted graded
    partially-monoidal $\sT$-stack with ($\kappa$-symmetric) bilinear
    element} is the data of an injective morphism of partial monoids
  $c\colon \sB \to \sA$ and locally closed $\sT$-substacks
  \[ i_\beta\colon \fN_\beta \subset \fM_{c(\beta)} \]
  for every $\beta \in \sB$, such that:
  \begin{enumerate}[label = (\alph*)]
  \item $\fN \coloneqq \bigsqcup_{\beta \in \sB} \fN_\beta$ has the
    structure of a graded partially monoidal $\sT$-stack with
    ($\kappa$-symmetric) bilinear element;
  \item the graded monoidal structure on $\fN$ is inherited from
    $\fM$, i.e.
    \[ \Phi^{\fM}_{c(\beta_1),c(\beta_2)} \circ (i_{\beta_1} \times i_{\beta_2}) = i_{\beta_1+\beta_2} \circ \Phi^{\fN}_{\beta_1,\beta_2}, \quad \Psi^{\fM}_{c(\beta)} \circ (\id_{[\pt/\bC^\times]} \times i_\beta) = i_\beta \circ \Psi^\fN_{\beta} \]
    for any $\beta_1, \beta_2 \in \sB$ such that $\beta_1 + \beta_2$
    is defined and for any $\beta \in \sB$. Here, the superscripts on
    $\Phi$ and $\Psi$ are to distinguish the graded monoidal structure
    on $\fM$ from that on $\fN$.
  \end{enumerate}
  A {\it morphism} of two restricted graded partially-monoidal
  $\sT$-stacks with ($\kappa$-symmetric) bilinear elements is a
  morphism of their ambient graded monoidal $\sT$-stacks which
  restricts to a morphism of graded partially monoidal $\sT$-stacks
  with ($\kappa$-symmetric) bilinear elements.
\end{definition}

\subsubsection{}
\label{es:sec:partial-vertex-algebra}

Let $\fN$ be a restricted graded partially-monoidal $\sT$-stack with
bilinear element. All constructions in \S\ref{sec:vertex-algebra}
continue to work. In particular, they produce a {\it partial
  $\tilde\sT$-equivariant multiplicative vertex algebra} structure on
$\bfK_\circ^{\tilde\sT}(\fN)$ and a {\it partial Lie algebra}
structure on $\bfK_\circ^{\tilde\sT}(\fN)^\pl$, meaning that the
vertex product and Lie bracket are only defined on objects in classes
$\alpha_1, \alpha_2 \in \sA$ whenever the addition $\alpha_1 +
\alpha_2$ is defined. All results concerning the vertex product and
Lie bracket continue to hold whenever they are defined.

\subsubsection{}

\begin{assumption} \label{es:assump:exact-subcategory}
  Let $\cat{B} \subset \cat{A}$ be an exact sub-category. In order to
  replace $\cat{A}$ by $\cat{B}$ throughout, we require the following
  assumptions.
  \begin{enumerate}[label=(\alph*)]
  \item \label{es:assump:it:B-closed} $\cat{B} \subset \cat{A}$ is
    closed under isomorphisms in $\cat{A}$, i.e. if $B \in \cat{B}$
    and $B' \in \cat{A}$ with $B' \cong B$ then $B' \in \cat{B}$, and
    direct summands in $\cat{A}$, i.e. if $B_1 \oplus B_2 \in \cat{B}$
    for $B_1, B_2 \in \cat{A}$ then $B_1, B_2 \in \cat{B}$.

  \item \label{es:assump:it:permissible-classes} (cf.
    \S\ref{sec:abelian-category-moduli-stack}) Let $C(\cat{B}) \subset
    K(\cat{B})$ be the cone of non-zero effective
    classes.\footnote{Here $K(\cat{B}) \subset K(\cat{A})$ is the
      subset of classes arising from the Grothendieck group
      $K_0(\cat{B})$.} There is a partial sub-monoid $C(\cat{B})_{\pe}
    \subset C(\cat{B})$ of {\it permissible classes}, and, for every
    $\beta \in C(\cat{B})_{\pe}$, there is a moduli stack $\fM_\beta$,
    Artin and locally of finite type, parameterizing objects $B \in
    \cat{B}$ of class $\beta$. Direct sum and scaling automorphisms
    make $\fM \coloneqq \bigsqcup_{\beta \in C(\cat{B})_{\pe}}
    \fM_\beta$ into a graded partially-monoidal $\sT$-stack (not
    necessarily with bilinear element). We say an object $A \in
    \cat{A}$ is {\it permissible} if its class lies in
    $C(\cat{B})_{\pe}$.

  \item \label{es:assumption-restricted-substack} (cf.
    \S\ref{sec:equivariant-3CY-category}) For every $\beta \in
    C(\cat{B})_{\pe}$, there is a locally closed $\sT$-invariant
    moduli substack $\fN_\beta \subset \fM_\beta$, such that
    \[ \fN = \bigsqcup_{\beta \in C(\cat{B})_{\pe}} \fN_\beta \]
    is a restricted graded partially-monoidal $\sT$-substack of $\fM$
    with $\kappa$-symmetric bilinear elements defined from
    $\kappa$-symmetric bilinear perfect complexes
    \[ \scE_{\beta_1,\beta_2} \in \cat{Perf}_\sT(\fN_{\beta_1} \times \fN_{\beta_2}). \]
    If $\Delta\colon \fN \to \fN \times \fN$ denotes the diagonal map,
    then there are $\sT$-equivariant $\kappa$-symmetric obstruction
    theories
    \begin{align*}
      \varphi_\beta\colon \Delta^*\scE_{\beta,\beta}^\vee[-1] &\to \bL_{\fN_\beta} \\
      \varphi_\beta^\pl\colon \bE_{\beta}^\pl &\to \bL_{\fN_\beta^\pl},
    \end{align*}
    for each $\beta \in C(\cat{B})_\pe$, such that $\varphi_\beta$ and
    $\varphi_\beta^\pl$ are $\kappa$-symmetrically compatible under
    $\Pi_\beta^\pl$, and the various $\varphi_\beta$ are compatible
    with $\Phi$ whenever it is defined.
  \end{enumerate}
\end{assumption}

For $\beta \in C(\cat{B})_{\pe}$, we adopt the same notation for
various loci in $\fN_\beta$ as for $\fM_\beta$. For example, the
(semi)stable loci $\fN_\beta^{\mathrm{(s)st}}(\tau) \subset
\fN_\beta^\pl$ are given by restriction of
$\fM_\beta^{\mathrm{(s)st}}(\tau) \subset \fM_\beta^\pl$ along
$\fN_\beta^\pl \hookrightarrow \fM_\beta^\pl$. Recall that open
immersions are preserved under base change.

\subsubsection{}
\label{es:sec:auxiliary-stacks}

In the setting of Assumption~\ref{es:assump:exact-subcategory}, the
constructions in \S\ref{sec:auxiliary-stacks} are generalized as
follows.
\begin{itemize}
\item Define a {\it framing functor for $\cat{B}$} following
  Definition~\ref{bg:def:framing-functor}, by replacing $\cat{A}$ by
  $\cat{B}$ and imposing the conditions there only on the moduli
  substacks $\fN_\beta$ for $\beta \in C(\cat{B})_{\pe}$, and the
  objects $[E] \in \fN_\beta$ that they parameterize. For instance,
  $\fM_\beta^{\Fr} \subset \fM_\beta$ may not be an open substack, but
  its restriction to $\fN_\beta^{\Fr} \subset \fN_\beta$ must be open.

\item Define the auxiliary category $\tilde{\cat{B}}^{Q(\vec \Fr)}$
  following Definition~\ref{def:auxiliary-stack}, and let
  $\tilde\fN^{Q(\vec\Fr)}$ be the moduli stack parameterizing triples
  $(E, \vec V, \vec\rho)$ such that $[E] \in \fN^{\vec\Fr} \coloneqq
  \bigcap_{v \in Q_0^\circ} \fN^{\Fr_v}$. The content of
  \S\ref{bg:sec:aux-stacks} holds for $\tilde\fN^{Q(\vec\Fr)}$, but
  not necessarily for the moduli stack $\tilde\fM^{Q(\vec\Fr)}$ of
  $\tilde{\cat{B}}^{Q(\vec \Fr)}$ itself.
\end{itemize}
Then Theorem~\ref{thm:auxiliary-stack-vertex-algebra} produces a
restricted graded partially-monoidal $\sT$-stack structure on
$\tilde\fN^{Q(\vec\Fr)} \subset \tilde\fM^{Q(\vec\Fr)}$ with unchanged
$\kappa$-symmetric bilinear elements, a partial multiplicative vertex
algebra structure on $\bfK_\circ^{\tilde\sT}(\tilde\fN^{Q(\vec\Fr)})$,
and a partial Lie algebra structure on
$\bfK_\circ^{\tilde\sT}(\tilde\fN^{Q(\vec\Fr)})^\pl$.

\subsubsection{}
\label{es:sec:assumptions-semistable-invariants}

\begin{assumption} \label{es:assump:semistable-invariants}
  Let $\tau$ be a stability condition on $\cat{A}$. We modify some
  conditions in Assumption~\ref{assump:semistable-invariants} as
  follows, in order to construct semistable invariants.
  \begin{enumerate}[label = (\alph*')]
  \item \label{es:assump:it:tau-artinian} $\cat{B}$ admits {\it
    $\tau$-HN filtrations for permissible classes}: given an object $A
    \in \cat{B}$ of class in $C(\cat{B})_{\pe}$, the $\tau$-HN
    filtration $0 = A_0 \subsetneq A_1 \subsetneq \cdots \subsetneq
    A_n = A$ of $A \in \cat{A}$ (from
    Assumption~\ref{assump:semistable-invariants}\ref{assump:it:tau-artinian})
    satisfies $A_i \in \cat{B}$ for all $i = 1, \ldots, n$.

  \item \label{es:assump:it:semistable-loci} $\tau$-(semi)stability is
    {\it open for permissible classes}: $\fN^{\st}_\alpha(\tau)
    \subset \fN^{\sst}_\alpha(\tau) \subset \fN_\alpha^\pl$ are open
    substacks of finite type for all $\alpha \in C(\cat{B})_{\pe}$.

  \item \label{es:assump:it:framing-functor} There exists a set $\Frs$
    of framing functors (as discussed in
    \S\ref{es:sec:auxiliary-stacks}) such that for any finite
    collection of classes $\{\alpha_i\}_{i \in I} \subset
    C(\cat{B})_{\pe}$, there is some $\Fr \in \Frs$ such that
    $\fN_{\alpha_i}^{\sst}(\tau) \subset \fN_{\alpha_i}^{\Fr,\pl}$ for
    all $i \in I$.
  \end{enumerate}
  When $C(\cat{B})_{\pe} \subsetneq C(\cat{B})$, we gain the freedom
  to identify, in the new
  assumption~\ref{es:assump:it:inert-classes} below, a subset
  $C(\cat{B})_{\inert} \subsetneq C(\cat{B})$ which makes it easier to
  construct our weak stability conditions on auxiliary stacks.
  \begin{enumerate}[resume, label = (\alph*')]
  \item \label{es:assump:it:rank-function} There exists a ``rank
    function'' $r\colon C(\cat{A}) \to \bZ$ such that
    \begin{itemize}
    \item if $A \in \cat{A}$ is $\tau$-semistable and permissible then
      $r(A) > 0$, and moreover
    \item if $A' \subset A$ with $\tau(A') = \tau(A/A')$, then $r(A) =
      r(A') + r(A/A')$ and $r(A'), r(A/A') > 0$.
    \end{itemize}

  \customitem[(T)] \label{es:assump:it:inert-classes} There exists
    a set of {\it inert classes} $C(\cat{B})_{\inert} \subset
    C(\cat{B})$ such that:
    \begin{itemize}
    \item $\tau(C(\cat{B})_{\inert}) \cap \tau(C(\cat{B})_{\pe}) =
      \emptyset$; \footnote{Consequently, inert classes do not
      participate in wall-crossing --- hence the name ``inert''.}
    \item if $\gamma \in C(\cat{B})_{\inert}$ and $\delta \in C(\cat{B})$
      such that $\gamma+\delta \in C(\cat{B})_{\pe}$, then
      $\tau(\gamma+\delta) = \tau(\delta)$ and $r(\gamma+\delta) =
      r(\delta)$.
    \end{itemize}

  \item \label{es:assump:it:semi-weak-stability} Let $0 < \beta <
    \alpha$ be the classes of $\tau$-semistable and permissible
    objects $0 \neq B \subsetneq A$ in $\cat{A}$ with $\tau(\beta) =
    \tau(\alpha-\beta)$. Then for the class $\beta'$ of any sub-object
    $0 \neq B' \subsetneq B$:
    \begin{itemize}
    \item if $\tau(\beta') < \tau(\alpha-(\beta-\beta')) =
      \tau(\beta-\beta')$, then $\beta' \in C(\cat{B})_{\inert}$;
    \item if $\tau(\beta') = \tau(\alpha-\beta') <
      \tau(\beta-\beta')$, then $\beta-\beta' \in C(\cat{B})_{\inert}$.
    \end{itemize}
  \end{enumerate}
  We need to be able to construct virtual cycles and their enumerative
  invariants, now for the restricted stacks $\fN_\alpha$ instead of
  $\fM_\alpha$.
  \begin{enumerate}[resume, label = (\alph*')]
  \item \label{es:assump:it:properness} The following algebraic spaces
    are proper and have the resolution property (see
    Remark~\ref{rem:localization-resolution-property}):
    \begin{itemize}
    \item $\fN_\alpha^{\sst}(\tau)^\sT$ for all $\alpha \in
      C(\cat{B})_{\pe}$ with no strictly $\tau$-semistable objects in
      $\fN$;
    \item $\tilde\fN_{\alpha,1}^{Q(\Fr),\sst}(\tau^Q)^\sT$ (the
      auxiliary stack in Definition~\ref{def:pair-invariant}) for all
      $\alpha \in C(\cat{B})_{\pe}$ and $\Fr \in \Frs$;
    \item $\bN_\alpha^{\sT_w}$ (the master space in
      Proposition~\ref{prop:pairs-master-space-fixed-loci} after base
      change along $\fN \hookrightarrow \fM$) for all $w$ in
      Lemma~\ref{lem:master-space-no-poles}, and all $\alpha \in
      C(\cat{B})_{\pe}$ and $\Fr_1, \Fr_2 \in \Frs$.
    \end{itemize}
  \end{enumerate}
  Furthermore, we make the following extra assumption.
  \begin{enumerate}[label = (\alph*), resume]
  \item \label{es:assump:it:Bpe-sst-summands} Suppose $B_1, B_2 \in
    \cat{B}$ have classes $\beta_1$ and $\beta_2$, and are
    $\tau$-semistable with $\tau(B_1) = \tau(B_2)$. If $\beta_1 +
    \beta_2 \in C(\cat{B})_{\pe}$, then $\beta_1, \beta_2 \in
    C(\cat{B})_{\pe}$. Moreover, in this case, we require that the
    following commutative square is Cartesian:
    \[ \begin{tikzcd}
      \fN_{\beta_1} \times \fN_{\beta_2} \ar{d}{\Phi_{\beta_1,\beta_2}} \ar[hookrightarrow]{r} & \fM_{\beta_1} \times \fM_{\beta_2} \ar{d}{\Phi_{\beta_1,\beta_2}} \\
      \fN_{\beta_1+\beta_2} \ar[hookrightarrow]{r} & \fM_{\beta_1+\beta_2}.
    \end{tikzcd} \]
  \end{enumerate}
  Without loss of generality, by shrinking $C(\cat{B})_{\pe}$ if
  necessary, we may assume that if $\beta \in C(\cat{B})_{\pe}$ then
  $\fN_\beta^{\sst}(\tau) \neq \emptyset$.
\end{assumption}

\subsubsection{}

\begin{theorem}[Semistable invariants] \label{es:thm:sst-invariants}
  Suppose $\tau$ is a stability condition on $\cat{A}$ for which
  Assumption~\ref{es:assump:semistable-invariants} holds. Then there
  exists a unique collection
  \begin{equation}
    \left(\sz_{\alpha}(\tau)\in K_{\circ}^{\tilde{\sT}}(\fN_\alpha)^\pl_{\loc,\bQ}\right)_{\alpha\in C(\cat{B})_{\pe}}
  \end{equation}
  of K-homology classes satisfying the same properties as in
  Theorem~\ref{thm:sst-invariants} with $\fM$ replaced by $\fN$ and
  all classes required to belong to $C(\cat{B})_{\pe}$.
\end{theorem}

\begin{proof}
  All steps in \S\ref{sec:semistable-invariants} hold upon replacing
  $\cat{A}$ by $\cat{B}$ and $\fM$ by $\fN$ whenever a geometric
  construction is required, and working only with classes $\alpha \in
  C(\cat{B})_{\pe}$, and adapting the auxiliary weak stability
  conditions \eqref{eq:pair-stability} and
  \eqref{eq:master-pair-stability} according to
  \eqref{es:eq:joyce-framed-stack-stability} below. Note that
  stability conditions are always still defined on $\cat{A}$ and its
  auxiliary categories.
  Assumptions~\ref{es:assump:exact-subcategory}\ref{es:assump:it:B-closed},
  \ref{es:assump:semistable-invariants}\ref{es:assump:it:tau-artinian},
  and (the first half of)
  \ref{es:assump:semistable-invariants}\ref{es:assump:it:Bpe-sst-summands}
  guarantee that the classes $\alpha_1,\ldots,\alpha_n$ in
  \eqref{eq:sstable-def-intro} belong to $C(\cat{B})_{\pe} \subset
  C(\cat{B})$. Generally speaking, the restriction to
  $C(\cat{B})_{\pe}$ and to $\cat{B} \subset \cat{A}$ works exactly as
  in \cite{Joyce2021}.

  We discuss the restriction to $\fN \subset \fM$. To work with
  invariants supported on $\fN$, we must ensure that, in all steps
  where an object $[E] \in \fN$ is decomposed, the resulting summands
  still lie in $\fN$ instead of the ambient moduli stack $\fM$. This
  discussion is independent of any auxiliary structures, and is where
  we use the extra
  Assumption~\ref{es:assump:semistable-invariants}\ref{es:assump:it:Bpe-sst-summands}.
  Concretely, it is used in the explicit construction of semistable
  invariants in Definition~\ref{def:sstable-explicit} to ensure the
  inductive definition is valid, and again in
  Proposition~\ref{prop:pairs-master-space-fixed-loci}\ref{sst:it:complicated-locus}
  to guarantee that both instances of $\tilde\fM$ in
  \eqref{eq:pairs-master-space-complicated-locus} become $\tilde\fN$,
  i.e. that, in the splitting $E = E_1 \oplus E_2$ appearing there,
  the summands $E_i$ lie in $\fN_{\alpha_i} \subset \fM_{\alpha_i}$.

  We make some additional comments. The universal enumerative
  invariant $\tilde\sZ_{\alpha,1}^{\Fr}(\tau^Q)$ from
  \eqref{eq:pairs-stack-enumerative-invariant} is well-defined by
  Assumption~\ref{es:assump:exact-subcategory}\ref{es:assumption-restricted-substack}.
  The key Theorem~\ref{thm:sst-invariants-construction} becomes
  \[ I_*\tilde{\sZ}_{\alpha,1}^{\Fr}(\tau^Q) = \sum_{\substack{n>0, \, \alpha_1,\ldots,\alpha_n \in C(\cat{B})_{\pe} \\ \alpha = \alpha_1+\cdots+\alpha_n\\ \forall i: \,\tau(\alpha_i) = \tau(\alpha)\\ \;\;\fM_{\alpha_i}^{\sst}(\tau) \neq \emptyset}} \frac{1}{n!} \left[\iota^Q_*\sz^{\Fr}_{\alpha_n}(\tau), \left[\cdots,\left[\iota^Q_*\sz^{\Fr}_{\alpha_2}(\tau), \left[\iota^Q_*\sz^{\Fr}_{\alpha_1}(\tau), I_*\partial\right]\right]\cdots\right]\right] \]
  for $\alpha \in C(\cat{B})_{\pe}$ and framing functors $\Fr$ such
  that $\fM_\alpha^{\sst}(\tau) \subset \fM_\alpha^{\Fr,\pl}$, and the
  Lie bracket is defined by the discussion in
  \S\ref{es:sec:partial-vertex-algebra}. Note that, since
  $\sz_\beta(\tau)$ is supported on $\fN_\beta^{\sst}(\tau)$, the
  condition $\fM_{\alpha_i}^{\sst}(\tau) \neq \emptyset$ in the sum
  may be replaced by $\fN_{\alpha_i}^{\sst}(\tau) \neq \emptyset$.
\end{proof}

\subsubsection{}

\begin{assumption} \label{es:assump:wall-crossing}
  Let $\tau$ and $\mathring\tau$ be stability conditions on $\cat{A}$.
  We relax some conditions in Assumption~\ref{assump:wall-crossing} as
  follows, in order to obtain a wall-crossing formula.
  \begin{enumerate}[label = (\alph*')]
  \item \label{es:assump:it:tau-circ} The stability condition $\tau$
    satisfies Assumption~\ref{es:assump:semistable-invariants}. In
    addition, $\cat{B}$ is $\mathring\tau$-Artinian for permissible
    classes
    (Assumption~\ref{es:assump:semistable-invariants}\ref{es:assump:it:tau-artinian}),
    $\mathring\tau$-(semi)stability is open for permissible classes
    (Assumption~\ref{es:assump:semistable-invariants}\ref{es:assump:it:semistable-loci}),
    and $\mathring\tau$ satisfies
    Assumption~\ref{es:assump:semistable-invariants}\ref{es:assump:it:Bpe-sst-summands}.

  \item \label{es:assump:it:lambda} For any $\alpha \in
    C(\cat{B})_{\pe}$, there exists a group homomorphism
    \[ \lambda\colon K(\cat{A})\to \bR \]
    such that for any class $\beta \in R_\alpha$, we have
    $\lambda(\beta) > 0$ (resp. $\lambda(\beta) < 0$) if and only if
    $\mathring\tau(\beta) > \mathring\tau(\alpha-\beta)$ (resp.
    $\mathring\tau(\beta) < \mathring\tau(\alpha-\beta)$).

  \item \label{es:assump:it:properness-wcf} The following algebraic
    spaces are proper and have the resolution property (see
    Remark~\ref{rem:localization-resolution-property}):
    \begin{itemize}
    \item $\tilde\fN_{\alpha,\vec d}^{\vec Q(\Fr),\sst}(\tau_x^s)^\sT$
      for all $\Fr \in \Frs$ and $\alpha \in C(\cat{B})_{\pe}$ and
      $\vec d$ with no strictly $\tau^s_x$-semistable objects;
    \item $\bN_{\alpha,\vec d}^{\sT_w}$ for all $w$ in
      Lemma~\ref{lem:master-space-no-poles}, where $\bN_{\alpha,\vec
        d}$ is the master space in the proof of
      Proposition~\ref{wc:prop:horizontal-wc} after base change along
      $\fN \hookrightarrow \fM$, for all classes $\alpha \in
      C(\cat{B})_{\pe}$ and $\vec d$ and $\Fr \in \Frs$.
    \end{itemize}
  \end{enumerate}
\end{assumption}

\subsubsection{}

\begin{theorem}[Dominant wall-crossing formula] \label{es:thm:wcf}
  Let $\tau$ and $\mathring\tau$ be weak stability conditions on
  $\cat{A}$ for which Assumption~\ref{es:assump:wall-crossing} holds.
  Suppose $\tau$ weakly dominates $\mathring\tau$ at $\alpha \in
  C(\cat{B})_{\pe}$. Then the semistable invariants of
  Theorem~\ref{es:thm:sst-invariants} satisfy
  \[ \sz_\alpha(\mathring\tau) = \sum_{\substack{n>0, \, \alpha_1,\ldots,\alpha_n \in C(\cat{B})_{\pe} \\\alpha = \alpha_1 + \cdots + \alpha_n\\\forall i:\, \tau(\alpha_i) = \tau(\alpha)\\\fN^{\sst}_{\alpha_i}(\tau)\neq \emptyset}}\tilde U\left(\alpha_1,\dots,\alpha_n;\tau,\mathring\tau\right)\left[\left[\cdots\left[\sz_{\alpha_1}(\tau),\sz_{\alpha_2}(\tau)\right],\cdots\right],\sz_{\alpha_n}(\tau)\right] \]
  in $K_\circ^{\tilde\sT}(\fN_\alpha)^\pl_{\loc,\bQ}$, with Lie
  bracket $[-, -]$ as in \S\ref{es:sec:auxiliary-stacks}.
\end{theorem}

Note that all appearances of $\fM$ in the dominance conditions of
Definition~\ref{def:dominates-at} should be replaced with $\fN$.

\begin{proof}
  All steps in \S\ref{sec:wall-crossing} hold upon replacing $\cat{A}$
  by $\cat{B}$ and $\fM$ by $\fN$ whenever a geometric construction is
  required, and working only with classes $\alpha \in
  C(\cat{B})_{\pe}$, and replacing the auxiliary weak stability
  condition \eqref{wc:eq:joyce-framed-stack-stability} with the more
  general \eqref{es:eq:joyce-framed-stack-stability} below. The same
  considerations as in the proof of
  Theorem~\ref{es:thm:sst-invariants} continue to hold. The extra
  Assumption~\ref{es:assump:semistable-invariants}\ref{es:assump:it:Bpe-sst-summands}
  guarantees that $R_\alpha \subset C(\cat{B})_{\pe}$, and is also is
  used in the ``horizontal'' wall-crossing step in
  Proposition~\ref{wc:prop:horizontal-wc} to guarantee that all
  occurrences of $\fM$ in the right hand side of
  \eqref{wc:eq:master-space-complicated-locus} may be replaced by
  $\fN$.

  We now discuss the auxiliary weak stability condition. The extra
  Assumption~\ref{es:assump:semistable-invariants}\ref{es:assump:it:inert-classes}
  on inert classes, and the weakened
  Assumptions~\ref{es:assump:semistable-invariants}\ref{es:assump:it:rank-function}
  and \ref{es:assump:it:semi-weak-stability}, allow us to replace the
  auxiliary weak stability condition
  \eqref{wc:eq:joyce-framed-stack-stability} with the more general
  \begin{equation}\label{es:eq:joyce-framed-stack-stability}
    \tau_x^s\colon (\beta,\vec e) \mapsto \begin{cases}
      \left(\tau(\beta), \frac{s\lambda(\beta)+(\vec\mu+x\vec 1)\cdot \vec e}{r(\beta)}\right), & \beta \notin C(\cat{B})_{\inert} \sqcup \{0\}, \, \tau(\beta) = \tau(\alpha-\beta) \text{ or } \beta = \alpha, \\
      \left(\tau(\beta), \infty\right), & \beta \notin C(\cat{B})_{\inert} \sqcup \{0\}, \, \tau(\beta) > \tau(\alpha - \beta),\\
      \left(\tau(\beta), -\infty\right), & \beta \notin C(\cat{B})_{\inert} \sqcup \{0\}, \, \tau(\beta) < \tau(\alpha - \beta), \\
      \left(\infty, \frac{s\lambda(\beta)+(\vec\mu+x\vec 1)\cdot \vec e}{\vec 1\cdot \vec e}\right), & \beta \in C(\cat{B})_{\inert} \sqcup \{0\},\, s\lambda(\beta) + (\vec\mu+x\vec 1)\cdot \vec e>0,\\
      \left(-\infty, \frac{s\lambda(\beta)+(\vec\mu+x\vec 1)\cdot \vec e}{\vec 1\cdot \vec e}\right), &\beta \in C(\cat{B})_{\inert} \sqcup \{0\},\, s\lambda(\beta) + (\vec\mu+x\vec 1)\cdot \vec e\leq 0.
    \end{cases}
  \end{equation}
  Lemma~\ref{wc:lem:R-sets}\ref{it:R-sets-ii} continues to hold
  assuming that $\beta', \beta-\beta' \notin C(\cat{B})_{\inert}$.
  Lemma~\ref{lem:joyce-framed-stack-stability} also continues to hold,
  namely the three cases in its proof are now:
  \begin{itemize}
  \item $\beta'$ or $\beta-\beta'$ belong to $C(\cat{B})_{\inert}
    \sqcup \{0\}$, in which case we use
    Assumption~\ref{es:assump:semistable-invariants}\ref{es:assump:it:inert-classes}
    to make the original argument go through;
  \item $\beta', \beta-\beta' \notin C(\cat{B})_{\inert} \sqcup \{0\}$
    and $\tau(\beta') \neq \tau(\beta-\beta')$, in which case we use
    Assumption~\ref{es:assump:semistable-invariants}\ref{es:assump:it:semi-weak-stability}
    to make the original argument go through;
  \item $\beta', \beta-\beta' \notin C(\cat{B})_{\inert} \sqcup \{0\}$
    and $\tau(\beta') = \tau(\beta-\beta')$, in which case the
    original argument continues to hold verbatim.
  \end{itemize}
  Consequently, $\tau^s_x$ is still effectively a weak stability
  condition on $\tilde{\cat{A}}^{\vec Q(\Fr)}$.
\end{proof}

\subsection{Reduced obstruction theories}
\label{sec:reduced-obstruction-theories}

\subsubsection{}
\label{rc:sec:reduced-virtual-classes}

In this subsection, we explain how to prove generalizations of the
main Theorems~\ref{thm:sst-invariants} and \ref{thm:wcf} for {\it
  reduced} (semistable) enumerative invariants, given a collection of
cosections for the symmetric obstruction theory.

Recall from \S\ref{sec:equivariant-3CY-category} that the moduli stack
$\fM_\alpha$, for every $\alpha$, is equipped with a $\sT$-equivariant
$\kappa$-symmetric obstruction theory which we denote
\[ \varphi_\alpha\colon \bE_{\fM_\alpha} \to \bL_{\fM_\alpha}. \]
Throughout this subsection we suppose that, for each $\alpha$, we are
given a finite-dimensional vector space $\sU_\alpha$ and a morphism
\begin{equation} \label{eq:cosection}
  \sigma_\alpha\colon \bE_{\fM_\alpha}^\vee[1] \to \sU_\alpha \otimes \cO_{\fM_\alpha}
\end{equation}
whose $h^0$ induces a surjective morphism
\begin{equation} \label{eq:cosection-of-obs}
  h^0(\sigma_\alpha)\colon \cOb_\alpha \twoheadrightarrow \sU_\alpha \otimes \cO_{\fM_\alpha}.
\end{equation}
Let $o_\alpha \coloneqq \dim \sU_\alpha$. We say that {\it the
  obstruction theory on $\fM_\alpha$ has $o_\alpha$ cosections}, and
that the {\it the obstruction sheaf $\cOb_\alpha$ has $o_\alpha$
  surjective cosections}. Note that $o_\alpha$ may be $0$ for some or
even all classes $\alpha$.

\subsubsection{}

\begin{lemma} \label{lem:cosection-for-reductions}
  Let $\Pi^\pl\colon \fX \to \fX^\pl$ be a $\sT$-equivariant
  $\bC^\times$-gerbe in $\cat{Art}_\sT$ over a smooth equidimensional
  base Artin stack $\fB$. Suppose that $\fX$ has an obstruction theory
  $\varphi\colon \bE_{\fX/\fB} \to \bL_{\fX/\fB}$, of trivial
  $\bC^\times$-weight, with cosections $\sigma\colon
  \bE_{\fX/\fB}^\vee[1] \to \sU \otimes \cO_{\fX}$.
  \begin{enumerate}[label = (\roman*)]
  \item \label{it:cosection-for-smooth-reduction} If
    $\tilde\varphi^\pl\colon \tilde\bE_{\fX^\pl/\fB} \to
    \bL_{\fX^\pl/\fB}$ is an obstruction theory on $\fX^\pl$
    which is compatible with $\varphi$ under $\Pi^\pl$, then it also has
    cosections $\tilde\sigma^\pl\colon \tilde\bE_{\fX^\pl/\fB}^\vee[1]
    \to \sU \otimes \cO_{\fX^\pl}$.
  \item \label{it:cosection-for-symmetrized-reduction} If
    $\varphi^\pl\colon \bE_{\fX^\pl/\fB} \to \bL_{\fX^\pl/\fB}$ is a
    $\kappa$-symmetric obstruction theory on $\fX^\pl$ which is
    $\kappa$-symmetrically compatible with $\varphi$ under $\Pi^\pl$,
    then it also has cosections $\sigma^\pl\colon
    \bE_{\fX^\pl/\fB}^\vee[1] \to \sU \otimes \cO_{\fX^\pl}$.
  \end{enumerate}
\end{lemma}

Hence \eqref{eq:cosection} induces cosections
\[ \sigma_\alpha^\pl\colon \bE_{\fM_\alpha^\pl}^\vee[1] \to \sU_\alpha \otimes \cO_{\fM_\alpha^\pl} \]
for the $\kappa$-symmetrically compatible symmetric obstruction theory
$\varphi_\alpha^\pl\colon \bE_{\fM_\alpha^\pl} \to
\bL_{\fM_\alpha^\pl}$ on $\fM_\alpha^\pl$.

\begin{proof}
  As in the proof of Lemma~\ref{lem:obstruction-theory-pl}, we
  implicitly identify sheaves on $\fX^\pl$ with sheaves on $\fX$ of
  trivial $\bC^\times$-weight, and omit writing $(\Pi^\pl)^*$.
  
  For \ref{it:cosection-for-smooth-reduction}, apply $(-)^\vee[1]$ to
  the top row of \eqref{eq:obstruction-theories-compatibility}. Since
  $\Hom(\bL_{\Pi^\pl}^\vee[1], \cO) = 0$ for degree reasons, $\sigma$
  induces a morphism $\tilde\sigma^\pl\colon
  \tilde\bE_{\fX^\pl/\fB}^\vee[1] \to \sU \otimes \cO$. The associated
  long exact sequence provides a commutative square
  \[ \begin{tikzcd}
    h^0(\bE_{\fX/\fB}^\vee[1]) \ar{r}{\xi} \ar[twoheadrightarrow]{d}{h^0(\sigma)} & h^0(\tilde\bE_{\fX^\pl/\fB}^\vee[1]) \ar{d}{h^0(\tilde\sigma^\pl)} \\
    \sU \otimes \cO_{\fX} \ar[equals]{r} & \sU \otimes \cO_{\fX}
  \end{tikzcd} \]
  so surjectivity of $h^0(\sigma)$ implies surjectivity of
  $h^0(\tilde\sigma^\pl)$, making $\tilde\sigma^\pl$ into a cosection.

  For \ref{it:cosection-for-symmetrized-reduction}, take the middle
  row of \eqref{eq:obstruction-theories-symmetric-compatibility}
  (where $\tilde\bE_{\fX^\pl/\fB}^\vee[1]$ was denoted $\bF$) and
  consider the composition
  \[ \sigma^\pl\colon \bE_{\fX^\pl/\fB}^\vee[1] \xrightarrow{\eta} \tilde\bE_{\fX^\pl/\fB}^\vee[1] \xrightarrow{\tilde\sigma^\pl} \sU \otimes \cO. \]
  We know $h^0(\tilde\sigma^\pl)$ is surjective, and $h^0(\eta)$ is
  surjective because $h^0(\bL_{\Pi^\pl}) = 0$. Hence $h^0(\sigma^\pl)$
  is also surjective, making $\sigma^\pl$ into a cosection.
\end{proof}

\subsubsection{}

\begin{remark}
  It is possible that the results of this subsection continue to hold
  if we are only given the cosections \eqref{eq:cosection-of-obs} of
  the obstruction sheaf, instead of the cosections
  \eqref{eq:cosection} of the obstruction theory. Certainly, the
  construction (\S\ref{rc:sec:different-approaches}) of reduced
  virtual cycles uses only cosections of the obstruction sheaf.
  However, as the proof of Lemma~\ref{lem:cosection-for-reductions}
  demonstrates, in order to ``lift''/``reduce'' cosections along
  smooth or symmetrized pullbacks, it seems more convenient to assume
  that cosections are given at the level of the obstruction theory
  instead of the obstruction sheaf.
\end{remark}

\subsubsection{}
\label{rc:sec:different-approaches}

We review some generalities about reduced virtual cycles. Suppose we
have restricted to a locus $X \subset \fM_\alpha^\pl$ which is an
algebraic space, so that $\varphi_\alpha$ becomes a {\it perfect}
obstruction theory
(Lemma~\ref{lem:perfect-symmetric-obstruction-theory}). Suppose
furthermore that there are $o_\alpha > 0$ cosections
\eqref{eq:cosection-of-obs} of the obstruction sheaf. Then the virtual
cycle associated to the perfect obstruction theory $\varphi_\alpha$
will vanish, e.g. by \cite[Remark 3.3]{kiem_li_loc_vir_str_sh}.

In this setting, there are two different ways of obtaining a {\it
  reduced} virtual cycle. The most natural and conceptual method is to
take the co-cone of $\sigma_\alpha$ and hope it is still a perfect
obstruction theory. When it is, in which case we refer to it as a {\it
  reduced POT}, we may then take its associated virtual cycle.
However, in general it is not \cite[Appendix
  A]{mpt-curves-k3-modular}. Instead, one can reduce the virtual cycle
directly. This approach, due to Kiem and Li, works generally and will
be explained in detail in \S\ref{rc:sec:reduced-apot} below. The
result is a {\it reduced virtual structure sheaf}
\[ \cO_X^{\vir,\red} \in K_\sT(X) \]
which satisfies all the usual properties of a virtual structure sheaf
and agrees with the one defined via the reduced POT whenever it
exists. Since extensions by copies of $\cO$ do not affect the
determinant, the {\it symmetrized} reduced virtual structure sheaf
$\hat\cO_X^{\vir,\red}$ is defined by the same twist
\eqref{eq:symmetrized-virtual-sheaf} as in the non-reduced case. As in
Definition~\ref{bg:def:univ-enum-inv}, this allows us to define the
{\it reduced universal enumerative invariant}
\[ \sZ_X^{\red} \coloneqq \chi\left(X, \hat\cO_X^{\vir,\red} \otimes -\right) \in K_\circ^{\tilde\sT}(X)_{\loc}. \]
By convention, if $o_\alpha = 0$ then we set $\cO_X^{\vir,\red}
\coloneqq \cO_X^\vir$, and then in particular $\sZ_X^{\red} = \sZ_X$.

\subsubsection{}

\begin{assumption} \label{assump:cosections}
  We make the following assumptions on the cosections
  \eqref{eq:cosection}, in order to construct reduced semistable
  invariants and obtain a wall-crossing formula.
  \begin{enumerate}[label = (\alph*)]
  \item $o_\alpha + o_\beta \ge o_{\alpha + \beta}$ for all $\alpha$
    and $\beta$.
  \end{enumerate}
\end{assumption}

\subsubsection{}

\begin{theorem}[Reduced semistable invariants] \label{thm:red-sst-invariants}
  Suppose $\tau$ is a stability condition on $\cat{A}$ for which
  Assumption~\ref{assump:semistable-invariants} holds, and
  $\{\sigma_\alpha\}_{\alpha \in C(\cat{A})}$ is a collection of
  cosections \eqref{eq:cosection} for which
  Assumption~\ref{assump:cosections} holds. Then there
  exists a unique collection
  \begin{equation}
    \left(\sz_{\alpha}^{\red}(\tau)\in K_{\circ}^{\tilde{\sT}}(\fM_{\alpha})^\pl_{\loc,\bQ}\right)_{\alpha\in C(\cat{A})}
  \end{equation}
  of K-homology classes satisfying the same properties as in
  Theorem~\ref{thm:sst-invariants}, with the modifications:
  \begin{enumerate}[label = (\roman*')]
    \addtocounter{enumi}{1}

  \item for any $\alpha$ for
    which $\fM_\alpha^{\st}(\tau) = \fM_{\alpha}^{\sst}(\tau)$,
    \[ (\Pi^\pl_\alpha)_*\sz_{\alpha}^{\red}(\tau) = \chi\left(\fM^{\sst}_{\alpha}(\tau), \hat\cO^{\vir,\red}_{\fM_{\alpha}^{\sst}(\tau)}\otimes -\right); \]

    \addtocounter{enumi}{1}

  \item for any framing functor $\Fr \in \Frs$, in the notation of
    Definition~\ref{def:pair-invariant} and \S\ref{sec:pairs-maps},
    \[ I_*\tilde\sZ_{\alpha,1}^{\Fr,\red}(\tau^Q) = \sum_{\substack{n>0 \\ \alpha = \alpha_1+\cdots+\alpha_n\\o_\alpha = o_{\alpha_1} + \cdots + o_{\alpha_n}\\ \forall i: \,\tau(\alpha_i) = \tau(\alpha)\\ \;\;\fM_{\alpha_i}^{\sst}(\tau) \neq \emptyset}} \frac{1}{n!} \left[\iota^Q_*\sz_{\alpha_n}^{\red}(\tau), \left[\cdots,\left[\iota^Q_*\sz_{\alpha_2}^{\red}(\tau), \left[\iota^Q_*\sz_{\alpha_1}^{\red}(\tau), \partial\right]\right]\cdots\right]\right], \]
    where Lemma~\ref{lem:lifting-Obs-cosections} below provides the
    reduced version $\tilde\sZ_{\alpha,1}^{\Fr,\red}(\tau^Q)$ of
    $\tilde\sZ_{\alpha,1}^{\Fr}(\tau^Q)$.
  \end{enumerate}
\end{theorem}

\subsubsection{}

\begin{theorem}[Reduced dominant wall-crossing formula] \label{thm:red-wcf}
  Consider the situation of Theorem~\ref{thm:wcf}, and suppose
  $\{\sigma_\alpha\}_{\alpha \in C(\cat{A})}$ is a collection of
  cosections \eqref{eq:cosection} for which
  Assumption~\ref{assump:cosections} holds. Then the reduced
  semistable invariants of Theorem~\ref{thm:red-sst-invariants}
  satisfy
  \[ \sz_\alpha^{\red}(\mathring\tau) = \sum_{\substack{n>0\\\alpha = \alpha_1 + \cdots + \alpha_n\\ o_\alpha = o_{\alpha_1} + \cdots + o_{\alpha_n}\\\forall i:\, \tau(\alpha_i) = \tau(\alpha)\\\fM^{\sst}_{\alpha_i}(\tau)\neq \emptyset}}\tilde U\left(\alpha_1,\dots,\alpha_n;\tau,\mathring\tau\right)\left[\left[\cdots\left[\sz_{\alpha_1}^{\red}(\tau),\sz_{\alpha_2}^{\red}(\tau)\right],\cdots\right],\sz_{\alpha_n}^{\red}(\tau)\right]. \]
\end{theorem}

Note that, comparing Theorems~\ref{thm:red-sst-invariants} and
\ref{thm:red-wcf} to the original Theorems~\ref{thm:sst-invariants}
and \ref{thm:wcf}, the only difference is the extra constraint
$o_{\alpha_1}+\cdots+o_{\alpha_n}=o_\alpha$ in the sums over
decompositions $\alpha = \alpha_1 + \cdots + \alpha_n$.

\subsubsection{}
\label{rc:sec:wall-crossing-formula}

The proof of Theorems~\ref{thm:red-sst-invariants} and
\ref{thm:red-wcf} work exactly as in \S\ref{sec:semistable-invariants}
and \S\ref{sec:wall-crossing}, with the following two modifications.
\begin{enumerate}
\item Virtual cycles on auxiliary spaces are induced from APOTs
  obtained by symmetrized pullback (Theorem~\ref{thm:APOTs}). To
  obtain {\it reduced} virtual cycles on auxiliary spaces, we
  therefore need to lift the cosections \eqref{eq:cosection} to the
  symmetrized pullback APOT in a compatible way. This is done by
  Lemma~\ref{lem:lifting-Obs-cosections}.

\item In the master space wall-crossing steps of
  \S\ref{sec:pairs-master-term-3} and \S\ref{sec:horizontal-wc},
  virtual cycles induced by APOTs of different origins are compared
  using Theorem~\ref{thm:APOT-comparison}. To compare {\it reduced}
  virtual cycles, we therefore also need to compare cosections. This
  is done in \S\ref{rc:sec:master-space-argument} using a general
  deformation invariance result for reduced virtual cycles
  (Lemma~\ref{lem:red-APOT-virtual-cycle-functoriality}), and produces
  the $o_{\alpha_1}+\cdots+o_{\alpha_n}=o_\alpha$ condition.
\end{enumerate}
These constructions will occupy the remainder of this subsection. 

\subsubsection{}
\label{rc:sec:reduced-apot}

We first provide some detail on the construction of reduced
(K-theoretic) virtual cycles associated to APOTs with surjective
cosections, and their deformation invariance. The tools for this are
mostly already in \cite[\S 4]{kiem_li_loc_vir_str_sh}, and the basic
idea is already implicitly present in \cite[\S 4]{Kiem2013}; see also
\cite[\S 3.6]{mpt-curves-k3-modular}. We freely use the
notation in \S\ref{sec:APOTs}, and the following construction should
be compared with the construction there.

Suppose $\fX$ is Deligne--Mumford and of finite presentation, and let
$\varphi$ be a $\sT$-equivariant APOT on $\fX$ relative to $\fB$.
Given a surjective cosection
\[ \varsigma\colon \cOb_\varphi \twoheadrightarrow \sU \otimes \cO_\fX, \]
let $\cOb_\varphi(\varsigma)$ denote the sheaf stack associated with
$\ker(\varsigma)$, in the terminology of \cite[\S
  2]{kiem_savvas_apot}. Then the argument in \cite[\S
  4.2]{kiem_savvas_loc} shows that the coherent sheaf
$\cO_{\fc_\varphi}$ of the intrinsic normal cone stack $\fc_\varphi
\subset \cOb_\varphi$ is actually supported on the closed substack
$\cOb_\varphi(\varsigma)$. A standard d\'evissage argument, see e.g.
\cite[(4.9)]{kiem_savvas_loc}, expresses $\cO_{\fc_\varphi}$ as a
linear combination of classes supported on $\cOb_\varphi(\varsigma)$.
Pulling these classes back along $f\colon \cOb_\varphi(\varsigma) \to
\cOb_\varphi$ \cite[(3.10)]{kiem_savvas_loc} produces actual classes
on $\cOb_\varphi(\varsigma)$. Applying
$0^!_{\cOb_\varphi(\varsigma)}\colon K_\sT(\cOb_\varphi(\varsigma))
\to K_\sT(\fX)$, the K-theoretic Gysin map of the sheaf stack
$\cOb_\varphi(\varsigma)$, we can define a Gysin map \footnote{Since
  we start with a sheaf supported on $\cOb_\varphi(\varsigma)$, on
  charts $E$ of $\cOb_\varphi$, these are pushed forward from the
  associated chart $E(\varsigma)$ of $\cOb_\varphi(\varsigma)$, and
  this operation locally simply applies
  $0^!_{\cOb_\varphi(\varsigma)}$ to this underlying sheaf on
  $E(\varsigma)$. This way, we avoid having to use a genuine
  d\'evissage argument for sheaf stacks.}
\[ 0^!_{\cOb_\varphi,\varsigma}\coloneqq 0^!_{\cOb_\varphi(\varsigma)}\circ f^*\colon \cat{Coh}_\sigma(\cOb_\varphi) \to K_\sT(\fX). \]
The {\it reduced virtual structure sheaf} is then defined as
\[ \cO_\fX^{\vir,\red} \coloneqq 0^!_{\cOb_\varphi,\varsigma}\cO_{\fc_\varphi} \in K_\sT(\fX).\]

\subsubsection{}

\begin{lemma}[Lifting cosections along symmetrized pullbacks] \label{lem:lifting-Obs-cosections}
  Let $\fX$ be an algebraic space and $f\colon X \to \fY$ be a smooth
  morphism in $\cat{Art}_\sT$ over a base $\fB$. Suppose
  $\varphi_{\fY/\fB}\colon \bE_{\fY/\fB} \to \bL_{\fY/\fB}$ is a
  $\kappa$-symmetric obstruction theory on $\fY$ relative to $\fB$
  with cosections $\sigma_\fY\colon \bE_{\fY/\fB}^\vee[1] \to \sU
  \otimes \cO_{\fY}$. Then the obstruction sheaf of the symmetrized
  pullback APOT $\varphi_{X/\fB}$ has cosections
  $\cOb_{\varphi_{X/\fB}} \twoheadrightarrow \sU \otimes \cO_X$.
\end{lemma}

Combined with the construction of \S\ref{rc:sec:reduced-apot}, this
allows us to construct reduced universal enumerative invariants of
semistable=stable loci in auxiliary stacks.

In theory, we could define what it means for an APOT itself to have
cosections and check that the cosections of $\varphi_{\fY/\fB}$ lift
compatibly to cosections of $\varphi_{X/\fB}$ by lifting the
cosections along each affine \'etale chart, but this is cumbersome and
we will not need this.

\begin{proof}
  As in the proof of Theorem~\ref{thm:APOTs}, let $a\colon A \to X$ be
  a suitable affine bundle and $\bE_{X/\fB}^A$ be the object on $A$
  representing the symmetrized pullback of $\varphi_{\fY/\fB}$ along
  $f$. By construction,
  \[ h^0\left((\bE_{X/\fB}^A)^\vee[1]\right) = a^*\cOb_{\varphi_{X/\fB}}, \]
  and therefore it suffices to construct cosections for
  $\bE_{X/\fB}^A$ and then apply the isomorphism $(a^*)^{-1}$. This is
  done by lifting the cosection $\sigma_{\fY/\fB}$ along ($a^*$ of)
  various parts of the compatibility diagram
  \eqref{eq:obstruction-theories-symmetric-pullback}.
  \begin{itemize}
  \item Apply $(-)^\vee[1]$ to the middle row of
    \eqref{eq:obstruction-theories-symmetric-pullback} to get the
    composition
    \[ \sigma_{\bF}^A\colon \bF^\vee[1] \xrightarrow{\eta^\vee[1]} a^*f^*\bE_{\fY/\fB}^\vee[1] \to \sU \otimes \cO_A. \]
    Note that $h^0(\eta^\vee[1])$ is surjective because the next term
    is $h^0(\bL_f^\vee[2]) = 0$.
  \item Apply $(-)^\vee[1]$ to the right-most column of
    \eqref{eq:obstruction-theories-symmetric-pullback} to obtain an
    exact triangle
    \[ \kappa^{-1} \otimes a^*\bL_f[-1] \xrightarrow{\zeta^\vee[1]} \bF^\vee[1] \to (\bE_{X/\fB}^A)^\vee[1] \xrightarrow{+1}. \]
    We may choose the affine bundle $a$ such that $\sigma_{\bF}^A
    \circ \zeta^\vee[1] = 0$. Then $\sigma_{\bF}^A$ factors through
    \[ \sigma_X^A\colon (\bE_{X/\fB}^A)^\vee[1] \to \sU \otimes \cO_A. \]
    Note that $h^0(\sigma_X^A)$ is surjective because
    $h^0(\sigma_{\bF}^A)$ is.
  \end{itemize}
  So we have constructed the desired cosections $\sigma_X^A$.

\subsubsection{}

  We verify that this lifting procedure is independent of the choice
  of affine bundle as follows. Given any two suitable affine bundles
  $a_1\colon A_1\to X$ and $a_2\colon A_2\to X$, we may form their
  fiber product
  \[ \begin{tikzcd}
      A \ar{d}[swap]{\pr_1} \ar{r}{\pr_2} & A_2 \ar{d}{a_2} \\
      A_1 \ar{r}{a_1} & X,
  \end{tikzcd} \]
  which continues to be an affine bundle. By flatness of $\pr_i$, we
  can pull back the construction
  \eqref{eq:obstruction-theories-symmetric-pullback} of
  $\bE_{X/\fB}^{A_i}$ along $\pr_i$, and use $a_1 \circ \pr_1 = a_2
  \circ \pr_2$ to identify
  \[ \pr_1^* \bE_{X/\fB}^{A_1} = \pr_2^* \bE_{X/\fB}^{A_2} \]
  along with all the relevant morphisms, including
  $\pr_i^*\sigma_X^{A_i}$. Upon applying $((a_i \circ \pr_i)_*)^{-1}$,
  this shows that the two constructed cosections of
  $\cOb_{\varphi_{X/\fB}}$ are equal.
\end{proof}

\subsubsection{}

\begin{lemma}[Deformation invariance of $\cO^{\vir,\red}$] \label{lem:red-APOT-virtual-cycle-functoriality}
  Consider a Cartesian square
  \[ \begin{tikzcd}
    \fY \ar{r}{u}\ar{d} & \fX \ar{d} \\
    Z \ar{r}{v} & W
  \end{tikzcd} \]
  in $\cat{Art}_\sT$, where $v\colon Z \to W$ is a regular closed
  immersion of smooth varieties. Suppose we are given
  $\sT$-equivariant APOTs
  \[ \varphi = (U_i, \, \varphi_i\colon \bE_i \to \bL_{U_i})_{i \in I}, \qquad \varphi' = (V_i \coloneqq U_i \times_{\fX} \fY, \, \varphi_i'\colon \bE_i' \to \bL_{V_i})_{i \in I}, \]
  on $\fX$ and $\fY$ respectively, relative to a base $\fB$, which are
  compatible under $u$ in the sense that there is a morphism of exact
  triangles (cf.
  Definition~\ref{def:obstruction-theories-compatibility})
  \begin{equation} \label{eq:functoriality-compatibility-diagram}
    \begin{tikzcd}
      \cN_{Z/W}\big|_{V_i} \ar{d} \ar{r} & \bE_i|_{V_i} \ar{r}{g_i} \ar{d}{\varphi_i\big|_{V_i}} & \bE'_i \ar{d}{\varphi'_i} \ar{r}{+1} & {} \\
      \bL_{V_i/U_i}[-1] \ar{r} & \bL_{U_i/\fB}\big|_{V_i} \ar{r} & \bL_{V_i/\fB} \ar{r}{+1} & {}
    \end{tikzcd}
  \end{equation}
  compatible with the isomorphisms in the Definition~\ref{bg:def:apot}
  of an APOT. Then any surjective cosections $\sigma\colon
  \cOb_\varphi \twoheadrightarrow \cO_{\fX}^{\oplus o}$ induce
  surjective cosections
  \[ \sigma'\colon \cOb_{\varphi'} \twoheadrightarrow \cOb_\varphi|_\fY \xrightarrow{\sigma|_\fY} \cO_\fY^{\oplus o}, \]
  and moreover, letting $v^! \coloneqq v_*\cO_Z \otimes_{\cO_{\fX}}
  -\colon K_\sT(\fX) \to K_\sT(\fY)$ be the Gysin map,
  \[ \cO^\red_{\fY} = v^!\cO^\red_{\fX} \in K_\sT(\fY). \]
\end{lemma}

The proof, in \S\ref{af:sec:rd-functoriality-proof}, will follow from
various ingredients already in the literature.

\subsubsection{}
\label{af:rmk:deformation-invariance}

Recall that the universal enumerative invariant $\sZ_X \coloneqq
\chi(X, \hat\cO_X^\vir \otimes -)$ is {\it deformation invariant} in
the sense \cite[\S 3.5]{Fantechi2010} that if
\[ \begin{tikzcd}
  X_b \ar[hookrightarrow]{r} \ar{d} & \fX \ar{d} \\
  \{b\} \ar[hookrightarrow]{r}{i_b} & B,
\end{tikzcd} \]
is a flat family of such $X$'s with compatible $\sT$-equivariant
APOTs, then for every $E \in K_\sT^\circ(\fX)$, the quantity
$\sZ_{X_b}(E|_{X_b})$ is locally constant in $b$. This deformation
invariance follows directly from $i_b^!\cO^\vir_{\fX} =
\cO^\vir_{X_b}$ and the local constancy of Euler characteristic in
proper flat families.

Since Lemma~\ref{lem:red-APOT-virtual-cycle-functoriality} implies
$i_b^!\cO^{\vir,\red}_{\fX} = \cO^{\vir,\red}_{X_b}$, it immediately
implies the same deformation invariance for the reduced universal
enumerative invariant $\sZ_X^{\red}$.

Note that if $\fX = X \times B \to B$ is a trivial family over $B$,
with projection $p\colon \fX \to X$, then evaluation on the image of
$p^*\colon K_\sT^\circ(X) \to K_\sT^\circ(\fX)$ shows that the element
$\sZ_{X_b} \in K_\circ^{\tilde\sT}(X)$ is locally constant in $b$.

\subsubsection{}
\label{af:sec:rd-functoriality-proof}

\begin{proof}[Proof of Lemma~\ref{lem:red-APOT-virtual-cycle-functoriality}.]
  We follow the proof of \cite[Theorem A.1]{kiem_savvas_loc} closely.
  Let $\cM_{\fX}^\circ \to \bP^1$ denote the deformation of $\fX$ to
  the intrinsic normal cone $\fC_{\fX} = \fC_{\fX/\fB}$. Following a
  standard argument in \cite[Theorem 1]{kkp-funct-vir-class},
  Kiem--Savvas use a ``double deformation space'' $\fW$, which is the
  deformation of $\fY\times \bP^1$ to its normal cone stack in
  $\cM_{\fX}^\circ$, to prove \cite[(A.8)]{kiem_savvas_loc}
  \[ \cO_{\fC_{\fY}} = \cO_{\fC_{\fY/\fC_{\fX}}} \in K_\sT(\fN_{\fY\times\bP^1/\cM_{\fX}^\circ}). \]
  As in the proof of \cite[Theorem 2.1.6]{klt_dtpt}, applying the
  ``coarsifying'' operation $(\pi_{\fY\times
    \bP^1/\cM_\fX^\circ})_\diamond$ defined in \cite[Def.
    2.1.8]{klt_dtpt}, we obtain
  \begin{equation} \label{eq:coarse-normal-cone-eq-ks}
    \cO_{\fc_{\fY}} = \cO_{\fc_{\fY/\fC_{\fX}}}\in K_\sT(\fn_{\fY\times\bP^1/\cM_{\fX}^\circ}).
  \end{equation}
  Following \cite{kkp-funct-vir-class}, consider the morphism
  \begin{equation}\label{eq:kappa-def-cd}
    \begin{tikzcd}
      \bE_i\big|_{V_i}(-1) \ar{r}{\kappa_i} \ar{d} & \bE_i\big|_{V_i}\oplus \bE_i' \ar{r}\ar{d} & \cone(\kappa_i) \ar{r}{+1}\ar{d} & {}\\
      \bL_{U_i/\fB}\big|_{V_i}(-1) \ar{r}{\lambda_i} & \bL_{U_i/\fB}\big|_{V_i}\oplus \bL_{V_i/\fB} \ar{r} & \cone(\lambda_i) \ar{r}{+1} & {}
    \end{tikzcd}
  \end{equation}
  of exact triangles, for each $i \in I$. Here, $\kappa_i = (x_0 \cdot
  \id, x_1 \cdot g_i)$ where $[x_0 : x_1] \in \bP^1$ are coordinates,
  and $\lambda_i$ is the restriction of the global morphism $\lambda$
  from \cite[Prop. 1]{kkp-funct-vir-class}, which gives us
  $h^1/h^0(\cone(\lambda)^\vee)=\cN_{\fY\times \bP^1/\cM_\fX^\circ}$.
  By \cite{kkp-funct-vir-class}, we get a closed embedding
  \[ \cN_{\fY\times \bP^1/\cM_\fX^\circ}\hookrightarrow \cK \]
  of vector bundle stacks, where $\cK$ is glued from
  $h^1/h^0(\cone(\kappa_i)^\vee)$. By the defining gluing property of
  APOTs and \eqref{eq:functoriality-compatibility-diagram}, the closed
  embeddings $h^1(\cone(\lambda_i)^\vee)\hookrightarrow
  h^1(\cone(\kappa_i)^\vee)$ glue to a compatible closed embedding
  \[ \fn_{\fY\times \bP^1/\cM_\fX^\circ}\hookrightarrow \fK \]
  of sheaf stacks, where $\fK$ is glued from
  $h^1(\cone(\kappa_i)^\vee)$.

\subsubsection{}

  The upper row of \eqref{eq:kappa-def-cd} is used in
  \cite[Eq. (5.7)]{Kiem2013} to induce a surjective twisted
  cosection $\tilde{\sigma}\colon \cK \twoheadrightarrow
  \cO_{\fY\times \bP^1}(-1)$. Similarly, the upper row of
  \eqref{eq:kappa-def-cd} together with
  \eqref{eq:functoriality-compatibility-diagram} are used to obtain
  the commutative diagram
  \begin{equation} \label{eq:cd-K-cosection}
    \begin{tikzcd}
      \fK \ar{r}\ar{d} & \cOb_\varphi|_\fY\oplus \cOb_{\varphi'} \ar{r}\ar{d} & \cOb_\varphi|_\fY(1) \ar{r}\ar{d} & 0\\
      \cO_{\fY\times \bP^1}(-1) \ar{r} & \cO_{\fY\times \bP^1}\oplus \cO_{\fY\times \bP^1} \ar{r} & \cO_{\fY\times \bP^1}(1) \ar{r} & 0
    \end{tikzcd}
  \end{equation}
  of sheaves, which gives us a twisted cosection $\bar\sigma\colon
  \fK \to \cO_{\fY\times \bP^1}(-1)$. Collecting all the embeddings and
  cosections, we get a commutative diagram
  \begin{equation}\label{eq:cosections-emb-cd}
    \begin{tikzcd}
      \fC_{Y \times \bP^1/\cM_\fX^\circ} \ar[hookrightarrow]{r}\ar[twoheadrightarrow]{d} & \cN_{\fY\times \bP^1/\cM_\fX^\circ} \ar[hookrightarrow]{r}\ar{d}{\pi_{\fY\times \bP^1/\cM_\fX^\circ}} & \cK \ar[twoheadrightarrow]{rd}{\tilde{\sigma}}\ar{d}{\pi_\cK} & {}\\
      \fc_{Y\times \bP^1/\cM_\fX^\circ} \ar[hookrightarrow]{r} & \fn_{\fY\times \bP^1/\cM_\fX^\circ} \ar[hookrightarrow]{r} & \fK \ar[twoheadrightarrow]{r}{\bar\sigma} & \cO_{\fY\times \bP^1}(-1).
    \end{tikzcd}
  \end{equation}
  The fiber of $\fK$ over $0 \in \bP^1$ is $u^*\cOb_\varphi \oplus
  \cN_{Z/W}|_{\fY}$ and the fiber over $1 \in \bP^1$ is
  $\cOb_{\varphi'}$. By construction, the twisted cosection
  $\bar\sigma$ restricts to the corresponding cosections over these
  fibers. Then, as in \cite[Lemma 5.3]{Kiem2013}, the
  class $\cO_{\fC_{Y\times \bP^1/\cM_\fX^\circ}}$ has reduced support
  in $\cK(\tilde{\sigma})$, the vector bundle stack associated to
  $\ker \tilde\sigma$. Applying the ``coarsifying'' operation
  $(\pi_{\fY\times \bP^1/\cM_\fX^\circ})_\diamond$, and using
  \eqref{eq:cosections-emb-cd}, we see that $\fc_{Y\times
    \bP^1/\cM_\fX^\circ}$ has reduced support in $\fK(\bar\sigma)$, so
  that \eqref{eq:coarse-normal-cone-eq-ks} holds in $\fK(\bar\sigma)$.
  Applying $0^!_{\cOb_\fY(\sigma')}$ and using the usual properties of
  Gysin maps yields the desired equality.
\end{proof}

\subsubsection{}
\label{rc:sec:master-space-argument}

\begin{proof}[Proof of Theorems~\ref{thm:red-sst-invariants} and \ref{thm:red-wcf}.]
  With all auxiliary reduced universal enumerative invariants defined
  using Lemma~\ref{lem:lifting-Obs-cosections}, it remains to modify
  the master space wall-crossing steps of
  \S\ref{sec:pairs-master-term-3} and \S\ref{sec:horizontal-wc} to use
  the reduced invariants. The rest of the content in
  \S\ref{sec:semistable-invariants} and \S\ref{sec:wall-crossing} go
  through with only obvious modifications.

  The key modification is the usage of a reduced version of
  Theorem~\ref{thm:APOT-comparison} in the master space wall-crossing
  steps. Namely, let $\beta = \gamma + \delta \in C(\cat{A})$ and
  consider the following in $\cat{Art}_\sT$:
  \[ \begin{tikzcd}
    Z \ar{d}{f} \ar[hookrightarrow]{r} & \bM \ar[loop right]{r}{\bC^\times} \ar{d}{g} \\
    \fM_{\gamma} \times \fM_{\delta} & \fM_\beta, 
  \end{tikzcd} \]
  where $f$ and $g$ are smooth, $\bM$ is an algebraic space with
  $\bC^\times$-action compatible with the $\sT$-action (the master
  space), and $Z \subset \bM$ is a $\bC^\times$-fixed component (the
  complicated locus \eqref{eq:pairs-master-space-complicated-locus} or
  \eqref{wc:eq:master-space-complicated-locus}). Let $\sZ_Z^{\red}$
  and $\sZ_{Z \subset \bM}^{\red}$ be the reduced universal
  enumerative invariants associated to symmetrized pullback APOT on
  $Z$ and the $\bC^\times$-fixed part of the restriction of the
  symmetrized pullback APOT on $\bM$, respectively. We claim
  \[ \sZ_{Z \subset \bM}^{\red} = \begin{cases} \sZ_Z^{\red} & o_\beta = o_{\gamma} + o_{\delta} \\ 0 & \text{otherwise}. \end{cases} \]

  We prove this by adapting the proof of
  Theorem~\ref{thm:APOT-comparison}. First, build compatible affine
  bundles $a\colon A \to Z$ and $b\colon B \to \bM$ which remove the
  obstructions to the existence of the diagram
  \eqref{eq:obstruction-theories-symmetric-pullback} which constructs
  the would-be symmetrized pullback POTs $\hat\bE_Z^A \to a^*\bL_Z$
  and $\hat\bE_{\fM}^B \to b^*\bL_{\bM}$ for $Z$ and $\bM$
  respectively. We claim that
  \begin{equation} \label{eq:APOT-comparison-underlying-POTs}
    \hat\bE_Z^A \oplus \Omega_a \cong \left((\hat\bE_{\bM}^B \oplus \Omega_b)\big|_A\right)^{\text{fix}}
  \end{equation}
  as objects in the derived category. (Here, the superscript
  $\text{fix}$ means the $\bC^\times$-fixed part.) This is because
  both sides are constructed via taking cones, which is unique up to
  isomorphism, of objects which may be identified as
  \[ f^*\left((\bE_{\fM_{\gamma}} \boxplus \bE_{\fM_{\delta}}\right) = \left((g^*\bE_{\fM_{\beta}})\big|_Z\right)^{\text{fix}}, \quad \bL_f = \left(\bL_g\big|_Z\right)^{\text{fix}}, \]
  by the bilinearity of the obstruction theory on $\fM$ and of $\bL_f$
  and $\bL_g$, and moreover $\Omega_a \cong
  (\Omega_b|_A)^{\text{fix}}$ by construction.

  Hence, for any \'etale cover $\{U_i \to A\}_{i \in I}$, the two
  sides of \eqref{eq:APOT-comparison-underlying-POTs} define APOTs
  $\varphi_Z^A$ and $\varphi_{Z \subset \bM}^A$, respectively, whose
  underlying objects $\bE_{Z,i}^A \cong \bE_{Z \subset \bM,i}^A$ are
  isomorphic but whose local POTs $\varphi_{Z,i}^A, \varphi_{Z \subset
    \bM,i}^A\colon \bE_{Z,i}^A \to \bL_{U_i}$ may be different.
  Moreover, by the proof of Lemma~\ref{lem:lifting-Obs-cosections},
  the obstruction sheaves $\cOb_Z^A$ and $\cOb_{Z \subset \bM}^A$ of
  these APOTs carry surjective cosections $\sigma_Z^A$ and $\sigma_{Z
    \subset \bM}^A$, lifted from $\fM_{\gamma} \times \fM_{\delta}$
  and $\fM_\beta$ respectively, and these cosections may be different.

  We follow the idea of the deformation argument in \cite[Proof of
    Prop. 9.6(c)]{Joyce2021} in order to identify the reduced
  universal enumerative invariants. Let $d \coloneqq
  o_{\gamma}+o_{\delta}-o_\beta$; note that $d \ge 0$ by
  Assumption~\ref{assump:cosections}. Choose a projection $q\colon
  \cO_A^{o_{\gamma}+o_{\delta}} \twoheadrightarrow \cO_A^{o_{\beta}}$.
  We consider the family of objects
  \begin{align*}
    \varphi_{i,t} \coloneqq t\varphi_{Z,i}^A + (1-t)\varphi_{Z \subset \bM,i}^A&\colon \bE_{Z,i}^A \to \bL_{U_i}, \quad i \in I, \\
    \sigma_t \coloneqq t q \circ \sigma_Z^A + (1 - t) \sigma_{Z \subset\bM}^A&\colon  \to \cO_A^{o_{\beta}}
  \end{align*}
  on $A$ linearly interpolating between $(\varphi_Z^A, q \circ
  \sigma_Z^A)$ and $(\varphi_{Z \subset \bM}^A, \sigma_{Z \subset
    \bM}^A)$. Relative to the trivial flat family $A \times \bA^1 \to
  \bA^1$, the conditions that $\varphi_{i,t}$ is a {\it perfect}
  obstruction theory and that $\sigma_t$ is surjective are open
  conditions on $t \in \bA^1$. Restricting to the open subset where
  these conditions hold, which includes $0, 1 \in \bA^1$ by
  construction, we obtain a compatible family of APOTs (whose
  obstruction sheaves have surjective cosections) over a connected
  base. Then Lemma~\ref{lem:red-APOT-virtual-cycle-functoriality} and
  the discussion of \S\ref{af:rmk:deformation-invariance} apply, and
  yield the deformation invariance
  \[ \sZ_Z^{\red} = \sZ_{Z \subset \bM}^{\red} \in K_\circ^{\tilde\sT}(Z)_{\loc} \]
  if $q$ is an isomorphism, i.e. if $o_{\gamma} + o_{\delta} =
  o_\beta$. Otherwise, if $o_{\gamma} + o_{\delta} > o_\beta$, then
  the other projection $\cO_A^{o_{\gamma}+o_{\delta}}
  \twoheadrightarrow \cO_A^d$ gives $d > 0$ extra cosections of the
  $\cOb_{Z \subset \bM}^A$, which makes $\sZ_{Z \subset \bM}^{\red}$
  vanish.
\end{proof}

\section{Example: equivariant DT/PT/BS theory}
\label{sec:DT-PT-BS}

\subsection{Wall-crossings of simple type}
\label{sec:wall-crossings-simple-type}

\subsubsection{}

\begin{definition} \label{def:wall-crossing-simple-type}
  Let $\cat{A}$ be an abelian category, and let $A \subset C(\cat{A})$
  be a subset of effective classes. Let $\{\tau_t\}_{t \in [-1,1]}$ be
  a one-parameter family of weak stability conditions on $\cat{A}$,
  and let $\tau_\pm$ denote $\tau_{\pm 1}$ for short. We say
  $\{\tau_t\}_t$ defines a {\it wall-crossing problem of simple type
    for $A$} if there exists $B \subseteq \tilde B \subset C(\cat{A})
  \setminus A$ such that:
  \begin{enumerate}[label = (\alph*)]
  \item $\tilde B + A \subseteq A$;
  \item if $\alpha \in A$ and $\alpha = \alpha_1 + \cdots + \alpha_n$
    is a splitting such that $\tau_0(\alpha_1) = \cdots =
    \tau_0(\alpha_n)$, then there is a unique $i \in \{1, \ldots, n\}$
    such that
    \begin{equation} \label{eq:simple-type-splitting}
      \alpha_i \in A, \qquad \alpha_j \in \tilde B \text{ for all } j \neq i;
    \end{equation}
  \item for all $t \in [-1, 1]$, the function $\tau_t$ is constant on
    $A$ and on $B$, and, denoting these values by $\tau_t(A)$ and
    $\tau_t(B)$,
    \[ \tau_{t_-}(A) > \tau_{t_-}(B), \qquad \tau_0(A) = \tau_0(B), \qquad \tau_{t_+}(A) < \tau_{t_+}(B) \]
    for any $t_- < 0$ and any $t_+ > 0$;
  \item for any classes $\gamma, \gamma' \in C(\cat{A})$,
    \begin{equation} \label{eq:simple-type-inert-class-1}
      \tau_0(\gamma) < \tau_0(\gamma') \implies \tau_t(\gamma) < \tau_t(\gamma') \text{ for all } t \in [-1, 1];
    \end{equation}
  \item for any classes $\gamma \in \tilde B \setminus B$, we have
    $\tau_+(\gamma) \neq \tau_+(A)$. Moreover,
    \begin{equation} \label{eq:simple-type-inert-class-2}
      \begin{aligned}
        \tau_+(\gamma) > \tau_+(A) &\implies \tau_-(\gamma) > \tau_-(A), \\
        \tau_+(\gamma) < \tau_+(A) &\implies \tau_-(\gamma) < \tau_-(A),
      \end{aligned}
    \end{equation}
    and the same for $\tau_\pm(B)$ in place of $\tau_\pm(A)$.
  \end{enumerate}
  The intuition is that, for classes in $A$, there is a wall at $t=0$
  at which the slopes of classes in $A$ and classes in $\tilde B$ can
  coincide, but only the classes in $B \subset \tilde B$ truly
  ``cross'' the wall.
\end{definition}

\subsubsection{}

\begin{proposition} \label{prop:Utilde}
  Let $\{\tau_t\}_{t \in [-1,1]}$ define a wall-crossing problem of
  simple type. Let $\alpha \in A$ and
  \[ \alpha = \alpha_1 + \cdots + \alpha_n, \quad n > 0, \]
  be a splitting into effective classes. Then
  \[ \tilde U(\alpha_1, \ldots, \alpha_n; \tau_-, \tau_+) = \begin{cases} 1/(n-1)! & \alpha_1 \in A, \; \alpha_2, \ldots, \alpha_n \in B, \\ 0 & \text{otherwise}. \end{cases} \]
\end{proposition}

Proposition~\ref{prop:Utilde}, along with the wall-crossing formula
(Theorem~\ref{thm:wcf}), immediately implies
Proposition~\ref{prop:wcf-simple-type}. The goal of this subsection is
to prove Proposition~\ref{prop:Utilde} from first principles, i.e.
Definition~\ref{def:universal-coefficients} of $\tilde U$. So, for the
remainder of this subsection, assume we are in the setting of
Proposition~\ref{prop:Utilde} above.

\subsubsection{}

\begin{lemma} \label{lem:S-and-U-vanishing}
  Unless $\alpha_1, \ldots, \alpha_n \in A \sqcup B$, we have
  \[ S(\alpha_1, \ldots, \alpha_n; \tau_-, \tau_+) = U(\alpha_1, \ldots, \alpha_n; \tau_-, \tau_+) = \tilde U(\alpha_1, \ldots, \alpha_n; \tau_-, \tau_+) = 0. \]
\end{lemma}

Note that $\alpha_1, \ldots, \alpha_n \in A \sqcup B$ implies that
$\tau_0(\alpha_1) = \cdots = \tau_0(\alpha_n)$ and thus
\eqref{eq:simple-type-splitting} holds. Hence there is a unique $i \in
\{1, \ldots, n\}$ such that
\begin{equation} \label{eq:simple-type-splitting-true}
  \alpha_i \in A, \qquad \alpha_j \in B \text{ for all } j \neq i.
\end{equation}

\begin{proof}
  Suppose the $\tau_0(\alpha_i)$ are not all equal, so either $\min_i
  \tau_0(\alpha_i) < \tau_0(A)$ or $\tau_0(A) < \max_i
  \tau_0(\alpha_i)$. Suppose $\min_i \tau_0(\alpha_i) < \tau_0(A)$.
  Let $I$ be the set of $i$ for which $\tau_0(\alpha_i)$ is minimal.
  For such $i \in I$, we have $\tau_0(\alpha_i) < \tau_0(A)$ and
  $\tau_0(\alpha_i) < \tau_0(\alpha_j)$ for $j \notin I$. By
  \eqref{eq:simple-type-inert-class-1}, $\tau_+(\alpha_i) < \tau_+(A)$
  and $\tau_-(\alpha_i) < \tau_-(\alpha_j)$ for $j \notin I$. The
  desired vanishing then follows from
  Lemma~\ref{wc:lem:U-properties}\ref{it:U-vanishing} applied to $I$.
  The other case $\tau_0(A) < \max_i \tau_0(\alpha_i)$ is completely
  analogous.
  
  Suppose the $\tau_0(\alpha_i)$ are all equal, and therefore
  \eqref{eq:simple-type-splitting} holds, but at least $\alpha_i$
  belongs to $\tilde B \setminus B$. Set $I_> \coloneqq \{i : \alpha_i
  \in \tilde B \setminus B, \, \tau_+(\alpha_i) > \tau_+(A)\}$ and
  similarly define $I_<$. For $i \in I_>$,
  \eqref{eq:simple-type-inert-class-2} yields $\tau_-(\alpha_i) >
  \tau_-(A)$ ($> \tau_-(B)$). For $i \in I_<$,
  \eqref{eq:simple-type-inert-class-2} yields $\tau_-(\alpha_i) >
  \tau_-(A), \tau_-(B)$. Hence, if $I_> \neq \emptyset$ (resp. $I_<
  \neq \emptyset$), then the desired vanishing follows from
  Lemma~\ref{wc:lem:U-properties}\ref{it:U-vanishing} applied to $I_>$
  (resp. $I_<$).
\end{proof}

\subsubsection{}

\begin{lemma} \label{lem:S-calculation}
  \[ S(\alpha_1, \ldots, \alpha_n; \tau_-, \tau_+) = \begin{cases} (-1)^{n-1} & \alpha_n \in A, \; \alpha_1, \ldots, \alpha_{n-1} \in B, \\ (-1)^{n-2} & \alpha_{n-1} \in A, \; \alpha_1, \ldots, \alpha_{n-2}, \alpha_n \in B, \\ 0 & \text{otherwise.} \end{cases} \]
\end{lemma}

\begin{proof}
  By Lemma~\ref{lem:S-and-U-vanishing}, $S = 0$ unless the splitting
  \eqref{eq:simple-type-splitting-true} holds.
  \begin{itemize}
  \item If $i < n-1$, then $\tau_-(\alpha_{n-1}) =
    \tau_-(\alpha_n)=\tau_-(B)$, but
    $\tau_+(\alpha_1+\cdots+\alpha_{n-1}) = \tau_+(A) < \tau_+(B) =
    \tau_+(\alpha_n)$. Hence $S = 0$ by definition.
  \item If $i = n-1$, then condition~\ref{it:S-case-1} in the
    definition of $S$ is satisfied for all $i < n-1$, and
    condition~\ref{it:S-case-2} is satisfied for $i = n-1$. Hence $S =
    (-1)^{n-2}$.
  \item If $i = n$, then condition~\ref{it:S-case-1} in the definition
    of $S$ is satisfied for all $i$. Hence $S = (-1)^{n-1}$. \qedhere
  \end{itemize}
\end{proof}

\subsubsection{}

\begin{lemma} \label{lem:U-calculation}
  \[ U(\alpha_1, \ldots, \alpha_n; \tau_-, \tau_+) = \begin{cases} \frac{(-1)^{i-1}}{(n-i)!(i-1)!} & \alpha_i \in A, \; \alpha_j \in B \text{ for all } j \neq i \\ 0 & \text{otherwise}. \end{cases} \]
\end{lemma}

\begin{proof}
  By Lemma~\ref{lem:S-and-U-vanishing} and the definition of a simple
  type wall-crossing problem, $U = 0$ unless we are in the first case.
  We use the notation of Definition~\ref{def:universal-coefficients}
  for $U$. Note that $\tau_\pm(A) \neq \tau_\pm(B)$. Then the
  $\tau_-$-permissible condition for the double grouping implies that
  if $a_{j-1} < i \le a_j$ then actually $a_{j-1} + 1 = a_j$, i.e.
  $\beta_j = \alpha_i$, and the $\tau_+$-permissible condition for the
  double grouping implies that $l = 1$. Hence
  \[ U(\alpha_1, \ldots, \alpha_n; \tau_-, \tau_+) = \sum_{\substack{m>0\\0=a_0<a_1<\cdots<a_m=n\\\exists j:\, a_{j-1}+1=a_j}} S(\beta_1, \beta_2, \ldots, \beta_m; \tau_-, \tau_+) \prod_{k=1}^m \frac{1}{(a_k - a_{k-1})!}. \]
  By Lemma~\ref{lem:S-calculation}, a term in this sum is non-zero if
  and only if $a_{m-1} = i$ or $a_m = i$. The former case is only possible
  if $i < n$, giving
  \[ U(\alpha_1, \ldots, \alpha_n; \tau_-, \tau_+) = \frac{1}{(n-i)!} \sum_{\substack{m>0\\0=a_0<a_1<\cdots<a_{m-2}=i-1}} (-1)^{m-2} \prod_{k=1}^{m-2} \frac{1}{(a_k - a_{k-1})!}. \]
  The latter case is possible only if $i = n$, giving
  \[ U(\alpha_1, \ldots, \alpha_n; \tau_-, \tau_+) = \sum_{\substack{m>0\\0=a_0<a_1<\cdots<a_{m-1}=n-1}} (-1)^{m-1} \prod_{k=1}^{m-1} \frac{1}{(a_k - a_{k-1})!}. \]
  We conclude using the following Lemma~\ref{lem:U-combinatorics}.
\end{proof}

\subsubsection{}

\begin{lemma} \label{lem:U-combinatorics}
  \[ \sum_{\substack{m>0\\0=a_0<a_1<\cdots<a_m=n}} (-1)^m \prod_{k=1}^m \frac{1}{(a_k - a_{k-1})!} = \frac{(-1)^n}{n!}. \]
\end{lemma}

\begin{proof}
  Let $c_n$ denote the left hand side. We prove $c_n = (-1)^n/n!$ by
  induction on $n$. The base case $n = 1$ is clear. For the inductive
  step, note that a sequence $0 = a_0 < a_1 < \cdots < a_m = n+1$ with
  $m > 0$ either has $m = 1$, or is equivalent to the data of an
  integer $0 < b < n+1$ and a sequence $0 = a_0 < a_1 < \cdots <
  a_{m-1} = b$ with $m-1 > 0$. Hence
  \[ c_{n+1} = \frac{-1}{(n+1)!} + \sum_{b=1}^n c_b \cdot \frac{-1}{(n+1-b)!} = -\sum_{b=0}^n (-1)^b \frac{1}{b! (n+1-b)!} = \frac{(-1)^{n+1}}{(n+1)!} \]
  where the second equality is the induction hypothesis, and the third
  equality follows from $\sum_{b=0}^{n+1} (-1)^b \binom{n+1}{b} = 0$.
\end{proof}

\subsubsection{}

\begin{proof}[Proof of Proposition~\ref{prop:Utilde}.]
  By Lemma~\ref{lem:U-calculation},
  \begin{align*}
    &\sum_{\substack{n \ge 1\\\alpha = \alpha_1+\cdots+\alpha_n}} U(\alpha_1,\ldots,\alpha_n; \tau, \tilde\tau) \epsilon(\alpha_1) \epsilon(\alpha_2) \cdots \epsilon(\alpha_n) \\
    &= \sum_{n \ge 1} \sum_{i=1}^n \frac{(-1)^{i-1}}{(n-i)! (i-1)!} \sum_{\substack{\alpha_i \in A\\\forall j \neq i: \, \alpha_j \in B}} \epsilon(\alpha_1)\epsilon(\alpha_2) \cdots \epsilon(\alpha_n) \\
    &= \exp\bigg(-\sum_{\beta \in B} \epsilon(\beta)\bigg) \bigg(\sum_{\beta \in A} \epsilon(\beta)\bigg)\exp\bigg(\sum_{\beta \in B} \epsilon(\beta)\bigg).
  \end{align*}
  Using the identity $e^{-X} Y e^{X} = e^{-\ad_X} Y$, where $\ad_X
  \coloneqq [X, -]$, this becomes
  \[ \sum_{n \ge 1} \sum_{\substack{\alpha_1 \in A\\\alpha_2,\ldots,\alpha_n \in B}} \frac{1}{(n-1)!} \left[\left[ \cdots \left[[\epsilon(\alpha_1), \epsilon(\alpha_2)], \epsilon(\alpha_3)\right], \ldots\right], \epsilon(\alpha_n)\right]. \]
  The desired formula now follows from the definition
  \eqref{eq:U-vs-Utilde} of $\tilde U$.
\end{proof}

\subsection{DT/PT vertex correspondence}
\label{sec:DT-PT}

\subsubsection{}

Given a smooth quasi-projective variety $X$, let $\cat{Coh}_{\le
  d}(X)$ be the abelian category of coherent sheaves on $X$ whose
support is proper of dimension $\le d$. (The absence of a subscript
$\le d$ means no constraint on the dimension.) Let
\begin{align*}
  \cat{A}_X &\coloneqq \inner{\cO_X, \cat{Coh}_{\le 1}(X)[-1]}_{\text{ex}} \subset \cat{D}_X, \\
  \cat{D}_X &\coloneqq \inner{\cO_X, \cat{Coh}_{\le 1}(X)[-1]}_{\text{tr}} \subset D^b\cat{Coh}(X)
\end{align*}
be the smallest extension-closed\footnote{Extensions are taken in the
abelian category $\cat{Coh}^\dag(X)[-1]$, where $\cat{Coh}^{\dag}(X)$
is the tilt of $\cat{Coh}(X)$ with respect to the torsion pair
$(\cat{Coh}_{\le 1}(X), \cat{Coh}_{\le 1}(X)^\perp)$; see \cite[\S
  3.1]{toda_dtpt}.} (resp. triangulated) full subcategory
containing $\cO_X$ and $\cat{Coh}_{\le 1}(X)[-1]$. It is known that
$\cat{A}_X$ is the heart of a bounded t-structure on $\cat{D}_X$
\cite[Lemma 3.5]{toda_dtpt} and therefore is an abelian category, and
moreover it is Noetherian \cite[Lemma 6.2]{toda_dtpt} and clearly
$\bC$-linear. This is the ambient abelian category of interest in this
subsection.

Let $\fM_X$ denote the moduli stack parameterizing objects in
$\cat{A}_X$. When it is unambiguous or unimportant, we omit the
subscript $X$ from $\cat{A}_X$ and $\fM_X$. Clearly, direct sum and
scaling automorphisms in $\cat{A}$ make $\fM$ into a graded monoidal
stack.

\subsubsection{}

Specifically, for DT/PT, take $X = (\bP^1)^3$ with the natural scaling
action of $\sT \coloneqq (\bC^\times)^3$. Let $\iota_i\colon D_i
\hookrightarrow X$ be the three $\sT$-invariant boundary divisors
given by setting the $i$-th coordinate on $X$ to $\infty$, let $D
\coloneqq D_1 \cup D_2 \cup D_3$ be their union, and $\bC^3 \coloneqq
X \setminus D$. Let $\kappa$ denote the $\sT$-weight of the trivial
(but not equivariantly trivial) canonical bundle $\cK_{\bC^3}$.

Objects $I \in \cat{A}_X$ have Chern character
\[ \ch(I) = (r, 0, -\beta_C, -n) \]
where $r = \rank(I)$ and $n = \ch_3(I)$ are integers, and $\beta_C =
(\beta_1,\beta_2,\beta_3) \in H_2(X; \bZ)$. On $\cat{A}_X$, define the
family of weak stability conditions
\[ \tau_\xi(r, 0, -\beta_C, -n) \coloneqq \begin{cases} 3\pi/4 & r \neq 0, \\ \pi/2 & r = 0, \, \beta_C \neq 0, \\ 3\pi/4 + \xi & r = 0, \, \beta_C = 0, \, n \neq 0, \end{cases} \]
for $\xi \in (-\pi/4, \pi/4)$. This is a continuous family with
exactly one wall at $\xi = 0$ for objects of rank $r=1$, so let
$\tau^{\PT}$ and $\tau^{\DT}$ denote the cases $\xi > 0$ and $\xi < 0$
respectively.

We proceed to identify $\tau_\xi$-semistable objects of rank $r \le
1$.

\subsubsection{}

\begin{lemma}[{\cite[Lemma 3.11(i), (ii)]{toda_dtpt}}] \label{lem:DT-PT-rank-1-pairs}
  Take $I \in \cat{A}$ with $\rank I = 1$. There is an isomorphism
  \[ I \cong [\cO_X \xrightarrow{s} \cF], \quad \cF \in \cat{Coh}_{\le 1}(X), \, \coker s \in \cat{Coh}_{\le 0}(X), \]
  if and only if $\cH^1(I) \in \cat{Coh}_{\le 0}(X)$. Here $\cO_X$ is
  in degree zero and $\cF$ is in degree one.
\end{lemma}

Moreover, $\im s$ is the structure sheaf of some $1$-dimensional
subscheme $C \subset X$, but we will not use this.

\begin{proof}
  We review Toda's proof, for the reader's convenience.

  If $I$ is of the specified form, clearly $\cH^1(I) = \coker s \in
  \cat{Coh}_{\le 0}(X)$.
  
  Conversely, given any $I \in \cat{A}$ of rank one, decompose it into
  the short exact sequences
  \begin{equation} \label{eq:DT-PT-rank-1-objects}
    \begin{aligned}
      0 \to \cE[-1] \to &I_1 \to \cO_X \to 0, \\
      0 \to I_1 \to &I \to \cQ[-1] \to 0,
    \end{aligned}
  \end{equation}
  where $\cQ \coloneqq \cH^1(I)$, so $\cH^1(I_1) = 0$. Note that this
  means $I_1$ is the ideal sheaf of a $1$-dimensional subscheme $C
  \subset X$. Consider the composition $\cQ[-2] \to I_1 \to \cO_X$.
  Since $\cQ \in \cat{Coh}_{\le 0}(X)$, this map lies in
  \begin{equation} \label{eq:DT-PT-pair-factorization}
    \Hom(\cQ[-2], \cO_X) \cong H^1(X, \cQ \otimes \cK_X)^\vee = 0.
  \end{equation}
  Hence the following solid square exists and can be completed into
  the commutative diagram
  \begin{equation} \label{eq:pair-formation}
    \begin{tikzcd}
      \cQ[-2] \ar{r} \ar{d} & I_1 \ar[dashed]{r} \ar{d} & I \ar[dashed]{r}{+1} \ar[dashed]{d} & {} \\
      0 \ar{r} \ar[dashed]{d} & \cO_X \ar[dashed]{r} \ar[dashed]{d} & \cO_X \ar[dashed]{r}{+1} \ar[dashed]{d}{s} & {} \\
      \cQ[-1] \ar[dashed]{r} \ar[dashed]{d}{+1} & \cE \ar[dashed]{r} \ar[dashed]{d}{+1}  & F \ar[dashed]{r}{+1} \ar[dashed]{d}{+1} & {} \\
         {}& {}& {}
    \end{tikzcd}
  \end{equation}
  of exact triangles, for some $F \in D^b\cat{Coh}(X)$. Rotating the
  bottom row, $F$ is an extension of $\cE, \cQ \in \cat{Coh}_{\le
    1}(X)$ and so $F \in \cat{Coh}_{\le 1}(X)$ as well. Since
  $\cH^1(I_1) = 0$, we can identify $\coker s \cong \cQ$ in
  $\cat{Coh}(X)$. Finally, the right-most column gives the desired
  isomorphism.
\end{proof}

\subsubsection{}
\label{sec:DT-PT-N-stack}

We define a sequence of substacks of $\fM_{1,0,-\beta_C,-n}$, for any
$\beta_C$ and $n$. First, let
\[ \fM^\circ_{1,0,-\beta_C,-n} \subset \fM_{1,0,-\beta_C,-n} \]
be the moduli substack of objects $I \in \cat{A}$ with $\dim \cH^1(I)
= 0$. By Lemmas~\ref{lem:DT-PT-rank-1-pairs}, all objects
parameterized by this substack are isomorphic to pairs $s\colon\cO_X
\to \cF$ where $\cF \in \cat{Coh}_{\le 1}(X)$. Let $\cQ \coloneqq
\coker(s)$ ($= \cH^1(I)$). Second, let
\[ \fN_{\beta_C,n} \subset \fM^\circ_{1,0,-\beta_C,-n} \]
be the moduli substack consisting of the pairs such that:
\begin{enumerate}[label=(\alph*)]
\item \label{it:cond:DT-PT-N-stack-support} (support) $\supp \cQ$
  avoids $D$ and $\supp \cF$ lies in the non-singular locus $D^\circ
  \subset D$;
\item \label{it:cond:DT-PT-N-stack-transverality} (transversality)
  $L^k\iota_i^*\cF = 0$ for $k > 0$, for $i = 1, 2, 3$.
\end{enumerate}
These conditions imply the existence of {\it evaluation maps}
\begin{align*}
  \ev_i\colon \fN_{\beta_C,n} &\to \fHilb(D_i^\circ, \beta_i) \\
  [\cO_X \xrightarrow{s} \cF] &\mapsto [\cO_{D_i^\circ} \xrightarrow{s} \iota_i^*\cF]
\end{align*}
for $i = 1, 2, 3$, which land in the Hilbert stack\footnote{The moduli
stack which is the {\it trivial} $\bC^\times$-gerbe over the usual
Hilbert scheme.} of $\beta_i$ points on $D_i^\circ \coloneqq D_i \cap
D^\circ \cong \bC^2$. Finally, given points $p_i \in \fHilb(D_i^\circ,
\beta_i)$, define the substack
\[ \fN_{(p_1,p_2,p_3),n} \coloneqq (\ev_1 \times \ev_2 \times \ev_3)^{-1}(p_1, p_2, p_3) \subset \fN_{\beta_C,n}. \]

\subsubsection{}

\begin{lemma} \label{lem:DT-PT-rigidification}
  The $\bC^\times$-rigidification map $\Pi^\pl\colon
  \fM^\circ_{1,0,-\beta_C,-n} \to \fM^{\circ,\pl}_{1,0,-\beta_C,-n}$ is a
  trivial $\bC^\times$-gerbe.
\end{lemma}

Consequently, by Lemma~\ref{lem:k-homology-pl}, if $I$ is any section
of $\Pi^\pl$, then there are isomorphisms
\begin{equation} \label{eq:DT-PT-rigidification}
  K_\circ^{\tilde\sT}(\fM^{\circ,\pl}_{1,0,-\beta_C,-n})_{\loc} \mathrel{\mathop{\rightleftarrows}^{I_*}_{(\Pi^\pl)_*}} K_\circ^{\tilde\sT}(\fM^\circ_{1,0,-\beta_C,-n})^\pl_{\loc}.
\end{equation}

\begin{proof}
  By Lemma~\ref{lem:DT-PT-rank-1-pairs}, all objects parameterized by
  $\fM^\circ_{1,0,-\beta_C,-n}$ are pairs of the form
  \[ [\cL \xrightarrow{s} \cF], \qquad \cL \cong \cO_X, \; \cF \in \cat{Coh}_{\le 1}(X). \]
  Since $\Aut(\cL) \cong \bC^\times$, the $\bC^\times$-rigidification
  of such a pair is (non-canonically) equivalent to fixing the extra
  data of an isomorphism $\phi\colon \cL \xrightarrow{\sim} \cO_X$.
  Hence $\Pi^\pl$ has a section given by forgetting this extra data,
  and this section trivializes $\Pi^\pl$ \cite[Lemma
    3.21]{Laumon2000}.
\end{proof}

\subsubsection{}

\begin{lemma} \label{lem:DT-PT-moduli-stacks}
  The moduli stack $\fM_{1,0,-\beta_C,-n}$ is Artin and locally of
  finite type. The inclusions
  \[ \fN_{(p_1,p_2,p_3),n} \subset \fN_{\beta_C,n} \subset \fM^\circ_{1,0,-\beta_C,-n} \subset \fM_{1,0,-\beta_C,-n} \]
  are closed, open, and open respectively.
\end{lemma}

\begin{proof}
  Let $\fMum_X$ be the moduli stack of objects $I \in D\cat{Coh}(X)$
  with $\Ext_X^{<0}(I, I) = 0$. This is an Artin stack, locally of
  finite type \cite{Lieblich2006}. Since $\fM_{1,0,-\beta_C,-n}$ is an
  open substack of $\fMum_X$ \cite[\S 6.3, Step 1]{toda_dtpt}, it is
  also Artin and locally of finite type.

  The first inclusion is obviously closed. For the second inclusion,
  the support condition~\ref{it:cond:DT-PT-N-stack-support} is clearly
  open, and the transversality
  condition~\ref{it:cond:DT-PT-N-stack-transverality} is open by the
  upper semi-continuity of cohomology. Finally, the last inclusion is
  open because $\dim \cH^1(I) = 0$ is an open condition.
\end{proof}

\subsubsection{}

\begin{lemma} \label{lem:DT-PT-obstruction-theories}
  Let $\fN = \fN_{(p_1,p_2,p_3),n}$ for short. There is a
  $\kappa$-symmetric $\sT$-equivariant obstruction theory on
  $\fN^\pl$, given by $(\bF[1])^\vee$ where
  \begin{equation} \label{eq:DT-PT-obstruction-theories}
    \bF \coloneqq Rp_*R\cHom_X(\scI, \scI(-D)) \in \cat{Perf}_\sT(\fN^\pl)
  \end{equation}
  for the universal family $\scI$ on $p\colon \fN \times X \to \fN$.
  If $\kappa^{1/2}$ exists, then $\det \bF$ admits a square root.
\end{lemma}

To be precise, $\bF$ is a perfect complex on $\fN$, but we identify
sheaves on $\fN^\pl$ with sheaves on $\fN$ of weight $0$ with respect
to the global $\bC^\times$ stabilizer.

Consequently, after passing to a double cover $\tilde\sT \to \sT$ so
that $\kappa^{1/2}$ exists, a symmetrized virtual structure sheaf
$\hat\cO^\vir \otimes (\det \bF)^{1/2}$ exists on any stable locus of
$\fN$.

\begin{proof}
  We use Lemma~\ref{lem:DT-PT-moduli-stacks} and the notation in its
  proof. Recall that $\fMum_X$ has a $\sT$-equivariant obstruction
  theory corresponding to $Rp_*R\cHom_X(\scI, \scI)$ \cite[Theorem
    B]{Ricolfi2021}. Then the open locus $\fN_{\beta_C,n} \subset
  \fMum_X$ inherits this obstruction theory from $\fMum_X$. Then a
  standard argument \cite[Proposition 3.3.2]{klt_dtpt} using the
  compatibility of Atiyah classes with pullbacks (e.g. along
  $\iota_i$) constructs the desired obstruction theory on the closed
  substack $\fN \subset \fN_{\beta_C,n}$. Finally, $(\bF[1])^\vee$ is
  $\kappa$-symmetric by Serre duality, and $\det \bF$ admits a square
  root by \cite[\S 6]{Nekrasov2016}.
\end{proof}

\subsubsection{}

For $M \in \{\DT, \PT\}$ and torus-fixed points $\lambda, \mu, \nu \in
\Hilb(\bC^2)$, let
\begin{equation} \label{eq:DT-PT-vertices-1}
  \sV_{\lambda,\mu,\nu}^M = \sum_{n \in \bZ} Q^n \chi\left(\fN_{(\lambda,\mu,\nu),n}^{\sst}(\tau^M), \hat\cO^\vir \otimes -\right).
\end{equation}
Lemma~\ref{lem:DT-PT-stability-chambers} below shows that the spaces
$\fN_{(\lambda,\mu,\nu),n}^{\sst}(\tau^M)$ are exactly the DT and PT
moduli spaces $M^{\sst}_{(\lambda,\mu,\nu),n}$ from
\S\ref{sec:intro:equivariant-vertices}. In particular, they are
algebraic spaces with proper $\sT$-fixed loci, so the symmetrized
virtual cycle may be defined using the obstruction theory of
Lemma~\ref{lem:DT-PT-obstruction-theories}. Thus
\eqref{eq:DT-PT-vertices-1} is well-defined and agrees with the
definition \eqref{eq:DT-PT-vertices} of DT and PT vertices for $X$.

\subsubsection{}

\begin{lemma} \label{lem:DT-PT-wall-semistables}
  Take $I \in \cat{A}$ with $\tau_0(I) = 3\pi/4$. Then $I$ is
  $\tau_0$-semistable if and only if $\cH^1(I) \in \cat{Coh}_{\le
    0}(X)$.
\end{lemma}

\begin{proof}
  Note that $\tau_0(I) = 3\pi/4$ implies either $\rank I > 0$ or
  $\ch(I) = (0, 0, 0, *)$.

  Suppose $\dim \cH^1(I) > 0$. Then $\tau_0(\cH^1(I)[-1]) = \pi/2$ and
  therefore the quotient $I \twoheadrightarrow \cH^1(I)[-1]$ is
  $\tau_0$-destabilizing for $I$.

  Suppose $\dim \cH^1(I) = 0$. If $\ch(I) = (0, 0, 0, *)$, then $I$ is
  automatically $\tau_0$-semistable. Otherwise, let $0 \to A \to I \to
  B \to 0$ be a short exact sequence in $\cat{A}$.
  \begin{itemize}
  \item If $\rank A, \rank B, \rank I > 0$, then $\tau_0(A) =
    \tau_0(I) = \tau_0(B)$.
  \item If $\rank A = 0$, then $\rank B = \rank I > 0$ and therefore
    $\tau_0(A) \le \tau_0(B) = \tau_0(I)$.
  \item If $\rank B = 0$, then $\rank A = \rank I > 0$. The surjection
    $\cH^1(I) \twoheadrightarrow \cH^1(B)$, combined with $B \in
    \cat{Coh}_{\le 1}(X)[-1]$, shows that $\ch(B) = (0,0,0,*)$. Thus
    $\tau_0(A) = \tau_0(I) = \tau_0(B)$.
  \end{itemize}
  In all cases, the short exact sequence is not $\tau_0$-destabilizing
  for $I$, so $I$ is $\tau_0$-semistable.
\end{proof}

\subsubsection{}

\begin{lemma} \label{lem:DT-PT-stability-chambers}
  Let $I = [\cO_X \xrightarrow{s} \cF]$ where $\cF \in \cat{Coh}_{\le
    1}(X)$.
  \begin{enumerate}[label = (\roman*)]
  \item \label{it:DT-stability-X} $I$ is $\tau^{\DT}$-stable if and
    only if $I$ is a DT stable pair, i.e. $s$ is surjective.
  \item \label{it:PT-stability-X} $I$ is $\tau^{\PT}$-stable if and
    only if $I$ is a PT stable pair, i.e. $\cF \in \cat{Coh}_{\le
      0}(X)^\perp$ and $\coker s \in \cat{Coh}_{\le 0}(X)$.
  \end{enumerate}
\end{lemma}

Recall if $\cat{T} \subset \cat{C}$ is a subcategory of an abelian
category, $\cat{T}^\perp \coloneqq \{C \in \cat{C} : \Hom(\cat{T}, C)
= 0\}$. For consistency, note that the condition ``$s$ is surjective''
in \ref{it:DT-stability-X} is equivalent to ``$\cF \in \cat{T}^\perp$
and $\coker s \in \cat{T}$ for $\cat{T} = 0$''.

\begin{proof}
  This is \cite[Prop. 3.12]{toda_dtpt}, but for us it follows
  immediately from Lemmas~\ref{lem:DT-PT-wall-semistables} (and
  \ref{lem:PT-BS-purity}\ref{it:PT-purity}), because:
  \begin{itemize}
  \item $I$ is $\tau^{\DT}$-stable if and only if $I$ is
    $\tau_0$-semistable and has no quotient objects in $\cat{Coh}_{\le
    1}(X)$;
  \item $I$ is $\tau^{\PT}$-stable if and only if $I$ is
    $\tau_0$-semistable and has no sub-objects in $\cat{Coh}_{\le
    0}(X)$. \qedhere
  \end{itemize}
\end{proof}

\subsubsection{}
\label{sec:DT-PT-WCF-proof}

\begin{proof}[Proof of Theorem~\ref{thm:vertex-correspondences} for DT/PT]
  This is essentially a review of the setup in \cite[\S
    6.1]{klt_dtpt}. We provide the ingredients and check the necessary
  assumptions for the wall-crossing formula (Theorem~\ref{thm:wcf})
  with restricted classes and moduli stacks
  (\S\ref{sec:restrictions}). The wall-crossing from $\tau^{\DT}$ to
  $\tau^{\PT}$ is realized as the composition (as in
  \S\ref{sec:general-wall-crossing}) of the two dominant
  wall-crossings for $(\tau, \mathring\tau) = (\tau_0, \tau^{\DT})$
  and $(\tau, \mathring\tau) = (\tau_0, \tau^{\PT})$.

  Analogously to $\fN_{(\lambda,\mu,\nu),n}$, define the following
  open moduli substack of $\fM_{0,0,0,-n}$:
  \begin{align*}
    \fQ_n
    &\coloneqq \{[0 \to \cF] \in \fM_{0,0,0,-n} : L\iota_i^*\cF = 0, \, i = 1, 2, 3\} \\
    &\cong \{\cF \in \cat{Coh}_{\le 0}(\bC^3) : \ch_3(\cF) = n\}.
  \end{align*}
  Note that all such $\cF$ are automatically $\tau_0$-semistable, i.e.
  $\fQ_n^{\sst}(\tau_0) = \fQ_n^\pl$.

  It is easy to check that $\{\tau_\xi\}_{\xi}$ defines a
  wall-crossing problem of simple type
  (Definition~\ref{def:wall-crossing-simple-type}) for the subsets $A$
  and $B \eqqcolon \tilde B$ in \eqref{eq:DT-PT-A-B-classes} below.
  Thus, Proposition~\ref{prop:wcf-simple-type} applies. Note that the
  moduli stacks $\fQ_n$ are independent of $\lambda, \mu, \nu$. This
  produces the desired wall-crossing formula \eqref{eq:DT-PT}.

\subsubsection{Assumption~\ref{es:assump:exact-subcategory}}

  (\ref{es:assump:exact-subcategory}\ref{es:assump:it:B-closed})
  Automatic since we are taking $\cat{B} = \cat{A}$.

  (\ref{es:assump:exact-subcategory}\ref{es:assump:it:permissible-classes})
  Set $\beta_C \coloneqq (|\lambda|,|\mu|,|\nu|)$ and let
  $C(\cat{A})_{\pe} \coloneqq A \sqcup B$ where
  \begin{equation} \label{eq:DT-PT-A-B-classes}
    \begin{aligned}
      A &\coloneqq \{(1,0,-\beta_C,-n) : \fN_{(\lambda,\mu,\nu),n}^{\sst}(\tau_0) \neq \emptyset\}, \\
      B &\coloneqq \{(0,0,0,-n) : \fQ_n^{\sst}(\tau_0) \neq \emptyset\}.
    \end{aligned}
  \end{equation}
  This is a partial monoid
  (Definition~\ref{def:restricted-graded-monoidal-stack}) by defining
  the monoid addition $+$ only between classes in $B$ and $B$ and
  between classes in $B$ and $A$. Note that $B + A \subset A$ because
  $\tau_0(A) = \tau_0(B)$, and the direct sum of two
  $\tau_0$-semistable objects of the same slope is still
  $\tau_0$-semistable. By Lemma~\ref{lem:DT-PT-moduli-stacks} (and an
  easy modification of the argument there), the moduli stacks
  $\{\fM_\alpha\}_{\alpha \in C(\cat{A})_{\pe}}$ are Artin, locally of
  finite type, and form a graded partially-monoidal $\sT$-stack.

  (\ref{es:assump:exact-subcategory}\ref{es:assumption-restricted-substack})
  The locally closed substacks
  \[ \fN_{(\lambda,\mu,\nu),n} \subset \fM_{1,0,-\beta_C,-n} \text{ and } \fQ_n \subset \fM_{0,0,0,-n}, \]
  ranging over all classes in $C(\cat{A})_{\pe}$, form a restricted
  graded partially-monoidal $\sT$-stack with $\kappa$-symmetric
  bilinear elements given by
  \begin{align}
    \scE\Big|_{\fN \times \fN} &\coloneqq Rp_*R\cHom(\scI_1, \scI_2(-D)) \oplus \cO_{\Delta} \oplus \kappa \cO_{\Delta}[-3] \label{eq:unrigidified-DT-PT-bilinear-element-N} \\
    \scE\Big|_{\fQ \times \fQ}, \scE\Big|_{\fQ \times \fN}, \scE\Big|_{\fN \times \fQ} &\coloneqq Rp_*R\cHom(\scI_1, \scI_2) \label{eq:unrigidified-DT-PT-bilinear-element-Q}
  \end{align}
  where $\scI_i$ denotes the universal pair on $X$ times the $i$-th
  factor, $p$ is projection along $X$, and $\Delta\colon \fN \to \fN
  \times \fN$ is the diagonal. The $\cO_{\Delta}$ factors in
  \eqref{eq:unrigidified-DT-PT-bilinear-element-N} arise from
  symmetrized pullback of the symmetric obstruction theory of
  \eqref{eq:DT-PT-obstruction-theories} on $\fN^\pl$ along the trivial
  $\bC^\times$-gerbe $\Pi^\pl\colon \fN \to \fN^\pl$
  (Lemma~\ref{lem:DT-PT-rigidification}). Conversely, $\fQ$ is
  well-known to carry the symmetric obstruction theory corresponding
  to $Rp_*R\cHom(\scI, \scI)$, as in
  \eqref{eq:unrigidified-DT-PT-bilinear-element-Q}, which, by
  Lemma~\ref{lem:obstruction-theory-pl}, induces a symmetric
  obstruction theory on $\fQ^\pl$ by removing a copy of $\bC$ from
  $H^0$ and, by Serre duality, a copy of $\bC \otimes \kappa$ from
  $H^3$. (This step was done ``manually'' in \cite[\S
    3.3.8]{klt_dtpt}.)

  Note that, since the pairs $I$ parameterized by $\fQ_n$ avoid $D
  \subset X$, \eqref{eq:unrigidified-DT-PT-bilinear-element-Q} may be
  replaced by the equivalent $Rp_*R\cHom(\scI_1, \scI_2(-D))$.

\subsubsection{Assumptions~\ref{es:assump:semistable-invariants} and \ref{es:assump:wall-crossing}}
\label{sec:DT-PT-WCF-proof-end}

  (\ref{es:assump:semistable-invariants}\ref{es:assump:it:tau-artinian},
  \ref{es:assump:semistable-invariants}\ref{es:assump:it:semistable-loci},
  and \ref{es:assump:wall-crossing}\ref{es:assump:it:tau-circ}) By
  \cite[\S 6.2]{toda_dtpt}, $\cat{A}$ admits $\tau_\xi$-HN filtrations
  for permissible classes, for all $\xi$. By \cite[\S 6.3]{toda_dtpt},
  $\tau_\xi$-semistability is open for permissible classes, for all
  $\xi$.

  (\ref{es:assump:semistable-invariants}\ref{es:assump:it:framing-functor})
  For the framing functor(s), note that the full subcategory
  $\cat{A}^\circ \subset \cat{A}$ consisting of objects $I$ with $\dim
  \cH^1(I) = 0$ is extension-closed and therefore an exact
  subcategory. By Lemma~\ref{lem:DT-PT-rank-1-pairs}, objects of rank
  $r \le 1$ in $\cat{A}^\circ$ have the form $I = [\cO_X \otimes L \to
    \cF]$ where $L$ is an $r$-dimensional vector space (equivalently,
  a $\bC$-point in $[\pt/\GL(r,\bC)]$) and $\cF \in \cat{Coh}_{\le
    1}(X)$. Let $\cO_X(1)$ denote any very ample line bundle on $X$,
  e.g. $\cO_{\bP^1}(1)^{\boxtimes 3}$, take $k \gg 0$ and $p \in
  \bZ_{> 0}$, and define
  \[ \Fr_{k,p}\colon [\cO_X \otimes L \to \cF] \mapsto H^0(\cF \otimes \cO_X(k)) \oplus L^{\oplus p}. \]
  Comparing with Example~\ref{ex:framing-functor}, it is clear that
  $\Fr_{k,p}$ is a framing functor on the full exact subcategory
  $\cat{A}^{\Fr_{k,p}} \subset \cat{A}^\circ$ consisting of pairs
  where $\cF$ is $k$-regular. The factor $L^{\oplus p}$ ensures that
  $\Hom(I, I) \to \Hom(\Fr_{k,p}(I), \Fr_{k,p}(I))$ is injective,
  especially when $\lambda = \mu = \nu = \emptyset$. \footnote{It
    suffices to take $p = 1$ throughout, but the freedom to vary $p$
    was helpful in applying a combinatorial trick in \cite[\S
      6]{klt_dtpt}.}

  (\ref{es:assump:semistable-invariants}\ref{es:assump:it:rank-function})
  For the rank function $r$, we claim that there exists $r_0(\beta_C)
  \in \bZ$ such that $\fM_{1,0,-\beta_C,-n}^\sst(\tau_0) = \emptyset$
  for all $n \le n_0(\beta)$, and then it is straightforward to check
  that any $r$ satisfying
  \begin{align*}
    r(1, 0, -\beta_C, -n) &\coloneqq n - n_0(\beta_C) \\
    r(0, 0, 0, -n) &\coloneqq n
  \end{align*}
  is a valid rank function. Such a lower bound $n_0(\beta_C)$ exists
  because if $\fM_{1,0,-\beta_C,-n'}^{\sst}(\tau_0) \neq \emptyset$
  then $\dim \fM_{1,0,-\beta_C,-n}^{\sst}(\tau_0) \ge 3(n-n')$ by
  direct summing with $n-n'$ objects of the form $\cO_x[-1]$ for $x
  \in \bC^3$ (which are obviously $\tau_0$-semistable of the same
  slope as $(1, 0, -\beta_C, -n')$), but we already know
  $\fM_{1,0,-\beta_C,-n}^{\sst}(\tau_0)$ is of finite type.

  (\ref{es:assump:semistable-invariants}\ref{es:assump:it:inert-classes})
  Take $C(\cat{A})_{\inert} \coloneqq \emptyset$.

  (\ref{es:assump:semistable-invariants}\ref{es:assump:it:semi-weak-stability})
  We prove the stronger claim that $\tau_0(\beta') =
  \tau(\alpha-\beta')$ implies $\tau_0(\beta') =
  \tau_0(\beta-\beta')$, for {\it any} classes $0 < \beta' < \beta <
  \alpha$ where $\beta, \alpha \in C(\cat{A})_{\pe}$. If $\beta \in B$
  and $0 < \beta' < \beta$ then $\beta', \beta, \beta-\beta' \in B$ as
  well, and $\tau_0$ is constant on $B$. Otherwise if $\beta \in A$,
  then $\alpha \in A$ as well, and the hypothesis $\tau_0(\beta') =
  \tau_0(\alpha - \beta')$ implies $\beta' = (0, 0, 0, -m)$ or $\beta'
  = (1, 0, -\beta_C, -m)$ for some $m \in \bZ$. But then clearly
  $\tau_0(\beta') = \tau_0(\beta - \beta')$ as well.

  (\ref{es:assump:semistable-invariants}\ref{es:assump:it:properness}
  and \ref{es:assump:wall-crossing}\ref{es:assump:it:properness-wcf})
  Properness of various fixed loci may be proved via Langton's
  elementary modifications; see \cite[\S 4.3, Prop. 6.1.5]{klt_dtpt}
  for details.

  (\ref{es:assump:semistable-invariants}\ref{es:assump:it:Bpe-sst-summands}
  and \ref{es:assump:wall-crossing}\ref{es:assump:it:tau-circ}) If
  $\beta_1 + \beta_2 \in B$ then necessarily $\beta_1, \beta_2 \in B$.
  Else if $\beta_1 + \beta_2 \in A$ with $\tau_\xi(\beta_1) =
  \tau_\xi(\beta_2)$ then $\xi = 0$ and (without loss of generality)
  $\beta_1 \in A$ and $\beta_2 \in B$. In either case it is
  straightforward to verify the desired conditions, noting that the
  conditions
  \ref{sec:DT-PT-N-stack}\ref{it:cond:DT-PT-N-stack-support} and
  \ref{it:cond:DT-PT-N-stack-transverality} are preserved upon passing
  to sub-objects.

  (\ref{es:assump:wall-crossing}\ref{es:assump:it:lambda}) Let $\alpha
  \in C(\cat{A})_{\pe}$. If $\alpha \in B$, then $R_\alpha \subset B$
  and so $\mathring\tau(\beta) = \mathring\tau(\alpha - \beta)$ for
  any $\beta \in R_\alpha$, thus we may take $\lambda \coloneqq 0$. If
  $\alpha \in A$, then for the $\tau_0$ to $\tau^{\DT}$ (resp.
  $\tau^{\PT}$) dominant wall-crossing, take any group homomorphism
  $\lambda\colon K(\cat{A}) \to \bR$ such that $\lambda(\alpha) = 0$
  and $\lambda(0,0,0,-1) < 0$ (resp. $\lambda(0,0,0,-1) > 0$).
\end{proof}

\subsubsection{}
\label{sec:DT-PT-consequences}

Clearly only the trivial PT stable pair $\cO \to 0$ has curve class
$\beta_C = 0$, hence
\begin{equation} \label{eq:PT0-vertex-trivial}
  \sV^{\PT}_{\emptyset,\emptyset,\emptyset} = \partial
\end{equation}
where $\partial = \id \in K_\circ^{\tilde\sT}(\pt)$. This is the same
element which plays a significant role in the definition of semistable
invariants (Theorem~\ref{thm:sst-invariants}).

Specializing the equivariant vertex correspondence
(Theorem~\ref{thm:vertex-correspondences}) to the case
$(\lambda,\mu,\nu) = (\emptyset,\emptyset,\emptyset)$,
\eqref{eq:PT0-vertex-trivial} implies
\begin{equation} \label{eq:DT-PT-0}
  I_*\sV_{\emptyset,\emptyset,\emptyset}^{\DT} = \exp(\ad(\sz)) I_*\partial.
\end{equation}
Applying the operational equivariant vertex correspondence
(Theorem~\ref{thm:vertex-correspondences}) to $\cO$, by rigidity
(Proposition~\ref{prop:rigidity}) and that $(I_*\phi)(\cO) =
\phi(\cO)$,
\[ \sV_{\lambda,\mu,\nu}^{\DT}(\cO) = z \cdot \sV_{\lambda,\mu,\nu}^{\PT}(\cO) \]
for a factor $z$ which is {\it independent} of $\lambda, \mu, \nu$.
Since $\partial(\cO) = 1$, specializing to $(\lambda,\mu,\nu) =
(\emptyset,\emptyset,\emptyset)$ yields $z =
\sV_{\emptyset,\emptyset,\emptyset}^{\DT}(\cO)$. Hence we recover the
DT/PT primary vertex correspondence \eqref{eq:DT-PT-primary}.

\subsection{PT/BS vertex correspondence}
\label{sec:PT-BS}

\subsubsection{}

Let $\pi\colon X \to X_0$ be a projective morphism of relative
dimension one with $R\pi_*\cO_X = \cO_{X_0}$. Recall Bridgeland's
theory of perverse coherent sheaves \cite[\S 3]{Bridgeland2002}. Let
\[ \Per(\pi) \coloneqq \left\{E \in D^b\cat{Coh}(X) : \begin{array}{l} R\pi_*E \in \cat{Coh}(X_0),\\ \Hom^{<0}(E, \cat{C}) = \Hom^{<0}(\cat{C}, E) = 0\end{array}\right\} \]
where $\cat{C} \coloneqq \{\cF \in \cat{Coh}(X) : R\pi_*\cF = 0\}$.
(This is Bridgeland's ${}^p\!\Per(X/X_0)$ with the perversity $p = 0$.)
Equivalently, $\Per(\pi)$ is the tilt of $\cat{Coh}(X)$ with respect
to the torsion pair
\begin{align*}
  \cat{T}_\pi &\coloneqq \{\cT \in \cat{Coh}(X) : R^1\pi_*\cT = 0\} \\
  \cat{F}_\pi &\coloneqq \{\cF \in \cat{Coh}(X) : \pi_*\cF = 0, \, \Hom(\cat{C}, \cF) = 0\}
\end{align*}
on $\cat{Coh}(X)$ \cite[Lemma 3.1.2]{VandenBergh2004}, therefore
$\Per(\pi)$ is the heart of a bounded t-structure on $D^b\cat{Coh}(X)$
(the {\it perverse t-structure}) and in particular it is an abelian
category. Note that $\cO_X \in \Per(\pi)$.

Let $\PcH^i(-) \in \cat{Per}(\pi)$ denote cohomology with respect to
the perverse t-structure, to distinguish it from the cohomology
$\cH^i(-) \in \cat{Coh}(X)$ with respect to the ordinary t-structure.

\subsubsection{}

Let
\begin{align*} 
  \Per_0(\pi) &\coloneqq \{F \in \Per(\pi) : \dim \supp R\pi_*F = 0\} \\
  \Per_{\le 1}(\pi) &\coloneqq \{F \in \Per(\pi) : \dim \supp F \le 1\},
\end{align*}
which are the analogues of $\cat{Coh}_{\le 0}(X)$ and $\cat{Coh}_{\le
  1}(X)$ in \S\ref{sec:DT-PT}, and let
\[ \cat{A}_\pi \coloneqq \inner{\cO_X, \Per_{\le 1}(\pi)[-1]}_{\text{ex}} \subset \cat{D}_X \]
be the smallest extension-closed\footnote{Extensions are taken in the
abelian category $\cat{Per}^\dag(\pi)[-1]$, where
$\cat{Per}^{\dag}(\pi)$ is the tilt of $\cat{Per}(\pi)$ with respect
to the torsion pair $(\cat{Per}_{\le 1}(\pi), \cat{Per}_{\le
  1}(\pi)^\perp)$; see \cite[\S 3.1]{Toda2013}.} full subcategory
containing $\cO_X$ and $\cat{Per}_{\le 1}(\pi)[-1]$. Like $\Per(\pi)$,
both $\Per_{\le 1}(\pi)$ and $\cat{A}_\pi$ are tilts of
$\cat{Coh}_{\le 1}(X)$ and $\cat{A}_X$ respectively \cite[Lemmas 3.2,
  3.6]{Toda2013}, and are therefore both abelian categories. Moreover,
$\cat{A}_\pi$ is Noetherian \cite[Lemma 3.5(i)]{Toda2013} and clearly
$\bC$-linear. This is the ambient abelian category of interest in this
section.

Let $\fM_\pi$ denote the moduli stack parameterizing objects in
$\cat{A}_\pi$. When it is unambiguous or unimportant, we omit the
subscript $\pi$ from $\cat{A}_\pi$ and $\fM_\pi$. Clearly direct sum
and scaling automorphisms in $\cat{A}$ make $\fM$ into a graded
monoidal stack.

\subsubsection{}

\begin{remark}
  The abelian category $\cat{A}_\pi$ first appeared in Toda's work
  \cite{Toda2013} (where it was denoted ${}^0\cB_{X/X_0}$). Many
  objects and results there are implicitly about BS stable pairs,
  despite it predating the invention of BS stable pairs in
  \cite{Bryan2016} by several years. While the abelian category
  $\cat{A}_X$ appearing in \S\ref{sec:DT-PT} is important for the
  study of DT-type theories associated to $X$, Toda's insight was that
  $\cat{A}_\pi$ is the appropriate abelian category for the study of
  DT-type theories associated to the resolution $\pi$ and its
  birational transformations.
\end{remark}

\subsubsection{}

Specifically, for PT/BS, let the generator of the cyclic group
$\bZ_{m+1}$ act on $\bC^2$ by $(z_1, z_2) \mapsto (\zeta z_1,
\zeta^{-1} z_2)$, let $\cA_m \to \bC^2/\bZ_{m+1}$ be the minimal
smooth (crepant) resolution, and let
\[ \cA_m \times \bP^1 \eqqcolon X \xrightarrow{\pi} X_0 \coloneqq \bC^2/\bZ_{m+1} \times \bP^1 \]
be the trivial\footnote{It is certainly interesting and productive to
  take non-trivial families, see e.g. \cite{Kim2016}, but we do not do
  so here.} family of these resolutions over $\bP^1$. Clearly $\pi$ is
a projective morphism of relative dimension one, and it is well-known
that $R\pi_*\cO_X = \cO_{X_0}$ \cite{Viehweg1977}. The natural scaling
action of $\sT \coloneqq (\bC^\times)^3$ on $\bC^2 \times \bP^1$
commutes with the $\bZ_m$-action and therefore induces a $\sT$-action
on $X$. Let $\iota\colon D \hookrightarrow X$ be the $\sT$-invariant
boundary divisor $\cA_m \times \{\infty\}$. Let $\kappa$ denote the
$\sT$-weight of the trivial (but not equivariantly trivial) canonical
bundle $\cK_{X \setminus D}$.

Objects $I \in \cat{A}_\pi$ have Chern character
\[ \ch(I) = (r, 0, (-\beta_{\bP}, -\beta_{\cA}), -n) \]
where $r = \rank(I)$ and $2n = 2\ch_3(I)$ are integers, $\beta_{\bP}
\in H_2(\bP^1; \bZ) \cong \bZ$ and $\beta_{\cA} \in H_2(\cA_m; \bZ)$.
Let $\omega \in H^2(\cA_m)$ be the first Chern class of a
globally-generated\footnote{This is useful for the content of
  \S\ref{sec:PT-BS-WCF-proof-end}, but is not used otherwise.} ample
line bundle, and, on $\cat{A}_\pi$, define the family of weak
stability conditions
\begin{equation} \label{eq:PT-BS-stability-condition}
  \tau_\xi(r, 0, (-\beta_{\bP},-\beta_{\cA}), -n) \coloneqq \begin{cases} 3\pi/4 & r \neq 0, \\ \pi/2 & r = 0, \, \beta_{\bP} \neq 0, \\ 3\pi/4 + \xi & r = \beta_{\bP} = 0, \, \omega \cdot \beta_{\cA} > 0, \\ \pi & r = \beta_{\bP} = 0, \, \omega \cdot \beta_{\cA} \le 0 \end{cases}
\end{equation}
for $\xi \in (-\pi/4, \pi/4)$. This is a continuous family with
exactly one wall at $\xi = 0$ for objects of rank $r=1$, so let
$\tau^{\BS}$ and $\tau^{\PT}$ denote the cases $\xi > 0$ and $\xi < 0$
respectively. Convexity of the half spaces $\{\beta : \omega \cdot
\beta > 0\}$ and $\{\beta : \omega \cdot \beta \le 0\}$ ensures this
is a weak stability condition. Recall that $\beta_{\cA} > 0$, meaning
that $\beta_{\cA}$ is an effective curve class, implies $\omega \cdot
\beta_{\cA} > 0$, but the converse may not be true.

We proceed to identify $\tau_\xi$-semistable objects of rank $r \le
1$.

\subsubsection{}

\begin{lemma} \label{lem:PcH-vs-cH}
  Let $I \in \cat{A}_\pi$. Then $\cH^1(I) = \cH^0(\PcH^1(I)) \in
  \cat{T}_\pi$.
\end{lemma}

\begin{proof}
  Consider the short exact sequence $0 \to I_1 \to I \to \PcH^1(I)[-1]
  \to 0$. Since $\PcH^1(I_1) = 0$, i.e. $I_1 \in \cat{Per}(\pi)$, we
  know $\cH^{>0}(I_1) = 0$. Thus $\cH^1(I) = \cH^1(\PcH^1(I)[-1])$, as
  claimed.
\end{proof}

\subsubsection{}

\begin{lemma} \label{lem:PT-BS-rank-1-pairs}
  Take $I \in \cat{A}_\pi$ with $\rank I = 1$, and suppose $\supp
  \PcH^1(I)$ avoids $D \subset X$. There is an isomorphism
  \[ I \cong [\cO_X \xrightarrow{s} F], \quad F \in \cat{Per}_{\le 1}(\pi), \, \Pcoker s \in \cat{Per}_0(\pi), \]
  if and only if $\PcH^1(I) \in \cat{Per}_0(\pi)$. Here $\cO_X$ is in
  degree zero and $F$ is in degree one.
\end{lemma}

Here, the notation $\Pcoker$ reminds us that the cokernel is taken in
the category $\cat{Per}(\pi)$.

Moreover, $\im s$ is a perverse structure sheaf \cite[Def.
  3.4]{Bridgeland2002} numerically equivalent to $\cO_C$ for some
$1$-dimensional subscheme $C \subset X$, but we will not use this.

\begin{proof}
  This is the PT/BS analogue of Lemma~\ref{lem:DT-PT-rank-1-pairs}, so
  we indicate only the necessary modifications to that proof. First,
  $\cat{Coh}(X)$ must be replaced by $\cat{Per}(\pi)$ throughout, and
  $\cH^*$ and $\coker$ replaced by $\PcH^*$ and $\Pcoker$
  respectively. Second, given $I \in \cat{A}_\pi$ of rank one such
  that $Q \coloneqq \PcH^1(I)$ lies in $\cat{Per}_0(\pi)$ and is
  supported away from $D$, we need to show the analogue of the
  vanishing in \eqref{eq:DT-PT-pair-factorization}. It suffices to
  show
  \[ H^1(X, Q) \stackrel{?}{=} 0 \]
  since $\supp Q$ avoids $D$ by hypothesis. This vanishing follows
  from the exact sequence
  \[ 0 \to H^1(X_0, \pi_*Q) \to H^1(X, Q) \to H^0(X_0, R^1\pi_*Q) \to \cdots \]
  associated to the Leray spectral sequence for $\pi$ (which in fact
  degenerates because $\pi$ has relative dimension one, but we don't
  need this). Since $Q \in \Per_0(\pi)$, in particular $R\pi_*Q \in
  \cat{Coh}(X_0)$, the outer two terms in this exact sequence vanish
  and thus the middle term does too. The rest of the proof remains
  unchanged.
\end{proof}

\subsubsection{}
\label{sec:PT-BS-N-stack}

We define a sequence of substacks of $\fM_{1,0,(-\beta_{\bP},
  -\beta_{\cA}),-n}$, for any $\beta_{\bP}$, $\beta_{\cA}$ and $n$.
First, let
\[ \fM^\circ_{1,0,(-\beta_{\bP},-\beta_{\cA}),-n} \subset \fM_{1,0,(-\beta_{\bP},-\beta_{\cA}),-n} \]
be the moduli substack of objects $I \in \cat{A}$ such that
\begin{enumerate}[label = (\alph*)]
\item \label{it:cond:PT-BS-N-stack-support} (support) $\supp
  \PcH^1(I)$ avoids $D$
\end{enumerate}
and $\dim R\pi_*\PcH^1(I) = 0$. By Lemma~\ref{lem:PT-BS-rank-1-pairs},
all objects parameterized by this substack are isomorphic to pairs
$s\colon\cO_X \to F$ where $F \in \cat{Per}_{\le 1}(\pi)$. Note that
$\PcH^1(I) = \Pcoker(s)$. Second, let
\[ \fN_{\beta_{\bP},\beta_{\cA},n} \subset \fM^\circ_{1,0,(-\beta_{\bP}, -\beta_{\cA}),-n} \]
be the moduli substack consisting of the pairs such that:
\begin{enumerate}[resume,label = (\alph*)]
\item \label{it:cond:PT-BS-N-stack-transverality} (transversality)
  $L^k\iota^* F = 0$ for $k > 0$.
\end{enumerate}
This condition implies the existence of the {\it evaluation map}
\begin{align*}
  \ev\colon \fN_{\beta_{\bP}, \beta_{\cA},n} &\to \fHilb(D, \beta_{\bP}) \\
  [\cO_X \xrightarrow{s} F] &\mapsto [\cO_D \xrightarrow{s} \iota^*F],
\end{align*}
which lands in the Hilbert stack of $\beta_{\bP}$ points on $D$.
Finally, given a point $p \in \fHilb(D, \beta_{\bP})$, define the
substack
\[ \fN_{p,\beta_{\cA},n} \coloneqq \ev^{-1}(p) \subset \fN_{\beta_{\bP},\beta_{\cA},n}. \]

\subsubsection{}

\begin{lemma} \label{lem:PT-BS-stacks-summary}
  The analogues of Lemmas~\ref{lem:DT-PT-rigidification},
  \ref{lem:DT-PT-moduli-stacks}, and
  \ref{lem:DT-PT-obstruction-theories} hold for the moduli stacks
  \[ \fN_{p,\beta_{\cA},n} \subset \fN_{\beta_{\bP},\beta_{\cA},n} \subset \fM^\circ_{1,0,(-\beta_{\bP},-\beta_{\cA}),-n} \subset \fM_{1,0,(-\beta_{\bP},-\beta_{\cA}),-n}. \]
\end{lemma}

\begin{proof}
  We indicate only the necessary modifications to the proofs of
  Lemmas~\ref{lem:DT-PT-rigidification},
  \ref{lem:DT-PT-moduli-stacks}, and
  \ref{lem:DT-PT-obstruction-theories}. Note that $X$ is now
  quasi-projective instead of projective. Let $\bar X$ be any smooth
  (projective) $\sT$-equivariant compactification of $X$ --- all
  instances of $\fMum_X$ should be replaced by $\fMum_{\bar X}$. Then,
  for example, the moduli stack of rank-$1$ objects in $\cat{A}_X$ is
  the open substack of $\fMum_{\bar X}$ consisting of rank-$1$ objects
  whose restriction to $\bar X \setminus X$ is $\cO_{\bar X \setminus
    X}$. With this point in mind,
  $\fM_{1,0,(-\beta_{\bP},-\beta_{\cA}),-n}$ is an open substack of
  $\fMum_{\bar X}$ \cite[\S 6.4, Step 1]{Toda2013}. The last inclusion
  is open because both the support condition
  \ref{sec:PT-BS-N-stack}\ref{it:cond:PT-BS-N-stack-support} and $\dim
  R\pi_*\PcH^1(I) = 0$ are open conditions. Finally, the obstruction
  theory is inherited from $\fMum_{\bar X}$. Note that $R\cHom_X(\cI,
  \cI(-D))$ agrees with $R\cHom_{\bar X}(\cI, \cI(-D))$ because
  $R\Hom_{\bP^1}(\cO_{\bP^1}, \cO_{\bP^1}(-1)) = 0$.
\end{proof}

\subsubsection{}

For $M \in \{\PT, \BS\}$ and torus-fixed points $\lambda_1, \ldots,
\lambda_{m+1} \in \Hilb(\bC^2)$ specifying a $\sT$-fixed point
$\vec\lambda = (\lambda_1, \ldots, \lambda_{m+1}) \in \Hilb(\cA_m)$,
let
\begin{equation} \label{eq:PT-BS-vertices-1}
  \sV_{\vec\lambda}^M = \sum_{\substack{\beta_{\cA} \in H_2(\cA_m;\bZ)\\ n \in \bZ}} A^{\beta_{\cA}} Q^n \chi\left(\fN_{\vec\lambda,\beta_{\cA},n}^{\sst}(\tau^M), \hat\cO^\vir \otimes -\right),
\end{equation}
Lemma~\ref{lem:PT-BS-stability-chambers} below shows that the spaces
$\fN_{\vec\lambda,\beta_{\cA},n}^{\sst}(\tau^M)$ are exactly the PT
and BS moduli spaces $M^{\sst}_{\vec\lambda,\beta_{\cA},n}$ from
\S\ref{sec:intro:equivariant-vertices}. In particular, they are
algebraic spaces with proper $\sT$-fixed loci, so the symmetrized
virtual cycle may be defined using the obstruction theory of
Lemma~\ref{lem:PT-BS-stacks-summary}. Thus \eqref{eq:PT-BS-vertices-1}
is well-defined and agrees with the definition
\eqref{eq:PT-BS-vertices} of PT and BS vertices for $\pi$.

\subsubsection{}

\begin{lemma} \label{lem:PT-BS-wall-semistables}
  Take $I \in \cat{A}$ with $\tau_0(I) = 3\pi/4$. Then $I$ is
  $\tau_0$-semistable if and only if
  \begin{enumerate}[label = (\alph*)]
  \item \label{it:PT-BS-slope-pi-destabilizer} $I$ has no sub-object
    in $\cH^{-1}\cat{Per}_{\le 1}(\pi)$ or $\cat{Coh}_{\le 0}(X)[-1]$,
    and
  \item \label{it:PT-BS-slope-pi2-destabilizer} $\cH^1(I) \in
    \cat{T}_\pi \cap \cat{Per}_0(\pi) \subset \cat{T}_\pi$.
  \end{enumerate}
\end{lemma}

This is the PT/BS analogue of Lemma~\ref{lem:DT-PT-wall-semistables}.

\begin{proof}
  Suppose \ref{it:PT-BS-slope-pi-destabilizer} is violated by a
  sub-object $A \hookrightarrow I$. If $A \in \cat{Coh}_{\le
    0}(X)[-1]$, clearly $\tau_0(A) = \pi$. Otherwise, if $A \in
  \cH^{-1}\cat{Per}_{\le 1}(\pi)$, then $A \in \cat{F}_\pi$, in
  particular $\pi_*A = 0$, so $A$ is a $1$-dimensional sheaf supported
  only on exceptional fibers and thus indeed $\tau_0(A) = \pi$. In
  either case, $A \hookrightarrow I$ is $\tau_0$-destabilizing for
  $I$.

  Suppose \ref{it:PT-BS-slope-pi2-destabilizer} is violated. An object
  $\cF \in \cat{Coh}_{\le 1}(X)$ with $\dim \pi_*\cF > 0$ has
  non-trivial curve class along $\bP^1$ and thus $\tau_0(\cF) =
  \pi/2$. In particular $\tau_0(\cH^1(I)[-1]) = \pi/2$ and therefore
  the quotient $I \twoheadrightarrow \cH^1(I)[-1]$ is
  $\tau_0$-destabilizing for $I$.

  Suppose both \ref{it:PT-BS-slope-pi-destabilizer} and
  \ref{it:PT-BS-slope-pi2-destabilizer} are satisfied. Let $0 \to A
  \to I \to B \to 0$ be a short exact sequence in $\cat{A}$. Note that
  if $\rank I = 0$ then $\rank A = \rank B = 0$.
  \begin{itemize}
  \item If $\rank A, \rank B, \rank I > 0$, then $\tau_0(A) =
    \tau_0(I) = \tau_0(B)$.

  \item If $\rank A = 0$, then $A \in \cat{Per}_{\le 1}(\pi)[-1]$.
    Since $I$ has a sub-object $\cH^0(A) \hookrightarrow A
    \hookrightarrow I$, and $\cH^0(A) \in \cH^{-1}\cat{Per}_{\le
      1}(\pi)$, by \ref{it:PT-BS-slope-pi-destabilizer} $\cH^0(A) =
    0$. Thus $A \in \cat{Coh}_{\le 1}(\pi)[-1]$. Furthermore, by
    \ref{it:PT-BS-slope-pi-destabilizer}, $A \notin \cat{Coh}_{\le
      0}(\pi)[-1]$. So $A[1]$ is a $1$-dimensional sheaf and therefore
    $\tau_0(A) \le 3\pi/4$.

  \item If $\rank B = 0$, then $B \in \cat{Per}_{\le 1}(\pi)[-1]$.
    Since there is a quotient $\cH^1(I) \twoheadrightarrow \cH^1(B)$,
    and $\cat{T}_\pi \cap \cat{Per}_0(\pi)$ is closed under quotients,
    by \ref{it:PT-BS-slope-pi2-destabilizer} we get that $\dim
    \pi_*\cH^1(B) = 0$. Since $\cH^0(B) \in \cat{F}_\pi$, in
    particular $\pi_*\cH^0(B) = 0$, we conclude $B$ is supported only
    on exceptional fibers and therefore $\tau_0(B) \ge 3\pi/4$.
  \end{itemize}
  In all cases, the short exact sequence is not $\tau_0$-destabilizing
  for $I$, so $I$ is $\tau_0$-semistable.
\end{proof}

\subsubsection{}

\begin{lemma} \label{lem:PT-BS-stability-chambers}
  Let $I = [\cO_X \xrightarrow{s} F]$ where $F \in \cat{Per}_{\le
    1}(\pi)$.
  \begin{enumerate}[label = (\roman*)]
  \item \label{it:PT-stability-pi} $I$ is $\tau^{\PT}$-stable if and
    only if it is a PT stable pair, i.e. $F \in \cat{Coh}_{\le
      0}(X)^\perp$ and $\coker s \in \cat{Coh}_{\le 0}(X)$.
  \item \label{it:BS-stability-pi} $I$ is $\tau^{\BS}$-stable if and
    only if it is a BS stable pair, i.e. $F \in
    (\cat{T}^{\BS}_\pi)^\perp$ and $\coker s \in \cat{T}^{\BS}_\pi$
    where
    \[ \cat{T}^{\BS}_\pi \coloneqq \left\{\cF \in \cat{Coh}_{\le 1}(X) : R\pi_*\cF \in \cat{Coh}_{\le 0}(X_0)\right\}. \]
  \end{enumerate}
\end{lemma}

This is the PT/BS analogue of
Lemma~\ref{lem:DT-PT-stability-chambers}. Note that $\cat{T}^{\BS}_\pi
= \cat{T}_\pi \cap \cat{Per}_0(\pi)$.

\begin{proof}
  We prove \ref{it:BS-stability-pi} and leave the (analogous) case
  \ref{it:PT-stability-pi} to the reader. Note that if $[\cO_X
    \xrightarrow{s} \cF] \in \cat{A}_\pi \cap \cat{A}_X$, then
  Lemma~\ref{lem:PcH-vs-cH} identifies $\coker s = \cH^0(\Pcoker s)$.

  Suppose $I$ is a BS stable pair, so in particular $F \in
  \cat{Coh}_{\le 1}(X)$. By
  Lemma~\ref{lem:PT-BS-purity}\ref{it:pair-vs-F-subobject}, $I$ has no
  sub-objects in $\cH^{-1}\cat{Per}_{\le 1}(\pi)$, and by
  Lemma~\ref{lem:PT-BS-purity}\ref{it:BS-purity}, $I$ has no
  sub-objects in $\cat{Per}_0(\pi)[-1] \supset \cat{Coh}_{\le
    0}(X)[-1]$. Moreover, $\cH^1(I) = \coker s \in \cat{T}^{\BS}_\pi$.
  By Lemma~\ref{lem:PT-BS-wall-semistables}, we conclude $I$ is
  $\tau^{\BS}$-stable.

  Conversely, suppose $I$ is $\tau^{\BS}$-stable. Then it is
  $\tau_0$-semistable. In particular, by
  Lemma~\ref{lem:PT-BS-wall-semistables}, $\cH^{-1}(F) \hookrightarrow
  F[-1] \hookrightarrow I$ must be the zero sub-object, and $\Pcoker s
  \in \cat{Per}_0(\pi)$. These imply, respectively, that $F \in
  \cat{Coh}(X)$ and $\coker s = \cH^0(\Pcoker s) \in
  \cat{T}^{\BS}_\pi$. Finally, the $\tau^{\BS}$-slope of objects in
  $\cat{Per}_0(\pi)[-1]$ is $> 3\pi/4$, so $F \in
  (\cat{T}^{\BS}_\pi)^\perp$ by
  Lemma~\ref{lem:PT-BS-purity}\ref{it:BS-purity}. Hence $I$ is a BS
  stable pair.
\end{proof}

\subsubsection{}

\begin{lemma} \label{lem:PT-BS-purity}
  Let $I = [\cO_X \to \cF]$ be a pair with $\cF \in \cat{Coh}_{\le
    1}(X)$.
  \begin{enumerate}[label = (\roman*)]
  \item \label{it:pair-vs-F-subobject} Let $\cat{T} \subset \cat{A}$
    be a full subcategory whose elements have support of dimension
    $\le 1$. Then
    \[ \Hom(\cat{T}[-1], I) = \Hom(\tau^{\ge 0}\cat{T}, \cF) \]
    where $\tau^{\ge 0}$ denotes truncation (with respect to the
    ordinary t-structure).
  \item \label{it:PT-purity} $\Hom(\cat{Coh}_{\le 0}(X)[-1], I) = 0$
    if and only if $\cF \in \cat{Coh}_{\le 0}(X)^\perp$.
  \item \label{it:BS-purity} $\Hom(\cat{Per}_0(\pi)[-1], I) = 0$ if
    and only if $\cF \in (\cat{T}^{\BS}_\pi)^\perp$.
  \end{enumerate}
\end{lemma}

\begin{proof}
  For \ref{it:pair-vs-F-subobject}, let $T \in \cat{T}$ and apply
  $\Hom(T[-1], -)$ to the short exact sequence
  \[ 0 \to \cF[-1] \to I \to \cO_X \to 0 \]
  in $\cat{A}$. Since $\Hom(T[-k], \cO_X) \cong H^{3-k}(X, T \otimes
  \cK_X)^\vee$ vanishes for $k < 2$ for dimension reasons,
  \[ \Hom(T, \cF) \cong \Hom(T[-1], I). \]
  Finally, since $\cF$ is a sheaf, $\Hom(T, \cF) = \Hom(\tau^{\ge 0}T,
  \cF)$.

  For \ref{it:PT-purity} and \ref{it:BS-purity}, apply
  \ref{it:pair-vs-F-subobject} to $\cat{T} = \cat{Coh}_{\le 0}(X)$ and
  $\cat{Per}_0(\pi)$, noting that $\cat{T}^{\BS}_\pi = \tau^{\ge
    0}\cat{Per}_0(\pi)$.
\end{proof}

\subsubsection{}
\label{sec:PT-BS-WCF-proof}

\begin{proof}[Proof of Theorem~\ref{thm:vertex-correspondences} for PT/BS]
  The proof will be a more sophisticated version of the DT/PT case in
  \S\ref{sec:DT-PT-WCF-proof} -- \S\ref{sec:DT-PT-WCF-proof-end}. We
  will only specify the necessary modifications. We must verify
  Assumptions~\ref{es:assump:exact-subcategory},
  \ref{es:assump:semistable-invariants} and
  \ref{es:assump:wall-crossing}. However, it turns out that satisfying
  Assumption~\ref{es:assump:semistable-invariants} --- in particular,
  parts \ref{es:assump:it:rank-function},
  \ref{es:assump:it:inert-classes}, and
  \ref{es:assump:it:semi-weak-stability} --- is very difficult for the
  weak stability condition $\tau_0$ ``on the wall'' as defined by
  \eqref{eq:PT-BS-stability-condition}. Essentially, the issue is
  that, unlike in the DT/PT case, the relevant classes in
  wall-crossing do not always have maximal slope. We fix this in an
  ad-hoc way by modifying the family of weak stability conditions.
  This forces us to factorize the desired wall-crossing into an {\it
    sequence} of wall-crossings of simple type.

  By \cite[Lemma 6.1]{Toda2013}, there exists an integer $C(\pi) \in
  \bZ$ (depending on $\pi$) such that\footnote{Concretely,
    $\cat{Per}_0(\pi)$ is generated by the dualizing sheaf
    $\omega_E[1]$ of the exceptional divisor $E$, and $\cO_{C_i}(-1)$
    where $C_i \cong \bP^1$ are the irreducible components of $E$
    \cite[Lemma 2.20]{Toda2013}, hence it suffices to take $C(\pi) >
    \omega \cdot [E]$.}
  \[ \Pn(\beta_{\cA}, n) \coloneqq \omega \cdot \beta_{\cA} + C(\pi) n \ge 0 \]
  for all classes $(0, 0, (0, -\beta_{\cA}), -n)$ of objects in
  $\cat{Per}_0(\pi)[-1]$, with equality if and only if $(\beta_{\cA},
  n) = (0, 0)$, i.e. is the class of the zero object. Let
  \[ \mu(\beta_{\cA},n) \coloneqq \frac{\omega \cdot \beta_{\cA}}{\Pn(\beta_{\cA},n)} \in \bQ \]
  and, for every $s \in \bQ$, define the family of weak stability
  conditions
  \[ \tau_\xi^{(s)}(r, 0, (-\beta_{\bP}, -\beta_{\cA}), -n) \coloneqq \begin{cases} 3\pi/4 & r \neq 0, \\ \pi/2 & r=0, \, \beta_{\bP} \neq 0, \\ 3\pi/4 + \epsilon |\xi| \left(\mu(\beta_{\cA},n) - s\right) + \epsilon^2 \xi & r = \beta_{\bP} = 0, \end{cases} \]
  for $\xi \in (-\pi/4, \pi/4)$, where $\epsilon$ is a formal symbol
  with ordering defined by $x_0 + \epsilon x_1 + \epsilon^2 x_2 < y_0
  + \epsilon y_1 + \epsilon^2 y_2$ if and only if $(x_0, x_1, x_2) <
  (y_0, y_1, y_2)$ in the lexicographic order. In English, this is the
  family where objects in $\cat{Per}_0(\pi)[-1]$ of class
  $(\beta_{\cA},n)$ with $\mu(\beta_{\cA},n) = s$ move from slope
  $<3\pi/4$ to slope $>3\pi/4$, and slopes of all other objects remain
  unchanged. This is still a continuous family of stability conditions with exactly one wall at $\xi
  = 0$ for objects of rank one, so let $\tau_+^{(s)}$ and
  $\tau_-^{(s)}$ denote the cases $\xi > 0$ and $\xi < 0$
  respectively. 

  Analogously to $\fN_{(\lambda,\mu,\nu),n}$, define the following
  open moduli substack of $\fM_{0,0,-\beta_{\cA},-n}$:
  \begin{align*}
    \fQ_{\beta_{\cA},n}
    &\coloneqq \{[0 \to F] \in \fM_{0,0,-\beta_{\cA},-n} : L\iota^*F = 0\} \\
    &= \{[0 \to F] : F \in \cat{Per}_0(\pi|_{X \setminus D}), \, (\ch_2(F), \ch_3(F)) = (\beta_{\cA}, n)\}.
  \end{align*}
  Note that all such $F$ are automatically $\tau_0^{(s)}$-semistable,
  i.e. $\fQ_{\beta_{\cA},n}^{\sst}(\tau_0^{(s)}) =
  \fQ_{\beta_{\cA},n}^\pl$, and all sub- and quotient objects of $F$
  have the same $\tau_0^{(s)}$-slope as $F$. None of this is true for
  $\tau_0$ and is one reason why $\tau_0^{(s)}$ is better.

  It is easy to check that $\{\tau_\xi^{(s)}\}_{\xi}$ defines a
  wall-crossing problem of simple type
  (Definition~\ref{def:wall-crossing-simple-type}) for the subsets $A$
  and $\tilde B$ as in \eqref{eq:PT-BS-A-B-classes} below and for the
  subset $B \subset \tilde B$ consisting of classes with
  $\mu(\beta_{\cA},n) = s$. Thus,
  Proposition~\ref{prop:wcf-simple-type} still applies for each $s$.
  Note that the moduli stacks $\fQ_{\beta_{\cA},n}$ (for classes in
  $\tilde B$) are independent of $\vec\lambda$. Clearly
  $\tau_+^{(0)}$-semistability (resp. $\lim_{s \to \infty}
  \tau_+^{(s)}$) agrees with $\tau^{\PT}$-semistability (resp.
  $\tau^{\BS}$) for objects of rank one. For a given $\alpha \in A$,
  finiteness of $R_\alpha$ (see \S\ref{wc:sec:S-and-R-sets}) ensures
  there are only finitely many walls $s \in \bQ_{> 0}$ to consider.
  Hence, the desired wall-crossing from $\tau^{\PT}$ to $\tau^{\BS}$
  may be factorized through finitely many pairs of dominant
  wall-crossings (``chamber-to-wall'' followed by ``wall-to-chamber'')
  given by $(\tau, \mathring\tau) = (\tau_0^{(s)}, \tau_\pm^{(s)})$.
  This produces the desired wall-crossing formula \eqref{eq:PT-BS}.

\subsubsection{Assumption~\ref{es:assump:exact-subcategory}}

  (\ref{es:assump:exact-subcategory}\ref{es:assump:it:permissible-classes})
  Set $\beta_{\bP} \coloneqq \sum_{i=1}^{m+1} |\lambda_i|$ and let
  $C(\cat{A})_{\pe} \coloneqq A \sqcup B$ where
  \begin{equation} \label{eq:PT-BS-A-B-classes}
    \begin{aligned}
      A &\coloneqq \{(1,0,(-\beta_{\bP},-\beta_{\cA}),-n) : \fN_{\vec\lambda,\beta_{\cA},n}^{\sst}(\tau_0') \neq \emptyset\}, \\
      \tilde B &\coloneqq \{(0,0,(0,-\beta_{\cA}),-n) : \fQ_{\beta_{\cA},n}^{\sst}(\tau_0') \neq \emptyset\}.
    \end{aligned}
  \end{equation}
  By Lemma~\ref{lem:PT-BS-stacks-summary} (and an easy modification of
  the argument there), the moduli stacks $\{\fM_\alpha\}_{\alpha \in
    C(\cat{A})_{\pe}}$ are Artin and locally of finite type, and form
  a graded partially-monoidal $\sT$-stack.

  (\ref{es:assump:exact-subcategory}\ref{es:assumption-restricted-substack})
  The locally closed substacks
  \[ \fN_{\vec\lambda,\beta_{\cA},n} \subset \fM_{1,0,(-\beta_{\bP},-\beta_{\cA}),-n} \text{ and } \fQ_{\beta_{\cA},n} \subset \fM_{0,0,(0,-\beta_{\cA}),-n}, \]
  ranging over all classes in $C(\cat{A})_{\pe}$, form a restricted
  graded partially-monoidal $\sT$-stack with $\kappa$-symmetric
  bilinear elements given by
  \eqref{eq:unrigidified-DT-PT-bilinear-element-N} and
  \eqref{eq:unrigidified-DT-PT-bilinear-element-Q}.

\subsubsection{Assumptions~\ref{es:assump:semistable-invariants} and \ref{es:assump:wall-crossing}}
\label{sec:PT-BS-WCF-proof-end}

  (\ref{es:assump:semistable-invariants}\ref{es:assump:it:tau-artinian},
  \ref{es:assump:semistable-invariants}\ref{es:assump:it:semistable-loci},
  and \ref{es:assump:wall-crossing}\ref{es:assump:it:tau-circ}) By the
  same argument as in \cite[\S 6.2]{Toda2013}, $\cat{A}$ admits
  $\tau_\xi$-HN filtrations for permissible classes, for all $\xi$. By
  the same argument as in \cite[\S 6.4, Step 2]{Toda2013},
  $\tau_\xi$-semistability is open for permissible classes, for all
  $\xi$.

  (\ref{es:assump:semistable-invariants}\ref{es:assump:it:framing-functor})
  For the framing functor(s), as in the DT/PT case
  (\S\ref{sec:DT-PT-WCF-proof-end}), we consider pairs $I = [\cO_X
    \otimes L \to F]$ where now $F \in \cat{Per}_{\le 1}(\pi)$. But
  there is an equivalence of abelian categories \cite[Theorem
    1.4]{Calabrese2016}
  \[ \Phi\colon \Per(\pi) \xrightarrow{\sim} \cat{Coh}(\fX), \]
  where $\fX \coloneqq [\bC^2/\bZ_{m+1}] \times \bP^1$ is the orbifold
  corresponding to $X_0$, and we can view a sheaf on the quotient
  stack $[\bC^2/\bZ_{m+1}]$ as a $\bZ_{m+1}$-equivariant sheaf on
  $\bC^2$. So, let $\cO_{\fX}(1)$ denote any very ample line bundle on
  $\fX$, take $k \gg 0$ and $p \in \bZ_{> 0}$, and define
  \[ \Fr_{k,p}\colon [\cO_X \otimes L \to F] \mapsto H^0_{\bZ_{m+1}}(\Phi(F) \otimes \cO_{\fX}(k)) \oplus L^{\oplus p}. \]
  Here $H^0_{\bZ_{m+1}}$ denotes the vector space (not
  $\bZ_{m+1}$-representation!) of $\bZ_{m+1}$-equivariant global
  sections, i.e. pushforward along $\fX \to [\pt/\bZ_{m+1}]$. Since a
  morphism of $\bZ_{m+1}$-representations is in particular a morphism
  of vector spaces, $\Hom(I, I) \to \Hom(\Fr_{k,p}(I), \Fr_{k,p}(I))$
  remains injective. Comparing with Example~\ref{ex:framing-functor},
  the rest of the properties making $\Fr_{k,p}$ into a framing functor
  is clear.

  (\ref{es:assump:semistable-invariants}\ref{es:assump:it:rank-function})
  For the rank function, by the same argument as in the DT/PT case
  (\S\ref{sec:DT-PT-WCF-proof-end}), there exists $\Pn_0(\beta_{\bP})
  \in \bZ$ such that $\fM_{1,0,(-\beta_{\bP}, -\beta_{\cA}),
    -n}^{\sst}(\tau_0') = \emptyset$ for all $(\beta_{\cA},n)$ such
  that $\Pn(\beta_{\cA},n) \le \Pn_0(\beta_{\bP})$. Thus, any $r$
  satisfying
  \begin{align*}
    r(1, 0, (-\beta_{\bP}, -\beta_{\cA}), -n) &\coloneqq \Pn(\beta_{\cA}, n) - \Pn_0(\beta_{\bP}) \\
    r(0, 0, (0, -\beta_{\cA}), -n) &\coloneqq \Pn(\beta_{\cA}, n)
  \end{align*}
  is a valid rank function.

  (\ref{es:assump:wall-crossing}\ref{es:assump:it:lambda}) Let $\alpha
  \in C(\cat{A})_{\pe}$. If $\alpha \in \tilde B$, then $R_\alpha
  \subset \tilde B$ and we can take $\lambda$ to be any homomorphism
  such that
  \[ \lambda(\beta) \coloneqq \omega \cdot \beta_{\cA}(\beta) \Pn(\alpha) - \omega \cdot \beta_{\cA}(\alpha) \Pn(\beta) \]
  for $\beta \in \tilde B$, where $\beta_{\cA}(\beta)$ denotes the
  $\beta_{\cA}$ component of the class $\beta$. If $\alpha \in A$,
  take $\lambda$ to be any homomorphism such that $\lambda(\alpha)
  \coloneqq 0$ and
  \[ \lambda(\beta) \coloneqq \left(\omega \cdot \beta_{\cA}(\beta) - s \Pn(\beta)\right) \pm \delta \Pn(\beta), \quad \beta \in \tilde B, \]
  for the $\tau_0^{(s)}$ to $\tau_\pm^{(s)}$ dominant wall-crossing,
  where $\delta \in \bR$ is positive but sufficiently small such that
  the bracketed term always dominates the $\delta \Pn(\beta)$ term
  when $\beta \in R_\alpha$. Such a $\delta$ exists because $R_\alpha$
  is a finite set (see \S\ref{wc:sec:S-and-R-sets}).
\end{proof}

\subsubsection{}
\label{sec:PT-BS-consequences}

Only the trivial BS stable pair $\cO \to 0$ has curve class
$\beta_{\bP} = 0$ \cite[Prop. 18]{Bryan2016}, hence
\[ \sV^{\BS(\pi)}_{\vec\emptyset} = \partial \]
where $\partial = \id \in K_\circ^{\tilde\sT}(\pt)$. By the same
reasoning as in \S\ref{sec:DT-PT-consequences}, we obtain
\begin{equation} \label{eq:PT-BS-0}
  I_*\sV_{\vec\emptyset}^{\PT(\pi)} = \bigg(\prod^{\rightarrow}_{s \in \bQ_{>0}} \exp\left(\ad(\sz^{\pi}_{(s)})\right)\bigg) I_*\partial.
\end{equation}
and the PT/BS primary vertex correspondence \eqref{eq:PT-BS-primary}.

\subsubsection{}

\begin{remark} \label{rem:PT-BS-wcf-shape}
  If we were able to ignore the technical issues in
  \S\ref{sec:PT-BS-WCF-proof} and use the original family
  $\{\tau_\xi\}_\xi$ of weak stability conditions --- which is already
  a wall-crossing problem of simple type for $A$ --- instead of the
  factorization into the families $\{\tau_\xi^{(s)}\}_\xi$, then we
  would obtain a wall-crossing formula of the form
  \begin{equation} \label{eq:PT-BS-expected}
    I_*\sV_{\vec\lambda}^{\PT(\pi)} = \exp\left(\ad(\sz^{\pi})\right) I_*\sV_{\vec\lambda}^{\BS(\pi)}
  \end{equation}
  for the natural semistable invariants $\sz^{\pi} \coloneqq
  \sum_{\omega \cdot \beta_{\cA} > 0} \sum_{n \in \bZ}
  \sz_{\beta_{\cA},n}^{\pi} A^{\beta_{\cA}} Q^n$ associated to the
  semistable loci $\fQ_{\beta_{\cA},n}^{\sst}(\tau^{\PT})$. Note that
  this may be a {\it different} set of semistable invariants than the
  ones appearing in \eqref{eq:PT-BS}. To compute descendent
  transformations, it may be easier to have a formula of the shape
  \eqref{eq:PT-BS-expected} rather than \eqref{eq:PT-BS}, but both
  work equally well for the purpose of obtaining the equivariant
  primary PT/BS vertex correspondence \eqref{eq:PT-BS-primary}.
\end{remark}

\subsection{Explicit descendent transformations}
\label{sec:DT-descendent-transformations}

\subsubsection{}

In this subsection, we consider the cohomological DT/PT vertex
correspondence
(Theorem~\ref{thm:vertex-correspondences}\ref{it:dt-pt-vertex}),
following the notation in \S\ref{sec:cohomological-version}, and give
an explicit computation of the cohomological Lie bracket
(Theorem~\ref{thm:cohVOA-monoidal-stack} and
\S\ref{sec:coh-lie-algebra})
\[ (\Pi^\pl)_*[I_*(-), -]\colon A_*^{\sT}(\fN_{(\lambda,\mu,\nu),n}^\pl)_{\loc} \otimes A_*^{\sT}(\fQ_m)_{\loc}^\pl \to A_*^{\sT}(\fN_{(\lambda,\mu,\nu),n+m}^\pl)_{\loc} \]
for the DT/PT moduli stacks (\S\ref{sec:DT-PT-N-stack} and
\S\ref{sec:DT-PT-WCF-proof}), and thereby prove
Corollary~\ref{cor:descendent-correspondences} and
Theorem~\ref{thm:DT-PT-descendents} from the introduction. Here we are
using the isomorphisms \eqref{eq:DT-PT-rigidification}; see also
Definition~\ref{def:DT-rigidified-descendents}. More precisely, we
will actually compute the dual operation
\begin{align*}
  \Delta\colon A_{\sT}^*(\fN_{\lambda,\mu,\nu}^\pl)_{\loc} &\to A_{\sT}^*(\fN_{\lambda,\mu,\nu}^\pl)_{\loc} \otimes A_{\sT}^*(\fQ)_{\loc} \\
  \tau &\mapsto \frac{1}{\hbar} \Res_{u=0} (I \times \id)^* \left(\Theta(u) \cup (\deg_u \times \id)\Phi^*(\Pi^\pl)^*\tau\right)
\end{align*}
so that
\[ \left((\Pi^\pl)_*[I_*\phi, \psi]\right)(\tau) = (\phi \boxtimes \psi)\left(\Delta(\tau)\right). \]
Iterating this, and letting $\Delta^{\boxtimes n} \coloneqq (\Delta \boxtimes
\id^{\boxtimes n-2})(\Delta \boxtimes \id^{\boxtimes n-3}) \cdots (\Delta
\boxtimes \id)\Delta$,
\begin{equation} \label{eq:coproduct-iterated}
  \left(\exp(\ad_{\sz}) I_*\phi\right)((\Pi^\pl)^*\tau) = \bigg(\sum_{n \ge 0} \frac{(-1)^n}{n!} (\phi \boxtimes \sz^{\boxtimes n})\bigg)\left(\Delta^{\boxtimes n} \tau\right).
\end{equation}
For short, we write the right hand side as $(\sum_{n \ge 0} \phi
\boxtimes \sz^{\boxtimes n})(\exp(-\Delta)\tau)$ with the convention
for each term that if the number of external tensor factors do not
match then the result is zero. Then, from the equivariant vertex
correspondence \eqref{eq:DT-PT},
\begin{equation} \label{eq:DT-PT-with-coproduct}
  \sV_{\lambda,\mu,\nu}^{\DT} = \bigg(\sV_{\lambda,\mu,\nu}^{\PT} \boxtimes \sum_{n \ge 0} \sz^{\boxtimes n}\bigg)\left(e^{-\Delta} \tau\right).
\end{equation}
The goal is to compute matrix elements of $e^{-\Delta}$ in the
tautological sub-ring of $A_{\sT}^*(\fN_{\lambda,\mu,\nu}^\pl)_{\loc}$
generated by the following descendent classes.

\subsubsection{}

\begin{definition}[Descendents] \label{def:DT-rigidified-descendents}
  Let $\pi\colon \fN \to \fN^\pl$ be a trivial $\bC^\times$-gerbe with
  section denoted $I$. Write $\fN = \fN^\pl \times [\pt/\bC^\times]$,
  let $\cL$ denote weight-$1$ line bundle on $[\pt/\bC^\times]$, and
  let
  \[ v \coloneqq c_1(\cL). \]
  If there is a universal family $\scF$ on $\fN \times X$ of
  $\bC^\times$-weight $1$, then $\cL^\vee \otimes \scF$ has weight $0$
  and therefore $\scF$ descends to the ``rigidified'' universal family
  \[ \scF^\pl \coloneqq (I \times \id)^*(\cL^\vee \otimes \scF) \]
  on $\fN^\pl$. For any homogeneous $\xi \in \CH_*^\sT(X)$, define the
  {\it descendents} and {\it unrigidified descendents}
  \begin{align*}
    \tau_n(\xi) &\coloneqq \pi_{\fN^\pl *}(\ch_n(\scF^\pl) \cdot \pi_X^*(\xi) \cap \pi_{\fN^\pl}^*(-)) \in A_{\sT}^{n-3+\deg \xi}(\fN^\pl; \bQ), \\
    \tilde\tau_n(\xi) &\coloneqq \pi_{\fN*}(\ch_n(\scF) \cdot \pi_X^*(\xi) \cap \pi_{\fN}^*(-)) \in A_{\sT}^{n-3+\deg \xi}(\fN; \bQ),
  \end{align*}
  where $\pi_{\fN}$, $\pi_{\fN^\pl}$, and $\pi_X$ are projections to
  the factor specified by the subscript. They are related by
  \[ (\Pi^\pl)^* \tau_n(\xi) = e^{-v} \tilde\tau_n(\xi) \]
  using that $(\Pi^\pl)^* \scF^\pl = \cL^\vee \otimes \scF$ and base
  change.

  By Lemma~\ref{lem:DT-PT-rigidification}, these considerations apply
  to the moduli stacks $\fN_{(\lambda,\mu,\nu),n}$ with their
  universal families $\scF$ of objects $F$ in pairs $[\cL \to F]$. On
  the other hand, although the moduli stacks $\fQ_m$ are not trivial
  $\bC^\times$-gerbes over their rigidifications, we continue to use
  $\scF$ to denote their universal families and $\tilde\tau_n(\xi)$ to
  denote their unrigidified descendents.
\end{definition}

\subsubsection{}

We fix some notation. Let $\bh_{\sT}^* = \bZ[s_1, s_2, s_3]$, and let
$\hbar \coloneqq s_1 + s_2 + s_3$ be the cohomological Calabi--Yau
weight, i.e. the cohomological version of $\kappa$. We freely use the
notation of Theorem~\ref{thm:DT-PT-descendents} throughout this
subsection.

For the remainder of this subsection only, it is convenient to let
$\psi^k$ denote the $k$-th ``cohomological Adams operation'', which
acts as multiplication by $k^n$ on $H_\sT^{2n}$. A superscript $(k)$
on an object will mean $\psi^k$ applied to that object. This is
consistent with
\[ \tau^{(k)}(\xi) = \psi^k \tau(\xi), \text{ for } \tau(\xi) \coloneqq \sum_{n \ge 0} \tau_n(\xi). \]
Clearly $\psi^k$ commutes with all pullbacks, which preserve
cohomological degree. 

\subsubsection{}

\begin{proposition} \label{prop:DT-coproduct}
  Let $\Theta(u) \eqqcolon \sum_{n \ge 0} u^{-n} \theta_n$ be its
  expansion in $u$ and define
  \[ \Xi \coloneqq \frac{1}{\hbar} \sum_{n \ge 0} \frac{\theta_{n+1}}{n!} \]
  and $\Xi^{(k)} \coloneqq \psi^k \Xi$ (see Lemma~\ref{lem:DT-Theta}
  below). Then
  \begin{equation} \label{eq:DT-coproduct}
    \Delta\bigg(\prod_{i=1}^N \tau^{(k_i)}(\alpha_i)\bigg) = \sum_{I \sqcup J = \underline{N}} (I \times \id)^*(e^{-K_J v} \Xi^{(-K_J)}) \otimes \bigg(\prod_{i \in I} \tau^{(k_i)}(\alpha_i) \boxtimes \prod_{j \in J} \tilde\tau^{(k_j)}(\alpha_j)\bigg)
  \end{equation}
  where, in the sum, $K_J \coloneqq \sum_{j \in J} k_j$.
\end{proposition}

\begin{proof}
  Note that all pullbacks, and therefore $\deg_u$ too, are algebra
  homomorphisms, so it is enough to consider a single
  $\tau^{(k)}(\alpha)$ when computing
  \begin{align}
    &(\deg_u \times \id)\Phi^*(\Pi^\pl)^*\tau(\alpha) \nonumber \\
    &= (\deg_u \times \id)\Phi^*\left(e^{-kv} \tilde\tau^{(k)}(\alpha)\right) \nonumber \\
    &= (\deg_u \times \id)\left(e^{-kv} (\tilde\tau^{(k)}(\alpha) \boxtimes 1 + 1 \boxtimes \tilde\tau^{(k)}(\alpha))\right) \nonumber \\
    &= e^{-kv} \tilde\tau^{(k)}(\alpha) \boxtimes 1 + e^{-ku} e^{-kv} \boxtimes \tilde\tau^{(k)}(\alpha), \label{eq:DT-coproduct-1}
  \end{align}
  using that $\Phi^*\ch^{(k)}(\scF) = \ch^{(k)}(\scF \boxplus \scF) =
  \ch^{(k)}(\scF) \boxtimes 1 + 1 \boxtimes \ch^{(k)}(\scF)$ and that
  \[ \deg_u \ch^{(k)}(\scF) = e^{ku} \ch(\scF), \quad \deg_u e^{-kv} = e^{-ku} e^{-kv}. \]
  Now we cup $\Theta(u)/\hbar$ with a product of
  \eqref{eq:DT-coproduct-1}:
  \begin{align*}
    &\frac{1}{\hbar} \Res_{u=0} \bigg(\Theta(u) \cup \prod_{i=1}^N (e^{-k_i v} \tilde\tau^{(k_i)}(\alpha_i) \boxtimes 1 + e^{-k_i u} e^{-k_i v} \boxtimes \tilde\tau^{(k_i)}(\alpha_i))\bigg) \\
    &= \sum_{I \sqcup J = \underline N} \frac{1}{\hbar} \Res_{u=0} \bigg(\Theta(u) \cup e^{-K_J u} e^{-K_Jv} \prod_{i \in I} e^{-k_i v}\tilde\tau^{(k_i)}(\alpha_i) \boxtimes \prod_{j \in J} \tilde\tau^{(k_j)}(\alpha)\bigg) \\
    &= \sum_{I \sqcup J = \underline N} e^{-K_Jv}\Xi^{(-K_J)} \otimes \bigg(\prod_{i \in I} e^{-k_i v} \tilde\tau^{(k_i)}(\alpha_i) \boxtimes \prod_{j \in J} \tilde\tau^{(k_j)}(\alpha)\bigg).
  \end{align*}
  where in the second equality we used that $\theta_k$ has
  cohomological degree $2k$ (Proposition~\ref{lem:DT-Theta}) and
  therefore
  \[ \frac{1}{\hbar} \Res_{u=0} e^{N u} \sum_{n \ge 0} u^{-n} \theta_n = \frac{1}{\hbar} \sum_{k \ge 0} \frac{N^k \theta_{k+1}}{k!} = \psi^N \Xi. \]
  Finally, by Proposition~\ref{lem:DT-Theta} below, $e^v \theta_k$
  and therefore $e^{-K_J v} \Xi^{(-K_J)}$ has $\bC^\times$-weight zero
  in the first factor.
\end{proof}

\subsubsection{}

\begin{lemma} \label{lem:DT-Theta}
  Let $\scE$ be the bilinear element defining $\Theta(u)$. Then
  \begin{equation} \label{eq:DT-Theta-plethystic}
    \Theta(u) = \frac{e_{-u}(\scE)}{e_u(\kappa^{-1} \scE^\vee)} = \left(-\frac{u}{u - \hbar}\right)^{\rank \scE} \exp \bigg(\sum_{k>0} ((u - \hbar)^{-k} - u^{-k}) (k-1)! \ch_k(\scE)\bigg).
  \end{equation}
\end{lemma}

From \eqref{eq:DT-Theta-plethystic}, clearly $\Theta(u)$ has the form
$\sum_{n \ge 0} u^{-n} \theta_n$ and each $\theta_n$ is divisible by
$\hbar$. Explicitly, using the binomial theorem or otherwise,
\begin{equation} \label{eq:DT-theta}
  \theta_n = (-1)^{\rank \scE} \!\!\!\!\sum_{\substack{k \ge 0\\n_1,\ldots,n_k \ge 2\\m\ge 0\\n=n_1+\cdots+n_k+m}} \!\!\!\!\!\!\frac{(-1)^k}{k!} (-\hbar)^m \binom{\rank \scE}{m} \prod_{i=1}^k \bigg[\sum_{a=1}^{n_i-1} \frac{(n_i-1)!}{(n_i-a)!} \hbar^{n_i-a} \ch_a(\scE)\bigg].
\end{equation}

\begin{proof}
  The first equality is the definition of $\Theta(u)$. To prove the
  second equality, note that both sides are multiplicative in $\scE$,
  so it suffices to prove it when $\scE = \cL$ is a line bundle. Let
  $\zeta \coloneqq c_1(\cL)$. Then the left hand side becomes
  \[ \frac{-u + \zeta}{u - \hbar - \zeta} = -\frac{u}{u - \hbar} \frac{1 - u^{-1} \zeta}{1 - (u - \hbar)^{-1}\zeta}. \]
  This equals the right hand side using the series expansion (in
  $u^{-1}$)
  \[ 1 - u^{-1}\zeta = \exp\left(\log(1 - u^{-1}\zeta)\right) = \exp\bigg(-\sum_{k>0} \frac{(u^{-1} \zeta)^k}{k}\bigg) \]
  and that $\zeta^k = k! \cdot \ch_k(\cL)$.
\end{proof}

\subsubsection{}

\begin{lemma} \label{lem:DT-bilinear-element-chern-character}
  Let $\td(\bC^3) \coloneqq \td(s_1)\td(s_2)\td(s_3)$ where $\td(s)
  \coloneqq s/(1 - e^{-s})$. Then
  \[ \ch(\scE) = -\td(\bC^3) (e^{-v} - \tilde\tau^{(-1)}(\spt)) \boxtimes \tilde\tau(1). \]
\end{lemma}

\begin{proof}
  Write $\scI_1 = [\cL \to \scF_1]$ and $\scI_2 = [0 \to \scF_2]$ and
  let $\pi = \pi_{\fN \times \fQ}$ be projection along $X$. Then
  \begin{align*}
    \ch(\scE)
    &= \ch(R\pi_*R\cHom(\scI_1, \scI_2)) \\
    &= \pi_*\left(\ch\left(R\cHom(\scI_1, \scI_2)\right) \td(X)\right) \\
    &= \pi_*\left(\ch^{(-1)}(\scI_1) \cdot \ch(\scI_2) \cdot \td(X)\right) \\
    &= \pi_*\left((e^{-v} - \ch^{(-1)}(\scF_1))\cdot (-\ch(\scF_2)) \cdot \td(X)\right).
  \end{align*}
  Here, we used that $\supp \scI_2$ is a $0$-dimensional subscheme of
  $\bC^3 \subset X$, so the twist by $-D$ in $\scE$ may be neglected.

  To express this in terms of descendents, write the class $\delta$ of
  the diagonal $X \subset X \times X$ as $\delta = \sum_i \zeta_i
  \boxtimes \zeta_i^\star$ where $\{\zeta_i\}$ and $\{\zeta_i^\star\}$
  are dual bases for $H^*_\sT(X)$ under the pairing $x \otimes y
  \mapsto \int_X xy$. Then $\ch(\scF_i) = \sum_a \tilde\tau(\zeta_i)
  \boxtimes \zeta_i^\star$ and therefore we get
  \begin{equation} \label{eq:bilinear-element-chern-character}
    \ch(\scE) = -\sum_j (e^{-v} \boxtimes \tilde\tau(\zeta_j)) \int_X \zeta_j^\star \td(X) + \sum_{i,j} (\tilde\tau^{(-1)}(\zeta_i) \boxtimes \tilde\tau(\zeta_j)) \int_X \zeta_i^\star \zeta_j^\star \td(X).
  \end{equation}
  Pick the basis $\{1, h\} \subset H^*_{\bC^\times}(\bP^1)$ where $h$
  is the hyperplane at $\infty$. Then, on $\fQ$, only $\tilde\tau(1)$
  is non-vanishing. The dual of $\zeta_j = 1$ is $\zeta_j^\star =
  \spt$, so the first sum becomes $\td(\bC^3) e^{-v} \boxtimes
  \tilde\tau(1)$. Then, in the second sum, for support reasons
  $\zeta_i^\star = 1$ is the only non-zero term. Therefore the second
  sum becomes $\td(\bC^3) \tilde\tau^{(-1)}(\spt) \boxtimes
  \tilde\tau(1)$.
\end{proof}

\subsubsection{}

\begin{remark} \label{rem:DT-PT-coproduct-generality}
  Up until \eqref{eq:bilinear-element-chern-character}, we have used
  nothing specific to the DT/PT geometry, i.e.
  Proposition~\ref{prop:DT-coproduct} and Lemma~\ref{lem:DT-Theta} and
  \eqref{eq:bilinear-element-chern-character} continue to hold for the
  PT/BS geometry of \S\ref{sec:PT-BS}. A more complicated calculation
  starting from \eqref{eq:bilinear-element-chern-character} will yield
  a PT/BS version of
  Lemma~\ref{lem:DT-bilinear-element-chern-character}.
\end{remark}

\subsubsection{}

\begin{proof}[Proof of Corollary~\ref{cor:descendent-correspondences}.]
  We will state the proof for the DT/PT case, i.e. for
  \[ (\sV, \sV_0, \sV') = (\sV_{\lambda,\mu,\nu}^{\DT}, \sV_{\emptyset,\emptyset,\emptyset}^{\DT}, \sV_{\lambda,\mu,\nu}^{\PT}), \]
  but by Remark~\ref{rem:DT-PT-coproduct-generality} the same proof
  strategy will yield the PT/BS case.

  Combining Proposition~\ref{prop:DT-coproduct} with
  Lemmas~\ref{lem:DT-Theta} and
  \ref{lem:DT-bilinear-element-chern-character}, clearly $\Delta$
  preserves the {\it tautological subalgebras} of
  $A_{\sT}^*(\fN_{\lambda,\mu,\nu}^\pl)$ and $A_{\sT}^*(\fQ)$, i.e.
  the $\bh_{\sT}$-subalgebras generated by descendent classes. Hence,
  using \eqref{eq:DT-PT-with-coproduct},
  \[ \sV(\ff) = \sum_{\ff'} \tilde c_{\ff}^{\ff'}(\sz) \cdot \sV'(\ff') \]
  where $\tilde c_{\ff}^{\ff'}(\sz)$ is a polynomial in $\sz(\ff'')$.
  To conclude, it suffices to show that $\sz_n(\ff)$, for a descendent
  $\ff$, is itself a polynomial in $(\sV_0)_m(\ff')$ for $m \le n$ and
  descendents $\ff'$. This follows because the $Q^n$ coefficient of
  \eqref{eq:DT-PT-0}, for $n > 0$, says
  \[ I_*(\sV_0)_n = [\sz_n, \partial] + \cdots, \]
  where $\cdots$ involves only $\sz_m$ for $m < n$, so using that
  $\Delta$ preserves the tautological subalgebras, by induction
  $\sz_n(\ff)$ may be written using only iterated Lie brackets of
  various $(\sV_0)_m(\ff')$ for $m \le n$. (The same argument was also
  used in Definition~\ref{def:sstable-explicit} to explicitly define
  semistable invariants.)
\end{proof}

\subsubsection{}

For the remainder of this subsection, we take the Calabi--Yau
specialization $\hbar \to 0$ with the goal of proving the explicit
Calabi--Yau descendent correspondence
Theorem~\ref{thm:DT-PT-descendents}. From either
\eqref{eq:DT-Theta-plethystic} or \eqref{eq:DT-theta},
\[ \frac{\theta_{n+1}}{\hbar} = -(-1)^{\rank \scE} n! \ch_n(\scE) \]
and therefore, in \eqref{eq:DT-coproduct},
\[ (I \times \id)^*\left(e^{-K_J} \Xi^{(-K_J)}\right) = (-1)^{\rank \scE-1} \ch^{-(K_J)}(\scE). \]
Using Lemma~\ref{lem:DT-bilinear-element-chern-character}, the
coproduct \eqref{eq:DT-coproduct} simplifies into
\begin{equation} \label{eq:DT-coproduct-CY}
  \Delta\bigg(\prod_{i=1}^N \tau^{(k_i)}(\spt)\bigg) = \sum_{I \sqcup J = \underline{N}} \bigg((1 - \tau^{(K_J)}(\spt)) \prod_{i \in I} \tau^{(k_i)}(\spt)\bigg) \boxtimes a\{k_j\}_{j \in J}.
\end{equation}
where, using that $\rank \scE = -m$ depends only on the factor $\fQ_m$
and not the factor $\fN_{(\lambda,\mu,\nu),n}$,
\[ a\{k_1, \ldots, k_M\} \coloneqq (-1)^{\rank \scE} \td^{(K)}(\bC^3) \tilde\tau^{(-K)}(1) \prod_{j \in J} \tilde\tau^{(k_j)}(\spt), \qquad K \coloneqq \sum_{i=1}^M k_i. \]
Formulas are more economical in the basis of descendents given by
$\sigma\{k_1,\ldots,k_N\}$, in which
\begin{equation} \label{eq:DT-coproduct-CY-sigma}
  \Delta\left(\sigma\{k_1,\ldots,k_N\}\right) = \sum_{I \sqcup J = \underline{N}} (-1)^{|J|} \sigma\{K_J, \{k_i\}_{i \in I}\} \boxtimes a\{k_j\}_{j \in J}
\end{equation}
by collecting all terms in \eqref{eq:DT-coproduct-CY} of the form $-
\boxtimes a\{k_j\}_{j \in J}$.

\subsubsection{}

\begin{proof}[Proof of Theorem~\ref{thm:DT-PT-descendents}.]
  We define a ``formal'' version of $\Delta$ for notational
  convenience. Fix $k_1, \ldots, k_N$ and let $V = V_{\underline{N}}$
  be the free $\bh_\sT$-module spanned by formal symbols in the set
  \begin{equation} \label{eq:descendent-sigma-basis}
    \Sigma = \Sigma_{\underline{N}} \coloneqq \left\{\sigma\{K_{S_1},\ldots,K_{S_n}\} : S_1 \sqcup \cdots \sqcup S_n = \underline{N} \text { is a set partition}\right\},
  \end{equation}
  which we take to be an orthogonal basis in order to write matrix
  elements later. For a subset $S \subset \underline{N}$, the symbol
  $K_S$ corresponds to $\sum_{j \in S} k_j$, and if $S = \{i\}$ then
  we write $k_i$ instead of $K_S$. View $\Delta$ as a linear operator
  on $V$ valued in the (external) tensor $\bh_\sT$-algebra in formal
  variables $a\{k_i\}_{i \in S}$, for all subsets $S \subset
  \underline{N}$. Then \eqref{eq:DT-coproduct-CY-sigma} becomes
  \begin{equation} \label{eq:DT-coproduct-CY-matrix-elements}
    \Braket{\sigma' | \Delta | \sigma\{k_1,\ldots,k_N\}} = \begin{cases} a\{\} - \sum_{i=1}^N a\{k_i\} & \sigma' = \sigma\{k_1,\ldots,k_N\} \\ (-1)^{|J|} a\{k_j\}_{j \in J} & \sigma' = \sigma\{K_J, \{k_i\}_{i \in I}\} \\ & \quad\text{for } I \sqcup J = \underline{N}, \, |I| > 1 \\ 0 & \text{otherwise}. \end{cases}
  \end{equation}
  In this notation, the vertex correspondence
  \eqref{eq:DT-PT-with-coproduct} becomes
  \begin{equation} \label{eq:DT-PT-CY-with-coproduct}
    \sV_{\lambda,\mu,\nu}^{\DT}(\sigma\{k_1,\ldots,k_N\}) = \sum_{\sigma' \in \Sigma} \sV_{\lambda,\mu,\nu}^{\PT}(\sigma') \sz^{\boxtimes}\left(\Braket{\sigma' | e^{-\Delta} | \sigma\{k_1,\ldots,k_N\}}\right)
  \end{equation}
  where $\sz^{\boxtimes} \coloneqq \sum_{n \ge 0} \sz^{\boxtimes n}$
  for short. Specializing $(\lambda,\mu,\nu) =
  (\emptyset,\emptyset,\emptyset)$, this becomes
  \begin{equation} \label{eq:DT-PT-0-CY-with-coproduct}
    \sV_{\emptyset,\emptyset,\emptyset}^{\DT}(\sigma\{k_1,\ldots,k_N\}) = \sum_{\sigma' \in \Sigma} \sz^{\boxtimes}\left(\Braket{\sigma' | e^{-\Delta} | \sigma\{k_1,\ldots,k_N\}}\right).
  \end{equation}

\subsubsection{}

  We will compute the matrix exponential $e^{-\Delta}$ as follows.
  First, define a partial ordering $\preceq$ on $\Sigma$ by
  \[ \sigma\{K_{S_1'},\ldots,K_{S_m'}\} \preceq \sigma\{K_{S_1},\ldots,K_{S_n}\} \iff S_1 \sqcup \cdots \sqcup S_n \text{ refines } S_1' \sqcup \cdots \sqcup S_m' \]
  in the sense that each $S_j'$ for $1 \le j \le m$ is the union of
  one or more $S_i$ for $1 \le i \le n$. Then, from
  \eqref{eq:DT-coproduct-CY-matrix-elements},
  \[ \Braket{\sigma' | \Delta | \sigma} \neq 0 \text{ only if } \sigma' \preceq \sigma, \]
  i.e. $\Delta$ is upper-triangular with respect to $\preceq$.

  Second, we may assume without loss of generality that $\Delta$ is
  valued in the {\it polynomial ring} in the variables $a\{k_i\}_{i
    \in S}$ instead of the tensor algebra, i.e. that these variables
  commute. This is because, in \eqref{eq:DT-PT-CY-with-coproduct},
  they ultimately form the input to $\sz^{\boxtimes}$, an element
  invariant under permutation of its tensor factors.

  Finally, to determine the evaluation of $\sz^{\boxtimes}$ on matrix
  elements of $e^{-\Delta}$, it suffices to find formulas expressing
  the evaluations of $\sz^{\boxtimes}$ on all but one entry of
  $e^{-\Delta} \ket{\sigma}$ in terms of the evaluations of
  $\sz^{\boxtimes}$ of entries of $e^{-\Delta}\ket{\sigma'}$ for
  $\sigma' \prec \sigma$. The evaluation of $\sz^{\boxtimes}$ on the
  unknown entry is then given by the relation
  \eqref{eq:DT-PT-0-CY-with-coproduct}, and by induction to compute
  the smaller matrix elements. Specifically, our unknowns will be
  \begin{equation} \label{eq:DT-PT-CY-Y-coefficients}
    \sY\{k_1,\ldots,k_N\} \coloneqq \sz^{\boxtimes}\Braket{\sigma\{K_{\underline{N}}\} | e^{-\Delta} | \sigma\{k_1,\ldots,k_N\}}.
  \end{equation}
  
\subsubsection{}

  \begin{lemma} \label{lem:DT-PT-CY-non-extremal-matrix-element}
    Let $S_1 \sqcup \cdots \sqcup S_n = \underline{N}$. Then
    \[ \sz^{\boxtimes}\Braket{\sigma\{K_{S_1},\ldots,K_{S_n}\} | e^{-\Delta} | \sigma\{k_1,\ldots,k_N\}} = \sY\{\} \prod_{i=1}^n \frac{\sY\{k_j\}_{j \in S_i}}{\sY\{\}}. \]
  \end{lemma}

  \begin{proof}
    For the purpose of computing this matrix element, we may restrict
    $\Delta = \Delta_{\underline{N}}$ to the $\bh_\sT$-submodule of
    $V_{\underline{N}}$ generated by
    $\sigma\{K_{S_1'},\ldots,K_{S_m'}\}$ for all $S_1' \sqcup \cdots
    \sqcup S_m' \succeq S_1 \sqcup \cdots \sqcup S_n$. Then
    \[ \Delta_{\underline{N}} = -(n-1)a\{\} \cdot \id + \Delta_{S_1} + \cdots + \Delta_{S_n} \]
    where $\Delta_{S_i}$ denotes the natural lift of $\Delta_{S_i}
    \in \End(V_{S_i})$ to $V_{\underline{N}}$. Clearly $[\Delta_{S_i},
      \Delta_{S_j}] = 0$ for all $i$ and $j$. Then
    \begin{align*}
      &e^{(n-1)a\{\}} \Braket{\sigma\{K_{S_1},\ldots,K_{S_n}\} | e^{\Delta_{\underline{N}}} | \sigma\{k_1,\ldots,k_N\}} \\
      &= \Braket{\sigma\{K_{S_1},\ldots,K_{S_n}\} | e^{\Delta_{S_1}}\cdots e^{\Delta_{S_n}} | \sigma\{k_1,\ldots,k_N\}} \\
      &= \Braket{\sigma\{K_{S_1},\ldots,K_{S_n}\} | e^{\Delta_{S_1}} | \sigma\{\{k_j\}_{j \in S_1}, K_{S_2},\ldots, K_{S_n}\}} \\
      &\quad \cdot \Braket{\sigma\{\{k_j\}_{j \in S_1}, K_{S_2},\ldots, K_{S_n}\} | e^{\Delta_{S_2}} | \sigma\{\{k_j\}_{j \in S_1}, \{k_j\}_{j \in S_2}, K_{S_3},\ldots,K_{S_n}\}} \\
      &\quad \cdot \cdots \\
      &\quad \cdot \Braket{\sigma\{\{k_j\}_{j \in S_1},\ldots,\{k_j\}_{j \in S_{n-1}}, K_{S_n}\} | e^{\Delta_{S_n}} | \sigma\{k_1,\ldots,k_N\}},
    \end{align*}
    where the last equality follows because all other possible
    intermediate insertions $\ket{\sigma'}\bra{\sigma'}$ produce zero.
    We conclude by the definition \eqref{eq:DT-PT-CY-Y-coefficients}
    of $\sY$, noting that $\sY\{\} = e^{a\{\}}$.
  \end{proof}




\subsubsection{}

  Combining the definition \eqref{eq:DT-PT-CY-Y-coefficients} of $\sY$
  with Lemma~\ref{lem:DT-PT-CY-non-extremal-matrix-element}, the
  relation \eqref{eq:DT-PT-0-CY-with-coproduct} yields the recursion
  \begin{equation} \label{eq:DT-PT-Y-recursion}
    \sY\{k_1, \ldots, k_N\} \coloneqq \sV^{\DT}_{\emptyset,\emptyset,\emptyset}(\sigma\{k_1, \ldots, k_N\}) - \sum_{\substack{n>1\\S_1 \sqcup \cdots \sqcup S_n = \underline{N}}} \frac{\prod_{i=1}^n \sY\{k_j\}_{j \in S_i}}{\sY\{\}^{n-1}},
  \end{equation}
  so that $\sY\{\} = \sV^{\DT}_{\emptyset,\emptyset,\emptyset}(1)$ in
  particular, and the vertex correspondence
  \eqref{eq:DT-PT-CY-with-coproduct} becomes
  \begin{equation} \label{eq:DT-PT-descendent-transformation-1}
    \sV_{\lambda,\mu,\nu}^{\DT}\left(\sigma\{k_1, \ldots, k_N\}\right) = \sum_{\substack{n>0\\ S_1 \sqcup \cdots \sqcup S_n = \underline{N}}} \sV_{\lambda,\mu,\nu}^{\PT}\bigg(\sigma\big\{\sum_{j \in S_i} k_j\big\}_{i=1}^n\bigg) \frac{\prod_{i=1}^n \sY\{k_j\}_{j \in S_i}}{\sY\{\}^{n-1}}.
  \end{equation}
  The recursion \eqref{eq:DT-PT-Y-recursion} can be explicitly solved
  as follows.

\subsubsection{}

  \begin{lemma}
    \eqref{eq:DT-PT-Y-recursion} is equivalent to
    \begin{equation} \label{eq:DT-PT-Y-explicit}
      \sY\{k_1, \ldots, k_N\} \coloneqq \sum_{\substack{n>0\\S_1 \sqcup \cdots \sqcup S_n = \underline{N}}} (-1)^{n-1} (n-1)! \frac{\prod_{i=1}^n \sV^{\DT}_{\emptyset,\emptyset,\emptyset}(\sigma(\{k_j\}_{j \in S_i}))}{\sY\{\}^{n-1}}. 
    \end{equation}
  \end{lemma}

  \begin{proof}
    Clearly \eqref{eq:DT-PT-Y-recursion} has the same form as
    \eqref{eq:DT-PT-Y-explicit}, say with unknown coefficient
    $c(S_1,\ldots,S_n)$ for the term corresponding to $S_1 \sqcup
    \cdots \sqcup S_n = \underline{N}$, and we must show
    \[ c(S_1,\ldots,S_n) \stackrel{?}{=} (-1)^{n-1} (n-1)!. \]
    We do this by induction on $n$. The base case $n=1$ is clear from
    the first term of the right hand side of
    \eqref{eq:DT-PT-Y-recursion}. For general $n>1$, plug
    \eqref{eq:DT-PT-Y-explicit} into the right hand side of
    \eqref{eq:DT-PT-Y-recursion}. Then
    \[ c(S_1,\ldots,S_n) = \sum_{\substack{m>1\\T_1 \sqcup \cdots \sqcup T_m = \underline{n}}} \prod_{j=1}^m c(\{S_i\}_{i \in T_j}). \]
    To conclude the induction, it therefore suffices to prove that the
    polynomial
    \[ \mu_n(\kappa_1, \kappa_2, \ldots) \coloneqq \sum_{\substack{m>0\\T_1 \sqcup \cdots \sqcup T_m = \underline{n}}} \prod_{i=1}^m \kappa_{|T_i|} \]
    vanishes when $\kappa_j = (-1)^{j-1} (j-1)!$. This follows from the
    equality of generating series
    \[ 1 + \sum_{n > 0} \frac{\mu_n t^n}{n!} = \exp \sum_{n>0} \frac{\kappa_n t^n}{n!} \]
    and the relation $\exp \log (1 + t) = 1 + t$.
  \end{proof}
  
  Plugging \eqref{eq:DT-PT-Y-explicit} into
  \eqref{eq:DT-PT-descendent-transformation-1} concludes the proof of
  Theorem~\ref{thm:DT-PT-descendents}.
\end{proof}

\subsubsection{}

\begin{remark} \label{rem:DT-PT-descendents-basis}
  We make some observations on equivalent forms of
  Theorem~\ref{thm:DT-PT-descendents}.

  First, the change of basis from $\sigma\{k_1, \ldots, k_N\}$ to
  monomials in $\{\tau^{(k)}(\spt)\}_{k \in \bZ}$ is upper-triangular
  (with respect to degree in the symbols $\tau^{(k)}(\spt)$) and
  invertible over $\bZ$. So \eqref{eq:DT-PT-descendent-transformation-1}
  is equivalently a formula for the transformation of the DT
  descendent $\prod_{i=1}^n \tau^{(k_i)}(\spt)$.

  Second, since \eqref{eq:DT-PT-descendent-transformation-1} holds for
  all $k_1, \ldots, k_N \in \bZ$, we may treat it as an equality of
  formal power series in the variables $k_1, \ldots k_N$ and extract
  coefficients. Since the coefficient of $k_1^{n_1} \cdots k_N^{n_N}$
  in $\prod_{i=1}^N \tau^{(k_i)}(\spt)$ is $\prod_{i=1}^N
  \tau_{n_i}(\spt)$, the result is a descendent vertex correspondence
  for the insertion $\prod_{i=1}^N \tau_{n_i}(\spt)$ on the DT side.

  Finally, for vertices, note that $\xi = \spt$ is the only
  interesting class. Namely, any $\xi \in \CH_*^\sT(X)_{\loc}$ may be
  written in the basis of (classes of the) $\sT$-fixed points in $X$,
  and the universal sheaf $\scF$ restricted to any fixed point other
  than $0 \in \bC^3$ is constant over the moduli space.
\end{remark}

\subsubsection{}

\begin{remark}[Partition functions] \label{rem:DT-PT-partition-functions}
  Let $X$ be a smooth toric $3$-fold and $\sT$ be its dense open
  torus. Let $\Delta(X)$ be its toric $1$-skeleton, whose vertices $v$
  are toric charts $\bC^3$ with toric coordinates $\vec t_v$, and
  whose (half-)edges $e$ are the non-empty double intersections
  $\bC^\times \times \bC^2$ with toric coordinates $\vec t_e$. The
  (equivariant, cohomological) {\it operational DT partition function}
  of $X$ is
  \[ \sZ^{\DT}_X \coloneqq \sum_{\beta \in H_2(X; \bZ)} \sum_{n \in \bZ} Q^n A^\beta \int_{[\DT^{\sst}_{\beta,n}(X)]^\vir} (-). \]
  Here $\DT^{\sst}_{\beta,n}(X)$ is the moduli space of DT-stable
  pairs (i.e. ideal sheaves) $[\cO_X \to \cF]$ of Chern character
  $(1,0,-\beta,-n)$, and the variable $A$ records the multi-index
  $\beta$. Analogously, define the operational PT partition function
  $\sZ^{\PT}_X$ of $X$.

  Our explicit DT/PT descendent correspondence
  (Theorem~\ref{thm:DT-PT-descendents}) for vertices yields an
  explicit DT/PT descendent correspondence for partition functions, as
  follows. First, by writing classes $\xi \in \CH_*^\sT(X)_{\loc}$ in
  the basis of $\sT$-fixed points on $X$, assume without loss of
  generality that all descendents are of the form $\ff = \prod_v
  \ff_v$ where each $\ff_v$ is a polynomial of only descendents of the
  $\sT$-fixed point corresponding to $v$. Then $\sZ^{\DT}_X(\ff)$
  admits a standard factorization (see e.g. \cite[\S 4]{Maulik2006})
  into contributions $\sV$ and $E$ (which record the contribution to
  $A^\beta$) from vertices and edges of $\Delta(X)$ respectively:
  \begin{equation} \label{eq:DT-vertex-edge-factorization}
    \sZ^{\DT}_X(\ff) = \sum_{\vec\lambda} \prod_e E_{\vec\lambda(e)} \prod_v \sV_{\vec\lambda(e_1),\vec\lambda(e_2),\vec\lambda(e_3)}^{\DT}(\ff_v)
  \end{equation}
  where $e_1, e_2, e_3$ denotes the three incident (half-)edges at
  each vertex $v$, and the sum is over all assignments $\vec\lambda$
  of an integer partition to each (half-)edge in $\Delta(X)$. Note
  that the vertex at $v$ (resp. edge at $e$) uses the coordinates
  $\vec t_v$ (resp. $\vec t_e$); we omit this from the notation to
  avoid clutter. The same factorization holds for $\sZ^{\PT}_X(\ff)$.
  Now apply Theorem~\ref{thm:DT-PT-descendents} to each DT vertex on
  the right hand side of \eqref{eq:DT-vertex-edge-factorization}.
  Write the descendent correspondence
  \eqref{eq:DT-PT-descendent-transformation} at the vertex $v$ in the
  form
  \[ \sV^{\DT}(\ff_v) \equiv \sum_{i\in I(v)} \sV^{\PT}(\ff_{i}) C_i \bmod{\hbar}, \]
  for some coefficients $C_i$ and some index sets $I(v)$ implicitly
  dependent on $\ff_v$. Plugging this into
  \eqref{eq:DT-vertex-edge-factorization} and pulling out the sum over
  $I(v)$ for each vertex $v$, we get
  \[ \sZ^{\DT}_X(\ff) \equiv \sum_{(i_v\in I(v))_{v \in \Delta(X)}} \sum_{\vec\lambda} \bigg(\prod_e E_{\vec\lambda(e)}\bigg) \bigg(\prod_v \sV^{\PT}_{\vec\lambda(e_1),\vec\lambda(e_2),\vec\lambda(e_3)}(\ff_{i_v}) C_{i_v}\bigg) \bmod{\hbar}. \]
  Since each $C_{i_v}$ is independent of $\vec\lambda$, we can collect
  factors to find PT invariants of $X$ in this expression
  \begin{equation*}
    \sZ^{\DT}_X(\ff) \equiv \sum_{(i_v\in I(v))_{v \in \Delta(X)}} \sV^{\PT}_X\left({\textstyle\prod_{v}}\ff_{i_v}\right) \bigg(\prod_v C_{i_v}\bigg) \bmod{\hbar}.
  \end{equation*}
  Moreover, since \eqref{eq:DT-vertex-edge-factorization} takes the
  form $\sZ^{\DT_0}_X(\ff) = \prod_v
  \sV_{\emptyset,\emptyset,\emptyset}^{\DT}(\ff_v)$ for the degree-$0$
  (i.e. $\beta = 0$) partition function, $\prod_v C_{i_v}$ can be
  written as a polynomial in $\sZ^{\DT_0}_X(1)^{-1}$ and
  $\sZ^{\DT_0}_X(\tilde{\ff})$, for various descendents $\tilde{\ff}$,
  $\sZ^{\DT_0}_X(1)^{-1}$, and also possibly extra factors of
  $\sV_{\emptyset,\emptyset,\emptyset}^{\DT}(1)^{-1}$. These extra
  factors are necessary because the $C_{i_v}$, as $v$ varies, may not
  have the same degree in $\sV^{\DT}_{\emptyset,\emptyset,\emptyset}$,
  but each $\sZ^{\DT_0}_X$ is a product of exactly one $\DT_0$ vertex
  at each vertex $v$. Hence, we have obtained an explicit DT/PT
  descendent correspondence for partition functions of $X$.
\end{remark}

\section{Example: refined Vafa--Witten theory}
\label{sec:VW}

\subsection{The geometric setup}
\label{sec:VW-setup}

\subsubsection{}

Let $S$ be a smooth projective surface, and let $X \coloneqq
\tot(\cK_S)$ denote the total space of its canonical bundle, with
projection $\pi\colon X \to S$. Then $X$ is a smooth quasi-projective
Calabi--Yau threefold. Let $\sT \coloneqq \bC^\times$, with coordinate
denoted $\kappa$, act on $X$ by scaling the fibers of $\pi$ with
weight one. For example, $\cK_X \cong \kappa^{-1} \otimes \cO_X$.

Let $\cat{Coh}_c(X)$ denote the Noetherian and $\bC$-linear abelian
category of compactly-supported coherent sheaves on $X$. Let $\fM =
\fM_X$ denote the moduli stack parameterizing objects in
$\cat{Coh}_c(X)$ with the natural $\sT$-action inherited from $X$.

Since $\sT$ acts trivially on the surface $S$, our convention is that
all sheaves written on $S$ have $\sT$-weight zero.

\subsubsection{}

By the spectral construction \cite[\S 2]{Tanaka2020}, there is an
equivalence of categories
\[ \cat{Coh}_c(X) \cong \cat{Higgs}(S) \]
with the category of {\it Higgs sheaves} on $S$: a pair $(\bar\cE,
\phi)$ where $\bar\cE \in \cat{Coh}(S)$ is a coherent sheaf and
$\phi\colon \bar\cE \to \bar\cE \otimes \cK_S$ is a
morphism.\footnote{Passing to appropriate derived enhancements
  \cite{Pantev2013}, this identifies $\fM_X = T^*[-1]\fM_S$ as the
  $(-1)$-shifted cotangent bundle of the derived moduli stack of
  coherent sheaves on $S$ (with the natural quasi-smooth structure on
  $\fM_S$). This perspective is useful psychologically but we will
  avoid any actual use of derived algebraic geometry.} Explicitly,
$\cE \in \cat{Coh}_c(X)$ is identified with $(\pi_*\cE, \phi)$ where
$\phi$ is the operator of multiplication by the fiber coordinate of
$\pi$. We take the Chern character
\[ \alpha = (r, c_1, \ch_2) \coloneqq \ch(\bar\cE) \in H^{\text{even}}(S; \bQ) \]
to be the topological class of an object $\cE \in \cat{Coh}_c(X)$ and
consider only the sub-monoid $C_+(S) \subset H^{\text{even}}(S; \bQ)$
of classes with $r > 0$ unless stated otherwise.

\subsubsection{}
\label{vw:sec:M-stack}

As a moduli stack of sheaves on an equivariant Calabi--Yau threefold,
$\fM$ carries the $\kappa$-symmetric $\sT$-equivariant obstruction
theory given at the point $[\cE] \in \fM$ by the shifted dual of
$R\Hom_X(\cE, \cE)$ \cite[Theorem B]{Ricolfi2021}. The exact triangle
\begin{equation} \label{eq:ex-tri-VW}
  R\Hom_X(\cE, \cE) \to R\Hom_S(\bar \cE, \bar \cE) \xrightarrow{\circ \phi - \phi\circ} \kappa^{-1} R\Hom_S(\bar \cE, \bar \cE \otimes \cK_S) \xrightarrow{+1} 
\end{equation}
relates the obstruction theories of $\cE$ and $(\bar \cE, \phi)$
\cite[Prop. 2.14]{Tanaka2020}. Note that the Higgs field $\phi$
carries the weight $\kappa^{-1}$. Applying
Lemma~\ref{lem:obstruction-theory-pl} removes $H^0(\cO_S)$ and its
dual $\kappa^{-1} H^2(\cK_S)$ from the second and third terms of
\eqref{eq:ex-tri-VW}, and produces a symmetrically-compatible
symmetric obstruction theory for $\fM^\pl$. The resulting
K-homological invariants
\[ \sVW^{U}_\alpha(H) \coloneqq \chi\left(\fM^{\sst}_{\alpha}(\tau^H), \hat\cO^{\vir}_{\fM_{\alpha}^{\sst}(\tau^H)} \otimes -\right) \]
for stable=semistable loci of $\fM_{r,c_1,\ch_2}^\pl$ are called {\it
  $U(r)$ refined VW invariants}, and may be viewed as (local)
K-theoretic DT invariants of $X$. However, as in \cite[Remark
  4.5]{Tanaka2020} and \cite{liu_ss_vw}:
\begin{itemize}
\item if $h^2(\cO_S) \neq 0$, then $\sVW^{U}_\alpha(H) = 0$ because
  the trivial summand $H^2(\cO_S)$ in $R\Hom_S(\bar \cE, \bar \cE)$
  yields a non-trivial cosection (as in
  \S\ref{rc:sec:reduced-virtual-classes}) that forces any virtual
  cycle to be zero;

\item if $h^1(\cO_S) \neq 0$, then there is a non-trivial
  $\Pic_0(S)$-action on $\fM_\alpha$ given by
  \[ \cL \cdot (\bar\cE, \phi) \mapsto (\cL \otimes \bar\cE, \phi), \]
  and then, by Lemma~\ref{lem:VW-H1-vanishing} below,
  $\sVW^{U}_\alpha(H)(\cE) = 0$ for any $\cE \in
  K_{\sT}^\circ(\fM_\alpha^\pl)$ admitting a $\Pic_0(S)$-equivariant
  structure.\footnote{This is a non-trivial condition, i.e. the
    canonical map $K_{\Pic_0(S)}(-) \to K(-)$ which forgets the
    equivariant structure is typically not surjective. In other words,
    it could be that $\sVW^{U}_\alpha(H) \neq 0$ even if $h^1(\cO_S) \neq
    0$.}
\end{itemize}

\subsubsection{}

\begin{lemma} \label{lem:VW-H1-vanishing}
  Let $\sP \coloneqq \Pic_0(S)$. Suppose $M$ is a proper algebraic
  space with $(\sT \times \sP)$-action and a $(\sT \times
  \sP)$-equivariant map
  \[ \det\colon M \to r\!\Pic_{c_1}(S), \quad r \in \bZ_{> 0}, \]
  where $\sP$ acts on $r\!\Pic_{c_1}(S) \coloneqq \{\cL^{\otimes r} :
  \cL \in \Pic(S), \, c_1(\cL) = c_1\}$ by $\cL_0 \cdot \cL^{\otimes
    r} \coloneqq (\cL_0 \otimes \cL)^{\otimes r}$, and $\sT$ acts
  trivially. If $\cE \in \cat{Perf}_{\sT}(M)$ admits a $(\sT \times
  \sP)$-equivariant structure, then
  \[ \chi(M, \cE) = 0. \]
  Here $\chi$ denotes the $\sT$-equivariant Euler characteristic.
\end{lemma}

In particular, if $M$ carries a $\sT$-equivariant perfect obstruction
theory which is $\sP$-invariant, like \eqref{eq:ex-tri-VW}, then its
virtual cycle $\cO_M^\vir$ has a natural $(\sT \times
\sP)$-equivariant structure.

\begin{proof}
  Since $M$ is proper, so is $\det$, and we have
  \[ \chi(M, \cE) = \chi(r\!\Pic_{c_1}(S), \det\nolimits_*\cE). \]
  We claim $\det_*\cE = w \cdot \cO_{r\!\Pic_{c_1}(S)} \in
  K_{\sT}(r\!\Pic_{c_1}(S))$ for some $w \in K_\sT(\pt)$. Since $\cE$
  admits a $(\sT \times \sP)$-equivariant structure and $\det$ is
  $(\sT \times \sP)$-equivariant, $\det_*\cE \in K_{\sT \times
    \sP}(r\!\Pic_{c_1}(S))$. Since $\sP$ acts freely on
  $r\!\Pic_{c_1}(S)$,
  \[ K_{\sT \times \sP}(r\!\Pic_{c_1}(S)) = K_{\sT}(r\!\Pic_{c_1}(S) / \sP) = K_{\sT}(\pt) \]
  is generated by $\cO_{r\!\Pic_{c_1}(S)}$. Hence $\det_*\cE = w \cdot
  \cO_{r\!\Pic_{c_1}(S)} \in K_{\sT \times \sP}(r\!\Pic_{c_1}(S))$,
  and forgetting the $\sP$-equivariant structure yields the claim. We
  conclude because, non-equivariantly, $\chi(\cO_A) = 0$ on any
  abelian variety $A$.
\end{proof}

\subsubsection{}
\label{sec:VW-stack}

Let $\fPic(S)$ denote the Picard stack of $S$.\footnote{The trivial
$\bC^\times$-gerbe over the Picard {\it scheme} $\Pic(S)$.} There is a
natural map
\[ \det \times \tr\colon \fM \to [(\Pic(S) \times H^0(\cK_S)) / \bC^\times] \cong \fPic(S) \times H^0(\cK_S) \]
sending a pair $(\bar \cE, \phi)$ to $(\det \bar \cE, \tr \phi)$.
Here, $\bC^\times$ acts by diagonal scaling and $\sT$ acts with weight
$(0, 1)$. Let $\fN_{\alpha,\cL}$ denote the closed substack defined by
the fiber product
\[ \begin{tikzcd}
  \fN_{\alpha,\cL} \ar{d} \ar[hookrightarrow]{r} & \fM_{\alpha} \ar{d}{\det \times \tr} \\
  \{[\cL]\}\times \{0\} \ar[hookrightarrow]{r} & \fPic(S) \times H^0(\cK_S),
\end{tikzcd} \]
parameterizing Higgs sheaves $(\bar \cE, \phi)$ such that $\tr\phi =
0$ and $\det \bar \cE \cong \cL$.\footnote{To be pedantic, this is
{\it without} a specific choice of isomorphism $\det \bar \cE
\xrightarrow{\sim} \cL$. After rigidification, this choice of
isomorphism is fixed.} We also get a map of rigidified stacks
$\fN_{\alpha,\cL}^\pl \to \fM_\alpha^\pl$, over $\{[\cL]\} \times
\{0\} \subset \Pic(S) \times H^0(\cK_S)$. Later, as in
\S\ref{sec:intro-VW-setup}, it will be important to choose
$\{\cL(\alpha)\}_{\alpha \in C_+(S)} \subset \Pic(S)$ such that
$\cL(\alpha+\beta) = \cL(\alpha) \otimes \cL(\beta)$.


\subsubsection{}
\label{sec:VW-TT-obstruction-theory}

The main technical innovation of Tanaka--Thomas \cite[\S
  5]{Tanaka2020} is that the obstruction theory \eqref{eq:ex-tri-VW}
for $\fM_\alpha$ exists relative to $\fPic(S) \times H^0(\cK_S)$, and
its ``symmetrized restriction'' to $\fN_{\alpha,\cL}$ continues to be
$\sT$-equivariant and symmetric \cite[\S 5]{Tanaka2020}. In other
words, $\fN_{\alpha,\cL}$ (resp. $\fN_{\alpha,\cL}^\pl$) carries a
$\sT$-equivariant symmetric obstruction theory given by $R\Hom_X(\cE,
\cE)_{\lrcorner}$ (resp. $R\Hom_X(\cE, \cE)_{\perp}$) obtained from
the exact triangle \eqref{eq:ex-tri-VW} by removing the diagonal copy
of $\tau^{>0}R\Gamma(S,\cO_S)$ (resp. $R\Gamma(S,\cO_S)$) from the
middle term and the dual copy of $\tau^{<2}R\Gamma(S,\cK_S)$ (resp.
$R\Gamma(S,\cK_S)$) from the rightmost term. In particular there is an
exact triangle
\[ R\Hom_X(\cE, \cE)_\perp \to R\Hom_S(\bar\cE, \bar\cE)_0 \xrightarrow{\circ\phi - \phi\circ} R\Hom_S(\bar\cE, \bar\cE \otimes \cK_S)_0 \xrightarrow{+1} \]
where the subscript $0$ denotes traceless part. The resulting
(K-)homological invariants
\[ \sVW_\alpha(H) \coloneqq \chi\left(\fN^{\sst}_{\alpha,\cL}(\tau^H), \hat\cO^{\vir}_{\fN_{\alpha,\cL}^{\sst}(\tau^H)} \otimes -\right) \]
for stable=semistable loci of $\fN_{r,c_1,\ch_2,\cL}^\pl$ are called
{\it $\SU(r)$ VW invariants}. They are independent of the choice of
$\cL$ by deformation invariance along $\Pic(S)$ \cite[Remark
  6.4]{Tanaka2020}. These will be the stable=semistable VW invariants
of interest in this section.

\subsubsection{}

\begin{remark} \label{rem:VW-refinement}
  Since $\sT = \bC^\times$ has rank one, general principles imply that
  $\kappa$-symmetrized enumerative invariants like $\sVW_\alpha(H)$
  have {\it no poles at $\kappa = 1$}, i.e. if $\cE$ is a
  non-localized class then $\sVW_\alpha(H)(\cE)$ is a rational
  function in $\kappa$ which is well-defined upon specializing $\kappa
  \to 1$, and moreover this specialization recovers the unrefined VW
  invariant
  \[ \sVW_\alpha(H)(\cE)\big|_{\kappa=1} = \int_{[\fN^{\sst}_{\alpha,\cL}(\tau^H)]^\vir} \rank(\cE) \in \bQ, \]
  see \cite[Prop. 2.22]{Thomas2020}. In this sense, refined VW
  invariants are a $\kappa$-refinement of unrefined VW invariants.

  Motivic invariants --- meaning, Euler characteristics weighted by
  Behrend function --- may also be refined by replacing Euler
  characteristic with the Hirzebruch $\chi_{-\kappa}$-genus. In some
  settings, for instance when $\deg \cK_S < 0$ (in particular,
  $h^1(\cO_S) = h^2(\cO_S) = 0$) \cite[Theorem 5.15]{Thomas2020}, such
  refined motivic invariants of $\fN_{\alpha,\cL}^{\sst}(\tau^H)$
  agree with our refined VW invariants $\sVW_\alpha(H)(\cO)$, and,
  more generally, refined DT invariants in the sense of Joyce--Song
  (see also \S\ref{sec:comparison-with-joyce-song}) agree with our
  semistable refined VW invariants $\svw_\alpha(H)(\cO)$. But in
  general they disagree.
\end{remark}

\subsubsection{}

\begin{definition} \label{def:VW-stability}
  Let $\NS(S)$ be the N\'eron--Severi group of $S$, and $\NS_{\bR}(S)
  \coloneqq \NS(S) \otimes_{\bZ} \bR$. Recall that an element $H \in
  \NS_{\bR}(S)$ is {\it ample} if $H = \sum_i a_i \cL_i$ for $\cL_i
  \in \NS(X)$ ample and $a_i \in \bR_{>0}$. Given an ample $H \in
  \NS_{\bR}(S)$, viewed as a $(1,1)$-class, let
  \[ P_\alpha^H(n) \coloneqq \int_X \alpha \cup e^{nH} \cup \td(X) \]
  be the Hilbert polynomial of $\alpha$ (a Chern character of some
  coherent sheaf) with respect to $H$. This is a real polynomial in
  $n$ of degree $\dim \alpha$. If $H$ is a polarization of $S$, then,
  abusing notation slightly, $P_\alpha^H(n) = \chi(\alpha \otimes H^n)
  \in \bZ$. In particular $P_\alpha^H(n) = \dim H^0(\alpha \otimes
  H^n) > 0$ for $n \gg 0$, so the leading coefficient $r_\alpha^H$ is
  positive. By continuity this is therefore also true for general $H$.
  The normalized Hilbert polynomial
  \[ \tau^H(\alpha) \coloneqq \frac{P_\alpha^H}{r_\alpha^H} = n^d + a_{d-1} n^{d-1} + \cdots + a_0 \]
  is a monic polynomial of degree $d = \dim \alpha$. Following
  \cite[Def. 7.7]{Joyce2021}, put a total order $\le$ on monic
  polynomials: $f \le g$ if either $\deg f > \deg g$, or $\deg f =
  \deg g$ and $f(n) \le g(n)$ for all $n \gg 0$. Then $\tau^H$ is a
  stability condition on $\cat{Coh}(S)$ \cite[Lem. 2.5]{Rudakov1997}
  called {\it $H$-Gieseker stability}. Similarly define
  \[ \mu^H(\alpha) \coloneqq \frac{\bar P_\alpha^H}{r_\alpha^H} \coloneqq n^d + a_{d-1} n^{d-1} \]
  where $\bar P_\alpha^H$ is the truncation of $P_\alpha^H$ to the
  first two leading terms. This is a weak stability condition on
  $\cat{Coh}(S)$ \cite[Def. 7.8]{Joyce2021} called {\it $H$-slope
    stability}.

  We continue to use $\tau^H$ and $\mu^H$ to denote the same (weak)
  stability conditions on $\cat{Higgs}(S)$ instead of $\cat{Coh}(S)$.
  Note that a sub-object of a Higgs sheaf $(\bar\cE, \phi)$ is a {\it
    $\phi$-invariant} sub-sheaf $\cF \subset \bar\cE$.
\end{definition}

\subsubsection{}

\begin{remark}
  This definition of Gieseker and slope stability is more general than
  the usual definition, e.g. in \cite[\S 1.2]{huybrechts_lehn_geom},
  which only considers pure sheaves of positive rank. The extension to
  arbitrary sheaves is what makes slope stability into a {\it weak}
  stability condition. Gieseker stability remains a stability
  condition. It is important that slope stability is also a stability
  condition when considering only short exact sequences of sheaves of
  positive rank. 

  In more generality, the ample class $H$ may taken inside the
  K\"ahler cone inside $H^{1,1}(S; \bR)$, rather than inside the ample
  cone inside $H^{1,1}(S) \otimes_{\bZ} \bR \subset H^{1,1}(S; \bR)$.
  We will not use this.
\end{remark}

\subsubsection{}
\label{sec:VW-auxiliary-stack}

Applying the discussion of \S\ref{es:sec:auxiliary-stacks}, for any
quiver $Q$ and framing functors $\vec\Fr$, we obtain the (rigidified)
auxiliary stacks
\[ \pi_{\fN_{\alpha,\cL}^{\vec\Fr,\pl}}\colon \fN_{(\alpha,\cL),\vec d}^{Q(\vec\Fr),\pl} \to \fN_{\alpha,\cL}^{\vec\Fr,\pl}. \]
If $H$ is an ample line bundle on $S$ then, for $k \in \bZ_{> 0}$,
\[ \Fr_{H,k}(\cE) \coloneqq H^0(\cE \otimes H^k) \]
is a framing functor on the open locus of sheaves $\cE$ which are
$k$-regular \cite[\S 1.7]{huybrechts_lehn_geom}, i.e. $H^i(\cE \otimes
H^{k-i}) = 0$ for all $i > 0$. In particular, for the pairs quiver $Q$
of Definition~\ref{def:pair-invariant} and $k \gg 0$ (depending on
$\alpha$), we let
\begin{equation} \label{eq:VW-pair-invariant}
  \tilde{\sVW}_\alpha(H; k) \coloneqq \chi\left(\tilde\fN_{(\alpha,\cL),1}^{Q(\Fr_{H,k}),\sst}((\tau^H)^Q), \hat\cO^\vir \otimes -\right) \in K_\circ^{\tilde\sT}(\tilde\fN_{(\alpha,\cL),1}^{Q(\Fr_{H,k}),\pl})_{\loc}
\end{equation}
denote the universal enumerative invariant associated to the auxiliary
stability condition $(\tau^H)^Q$, where $\tau^H$ is $H$-Gieseker
stability, and the framing functor $\Fr_{H,k}$. Here, regarding
$\tilde\fN_{(\alpha,\cL),1}^{Q(\Fr_{H,k}),\sst}((\tau^H)^Q)$:
\begin{itemize}
\item its the virtual cycle arises from the $\kappa$-symmetric APOT
  obtained by symmetrized pullback (Theorem~\ref{thm:APOTs}) along the
  smooth forgetful map $\pi_{\fN_{\alpha,\cL}^{\vec\Fr,\pl}}$;
\item properness of its $\sT$-fixed loci follows from the verification
  of
  Assumptions~\ref{es:assump:semistable-invariants}\ref{es:assump:it:properness}
  and \ref{es:assump:wall-crossing}\ref{es:assump:it:properness-wcf}
  in \S\ref{sec:VW-WCF-proof-end} below.
\end{itemize}

\subsubsection{}

\begin{remark}
  For the universal enumerative invariant of the auxiliary pairs
  stack, it is not actually necessary to use the same $H$ for both the
  framing functor $\Fr_{H,k}$ and the stability condition $\tau^H$. We
  did so for $\tilde{\sVW}_\alpha(H; k)$ in
  \eqref{eq:VW-pair-invariant} in order for
  Corollary~\ref{cor:numerical-VW}\ref{it:numerical-VW-sst-invariants}
  to match, on the nose, with previous work \cite{Tanaka2017,
    liu_ss_vw}. In general, we may take any ample line bundle
  $\cO_X(1)$ and consider $\Fr_{\cO(1),k}$, and take any $H \in
  \NS_{\bR}(S)$ and consider $\tau^H$. The resulting auxiliary
  universal enumerative invariant can equally well be used in place of
  $\tilde{\sVW}_\alpha(H; k)$ in Theorem~\ref{thm:VW-sst-invariants},
  with no change to the right hand sides of
  \eqref{eq:VW-sst-invariants-pg-zero} or
  \eqref{eq:VW-sst-invariants-pg-positive}.
\end{remark}

\subsection{Positive-rank invariants and wall-crossing}
\label{sec:VW-WCF}

\subsubsection{}

In this subsection, we prove Theorems~\ref{thm:VW-sst-invariants} and
\ref{thm:VW-WCF}, on the existence and wall-crossing of refined
semistable VW invariants of classes $\alpha \in C_+(S)$, by
specializing the main theorems to the VW setting of
\S\ref{sec:VW-setup} and verifying all the relevant assumptions. There
will be the following cases.
\begin{itemize}
\item If $h^1(\cO_S) = h^2(\cO_S) = 0$, then $\Pic_0(S) = 0 =
  H^0(\cK_S)$ and therefore $\fN_{\alpha,\cL} = \fM_\alpha$. Then the
  ``plain'' Theorems~\ref{es:thm:sst-invariants} and \ref{es:thm:wcf}
  can be applied directly to the set $C_+(S)$ of permissible classes;
  see \S\ref{sec:VW-proof-H1-H2-zero}.
\item If $h^1(\cO_S) \neq 0$ or $h^2(\cO_S) \neq 0$, then
  $\fN_{\alpha,\cL} \subsetneq \fM_\alpha$, and in fact these do not
  satisfy
  Assumption~\ref{es:assump:semistable-invariants}\ref{es:assump:it:Bpe-sst-summands}!
  Nonetheless, we can pretend that our wall-crossing framework is
  applicable and obtain non-trivial results by a more careful analysis
  of the master space arguments used in
  \S\ref{sec:semistable-invariants} and \S\ref{sec:wall-crossing}; see
  \S\ref{sec:VW-proof-H1-H2-nonzero}--\S\ref{sec:VW-proof-end}.
\end{itemize}

\subsubsection{}

To prove Theorem~\ref{thm:VW-WCF} for a given class $\alpha \in
C_+(S)$, i.e. to cross between Gieseker stability for two
polarizations $H_1$ and $H_2$, following the discussion of
\S\ref{sec:general-wall-crossing} we first reduce to the dominant case
where we cross between a weak stability condition $\tau$ which weakly
dominates another weak stability condition $\mathring \tau$ at
$\alpha$. There are two main technical issues here. First, Gieseker
stability $\tau^H$ does not depend continuously on the class $H \in
\NS_{\bR}(S)$, e.g. it is possible to have a sub-object $0 \neq \cE'
\subsetneq \cE$ and a family $\{\tau^{H_t}\}_t$ such that $\cE'$ is
destabilizing for $t = t_0$ but not for any $t < t_0$. Second, we must
ensure that we only cross finitely many walls between $\tau^{H_1}$ and
$\tau^{H_2}$. Both problems are resolved by using the slope stability
$\mu^H$, which {\it does} depend continuously on $H$.

\subsubsection{}

\begin{lemma} \label{lem:VW-WCF-dominant-to-general}
  Suppose that Theorem~\ref{thm:VW-WCF} holds for all slope stability
  conditions $\mu$ and $\mathring\mu$ such that $\mu$ weakly dominates
  $\mathring\mu$ at $\alpha \in C_+(S)$. Then it holds for all
  Gieseker (and slope) stability conditions $\tau^{H_1}$ and
  $\tau^{H_2}$.
\end{lemma}

\begin{proof}
  This follows from the following three ingredients.
  \begin{enumerate}[label = (\roman*)]
  \item \label{vw:lem:dom-suff-i} The wall-crossing formula is
    transitive, as discussed in \S\ref{sec:general-wall-crossing},
    following \cite[\S 11]{Joyce2021}. So any finite sequence of
    dominant wall-crossings may be composed to get a general
    wall-crossing formula of exactly the same shape.
    
  \item \label{vw:lem:dom-suff-ii} For any ample $H \in \NS_{\bR}(S)$,
    the slope stability condition $\mu^H$ dominates the Gieseker
    stability condition $\tau^H$. This follows directly from
    definitions.

  \item \label{vw:lem:dom-suff-iii} Given ample $H_1, H_2 \in
    \NS_{\bR}(S)$, let $\mu_t$ be the family of slope stability
    conditions corresponding to the family $H_t = tH_1 + (1-t)H_0$ for
    $t\in [0,1]$. Then there exists finitely many values $0 = c_0 <
    c_1 < \cdots < c_N = 1$ so that the following is true (setting
    $c_{-1} \coloneqq 0$ and $c_{N+1}\coloneqq 1$):
    \begin{itemize}
    \item if $c_i < s,t < c_{i+1}$, then $\mu_s$ and $\mu_t$ weakly
      dominate each other at $\alpha$;
    \item if $c_{i-1} < t < c_{i+1}$, then $\mu_{c_i}$ weakly
      dominates $\mu_t$ at $\alpha$.
    \end{itemize}
    This is the content of
    Lemma~\ref{vw:lem:finite-dominant-wcs-at-alpha} below.
  \end{enumerate}
  To see how these imply the desired result, note that by
  \ref{vw:lem:dom-suff-iii}, for any given $\alpha \in C_+(S)$, the
  wall-crossing between different slope stability conditions breaks
  into a sequence of dominant wall-crossings: for each $c_i$, we may
  cross ``onto the wall'' and then ``off the wall''. By
  \ref{vw:lem:dom-suff-i}, this implies the wall-crossing formula
  between any two slope stability conditions. Combining
  \ref{vw:lem:dom-suff-i} and \ref{vw:lem:dom-suff-ii}, we conclude
  the result for Gieseker stability from the one for slope stability.
\end{proof}

\subsubsection{}

\begin{definition}[{\cite[Def. 2.1]{Yoshioka1996}}]
  Let $\alpha = (r, c_1, \ch_2) \in C_+(S)$ and let $c_2 \coloneqq
  c_1^2/2 - \ch_2$. Recall that the {\it discriminant} of $\alpha$ is
  \[ \Delta(\alpha) \coloneqq \frac{1}{r}\left(c_2 - \frac{r-1}{2r} c_1^2\right) \in H^4(S; \bQ). \]
  A hyperplane $W \subset \NS_{\bR}(S)$ is a {\it (Yoshioka) wall for
    $\alpha$} if there exists an effective decomposition $\alpha =
  \beta_1 + \beta_2$ such that:
  \begin{itemize}
  \item $\rank \beta_i > 0$ and $\Delta(\beta_i) \ge 0$ for $i = 1,2$;
  \item $W$ is the orthogonal complement of the class
    $c_1(\beta_1)/\rank(\beta_1) - c_1(\beta_2)/\rank(\beta_2)$.
  \end{itemize}
  This coincides with Yoshioka's definition.\footnote{More precisely,
  the object $W^F$ defined by Yoshioka is the union of Yoshioka walls,
  as defined here, over all classes $\beta_1$ of a fixed filtration
  $F$.}
\end{definition}

\subsubsection{}

\begin{lemma}\label{vw:lem:finite-dominant-wcs-at-alpha}
  \begin{enumerate}[label = (\roman*)]
  \item \label{vw:item:fdwcsi} Let $K \subset \NS_{\bR}(S)$ be a
    compact subset contained in the interior of the ample cone. Then
    $K$ intersects only a finite set of walls for $\alpha$.
  \item \label{vw:item:fdwcsii} Let $H, H' \in \NS_{\bR}(S)$ be ample
    with associated slope stabilities $\mu, \mu'$, and let $\ell$
    denote the line connecting $H$ and $H'$. If neither $\mu$ weakly
    dominates $\mu'$ at $\alpha$ nor $\mu'$ weakly dominates $\mu$ at
    $\alpha$, then there exists a wall $W$ for class $\alpha$
    which intersects $\ell$ in a single point.
  \item \label{vw:item:fdwcsiii} If there is no wall
    intersecting $\ell$ outside of the point $\{H\}$, then $\mu$
    weakly dominates $\mu'$ at $\alpha$.
  \end{enumerate}
\end{lemma}

\begin{proof}
  \ref{vw:item:fdwcsi} follows from the finiteness result \cite[Lemma
    2.2]{Yoshioka1996}.
  
  \ref{vw:item:fdwcsii} follows from \ref{vw:item:fdwcsiii} by
  splitting $\ell$ into segments at the (finitely many) proper
  intersection points with walls for class $\alpha$.

  Hence, we focus on \ref{vw:item:fdwcsiii}. By assumption, any wall
  for $\alpha$ either contains $\ell$ or intersects it only at $H \in
  \ell$. We may assume $H' \neq H$; otherwise there is nothing to
  show.

  Let $\beta \in \mathring R_{\alpha} \setminus \{\alpha\}$. Note that
  $\mu'(\beta) = \mu'(\alpha-\beta)$ implies $\rank \beta, \rank
  \alpha-\beta > 0$. According to Definition~\ref{def:dominates-at}
  for weak dominance at $\alpha$, we must show:
  \begin{enumerate}
  \item \label{vw:it:fdwcsiii-1} $\mu(\beta) = \mu(\alpha-\beta)$;
  \item \label{vw:it:fdwcsiii-2} $\fM_{\beta}^{\sst}(\mu') \subseteq
    \fM_{\beta}^{\sst}(\mu)$.
  \end{enumerate}
  Namely, if we show these for all $\beta \in \mathring R_\alpha
  \setminus \{\alpha\}$, then in particular \ref{vw:it:fdwcsiii-2}
  holds for $\alpha-\beta$ as well, thus $\mathring R_\alpha \subseteq
  R_\alpha$, so $\mu$ weakly dominates $\mu'$ at $\alpha$, as desired.

  To see \ref{vw:it:fdwcsiii-1}, use that there are $\mu'$-semistable
  Higgs sheaves $\cE_1, \cE_2$ whose underlying sheaves $\bar\cE_1,
  \bar\cE_2$ have classes $\beta$ and $\alpha-\beta$ respectively, by
  hypothesis. By the Bogomolov--Gieseker inequality for Higgs sheaves
  \cite[Prop. 3.4]{Simpson1988} \cite[Theorem 7]{Langer2015}, we have
  $\Delta(\beta), \Delta(\alpha-\beta) \ge 0$, so that the
  decomposition $\alpha = \beta + (\alpha-\beta)$ defines a wall for
  class $\alpha$. By definition, this wall contains the polarization
  $H'$, and thus by assumption must also contain $H$. This in turn
  means $\mu(\beta) = \mu(\alpha-\beta)$, as desired.

  To see \ref{vw:it:fdwcsiii-2}, suppose for the sake of contradiction
  that $\cE$ is a Higgs sheaf of class $\beta$ that is
  $\mu'$-semistable but not $\mu$-semistable. By the same argument as
  in the proof of \cite[Lemma 2.3]{Yoshioka1996} (which makes crucial
  use of the Bogomolov--Gieseker inequality), we can find some $t>0$
  such that $H_t$ lies on a wall not containing $H_0$, contradicting
  our assumption.
\end{proof}

\subsubsection{}
\label{sec:VW-proof-H1-H2-zero}

\begin{proof}[Proof of Theorems~\ref{thm:VW-sst-invariants} and \ref{thm:VW-WCF} when $h^1(\cO_S) = h^2(\cO_S) = 0$.]
  Let $H \in \NS_{\bR}(S)$, let $\tau \coloneqq \mu^H$, and suppose
  $\tau$ weakly dominates another weak stability condition
  $\mathring\tau$ at $\alpha \in C_+(S)$. Note that
  Lemma~\ref{lem:VW-WCF-dominant-to-general} reduces
  Theorem~\ref{thm:VW-WCF} to this setting: $\mathring\tau$ could be
  either $\mu^{\mathring H}$ for some $\mathring H \in \NS_{\bR}(S)$
  not on a wall, or it could be $\tau^H$. It remains to provide the
  ingredients and check the necessary
  Assumptions~\ref{es:assump:exact-subcategory},
  \ref{es:assump:semistable-invariants}, and
  \ref{es:assump:wall-crossing}, in order to apply the main
  Theorems~\ref{thm:sst-invariants} and \ref{thm:wcf}.

  Note that Yoshioka walls are orthogonal complements of rational
  cohomology classes, and therefore, without loss of generality, $H
  \in \NS_{\bQ}(S) \coloneqq \NS(S) \otimes_{\bZ} \bQ$ is a rational
  polarization.

\subsubsection{Assumption~\ref{es:assump:exact-subcategory}}

  (\ref{es:assump:exact-subcategory}\ref{es:assump:it:B-closed})
  Automatic since we are taking $\cat{B} \coloneqq \cat{A} \coloneqq
  \cat{Coh}_c(X)$.

  (\ref{es:assump:exact-subcategory}\ref{es:assump:it:permissible-classes})
  Let $C(\cat{A})_{\pe} \coloneqq C_+(S)$. This is a monoid, and the
  entire moduli stack $\fM_X$ of $\cat{Coh}_c(X)$ is Artin, locally of
  finite type, and forms a graded monoidal $\sT$-stack.

  (\ref{es:assump:exact-subcategory}\ref{es:assumption-restricted-substack})
  The closed substacks $\fN_{\alpha,\cL} \subseteq \fM_\alpha$,
  ranging over all classes $\alpha \in C_+(S)$, form a restricted
  graded monoidal $\sT$-stack with $\kappa$-symmetric bilinear
  elements given by (see \S\ref{sec:VW-TT-obstruction-theory})
  \[ \scE_{\alpha_1,\alpha_2} \coloneqq Rp_*R\cHom(\scE_1, \scE_2)_{\lrcorner} \in \cat{Perf}_\sT(\fN_{\alpha_1} \times \fM_{\alpha_2}) \]
  where $\scE_i$ denotes the universal family on the $i$-th factor
  times $X$ and $p$ is projection along $X$. This defines an
  obstruction theory on $\fN_{\alpha,\cL}$ and $\fN^\pl_{\alpha,\cL}$.

  When $h^1(\cO_S) = h^2(\cO_S) = 0$, note that $\cL$ is uniquely
  determined and $\fN_{\alpha,\cL} = \fM_\alpha$ for all $\alpha \in
  C_+(S)$.

\subsubsection{Assumptions~\ref{es:assump:semistable-invariants} and \ref{es:assump:wall-crossing}}
\label{sec:VW-WCF-proof-end}

  (\ref{es:assump:semistable-invariants}\ref{es:assump:it:tau-artinian},
  \ref{es:assump:semistable-invariants}\ref{es:assump:it:semistable-loci},
  and \ref{es:assump:wall-crossing}\ref{es:assump:it:tau-circ}) Since
  $\cat{Coh}_c(X)$ is $\tau^H$-Artinian and $\mu^H$-Artinian for any
  $H \in \NS_{\bR}(S)$, see e.g. \cite[Lemma 4.19, Theorem
    4.20]{Joyce2007}, any non-zero object admits $\tau^H$- and
  $\mu^H$-HN filtrations. Furthermore, $\tau^H$- and
  $\mu^H$-semistability are open conditions \cite[Prop.
    2.3.1]{huybrechts_lehn_geom}\footnote{The proof there works for
    arbitrary $H \in \NS_{\bR}(S)$ and also for slope stability.} and
  form bounded families in a given class $\alpha \in C_+(S)$
  \cite[Theorem 6.8]{Greb2016}. \footnote{The Bogomolov--Gieseker
    inequality for Higgs sheaves must be used here.}

  (\ref{es:assump:semistable-invariants}\ref{es:assump:it:framing-functor})
  Let $\cO_X(1)$ denote any very ample line bundle on $X$, take $k \gg
  0$, and define
  \[ \Fr_k\colon \cE \mapsto H^0(\cE \otimes \cO_X(k)). \]
  This is an exact functor on the open locus of pairs where $\cF$ is
  $k$-regular \cite[\S 1.7]{huybrechts_lehn_geom}, i.e. such that
  $H^i(\cF(k-i)) = 0$ for all $i > 0$. Basic properties of regularity
  imply that $\Fr_k$ is a framing functor.

  (\ref{es:assump:semistable-invariants}\ref{es:assump:it:rank-function})
  For the rank function, take
  \[ r(\alpha) \coloneqq \rank \alpha. \]
  By definition, $r(\alpha) > 0$ for any $\alpha \in
  C(\cat{A})_{\pe}$, and clearly $r(\alpha_1+\alpha_2) = r(\alpha_1) +
  r(\alpha_2)$ for any $\alpha_1,\alpha_2 \in C(\cat{A})$. If $A$ is
  $\mu^H$-semistable, if $0 \neq A' \subsetneq A$ is a rank zero
  sub-object then $\mu(A') > \mu(A)$, and similarly if $\rank A/A' =
  0$ then $\mu(A) < \mu(A/A')$. Hence the hypothesis $\mu^H(A') =
  \mu^H(A/A')$ implies $r(A'), r(A/A') > 0$, as desired.

  (\ref{es:assump:semistable-invariants}\ref{es:assump:it:inert-classes})
  Define the set of inert classes
  \[ C(\cat{A})_{\inert} \coloneqq \{(0, 0, \ch_2)\} \subset C(\cat{A}). \]
  Clearly $\mu^H$ of such (zero-dimensional) classes is strictly
  greater than $\mu^H$ of any $\alpha \in C(\cat{A})_{\pe}$. Moreover,
  given $\gamma = (0,0,n)$ and $\delta = (r,c_1,m) \in
  C(\cat{A})_{\pe}$, clearly $\gamma+\delta$ have the same rank $r>0$
  and therefore the same leading coefficient $r^H$ in their Hilbert
  polynomials, thus
  \[ \mu^H(\gamma + \delta) = n^2 + \frac{H \cdot c_1}{r^H} n = \mu^H(\delta). \]

  (\ref{es:assump:semistable-invariants}\ref{es:assump:it:semi-weak-stability})
  Let $B' \subset B$ where $B$ is $\mu^H$-semistable. We prove the
  stronger claim that $\mu^H(B') < \mu^H(B) = \mu^H(B/B')$ never
  occurs, and if $\mu^H(B') = \mu^H(B) < \mu^H(B/B')$ then $B/B' \in
  C(\cat{A})_{\inert}$. First, it cannot be that $\rank(B') = 0$,
  because $B$ is $\mu^H$-semistable. Second, if $B', B, B/B'$ all have
  rank $>0$ then $\mu^H$ behaves like a (non-weak!) stability
  condition on them, so we can never have $\mu^H(B') < \mu^H(B) =
  \mu^H(B/B')$ or $\mu^H(B') = \mu^H(B) < \mu^H(B/B')$. Combining
  these two observations, we rule out the possibility $\mu^H(B') <
  \mu^H(B) = \mu^H(B/B')$. Hence $\rank B' = \rank B$, so that $\rank
  B/B' = 0$, is the only remaining possibility. Then $\mu^H(B') =
  \mu^H(B)$ implies $H \cdot c_1(B/B') = 0$. Since $H$ is ample,
  $c_1(B/B') = 0$ and thus $B/B' \in C(\cat{A})_{\inert}$.

  (\ref{es:assump:semistable-invariants}\ref{es:assump:it:properness}
  and \ref{es:assump:wall-crossing}\ref{es:assump:it:properness-wcf})
  Embed $X$ into the $\sT$-equivariant projective completion $\hat X
  \coloneqq \bP(\cK_S \oplus \cO_S)$. Since $\sT$-fixed semistable
  loci in auxiliary stacks for $X$ are components of $\sT$-fixed
  semistable loci in auxiliary stacks for $\hat X$, without loss of
  generality we can suppose $X$ is projective. Then the desired
  properness of various fixed loci may be proved via standard
  semistable reduction arguments; see \cite[\S
    4.3]{huybrechts_lehn_geom} \cite[Theorem 9.6, Remark
    9.7]{Greb2016} and also \cite[\S 4.3, Prop. 5.1.9, Prop.
    6.1.5]{klt_dtpt} for the specific case of our auxiliary framed
  stacks. 

  (\ref{es:assump:semistable-invariants}\ref{es:assump:it:Bpe-sst-summands}
  and \ref{es:assump:wall-crossing}\ref{es:assump:it:tau-circ}) Let
  $\beta_1,\beta_2 \in C(\cat{A})$ such that $\rank(\beta_1+\beta_2) >
  0$. If $\rank(\beta_i) = 0$ then $\mu^H(\beta_i) >
  \mu^H(\beta_1+\beta_2)$ for $i=1,2$ and any $H \in \NS_{\bR}(S)$. So
  the hypothesis $\mu^H(\beta_1) = \mu^H(\beta_2)$, which therefore
  also equals $\mu^H(\beta_1+\beta_2)$, implies $\rank(\beta_i) > 0$,
  i.e. $\beta_i \in C(\cat{A})_{\pe}$. The same holds for $\tau^H$ in
  place of $\mu^H$. Finally, since $\fN_{\alpha,\cL} = \fM_\alpha$ for
  all $\alpha \in C_+(S)$, the desired Cartesian square is trivially
  Cartesian.

  (\ref{es:assump:wall-crossing}\ref{es:assump:it:lambda}) Since
  $\beta \in R_\alpha$, we know $\rank \beta, \rank \alpha-\beta > 0$.
  If $(\tau, \mathring\tau) = (\mu^H, \tau^H)$ then take
  $\lambda(\beta) \coloneqq r_\alpha^H P_\beta^H(N) - r_\beta^H
  P_\alpha^H(N)$ for $N \gg 0$. Otherwise, if $(\tau, \mathring\tau) =
  (\mu^H, \mu^{\mathring H})$ then take $\lambda(\beta) \coloneqq
  r_\alpha^{\mathring H} \bar P_\beta^{\mathring H}(N) -
  r_\beta^{\mathring H} \bar P_\alpha^{\mathring H}(N)$ for $N \gg 0$.
  In both cases, the desired property for $\lambda$ follows because
  $\mathring\tau(\beta) > \mathring\tau(\alpha-\beta)$ if and only if
  $\mathring\tau(\beta) > \mathring\tau(\alpha)$ if and only if
  $\mathring\tau(\beta)(N) > \mathring\tau(\alpha)(N)$.
\end{proof}

\subsubsection{}
\label{sec:VW-proof-H1-H2-nonzero}

\begin{proof}[Proof of Theorems~\ref{thm:VW-sst-invariants} and \ref{thm:VW-WCF} when $h^1(\cO_S) \neq 0$ or $h^2(\cO_S) \neq 0$.]
  In this case,
  Assumption~\ref{es:assump:semistable-invariants}\ref{es:assump:it:Bpe-sst-summands}
  fails because the Cartesian square
  \[ \begin{tikzcd}
    \fZ_{\beta_1,\beta_2} \ar{d}{\Phi_{\beta_1,\beta_2}} \ar[hookrightarrow]{r} & \fM_{\beta_1} \times \fM_{\beta_2} \ar{d}{\Phi_{\beta_1,\beta_2}} \\
    \fN_{\beta} \ar[hookrightarrow]{r} & \fM_{\beta}
  \end{tikzcd} \]
  defines a locus
  \[ \fZ_{\beta_1,\beta_2} = \left\{\begin{array}{c} {}[(\bar\cE_1, \phi_1)] \in \fM_{\beta_1} \\ {}[(\bar\cE_2, \phi_2)] \in \fM_{\beta_2} \end{array} : \det(\bar\cE_1)\det(\bar\cE_2) = \cL(\beta), \, \tr(\phi_1) + \tr(\phi_2) = 0\right\} \]
  which is not isomorphic to $\fN_{\beta_1,\cL} \times
  \fN_{\beta_2,\cL}$. Rather, by taking $(\det(\bar\cE_1),
  \tr(\phi_1))$, it admits a map
  \[ \det\nolimits_1 \times \tr_1\colon \fZ_{\beta_1,\beta_2} \to \fPic(S) \times H^0(\cK_S) \]
  and it is the fiber over $(\cL(\beta_1), 0)$ which is isomorphic to
  $\fN_{\beta_1,\cL} \times \fN_{\beta_2,\cL}$.

  Despite the failure of
  Assumption~\ref{es:assump:semistable-invariants}\ref{es:assump:it:Bpe-sst-summands},
  we can analyze the {\it proof} of
  Theorems~\ref{es:thm:sst-invariants} and \ref{es:thm:wcf}, in
  particular the structure of each ``simple'' wall-crossing step, in
  \S\ref{sec:pairs-master-space-localization} and
  \S\ref{sec:horizontal-wc} respectively. In both steps, if we run the
  machinery over $\fM$ instead of over $\fN$, we have loci
  \[ Z^{\fM}_{\beta_1,\beta_2} \xrightarrow{\sim} \tilde M_{\beta_1} \times \tilde M_{\beta_2} \]
  in some master space over $\fM_\beta$, where $\tilde M_{\beta_i} \to
  \fM_{\beta_i}^\pl$ are some stable=semistable loci in auxiliary
  framed stacks for $\fM_{\beta_i}$; see
  \eqref{eq:pairs-master-space-complicated-locus} and
  \eqref{wc:eq:master-space-complicated-locus}. These are the loci
  which contribute the Lie bracket terms. In reality, we want to run
  the machinery over $\fN$, and use
  Assumption~\ref{es:assump:semistable-invariants}\ref{es:assump:it:Bpe-sst-summands}
  and base change along $\fN_\beta^\pl \hookrightarrow \fM_\beta^\pl$
  to conclude a similar isomorphism for the loci
  \[ Z^{\fN}_{\beta_1,\beta_2} \xrightarrow[?]{\sim} \tilde N_{\beta_1} \times \tilde N_{\beta_2} \]
  in the actually-desired master space over $\fN_\beta$. But it is
  easy to see that there is a Cartesian square
  \[ \begin{tikzcd}
    Z^{\fN}_{\beta_1,\beta_2} \ar{d}{\pi} \ar[hookrightarrow]{r} & Z^{\fM}_{\beta_1,\beta_2} \ar{d} \\
    \fZ_{\beta_1,\beta_2} \ar[hookrightarrow]{r} & \fM_{\beta_1} \times \fM_{\beta_2}
  \end{tikzcd} \]
  and only the fiber over $(\det_1 \times \tr_1) \circ \pi$ over
  $(\cL(\beta_1), 0)$ is the desired $\tilde N_{\beta_1} \times \tilde
  N_{\beta_2}$. So the arguments in
  \S\ref{sec:pairs-master-space-localization} and
  \S\ref{sec:horizontal-wc} cannot proceed as written.

  We will instead show that the contribution from
  $Z^{\fN}_{\alpha_1,\alpha_2}$ must vanish (under certain
  circumstances). In other words, each ``simple'' wall-crossing step
  becomes trivial (under certain circumstances), with the result that
  the overall wall-crossing also becomes trivial. This will be the
  operational version of the arguments in \cite[\S 5.12]{liu_ss_vw}.

\subsubsection{}
\label{sec:VW-proof-end}

  Observe that the natural obstruction theory on
  $\fZ_{\beta_1,\beta_2}$ contains
  \[ R\Gamma(\cO_S)^\vee[-1] \oplus \kappa \otimes R\Gamma(\cK_S)^\vee \]
  as a summand. This is because it is constructed by the ``symmetric
  reduction'' procedure of \S\ref{sec:VW-TT-obstruction-theory} from
  the external direct sum of the $\kappa$-symmetric obstruction
  theories on $\fM_{\beta_1}$ and $\fM_{\beta_2}$, each of which
  contains $R\Gamma(\cO_S)^\vee[-1] \oplus \kappa \otimes
  R\Gamma(\cK_S)^\vee$ as a summand, and the ``symmetric reduction''
  removes only one (diagonal) copy. In particular, we have
  $h^2(\cO_S)$ genuine cosections, and $h^1(\cK_S) = h^1(\cO_S)$
  cosections which are ``$\kappa$-twisted'', i.e. map to $\kappa
  \otimes \cO$ instead of $\cO$.

  Suppose $h^2(\cO_S) \neq 0$. By
  Lemma~\ref{lem:lifting-Obs-cosections}, the $h^2(\cO_S)$ genuine
  cosections on $\fZ_{\beta_1,\beta_2}$ can be lifted to $h^2(\cO_S)$
  surjective cosections of the obstruction sheaf of
  $Z^{\fN}_{\beta_1,\beta_2}$. By cosection localization, the virtual
  cycle on $Z^{\fN}_{\beta_1,\beta_2}$ vanishes. Thus its contribution
  to each ``simple'' wall-crossing step, e.g.
  \eqref{eq:pairs-master-term-3}, vanishes as well. We conclude that
  \eqref{thm:sst-invariants} and \eqref{thm:wcf}, applied to our VW
  setting, contain only their $n=1$ terms.

  Suppose $h^1(\cO_S) \neq 0$. Lemma~\ref{lem:lifting-Obs-cosections}
  continues to apply to the $h^1(\cO_S)$ ``$\kappa$-twisted''
  cosections on $\fZ_{\beta_1,\beta_2}$. The conclusion is that the
  virtual cycle on $Z^{\fN}_{\beta_1,\beta_2}$ always contains a
  factor of
  \[ (1 - \kappa)^{h^1(\cO_S)} \]
  and therefore vanishes modulo $(1 - \kappa)^{h^1(\cO_S)}$. This
  statement makes sense because there are never any poles at $\kappa =
  1$ --- see Remark~\ref{rem:VW-refinement} --- and therefore we may
  put a filtration on such numerical invariants in
  $K_{\tilde\sT}(\pt)$ by their order of vanishing at $\kappa = 1$. We
  conclude that \eqref{thm:sst-invariants} and \eqref{thm:wcf},
  applied to our VW setting modulo $(1 - \kappa)^{h^1(\cO_S)}$,
  contain only their $n=1$ terms.
\end{proof}

\subsubsection{}

\begin{remark}
  For the $h^2(\cO_S) \neq 0$ case, we would like to use the setup of
  \S\ref{sec:reduced-obstruction-theories}, but the content there is
  not directly applicable because the cosections are ``on $\fM$''
  instead of ``on $\fN$''. Regardless, in the language there, the idea
  is that there are $o_\alpha = h^2(\cO_S)$ cosections for every
  $\alpha$, so in particular $o_{\alpha_1} + o_{\alpha_2} >
  o_{\alpha_1+\alpha_2}$, and therefore \eqref{thm:red-sst-invariants}
  and \eqref{thm:red-wcf} contain only their $n=1$ terms.

  For the $h^1(\cO_S) \neq 0$ case, we would like to use the same
  argument (\S\ref{vw:sec:M-stack}) that showed the vanishing of
  $U(r)$ refined VW invariants. This would at least prove the genuine
  vanishing of the Lie bracket whenever it is applied to a
  $\Pic_0(S)$-equivariant class. However, we do not see a $\Pic_0(S)$
  action on $Z^{\fN}_{\beta_1,\beta_2}$ which is compatible with the
  map $\det_1\colon Z^{\fN}_{\beta_1,\beta_2} \to r\!\Pic_{c_1}(S)$.
\end{remark}

\phantomsection
\addcontentsline{toc}{section}{References}

\begin{small}
\bibliographystyle{alpha}
\bibliography{paper}
\end{small}

\end{document}